\documentclass[a4paper,12pt]{article}

\usepackage{amscd}
\usepackage{amssymb,amsfonts,amsmath,amsthm}
\usepackage{verbatim}

\usepackage{amsmath}
\usepackage{amsthm}
\usepackage{amssymb}

\usepackage{latexsym}
\usepackage{array}
\usepackage{enumerate}

\numberwithin{equation}{section}

\usepackage{amsmath,amssymb,mathrsfs,amsthm}
\usepackage[all]{xy}
\usepackage{graphicx}
\usepackage{ulem}
\usepackage{ascmac}
\usepackage{mathtools}

\newcommand{\rHom}{{\bfR}\mathrm{Hom}}
\newcommand{\BDC}{{\mathbf{D}}^{\mathrm{b}}}

\newcommand{\Mod}{\mathrm{Mod}}
\newcommand{\Hom}{\mathrm{Hom}}
\newcommand{\shom}{{\mathcal{H}}om}
\newcommand{\ihom}{\mathcal{I}hom}

\newcommand{\CC}{\mathbb{C}}

\newcommand{\RR}{\mathbb{R}}

\newcommand{\ZZ}{\mathbb{Z}}
\newcommand{\D}{\mathcal{D}}

\newcommand{\F}{\mathcal{F}}
\newcommand{\G}{\mathcal{G}}
\newcommand{\M}{\mathcal{M}}
\newcommand{\N}{\mathcal{N}}

\newcommand{\R}{\mathcal{R}}

\newcommand{\CP}{\mathcal{P}}
\newcommand{\PP}{{\mathbb P}}
\newcommand{\scrS}{\mathscr{S}}
\newcommand{\calS}{\mathcal{S}}

\newcommand{\dist}{{\rm dist}}

\newcommand{\LL}{{\mathbb L}}

\renewcommand{\(}{\left(}
\renewcommand{\)}{\right)}

\newcommand{\Ker}{\operatorname{Ker}}
\newcommand{\Coker}{\operatorname{Coker}}
\newcommand{\Image}{\operatorname{Im}}

\newcommand{\id}{{\rm id}}

\newcommand{\Sol}{{\rm Sol}}

\newcommand{\sub}{{\rm sub}}

\newcommand{\Db}{\D b}

\newcommand{\tl}[1]{\widetilde{#1}}

\newcommand{\simto}{\overset{\sim}{\longrightarrow}}

\newcommand{\op}{{\mbox{\scriptsize op}}}
\newcommand{\SD}{\mathcal{D}}

\newcommand{\SO}{\mathcal{O}}
\newcommand{\Ob}{\SO b}
\newcommand{\SA}{\mathcal{A}}
\newcommand{\SM}{\mathcal{M}}

\newcommand{\SL}{\mathcal{L}}
\newcommand{\SK}{\mathcal{K}}
\newcommand{\SE}{\mathcal{E}}
\newcommand{\SF}{\mathcal{F}}
\newcommand{\SG}{\mathcal{G}}
\newcommand{\SH}{\mathcal{H}}

\newcommand{\SC}{\mathcal{C}}
\newcommand{\U}{\mathcal{U}}
\newcommand{\Cov}{\SC ov}
\newcommand{\Op}{\SO p}

\newcommand{\BDCcoh}{{\mathbf{D}}^{\mathrm{b}}_{\mbox{\rm \scriptsize coh}}}

\newcommand{\BDChol}{{\mathbf{D}}^{\mathrm{b}}_{\mbox{\rm \scriptsize hol}}}
\newcommand{\BDCrh}{{\mathbf{D}}^{\mathrm{b}}_{\mbox{\rm \scriptsize rh}}}

\newcommand{\DD}{\mathbb{D}}

\newcommand{\Dotimes}{\overset{D}{\otimes}}

\newcommand{\Potimes}{\overset{+}{\otimes}}

\newcommand{\rhom}{{\bfR}{\mathcal{H}}om}
\newcommand{\rihom}{{\bfR}{\mathcal{I}}hom}
\newcommand{\Prihom}{{\bfR}{\mathcal{I}}hom^+}
\newcommand{\Prihomsub}{{\bfR}{\mathcal{I}}hom^{+, \sub}}

\newcommand{\I}{{\rm I}}
\newcommand{\che}[1]{\check{#1}}
\newcommand{\var}[1]{\overline{#1}}
\newcommand{\BEC}{{\mathbf{E}}^{\mathrm{b}}}

\newcommand{\Q}{\mathbf{Q}}
\newcommand{\q}{\mathbf{q}}
\newcommand{\bfr}{\mathbf{r}}
\newcommand{\bfl}{\mathbf{l}}
\newcommand{\EE}{\mathbb{E}}

\newcommand{\T}{{\mathsf{T}}}
\newcommand{\bfR}{\mathbf{R}}
\newcommand{\bfL}{\mathbf{L}}
\newcommand{\bfD}{\mathbf{D}}

\newcommand{\rmR}{{\rm R}}
\newcommand{\rmE}{{\rm E}}
\newcommand{\rmD}{{\rm D}}
\newcommand{\rmt}{{\rm t}}
\newcommand{\bfE}{\mathbf{E}}

\newcommand{\sh}{{\rm sh}}
\newcommand{\RH}{{\rm RH}}

\newcommand{\pt}{{\rm pt}}

\newtheorem{theorem}{Theorem}[section]

\newtheorem{corollary}[theorem]{Corollary}
\newtheorem{lemma}[theorem]{Lemma}

\newtheorem{proposition}[theorem]{Proposition}

\theoremstyle{definition}
\newtheorem{definition}[theorem]{Definition}
\theoremstyle{remark}
\newtheorem{remark}[theorem]{\sc Remark}
\newtheorem{example}[theorem]{\sc Example}

\title{Note on Relation between Enhanced Ind-Sheaves and Enhanced Subanalytic Sheaves}

\author{Yohei ITO\footnote{Department of Mathematics, Faculty of Science Division II, Tokyo University of Science, 1-3, Kagurazaka, Shinjuku-ku, Tokyo, 162-8601, Japan. E-mail: yitoh@rs.tus.ac.jp }}

\date{}

\begin{document}
\maketitle

\begin{abstract}
In this paper, 
we will explain a relation between \cite[Thm.\:9.5.3]{DK16} and \cite[Thm.\:6.3]{Kas16}.
For this purpose, we will prove that there exists a fully faithful functor
from the triangulated category of enhanced subanalytic sheaves to the one of enhanced ind-sheaves.
\end{abstract}

\section{Introduction}
In \cite{KS01}, M.\:Kashiwara and P.\:Schapira introduced the notion of ind-sheaves and subanalytic sheaves
to treat ``sheaves" of functions with tempered growth conditions.
Ind-sheaves are defined as ind-objects of the category of sheaves of vector spaces with compact support.
Subanalytic sheaves are defined as sheaves on subanalytic sites.
Moreover, the authors proved that there exists a fully faithful functor
from the category of subanalytic sheaves to the category of ind-sheaves,
and its essentially image is equal to the category of ind-objects of $\RR$-constructible sheaves with compact support.


After a groundbreaking development in the theory of irregular meromorphic connections
by K.\:S.\:Kedlaya \cite{Ked10, Ked11} and T.\:Mochizuki \cite{Mochi09, Mochi11},
A.\:D'Agnolo and M.\:Kashiwara introduced the notion of enhanced ind-sheaves extending the notion of ind-sheaves
and established the Riemann--Hilbert correspondence for analytic irregular holonomic $\D$-modules
in \cite{DK16} as below (see \cite{Ito21} for the algebraic case).
Let $X$ be a complex manifold.
Then there exists a fully faithful functor $\Sol^\rmE_X$ which is called enhanced solution functor
(see \cite[Def.\:9.1.1]{DK16}, also Definition \ref{def3.33} for the definition)
from the full triangulated subcategory $\BDChol(\D_X)$
of the derived category of $\D_X$-modules
consisting of objects with holonomic cohomologies 
to the triangulated category $\BEC_{\RR-c}(\I\CC_X)$ of $\RR$-constructible enhanced ind-sheaves:
\begin{align}\label{1}
\Sol_X^{\rmE} : \BDChol(\D_X)^{\op}\hookrightarrow\BEC_{\RR-c}(\I\CC_X).
\end{align}
Moreover, 
T.\:Mochizuki characterized its essentially image by the curve test \cite[Thm.\:12.1]{Mochi16}.
In \cite{Ito20}, the author defined $\CC$-constructability for enhanced ind-sheaves
and proved that they are nothing but objects of its essentially image.
Namely, we obtain an equivalence of categories
between the triangulated category $\BDChol(\D_X)$
and the one $\BEC_{\CC-c}(\I\CC_X)$ of $\CC$-constructible enhanced ind-sheaves:
\[\Sol_X^{\rmE} \colon \BDChol(\D_X)^{\op}\simto \BEC_{\CC-c}(\I\CC_X).\]

At the 16th Takagi Lectures\footnote{
The 16th Takagi Lectures took place at Graduate School of Mathematical Sciences,
The University of Tokyo, on November 28 and 29, 2015.},
M.\:Kashiwara explained a similar result of (\ref{1})
by using ``enhanced subanalytic sheaves" instead of enhanced ind-sheaves as below.
We denote by $\BDC(\CC_{X\times\RR_\infty}^{\sub})$ the derived category of subanalytic sheaves
on a bordered space $X\times\RR_\infty$, see \S \ref{subsec3.1} for the definition. 
Then there exists a fully faithful functor $\Sol_X^{\T}$ (see \cite[\S 5.4]{Kas16}, also Definition \ref{def3.36} for the definition)
from $\BDChol(\D_X)$ to $\BDC(\CC_{X\times\RR_\infty}^{\sub})$:
\begin{align}\label{2}
\Sol_X^{\T} : \BDChol(\D_X)^{\op}\hookrightarrow\BDC(\CC^\sub_{X\times \RR_\infty}).
\end{align}

In this paper, we will explain a relation between (\ref{1}) and (\ref{2}).
For this purpose, we will prove that there exists a fully faithful functor
from the triangulated category of enhanced subanalytic sheaves to the one of enhanced ind-sheaves.
Although it may be known by experts, it is not in the literature to our knowledge.
The main results of this paper are
Theorems \ref{main1}, \ref{main2}, \ref{main3} and \ref{main4}.
One can summarize these results in the following commutative diagram:
\[\xymatrix@M=7pt@R=35pt@C=60pt{
{}&{}&\BDC(\CC_{X\times\RR_\infty}^\sub) & {}\\
\BDChol(\D_X)^\op\ar@{^{(}->}[r]_-{\Sol_X^{\rmE, \sub}}
\ar@{^{(}->}[rru]^-{\Sol_X^{\T, \sub}(\cdot)[1]}\ar@{^{(}->}[rd]_-{\Sol_X^{\rmE}}
 & \BEC_{\RR-c}(\CC_X^\sub)\ar@{}[r]|-{\text{\large $\subset$}}
 \ar@<-1.0ex>@{->}[d]_-{I_X^\rmE}\ar@{}[d]|-\wr
  & \BEC(\CC_X^\sub)\ar@{^{(}->}[u]_-{\bfR_X^{\rmE, \sub}}
  \ar@<-1.0ex>@{^{(}->}[rd]_-{I_X^\rmE}
   \ar@<-1.0ex>@{->}[d]_-{I_X^\rmE}\ar@{}[d]|-\wr\\
{}&\BEC_{\RR-c}(\I\CC_X)\ar@{}[r]|-{\text{\large $\subset$}}
\ar@<-1.0ex>@{->}[u]_-{J_X^\rmE}
&\BEC_{\I\RR-c}(\I\CC_X)\ar@{}[r]|-{\text{\large $\subset$}}
\ar@<-1.0ex>@{->}[u]_-{J_X^\rmE}
& \BEC(\I\CC_X).\ar@<-1.0ex>@{->}[lu]_-{J_X^\rmE}
}\]
See \S\ref{subsec2.7} for the definition of $\BEC(\I\CC_X)$,
\S\ref{subsec3.1} for the definition of $\BDC(\CC_{X\times\RR_\infty}^\sub)$,
\S\ref{subsec3.3} for the definitions of $\BEC(\CC_X^\sub), \bfR_X^{\rmE, \sub}$,
\S\ref{subsec3.4} for the definitions of $\BEC_{\I\RR-c}(\CC_X^\sub), I_X^\rmE, J_X^\rmE, \BEC_{\RR-c}(\I\CC_X)$,
Definition \ref{def3.21} for $\BEC_{\RR-c}(\CC_X^\sub)$,
Definition \ref{def3.33} for $\Sol_X^\rmE$,
Definition \ref{def3.36} for $\Sol_X^{\T, \sub}$,
Definition \ref{def3.37} for $\Sol_X^{\rmE, \sub}$.

\section*{Acknowledgement}
I would like to thank Dr. Tauchi of Kyushu University
for many discussions and giving many comments.

This work was supported by Grant-in-Aid for Research Activity Start-up (No. 21K20335), 
Japan Society for the Promotion of Science.

\newpage

\section{Preliminary Notions and Results}\label{sec-2}
In this section,
we briefly recall some basic notions
and results which will be used in this paper. 

\subsection{Subanalytic Subsets}
The theory of semi-analytic and subanalytic sets originates
in the work of S.\:{\L}ojasiewicz \cite{Loj59, Loj64a, Loj64b}
and has been elaborated by A.\:M.\:Gabrielov \cite{Gab68}, H.\:Hironaka \cite{Hiro73a, Hiro73b}
and R.\:M.\:Hardt \cite{Har75, Har77} for subanalytic sets. 
See also \cite{BM88}.
In this subsection, we just recall the definition of semi-analytic subsets and subanalytic subsets,
and some properties.

Let $M$ be a real analytic manifold and denote by $\SA_M$ the sheaf of rings of real analytic functions.
For an open subset $U$ of $M$, we denote by $\scrS(\SA_M(U))$
the family of subsets of $M$ of the form $$\bigcup_{i=1}^p\bigcap_{j=1}^qX_{ij},$$
where each $X_{ij}$ is either $\{x\in U\ | f_{ij}(x) = 0, \ f_{ij}\in\SA_M(U)\} \text{ or } \{x\in U\ | f_{ij}(x) > 0,\ f_{ij}\in\SA_M(U)\}.$

\begin{definition}
A subset $A$ of $M$ is called semi-analytic
if for any $x\in M$ there exist an open neighborhood $U$ of $x$
such that $A\cap U \in \scrS(\SA_M(U))$.
\end{definition}

Note that for any open subset $U$ of $M$,
 $\scrS(\SA_M(U))$ is stable under finite intersection, finite union and complement.
 Hence, a finite union and a finite intersection of semi-analytic subsets are semi-analytic
 and the complement of a semi-analytic subset is also semi-analytic.
 Furthermore, the closure and the interior of a semi-analytic subset are semi-analytic,
see \cite[Cor.\:2.8]{BM88} for the details.
Although, the inverse image of a semi-analytic subset by a morphism of real analytic manifolds is semi-analytic, 
the operation of direct image by a proper morphism of real analytic manifolds does not in general
preserve semi-analyticity.
 However, the class of subanalytic subsets is closed with respect to those operations as below.
 Subanalytic subsets are ``locally proper projections of relatively compact semi-analytic sets".

 \begin{definition}
A subset $S$ of $M$ is called subanalytic
if for any $x\in M$ there exist an open neighborhood $U$ of $x$,
a real analytic manifold $N$ and a relatively compact semi-analytic subset $A$ of $M\times N$
such that $S\cap U = {\rm pr}_1(A)$, where ${\rm pr}_1\colon M\times N\to M$ is the first projection.
 \end{definition}
 
From the basic properties of semi-analytic subsets,
a finite union and a finite intersection of subanalytic subsets are subanalytic,
and the closure of a subanalytic subset is subanalytic.
Furthermore,  the complement (and thus the interior) of a subanalytic subset is also subanalytic,
see \cite[Thm.\:3.10]{BM88} for details.
Note also that the inverse image of a subanalytic subset by a morphism of real analytic manifolds is subanalytic.
Moreover, the direct image of a subanalytic subset by a proper morphism of real analytic manifolds is subanalytic,
see \cite[Prop.\:3.8]{Hiro73a} for details.



\subsection{Sheaves on Sites}\label{subsec2.2}
The theory of sheaves on topological spaces was created by J.\:Leray \cite{Ler50}
and this notion was extended to sheaves on sites by A.\:Grothendieck \cite{SGA4}.
A site is a category endowed with a ``Grothendieck topology" on it.
The Grothendieck topology was introduced by A.\:Grothendieck in order to have a cohomology theory on algebraic varieties.
In this subsection, we shall briefly recall the definition of sheaves on sites and some properties based on \cite[\S2]{KS01}.
See \cite[\S\S16, 17, 18]{KS06} for more general settings.

Let $\U$ be a universe.
A set is called $\U$-small if it is isomorphic to a set belonging to $\U$.
A small category\footnote{In \cite{KS06}, a category means a small category.}
 $\SC$ is called $\U$-category if for any objects $X, Y$ of $\SC$ the set $\Hom_{\SC}(X, Y)$ is $\U$-small.
If moreover the family $\Ob(\SC)$ of objects (a set in bigger universe) of $\SC$ is $\U$-small, then $\SC$ is called $\U$-small.

\begin{definition}\label{def2.3}
Let $\SC$ be a $\U$-small category admitting finite products and fiber products.
\begin{itemize}
\item[(1)]
For an object $U$ of $\SC$, we denote by $\SC_U$ the category of arrows $V\to U$.
Namely, the category $\SC_U$ is given by
\begin{align*}
\Ob(\SC_U) &:= 
\{(V, i_V)\ |\ V\in\Ob(\SC), i_V\in\Hom_{\SC}(V, U)\},\\
\Hom_{\SC_U}\big((V, i_V), (W, i_W)\big) &:=
\{\varphi\in\Hom_{\SC}(V, W)\ |\ i_W\circ\varphi = i_V\}.
\end{align*}
For simplicity, we sometimes write $V\to U$ instead of $(V, i_V)$.

\item[(2)]
For an object $(V, i_V)\in\Ob(\SC_U)$ and a subset $S\subset \Ob(\SC_U)$,
let us set $$V\times_U S := \{(V\times_U W, {\rm pr}_V)\ |\ W\in S\}\subset \Ob(\SC_V),$$
where ${\rm pr}_V\colon V\times_U W\to V$ is the projection.

\item[(3)]
For $S_1, S_2\subset \Ob(\SC_U)$,
$S_1$ is said to be a refinement of $S_2$
if for any $(V_1, i_{V_1})\in S_1$ there exist $(V_2, i_{V_2})\in S_2$
and a morphism $\varphi\in\Hom_{\SC}(V_1, V_2)$ such that $i_{V_2}\circ\varphi = i_{V_1}$.
In such a situation, we write $S_1\preceq S_2$.
\end{itemize}
\end{definition}
\begin{definition}\label{def-GT}
Let $\SC$ be a $\U$-small category admitting finite products and fiber products.
A Grothendieck topology on $\SC$ is the data associating to any $U\in \Ob(\SC)$
a family $\Cov(U)$ of subsets of $\Ob(\SC_U)$ satisfying the axioms
\begin{itemize}
\item[(GT1)]
$\{\id_U\}\in\Cov(U)$,

\item[(GT2)]
if $S_1\in\Cov(U)$ is a refinement of $S_2\subset \Ob(\SC_U)$,
then $S_2\in\Cov(U)$,

\item[(GT3)]
if $S\in\Cov(U)$, then $V\times_U S\in \Cov(V)$ for any $(V, i_V)\in\Ob(\SC_U)$,

\item[(GT4)]
if $S_1, S_2\subset\Ob(\SC_U)$, $S_1\in \Cov(U)$ and 
$V\times_U S_2\in \Cov(V)$ for any $(V, i_V)\in S_1$,
then $S_2\in \Cov(U)$.
 \end{itemize}
A subset $S\in\Cov(U)$ is called a covering of $U$.

A site $X$ is a $\U$-small category $\SC_X$\footnote{Remark that the category $\SC_X$ is different from the category $\SC_X$ in Definition \ref{def2.3} (1).}
 admits finite products and fiber products
and endowed with a Grothendieck topology on $\SC_X$.
\end{definition}

\begin{example}\label{ex2.5}
Let $X$ be a topological space and denote by $\Op_X$ the category of open subsets of $X$.
Namely, the category $\Op_X$ is given by
\begin{align*}
\Ob(\Op_X) := \{U\in\CP(X)\ |\ U\text{ is open }\},\hspace{17pt}
\Hom_{\Op_X}(U, V) :=
\begin{cases}
\ \{\pt\}\quad&(U\subset V),\\
\ \emptyset\quad &(\text{otherwise}),
\end{cases}
\end{align*}
where $\CP(X)$ is the power set of $X$.
Note that for any $U\in\Op_X$ we have $(\Op_X)_U = \Op_U$.
We can endow $\Op_X$ with the following Grothendieck topology:
a subset $S\subset \Ob((\Op_X)_U)$ is a covering of $U\in\Ob(\Op_X)$
if $U = \bigcup_{V\in S}V$.
We will keep the same symbol $X$ to denote its site.

\end{example}

Let $\Bbbk$ be a commutative unital ring and denote by $\Mod(\Bbbk)$ the category of $\Bbbk$-modules.

\begin{definition}
A presheaf of $\Bbbk$-modules on a site $X$ is nothing but a contravariant functor from $\SC_X$ to $\Mod(\Bbbk)$.
We denote by $s|_V$ the image of $s\in\SF(U)$
by the restriction morphism $\SF(U)\to\SF(V)$ associated a morphism $V\to U$ in $\SC_X$.

Let us denote by ${\rm Psh}(\Bbbk_X)$ the category of presheaves of $\Bbbk$-modules on a site $X$. 
\end{definition}
Note that the category ${\rm Psh}(\Bbbk_X)$ is abelian because the category $\Mod(\Bbbk)$ is abelian.

\begin{definition}
Let $X$ be a site.
A presheaf $\SF$ of $\Bbbk$-modules on $X$ is called separated
if for any $U\in\Ob(\SC_X)$ and any covering $S\in\Cov(U)$ of $U$,
the natural morphism
$$\SF(U) \to \ker\left(\prod_{V\in S}\SF(V)\rightrightarrows\prod_{V', V''\in S}\SF(V'\times_U V'')\right)$$
is a monomorphism.
Here, the two arrows are associated to $\SF(V')\to \SF(V'\times_UV'')$,
 and $\SF(V'')\to \SF(V'\times_UV'')$
and the kernel of the double arrow is the kernel of the difference of the two arrows.

If moreover the natural morphism is an isomorphism, a presheaf $\SF$ is a sheaf of $\Bbbk$-modules on $X$.
Let us denote by $\Mod(\Bbbk_X)$ the category of shaves of $\Bbbk$-modules on $X$.
\end{definition}

Note that the category $\Mod(\Bbbk_X)$ is a full additive subcategory of ${\rm Psh}(\Bbbk_X)$.
Furthermore, it is abelian as below.
To explain this, let us recall the sheaf associated with a presheaf.
Let $X$ be a site.
For a presheaf $\SF$ of $\Bbbk$-modules on $X$ and $U\in \Ob(\SC_X)$,
we set
$$\SF^+(U) :=
\varinjlim_{S\in\Cov(U) } \SF(S),$$
where $\SF(S) := \ker\left(\prod_{V\in S}\SF(V)\rightrightarrows\prod_{V', V''\in S}\SF(V'\times_U V'')\right)$.
Remark that the relation $\preceq$ is a pre-order on $\Cov(U)$
and hence $\Cov(U)$ inherits a structure of a small category.
Moreover, $\Cov(U)^\op$ is filtrant.
See \cite[p.\,19]{KS01} for the details.
Then we have the presheaf $$\SF^+\colon \SC_X^\op\to\Mod(\Bbbk),\hspace{7pt}U\mapsto\SF^+(U)$$
and a functor $$(\cdot)^+\colon{\rm Psh}(\Bbbk_X)\to {\rm Psh}(\Bbbk_X),\hspace{7pt}\SF\mapsto \SF^+.$$
\newpage
\begin{theorem}[{\cite[Thm.\:2.1.7]{KS01}}, see also {\cite[Thm.\:17.4.7 (i), (iii)]{KS06}} for (3), (4)]\label{thm2.8}
Let $X$ be a site.
\begin{itemize}
\item[\rm (1)]
The functor $(\cdot)^+\colon{\rm Psh}(\Bbbk_X)\to {\rm Psh}(\Bbbk_X)$ is left exact.

\item[\rm (2)]
For any $\SF\in{\rm Psh}(\Bbbk_X)$, $\SF^+$ is a separated presheaf.

\item[\rm (3)]
For any separated presheaf $\SF$, $\SF^+$ is a sheaf.

\item[\rm (4)]
The functor $(\cdot)^{++}\colon{\rm Psh}(\Bbbk_X)\to\Mod(\Bbbk_X)$ is a left adjoint
to the embedding functor $\Mod(\Bbbk_X)\to {\rm Psh}(\Bbbk_X)$.
Namely, for any $\SF\in{\rm Psh}(\Bbbk_X)$ and any $\SG\in\Mod(\Bbbk_X)$, we have
$$\Hom_{{\rm Psh}(\Bbbk_X)}(\SF, \SG) \simeq \Hom_{\Mod(\Bbbk_X)}(\SF^{++}, \SG).$$ 
\end{itemize}
\end{theorem}

For a presheaf $\SF\in{\rm Psh}(\Bbbk_X)$, the sheaf $\SF^{++}$ is called the sheaf associated with $\SF$.
Let us call the functor $(\cdot)^{++}$ the sheafification functor.
Hence, we have:
\begin{theorem}[{\cite[Thm.\:2.1.10]{KS01}},
see also {\cite[Thms.\:17.4.9, 17.4.7 (iv), 18.1.6 (v)]{KS06}}]\label{thm2.9}
Let $X$ be a site.
\begin{itemize}
\item[\rm (1)]
The category $\Mod(\Bbbk_X)$ admits projective limits.
More precisely, for any projective system $\{\SF_i\}_{i\in I}$ of sheaves,
its projective limit in ${\rm Psh}(\Bbbk_X)$ is a sheaf and is a projective limit in $\Mod(\Bbbk_X)$.

\item[\rm (2)]
The category $\Mod(\Bbbk_X)$ admits inductive limits.
More precisely, for any inductive system $\{\SF_i\}_{i\in I}$ of sheaves,
its inductive limit in $\Mod(\Bbbk_X)$ is the sheaf associated with its inductive limit in ${\rm Psh}(\Bbbk_X)$.

\item[\rm (3)]
The category $\Mod(\Bbbk_X)$ is abelian.

\item[\rm (4)]
The embedding functor $\Mod(\Bbbk_X)\to {\rm Psh}(\Bbbk_X)$ is fully faithful and left exact,
and the functor $(\cdot)^{++}\colon{\rm Psh}(\Bbbk_X)\to\Mod(\Bbbk_X)$ is exact.

\item[\rm (5)]
Filtrant inductive limits in $\Mod(\Bbbk_X)$ are exact.

\item[\rm (6)]
The $\U$-category $\Mod(\Bbbk_X)$ admits enough injective.
\end{itemize}
\end{theorem}

For $M\in\Mod(\Bbbk)$, we denote by $M_X$ the sheaf associated with the presheaf $U\mapsto M$
and call $M_X$ the constant sheaf with stalk $M$.  

There are many operations for sheaves on sites as similar to classical sheaves.
\begin{definition}
Let $X$ and $Y$ be sites.
A morphism $f\colon X\to Y$ of sites from $X$ to $Y$
is a functor ${}^t\!f\colon \SC_Y\to \SC_X$ such that 
it commutes with fiber products
and if for any $V\in \Ob(\SC_Y)$ and any $S\in\Cov(V)$ then ${}^t\!f(S)\in\Cov({}^t\!f(V))$.
\end{definition}
For a morphism $f\colon X\to Y$ of sites associated with a functor ${}^t\!f\colon \SC_Y\to\SC_X$
and $\SF\in{\rm Psh}(\Bbbk_X), \SG\in{\rm Psh}(\Bbbk_Y)$,
we set
\begin{align*}
(f_\ast\SF)(V) &:= \SF({}^t\!f(V))\hspace{50pt} \text{ for any } V\in\SC_Y,\\
(f^{-1}_{\rm pre})\SG(U) &:= \varinjlim_{U\to {}^t\!f(V)}\SG(V)\hspace{50pt} \text{ for any } U\in\SC_X.
\end{align*} 
Then we have functors:
\begin{align*}
f_\ast&\colon {\rm Psh}(\Bbbk_X)\to {\rm Psh}(\Bbbk_Y),\\
f^{-1}_{\rm pre} &\colon {\rm Psh}(\Bbbk_Y)\to{\rm Psh}(\Bbbk_X).
\end{align*}
Note that if $\SF\in\Mod(\Bbbk_X)$ then the presheaf $f_\ast\SF$ is a sheaf,
see \cite[Prop.\:17.5.1]{KS06}.
\begin{definition}\label{def2.11}
Let $f\colon X\to Y$ be a morphism of sites.
The direct image functor $f_\ast$ and inverse image functor $f^{-1}$ are defined by
\begin{align*}
f_\ast&\colon\Mod(\Bbbk_X)\to\Mod(\Bbbk_Y),\hspace{7pt}
\SF\mapsto f_\ast\SF,\\
f^{-1}&\colon\Mod(\Bbbk_Y)\to\Mod(\Bbbk_X),\hspace{7pt}
\SF\mapsto \left(f^{-1}_{\rm pre}\SF\right)^{++}.
\end{align*}
\end{definition}

They have the following properties:
\begin{proposition}[{\cite[Thm.\:2.2.1]{KS01}}, see also {\cite[Thm.\:17.5.2]{KS06}}]\label{prop2.12}
Let $f\colon X\to Y$ be a morphism of sites.
\begin{itemize}
\item[\rm (1)]
The inverse image functor $f^{-1}\colon\Mod(\Bbbk_Y)\to\Mod(\Bbbk_X)$ is left adjoint to 
the direct image functor $f_\ast\colon\Mod(\Bbbk_X)\to\Mod(\Bbbk_Y)$.
Namely, for any $\SF\in\Mod(\Bbbk_X)$ and any $\SG\in\Mod(\Bbbk_Y)$ we have
$$\Hom_{\Mod(\Bbbk_X)}(f^{-1}\SF, \SG)\simeq \Hom_{\Mod(\Bbbk_Y)}(\SF, f_\ast\SG).$$

\item[\rm (2)]
The direct image functor $f_\ast$ is left exact and commutes with small projective limits.

\item[\rm (3)]
The inverse image functor $f^{-1}$ is exact and commutes with small inductive limits.
\end{itemize}

\end{proposition}

Note that for a site $X$ and $U\in\Ob(\SC_X)$,
we can endow $(\SC_X)_U$ with the following Grothendieck topology:
for $V\in \Ob((\SC_X)_U)$, a subset of $\Ob((\SC_X)_V)$ is a covering of $V$ if it is a covering in $\SC_X$.
We denote by $U$ its site and by $i_U\colon U\to X$ the natural morphism of sites from $U$ to $X$
associated with a functor ${}^ti_U\colon \SC_X\to(\SC_X)_U,\hspace{5pt}V\mapsto (U\times V, {\rm pr}_1)$,
where ${\rm pr}_1\colon U\times V\to U$ is the first projection.
On the other hand,
we have the morphism $j_U\colon X\to U$ of sites from $X$ to $U$ associated
with a functor ${}^tj_U\colon(\SC_X)_U\to \SC_X,\ V\to V$.
Hence we get functors:
\[\xymatrix@M=7pt@C=45pt{
\Mod(\Bbbk_U)\ar@<0.7ex>@{->}[r]^-{j_U^{-1}}
&
\Mod(\Bbbk_X)
\ar@<0.7ex>@{->}[r]^-{i_U^{-1}}
\ar@<0.7ex>@{->}[l]^-{j_{U\ast}}
&
\Mod(\Bbbk_U)
\ar@<0.7ex>@{->}[l]^-{i_{U\ast}}.
}\]

\begin{definition}
Let $X$ be a site and $U\in\Ob(\SC_X)$.
\begin{itemize}
\item[(1)]
The exact functor $i_U^{-1}\colon \Mod(\Bbbk_X)\to\Mod(\Bbbk_U)$
is called the restriction functor to $U$.
For simplicity, we some times write $\SF|_U$ instead of $i_U^{-1}\SF$ for $\SF\in\Mod(\Bbbk_X)$. 

\item[(2)]
The exact functor $j_U^{-1}\colon \Mod(\Bbbk_U)\to\Mod(\Bbbk_X)$ is called the extension functor from $U$.
Let us set $i_{U!} := j_U^{-1}$.

\item[(3)]
We set $\Gamma_U\SF := i_{U\ast}i_U^{-1}\SF$ and $\SF_U := i_{U!}i_U^{-1}\SF$
for $\SF\in\Mod(\Bbbk_X)$. 
\end{itemize}
\end{definition}

Clearly, the functor $\Gamma_U$ is left exact and the functor $(\cdot)_U$ is exact,
and there exist a canonical morphism $\id\to \Gamma_U$ of functors
and an isomorphism $\Gamma(X;\ \cdot\ )\circ\Gamma_U \simeq \Gamma(U;\ \cdot\ )$.
Note that $j_{U\ast} = i_U^{-1}$ and hence the functor $i_{U!}$ is a left adjoint to the functor $i_U^{-1}$.
Namely, for any $\SF\in\Mod(\Bbbk_U)$ and any $\SG\in\Mod(\Bbbk_X)$ we have:
$$\Hom_{\Mod(\Bbbk_X)}(i_{U!}\SF, \SG)\simeq \Hom_{\Mod(\Bbbk_U)}(\SF, i_U^{-1}\SG).$$
See \cite[Prop.\:2.3.4]{KS01} for the details.
Moreover we have:
\begin{proposition}[{\cite[Prop.\:2.3.6]{KS01}}]
Let $X$ be a site and $U\in\Ob(\SC_X)$.
We assume that for any $V\in\Ob(\SC_X)$, $\Hom_{\SC_X}(V, U)$ has at most one element.
Then we have
\begin{itemize}
\item[\rm (1)]
$i_U^{-1}\circ i_{U\ast}\simto \id,\hspace{5pt}
\id \simto i_U^{-1}\circ i_{U!}$ 

\item[\rm (2)]
Functors $i_{U\ast}$ and $i_{U!}$ are fully faithful. 

\item[\rm (3)]
There exist a canonical morphism $(\cdot)_U\to\id$ of functors.
Moreover, for any $\SF\in\Mod(\Bbbk_X)$ the canonical morphism $\SF_U\to \SF$ is a monomorphism .
\end{itemize}
\end{proposition}

If sites $X$ and $Y$ have terminal objects which are denoted by same symbols,
and a morphism $f\colon X\to Y$ of sites associated
the functor ${}^t\!f\colon \SC_Y\to\SC_X$ satisfies ${}^t\!f(Y)=X$,
for any $V\in\SC_Y$ we have the following commutative diagrams:
\[\xymatrix@M=5pt@R=20pt@C=40pt{
X\ar@{->}[r]^-{f} & Y\\
U\ar@{->}[u]^-{i_U}\ar@{->}[r]_-{f|_U}  & V \ar@{->}[u]_-{i_V}
,}\hspace{17pt}
\xymatrix@M=5pt@R=20pt@C=40pt{
X\ar@{->}[r]^-{f}\ar@{->}[d]_-{j_U} & Y\ar@{->}[d]^-{j_V}\\
U\ar@{->}[r]_-{f|_U}  & V,}\]
where we set $U:={}^t\!f(V)$ and a morphism $f|_U\colon U\to V$ of sites is induced by $f$.
In this case, we have
$$i_{U!}\circ(f|_U)^{-1}\simeq f^{-1}\circ i_{V!}\ ,\hspace{20pt}
(f|_U)_\ast\circ i_U^{-1}\simeq i_V^{-1}\circ f_\ast\ .$$
By using the first one and the fact that $f^{-1}\Bbbk_Y\simeq \Bbbk_X$,
we have $f^{-1}(\Bbbk_{YV}) \simeq \Bbbk_{XU}$.

For $\SF, \SG\in{\rm Psh}(\Bbbk_X)$,
we set
\begin{align*}
\big(\shom_{\Bbbk_X}(\SF, \SG)\big)(U) &:=
\Hom_{\Mod(\Bbbk_U)}(\SF|_U, \SG|_U)\hspace{19pt} \text{ for any } U\in\SC_X,\\
\left(\SF\underset{\Bbbk_X}{\overset{\rm pre}{\otimes}}\SG\right)(U) &:=
\SF(U)\underset{\Bbbk}{\otimes}\SG(U)\hspace{39pt} \text{ for any } U\in\SC_X.
\end{align*}
Then we have bifunctors:
\begin{align*}
 \shom_{\Bbbk_X}(\cdot, \cdot) &\colon
 {\rm Psh}(\Bbbk_X)^\op\times{\rm Psh}(\Bbbk_X)\to{\rm Psh}(\Bbbk_X),\hspace{7pt}
 (\SF, \SG)\mapsto\shom_{\Bbbk_X}(\SF, \SG),\\
 \Hom_{\Mod(\Bbbk_X)}(\cdot, \cdot) &\colon
 {\rm Psh}(\Bbbk_X)^\op\times{\rm Psh}(\Bbbk_X)\to\Mod(\Bbbk),\hspace{7pt}
 (\SF, \SG)\mapsto\Hom_{\Mod(\Bbbk_X)}(\SF, \SG),\\
 (\cdot)\overset{\rm pre}{\underset{\Bbbk_X}{\otimes}}(\cdot) &\colon
 {\rm Psh}(\Bbbk_X)\times{\rm Psh}(\Bbbk_X)\to{\rm Psh}(\Bbbk_X),\hspace{7pt}
 (\SF, \SG)\mapsto\SF \overset{\rm pre}{\underset{\Bbbk_X}{\otimes}}\SG.
 \end{align*}

Note that if $\SF, \SG\in\Mod(\Bbbk_X)$ then the presheaf $\shom_{\Bbbk_X}(\SF, \SG)$ is a sheaf,
see \cite[Prop.\:17.7.1 (i)]{KS06}.

\begin{definition}
The internal hom functor $ \shom_{\Bbbk_X}$, the hom functor $\Hom_{\Mod(\Bbbk_X)}$
and the tensor product functor $\underset{\Bbbk_X}{\otimes}$ are defined by
\begin{align*}
 \shom_{\Bbbk_X}(\cdot, \cdot) &\colon
 \Mod(\Bbbk_X)^\op\times\Mod(\Bbbk_X)\to\Mod(\Bbbk_X),\hspace{7pt}
 (\SF, \SG)\mapsto\shom_{\Bbbk_X}(\SF, \SG),\\
  \Hom_{\Mod(\Bbbk_X)}(\cdot, \cdot) &\colon
 \Mod(\Bbbk_X)^\op\times\Mod(\Bbbk_X)\to\Mod(\Bbbk),\hspace{7pt}
 (\SF, \SG)\mapsto\Hom_{\Mod(\Bbbk_X)}(\SF, \SG),\\
 (\cdot)\underset{\Bbbk_X}{\otimes}(\cdot) &\colon
 \Mod(\Bbbk_X)\times\Mod(\Bbbk_X)\to\Mod(\Bbbk_X),\hspace{7pt}
 (\SF, \SG)\mapsto\left(\SF \overset{\rm pre}{\underset{\Bbbk_X}{\otimes}}\SG\right)^{++}.
 \end{align*}
 We sometimes write $\shom,\ \otimes$ instead of $\shom_{\Bbbk_X},\ \underset{\Bbbk_X}{\otimes}$, respectively.
\end{definition}
Remark that the internal hom functor is left exact in each variable
and the tensor product functor is right exact in each variable.
They have the following properties.
\begin{proposition}[{\cite[Props.\:2.4.2, 2.4.3]{KS01}},
see also {\cite[Prop.\:17.7.3, Thms.\:18.2.3, Lem.\:18.3.1]{KS06}}]
Let $f\colon X\to Y$ be a morphism of sites
and $\SF, \SF_1, \SF_2, \SH\in\Mod(\Bbbk_X),$ $\SG, \SG_1, \SG_2\in\Mod(\Bbbk_Y)$.
Then we have
\begin{itemize}
\item[\rm (1)]
for any $U\in\Ob(\SC_X)$, we have $i_U^{-1}\shom(\SF, \SG) \simeq \shom(i_U^{-1}\SF, i_U^{-1}\SG)$,

\item[\rm (2)]
$\shom_{\Bbbk_X}(\Bbbk_X, \SF) \simeq \SF$,

\item[\rm (3)]
$\Bbbk_X\otimes \SF \simeq \SF$,

\item[\rm (4)]
$\shom_{\Bbbk_X}(\SF_1\otimes \SF_2, \SH)\simeq \shom_{\Bbbk_X}(\SF_1, \shom_{\Bbbk_X}(\SF_2, \SH))$,

\item[\rm (5)]
$\shom_{\Bbbk_Y}(\SG, f_\ast\SF)\simto f_\ast\shom_{\Bbbk_X}(f^{-1}\SG, \SF)$,

\item[\rm (6)]
$f^{-1}(\SG_1\underset{\Bbbk_X}{\otimes}\SG_2)\simto f^{-1}\SG_1\underset{\Bbbk_X}{\otimes} f^{-1}\SG_2$.
\end{itemize}
\end{proposition}

Let us recall the notion of (sheaves of) $\R$-modules.
\begin{definition}
Let $X$ be a site and $\Bbbk$ a commutative unital ring.
A sheaf of $\Bbbk$-algebras is
 a sheaf $\R$ of $\Bbbk$-modules on a site $X$ such that
for any $U\in\Ob(\SC_X)$, $\R(U)$ is a $\Bbbk$-algebra and 
for any morphism $V\to U$ in $\SC_X$,
the restriction morphism $\R(U)\to\R(V)$ is a morphism of $\Bbbk$-algebras.

For sheaves $\R, \calS$ of $\Bbbk$-algebras,
a morphism of sheaves of $\Bbbk$-algebras from $\R$ to $\calS$
is a morphism $f\colon \R\to\calS$ of sheaves
such that for any $U\in\SC_X$ the morphism $f(U)\colon\R(U)\to\calS(U)$ is a morphism of $\Bbbk$-algebras.
\end{definition}

\begin{example}
The constant sheaf $\Bbbk_X$ on a site $X$ is a sheaf of  $\Bbbk$-algebras. 
\end{example}

A sheaf of $\Bbbk$-algebras is called a $\Bbbk_X$-algebra. 
A sheaf of $\ZZ$-algebras is simply called a sheaf of rings.
Moreover, for a $\Bbbk_X$-algebra $\R$, we denote $\R^\op$ the opposite $\Bbbk_X$-algebra
which is defined by $\R^\op(U) := \R(U)^\op$ for any $U\in\SC_X$.
Here, $\R(U)^\op$ is the opposite ring of $\R(U)$.

\newpage
\begin{definition}
Let $X$ be a site and $\R$ a $\Bbbk_X$-algebra.

A presheaf of $\R$-modules is 
a presheaf $\SF\in{\rm Psh}(\Bbbk_X)$ such that 
for any $U\in\SC_X$, $\SF(U)$ is a left $\R(U)$-module and 
for any morphism $V\to U$ in $\SC_X$,
the restriction morphism $\SF(U)\to\SF(V)$ commutes with the action $\R$,
that is, $(r\cdot s)|_V = r|_V\cdot s|_V$ for any $s\in\SF(U)$ and $r\in\R(U)$.

For presheaves $\M, \N$ of $\R$-modules,
a morphism of presheaves of $\R$-modules from $\M$ to $\N$
is a morphism $\varphi\colon \M\to\N$ of presheaves
such that for any $U\in\SC_X$ the morphism $\varphi(U)\colon \M(U)\to\N(U)$ is a morphism of $\R(U)$-modules.

Let us denote by ${\rm Psh}(\R)$ the category of presheaves of $\R$-modules.
\end{definition}
Clearly, the category ${\rm Psh}(\R)$ is abelian.

\begin{definition}
Let $X$ be a site and $\R$ be a $\Bbbk_X$-algebra.

A sheaf of $\R$-modules is a presheaf of $\R$-modules which is a sheaf of $\Bbbk$-modules.
For simplicity, a sheaf of $\R$-modules is called a $\R$-module.
A $\R^\op$-module is called a right $\R$-module.

A morphism of $\R$-modules is a morphism of presheaves of $\R$-modules.
Let us denote by $\Mod(\R)$ the category of $\R$-modules.
\end{definition}
Note that the category $\Mod(\R)$ has same properties of Theorem \ref{thm2.9}.
For example, it is abelian, see also \cite[Thm.\:18.1.6 (1)]{KS06}.
Moreover, the category $\Mod(\R)$ is a Grothendieck category,
see \cite[Def.\:8.3.24, Thm.\:18.1.6 (v)]{KS06} for details.
In particular, it admits enough injectives, see also \cite[Thm.\:9.6.2]{KS06}.

Note also that the forgetful functor $$for\colon\Mod(\R)\to \Mod(\Bbbk_X)$$ is faithful
and conservative but not fully faithful in general.

The sheafification functor $(\cdot)^{++}\colon {\rm Psh}(\Bbbk_X)\to\Mod(\Bbbk_X)$ induces
a functor $$(\cdot)^{++}\colon{\rm Psh}(\R)\to\Mod(\R),$$
which will be denoted by the same symbol.
Note that it is also left adjoint to
the embedding functor $\Mod(\R)\to {\rm Psh}(\R)$.
See \cite[Lem.\:18.1.4]{KS06} for details.
Moreover, for any $U\in\SC_X$, we have functors: 
\begin{align*}
(\cdot)_U&\colon\Mod(\R)\to\Mod(\R),\\
\Gamma_U&\colon\Mod(\R)\to\Mod(\R),\\
\Gamma(U;\ \cdot\ )&\colon\Mod(\R)\to\Mod(\R(U))
\end{align*}
and the functor $(\cdot)_U$ is exact and the functors $\Gamma_U, \Gamma(U;\ \cdot\ )$ are left exact,
see \cite[\S18.1]{KS06} for details.

For $\M, \N\in{\rm Psh}(\R)$ and $\SL\in{\rm Psh}(\R^\op)$,
we set
\begin{align*}
\big(\shom_{\R}(\M, \N)\big)(U) &:=
\Hom_{{\rm Psh}(\R|_U)}(\M|_U, \N|_U)\hspace{19pt} \text{ for any } U\in\SC_X,\\
\left(\SL\underset{\R}{\overset{\rm pre}{\otimes}}\N\right)(U) &:=
\SL(U)\underset{\R(U)}{\otimes}\N(U)\hspace{39pt} \text{ for any } U\in\SC_X.
\end{align*}
Then we have bifunctors:
\begin{align*}
 \shom_{\R}(\cdot, \cdot) &\colon
 {\rm Psh}(\R)^\op\times{\rm Psh}(\R)\to{\rm Psh}(\Bbbk_X),\hspace{7pt}
 (\M, \N)\mapsto\shom_{\R}(\M, \N),\\
  \Hom_{{\rm Psh}(\R)}(\cdot, \cdot) &\colon
 {\rm Psh}(\R)^\op\times{\rm Psh}(\R)\to\Mod(\Bbbk),\hspace{7pt}
 (\M, \N)\mapsto\Hom_{{\rm Psh}(\R)}(\M, \N),\\
 (\cdot)\overset{\rm pre}{\underset{\R}{\otimes}}(\cdot) &\colon
 {\rm Psh}(\R^\op)\times{\rm Psh}(\R)\to{\rm Psh}(\Bbbk_X),\hspace{7pt}
 (\M, \N)\mapsto\M \overset{\rm pre}{\underset{\R}{\otimes}}\N.
 \end{align*}

Note that if $\M, \N\in\Mod(\R)$ then the presheaf $\shom_{\R}(\M, \N)$ is a sheaf,
see \cite[Lem.\:18.2.1 (i)]{KS06}.

\begin{definition}
The internal hom functor $\shom_{\R}$, the hom functor $ \Hom_{\Mod(\R)}$
 and the tensor product functor $\underset{\R}{\otimes}$ are defined by
\begin{align*}
 \shom_{\R}(\cdot, \cdot) &\colon
 \Mod(\R)^\op\times\Mod(\R)\to\Mod(\Bbbk_X),\hspace{7pt}
 (\M, \N)\mapsto\shom_{\R}(\M, \N),\\
 \Hom_{\Mod(\R)}(\cdot, \cdot) &\colon
 \Mod(\R)^\op\times\Mod(\R)\to\Mod(\Bbbk_X),\hspace{7pt}
 (\M, \N)\mapsto\Hom_{\Mod(\R)}(\M, \N),\\
 (\cdot)\underset{\R}{\otimes}(\cdot) &\colon
 \Mod(\R^\op)\times\Mod(\R)\to\Mod(\Bbbk_X),\hspace{7pt}
 (\M, \N)\mapsto\left(\M \overset{\rm pre}{\underset{\R}{\otimes}}\N\right)^{++}.
 \end{align*}
\end{definition}

Note that the internal hom functor is left exact in each variable
 and the tensor product functor is right exact in each variable.
Note also that if $\R$ is a sheaf of commutative rings
then $\shom_{\R}(\M, \N),\ \M\underset{\R}{\otimes}\N\in\Mod(\R)$ for $\M, \N\in\Mod(\R)$.

Moreover we have:
\begin{proposition}[{\cite[Rem.\:18.2.6]{KS06}}]
Let $\R_{i}\ (i=1,2,3,4)$ be $\Bbbk_X$-algebras.
Then we have left exact functors in each variable
\begin{align*}
 \shom_{\R_1}(\cdot, \cdot) &\colon
 \Mod(\R_1\otimes\R_2^\op)^\op\times\Mod(\R_1\otimes\R_3^\op)
 \to\Mod(\R_2\otimes\R_3^\op),\\
 \Hom_{\Mod(\R_1)}(\cdot, \cdot) &\colon
 \Mod(\R_1\otimes\R_2^\op)^\op\times\Mod(\R_1\otimes\R_3^\op)
 \to\Mod((\R_2\otimes\R_3^\op)(X))
 \end{align*}
 and a right exact functor in each variable
 $$(\cdot)\underset{\R_2}{\otimes}(\cdot) \colon
 \Mod(\R_1\otimes\R_2^\op)\times\Mod(\R_2\otimes\R_3^\op)\to\Mod(\R_1\otimes\R_3^\op).$$
 
Moreover, for any ${}_i\M_j\in\Mod(\R_i\otimes\R_j^\op)\ (i, j = 1,2,3,4)$,
there exist natural isomorphisms in $\Mod(\R_1\otimes \R_4^\op)$
\begin{align*}
\left({}_1\M_2\underset{\R_2}{\otimes}{}_2\M_3\right)\underset{\R_3}{\otimes}{}_3\M_4
&\ \simeq\
{}_1\M_2\underset{\R_2}{\otimes}\left({}_2\M_3\underset{\R_3}{\otimes}{}_3\M_4\right),\\
\shom_{\R_2}\left({}_2\M_1,\ \shom_{\R_3}({}_3\M_2,\ {}_3\M_4)\right)
&\ \simeq\
\shom_{\R_3}\left({}_3\M_2\underset{\R_2}{\otimes}{}_2\M_1,\ {}_3\M_4\right).
\end{align*}
\end{proposition}

Let $f\colon X\to Y$ be a morphism of sites.
Then $f_\ast\R$ is a $\Bbbk_Y$-algebra for a $\Bbbk_X$-algebra $\R$, see \cite[Lem.\:18.3.1 (i)]{KS06}.
However, $f^{-1}\calS$ is not necessarily a ring for a $\Bbbk_Y$-algebra $\calS$.
If the morphism $f\colon X\to Y$ of sites associated to the functor ${}^t\!f\colon \SC_Y\to\SC_X$ is left exact,
that is, the functor ${}^t\!f$ is left exact, then the sheaf $f^{-1}\calS$ is a $\Bbbk_X$-algebra.
Moreover, in this case,
the direct image and the inverse image functors $f_\ast, f^{-1}$
in the sense of sheaves of $\Bbbk$-modules induce the functors:
\begin{align*}
f_\ast&\colon\Mod(f^{-1}\calS)\to\Mod(\calS),\\
f^{-1}&\colon\Mod(\calS)\to\Mod(f^{-1}\calS).
\end{align*}
Note also that the direct image functor $f_\ast$ is left exact and the inverse image functor $f^{-1}$ is exact,
see \cite[Lem.\:18.3.1 (ii)]{KS06} for the details of these results.
Moreover, we have:
\begin{proposition}[{\cite[Lems.\:18.3.1 (ii) (c), 18.3.2]{KS06}}]
Let $f\colon X\to Y$ be a left exact morphism of sites and $\calS$ be a $\Bbbk_Y$-algebra.
Then we have
\begin{itemize}
\item[\rm (1)]
for any $\M\in\Mod(\calS^\op)$ and any $\N\in\Mod(\calS)$,
$$f^{-1}\left(\M\underset{\calS}{\otimes}\N\right)\simeq
f^{-1}\M\underset{f^{-1}\calS}{\otimes}f^{-1}\N,$$

\item[\rm (2)]
for any $\M\in\Mod(\calS)$ and any $\N\in\Mod(f^{-1}\calS)$,
\begin{align*}
f_\ast\shom_{f^{-1}\calS}(f^{-1}\M, \N) &\simeq
\shom_{\calS}(\M, f_\ast\N),\\
\Hom_{\Mod(f^{-1}\calS)}(f^{-1}\M, \N) &\simeq
\Hom_{\Mod(\calS)}(\M, f_\ast\N).
\end{align*}
In particular, the functor $f^{-1}$ is left adjoint to the functor $f_\ast$.
\end{itemize}
\end{proposition}


At the end of this subsection, let us briefly recall the derived category
$\bfD^\ast(\Mod(\R))$ ($\ast = {\rm ub}, +, -, {\rm b}$)
of sheaves on a site. 
Let $\R$ be a $\Bbbk_X$-algebra which is flat as a $\Bbbk_X$-module.
We shall often write for short
$\bfD(\R)$ (resp.\,$\bfD^\ast(\R)$\ for $\ast = +, -, {\rm b}$)
instead of $\bfD^{\rm ub}(\Mod(\R))$ (resp.\,$\bfD^\ast(\Mod(\R))$ for $\ast = +, -, {\rm b}$).
For a left exact morphism $f\colon Z\to X$ of sites and $U\in\SC_X$,
we have (derived) functors
\begin{align*}
\bfR f_\ast&\colon\bfD(f^{-1}\R)\longrightarrow \bfD(\R),\\
f^{-1}&\colon\bfD(\R)\longrightarrow\bfD(f^{-1}\R),\\
(\cdot)\underset{\R}{\overset{\bfL}{\otimes}}(\cdot)&\colon
\bfD(\R^\op)\times \bfD(\R)\longrightarrow\bfD(\Bbbk_X),\\
\rhom_{\R}(\cdot, \cdot)&\colon
\bfD(\R)^\op\times \bfD(\R)\longrightarrow\bfD(\Bbbk_X),\\
\rHom_{\R}(\cdot, \cdot)&\colon
\bfD(\R)^\op\times \bfD(\R)\longrightarrow\bfD(\Bbbk),\\
(\cdot)_U&\colon\bfD(\R)\longrightarrow\bfD(\R),\\
\bfR\Gamma_U&\colon\bfD(\R)\longrightarrow\bfD(\R),\\
\bfR\Gamma(U;\ \cdot\ )&\colon\bfD(\R)\longrightarrow\bfD(\R(U)).
\end{align*}
Note that these functors have same properties of the case of classical sheaves.
We shall skip the explanation of it.
See \cite[\S\S18.4, 18.5, 18.6]{KS06} for the details.

\newpage
\subsection{Subanalytic Sheaves}
The notion of subanalytic sheaves are introduced by M.\:Kashiwara and P.\:Schapira in \cite{KS01}
to treat functions with growth conditions in the formalism of sheaves.
They are defined as sheaves on a subanalytic site and have examples such as
the subanalytic sheaf of tempered distributions,
the one of tempered $\SC^\infty$-functions, the one of Whitney $\SC^\infty$-functions.
In this subsection, 
let us briefly recall the notion of subanalytic sheaves.
Reference are made to \cite[\S 6]{KS01}, \cite{Pre08},
see also Subsection \ref{subsec2.2}.

Throughout this subsection, we assume that $\Bbbk$ is a field.
Let $M$ be a real analytic manifold
and denote by $\Op^\sub_M$ the category of subanalytic open subsets of $M$.
We can endow $\Op^\sub_M$ with the following Grothendieck topology:
a subset $S\subset \Ob((\Op_M^\sub)_U)$ is a covering of $U\in\Op_M^{\sub}$ 
if for any compact subset $K$ of $M$ there exists a finite subset $S_0\subset S$ of $S$
such that $U\cap K = \left(\cup_{V\in S_0}V\right)\cap K$.
We denote by $M^{\sub}$ such a site and call it the subanalytic site.

A subanalytic sheaf (resp.\,presheaf) of $\Bbbk$-modules on $M$ is
a sheaf (resp.\,presheaf) of $\Bbbk$-modules on the subanalytic site $M^{\sub}$.
We shall write $\Mod(\Bbbk_{M}^\sub)$ (resp.\,${\rm Psh}(\Bbbk_{M}^\sub$))
instead of $\Mod(\Bbbk_{M^\sub})$ (resp.\,${\rm Psh}(\Bbbk_{M^\sub})$).
For simplicity, we shall call subanalytic sheaves (resp.\,presheaves) of $\Bbbk$-modules
subanalytic sheaves (resp.\,presheaves).

Note that the forgetful functor induces an equivalence of categories (see e.g. \cite[Rem.\:1.1.2]{Pre08})
$$\Mod(\Bbbk_{M^{\sub}})\simto \Mod(\Bbbk_{M^{\sub,\,c}}),$$
where the site $M^{\sub,\,c}$ is the category $\Op_M^{\sub,\,c}$ of relatively compact subanalytic open subsets
with the following Grothendieck topology:
a subset $S\subset \Ob((\Op_M^\sub)_U)$ is a covering of $U\in\Op_M^{\sub,\,c}$ 
if for any compact subset $K$ of $M$ there exists a finite subset $S_0\subset S$ of $S$
such that $U\cap K = \left(\cup_{V\in S_0}V\right)\cap K$.
Hence we have:
\begin{proposition}[{\cite[Prop.\:1.1.4]{Pre08}}]
Let $M$ be a real analytic manifold and $\SF$ a presheaf on $M^{\sub,\,c}$.
If $\SF$ satisfies the following conditions then it is a sheaf on $M^{\sub,\,c}$,
that is, it is an object of $\Mod(\Bbbk_{M^{\sub,\,c}})\ \left(\xleftarrow{\ \sim\ }\Mod(\Bbbk_M^\sub)\right)$
\begin{itemize}
\item[\rm(1)]
$\SF(\emptyset)=0$,

\item[\rm(2)]
for any relatively compact subanalytic open subsets $U, V$,
the sequence
\[\xymatrix@M=5pt@R=0pt@C=20pt{
0\ar@{->}[r] & \SF(U\cup V)\ar@{->}[r]&\SF(U)\oplus\SF(V)\ar@{->}[r]&\SF(U\cap V)\\
{}& s \ar@{|->}[r] & (s|_U, s|_V) & {}\\
{}& {}  & (t, u) \ar@{|->}[r] & t|_{U\cap V}- u|_{U\cap V}
}\]
is exact.
\end{itemize}
\end{proposition}

Clearly, we have the natural morphism $\rho_M\colon M\to M^\sub$ of sites.
Hence, there exist functors
\begin{align*}
\rho_{M\ast}&\colon\Mod(\Bbbk_M)\to\Mod(\Bbbk_M^\sub),\\
\rho_{M}^{-1}&\colon\Mod(\Bbbk_M^\sub)\to\Mod(\Bbbk_M).
\end{align*}
We have already seen in Subsection \ref{subsec2.2},
a pair $(\rho_{M}^{-1}, \rho_{M\ast})$ is an adjoint pair, 
the functor $\rho_{M\ast}$ is left exact and the functor $\rho_{M}^{-1}$ is exact.
Note that there exists a canonical isomorphism $\rho_{M}^{-1}\circ\rho_{M\ast}\simto \id$ of functors.
Hence the functor $\rho_{M\ast}$ is a fully faithful functor, see \cite[Prop.\:1.1.8]{Pre08}.
Note that the restriction of $\rho_{M\ast}$ to the category of $\RR$-constructible sheaves is exact,
see \cite[Prop.\:1.1.9]{Pre08}.
Let us denote by $\rho_{M\ast}^{\RR-c}$ the exact functor.
Moreover there exists a left adjoint functor $\rho_{M!}$ of the functor $\rho_{M}^{-1}$
which is defined  by the following:
for any $\SF\in\Mod(\Bbbk_M)$, $\rho_{M!}\SF$ is given
by the sheaf associated to the presheaf $U\mapsto \Gamma\left(\var{U}; \SF|_{\var{U}}\right)$,
where $\var{U}$ is the closure of $U$ in $M$.
See \cite[Prop.\:1.1.13 (i)]{Pre08} for the details.
Note that the functor $\rho_{M!}$ is exact and fully faithful,
and there exist the canonical isomorphism $\id\simto\rho_{M}^{-1}\circ\rho_{M!}$ of functors 
and an isomorphism
$$\shom^\sub(\rho_{M!}\SF, \SG)\simeq\shom(\SF, \rho_{M}^{-1}\SG)$$
 in $\Mod(\Bbbk_M)$ for any $\SF\in\Mod(\Bbbk_M)$ and any $\SG\in\Mod(\Bbbk_M^\sub)$,
see \cite[Props.\:1.1.14, 1.1.5]{Pre08}.
Here, the functor $\shom^\sub$ is defined as below.

Let $f\colon M\to N$ be a morphism of real analytic manifolds.
Then it induces the morphism $f\colon M^\sub\to N^\sub$ of subanalytic sites from $M^\sub$ to $N^\sub$
which is denoted by the same symbol.
We have already seen in Subsection \ref{subsec2.2},
there exist many functors
\begin{align*}
(\cdot)^{++}&\colon{\rm Psh}(\Bbbk_M^\sub)\to\Mod(\Bbbk_M^\sub),\\
f_\ast&\colon\Mod(\Bbbk_M^\sub)\to \Mod(\Bbbk_N^\sub),\\
f^{-1}&\colon\Mod(\Bbbk_N^\sub)\to\Mod(\Bbbk_M^\sub),\\
(\cdot)\otimes(\cdot)&\colon
\Mod(\Bbbk_M^\sub)\times \Mod(\Bbbk_M^\sub)\to\Mod(\Bbbk_M^\sub),\\
\ihom^\sub(\cdot, \cdot)&\colon
\Mod(\Bbbk_M^\sub)^\op\times \Mod(\Bbbk_M^\sub)\to\Mod(\Bbbk_M^\sub),\\
\Hom_{\Mod(\Bbbk_M^\sub)}(\cdot, \cdot)&\colon
\Mod(\Bbbk_M^\sub)^\op\times \Mod(\Bbbk_M^\sub)\to\Mod(\Bbbk),\\
(\cdot)_U&\colon\Mod(\Bbbk_M^\sub)\to\Mod(\Bbbk_M^\sub),\\
\Gamma_U&\colon\Mod(\Bbbk_M^\sub)\to\Mod(\Bbbk_M^\sub),\\
\Gamma(U;\ \cdot\ )&\colon\Mod(\Bbbk_M^\sub)\to\Mod(\Bbbk).
\end{align*}
In this paper, we shall write $\ihom^\sub$ instead of the internal hom functor on $\Mod(\Bbbk_M^\sub)$,
that is, $\shom_{\Bbbk_{M^\sub}}(\cdot, \cdot)$ in the notion of Subsection \ref{subsec2.2}.
Let us set $\shom^\sub := \rho_{M}^{-1}\circ\ihom^\sub$.
Moreover, we have a functor which is called the proper direct image functor: 
$$f_{!!}\colon\Mod(\Bbbk_M^\sub)\to \Mod(\Bbbk_N^\sub),
\hspace{7pt}
f_{!!}\SF := \varinjlim_Uf_\ast\SF_U\simeq \varinjlim_Kf_\ast\Gamma_K\SF,$$
where $U$ ranges through the family of relatively compact open subanalytic subsets and
$K$ ranges through the family of subanalytic compact subsets. 
Note that the functor $f_{!!}$ is left exact and
if $f\colon M\to N$ is proper then we have $f_\ast\simeq f_{!!}$.
Moreover, these functors have many properties as similar to classical sheaves, see \cite{Pre08} for the details.
We shall skip the explanation of it.

We shall write $\bfD^\ast(\Bbbk_M^\sub)$ ($\ast = {\rm ub}, +, -, {\rm b}$)
instead of $\bfD^\ast(\Bbbk_{M^\sub})$.
Then there exist (derived) functors of them,
we have already seen in Subsection \ref{subsec2.2} for several ones.
Moreover, the functors below are well defined,
see \cite[Cor.\:2.3.3]{Pre08} for the functors $\bfR f_\ast, \bfR f_{!!}$:
\begin{align*}
\bfR\rho_{M\ast}&\colon\BDC(\Bbbk_M)\hookrightarrow\BDC(\Bbbk_M^\sub),\\
\rho_{M\ast}^{\RR-c}&\colon\BDC_{\RR-c}(\Bbbk_M)\hookrightarrow\BDC(\Bbbk_M^\sub),\\
\rho_{M}^{-1}&\colon\BDC(\Bbbk_M^\sub)\longrightarrow\BDC(\Bbbk_M),\\
\rho_{M!}&\colon\BDC(\Bbbk_M)\hookrightarrow\BDC(\Bbbk_M^\sub),\\
(\cdot)^{++}&\colon\BDC({\rm Psh}(\Bbbk_M^\sub))\longrightarrow\BDC(\Bbbk_M^\sub),\\
\bfR f_\ast&\colon\BDC(\Bbbk_M^\sub)\longrightarrow \BDC(\Bbbk_N^\sub),\\
f^{-1}&\colon\BDC(\Bbbk_N^\sub)\longrightarrow\BDC(\Bbbk_M^\sub),\\
\bfR f_{!!}&\colon\BDC(\Bbbk_M^\sub)\longrightarrow \BDC(\Bbbk_N^\sub),\\
(\cdot)\otimes(\cdot)&\colon
\BDC(\Bbbk_M^\sub)\times \BDC(\Bbbk_M^\sub)\longrightarrow\BDC(\Bbbk_M^\sub),\\
\rihom^\sub(\cdot, \cdot)&\colon
\BDC(\Bbbk_M^\sub)^\op\times \BDC(\Bbbk_M^\sub)\longrightarrow\bfD^+(\Bbbk_M^\sub),\\
\rhom^\sub(\cdot, \cdot)&\colon
\BDC(\Bbbk_M^\sub)^\op\times \BDC(\Bbbk_M^\sub)\longrightarrow\bfD^+(\Bbbk_M),\\
\rHom(\cdot, \cdot)&\colon
\BDC(\Bbbk_M^\sub)^\op\times \BDC(\Bbbk_M^\sub)\longrightarrow\bfD^+(\Bbbk),\\
(\cdot)_U&\colon\BDC(\Bbbk_M^\sub)\longrightarrow\BDC(\Bbbk_M^\sub),\\
\bfR\Gamma_U&\colon\BDC(\Bbbk_M^\sub)\longrightarrow\BDC(\Bbbk_M^\sub),\\
\bfR\Gamma(U;\ \cdot\ )&\colon\BDC(\Bbbk_M^\sub)\longrightarrow\BDC(\Bbbk),
\end{align*}
where $U$ is a subanalytic open subset of $M$.
Note that the functor $\bfR f_{!!}$ admits a right adjoint functor (see \cite[\S2.4]{Pre08} for the details):
$$f^!\colon \BDC(\Bbbk_N^\sub)\longrightarrow\BDC(\Bbbk_M^\sub).$$

Moreover, these functors have many properties as similar to classical sheaves.
In this subsection, we just recall the following properties,
see \cite{Pre08} for the details.

\begin{theorem}\label{thm2.25}
Let $f\colon M\to N$ be a morphism of real analytic manifolds
and $\SF, \SF_1, \SF_2\in\BDC(\Bbbk_M^\sub),
\SG, \SG_1, \SG_2\in\BDC(\Bbbk_N^\sub),
\SK\in\BDC(\Bbbk_M), \SL\in\BDC(\Bbbk_N), \mathcal{J}\in\BDC_{\RR-c}(\Bbbk_M)$.
\begin{itemize}
\setlength{\itemsep}{-1pt}
\item[\rm (1)]
$\bfR \rho_{M\ast}\shom(\rho_{M}^{-1}\SF, \SK)
\simeq
\rihom^\sub(\SF, \bfR \rho_{M\ast}\SK),\\[5pt]
\rhom^\sub(\rho_{M!}\SK, \SF)
\simeq
\rhom(\SK, \rho_{M}^{-1}\SF),$

\item[\rm (2)]
$\rihom^\sub(\bfR f_{!!}\SF, \SG)
\simeq
\bfR f_\ast\rihom^\sub(\SF, f^!\SG),\\[5pt]
\bfR f_\ast\rihom^\sub(f^{-1}\SG, \SF)
\simeq
\rihom^\sub(\SG, \bfR f_\ast\SF),$

\item[\rm (3)]
$\rihom^\sub(\SF_1\otimes\SF_2, \SF)
\simeq
\rihom^\sub\left(\SF_1, \rihom^\sub(\SF_2, \SF)\right),$

\item[\rm(4)]
$f^{-1}(\SF_1\otimes\SF_2)\simeq f^{-1}\SF_1\otimes f^{-1}\SF_2,\\[5pt]
\bfR f_{!!}\left(\SF\otimes f^{-1}\SG\right)\simeq \bfR f_{!!}\SF\otimes \SG,$

\item[\rm (5)]
$f^!\rihom^\sub(\SG_1, \SG_2)\simeq\rihom^\sub(f^{-1}\SG_1, f^!\SG_2),$

\item[\rm (6)]
for a cartesian diagram 
\[\xymatrix@M=5pt@R=20pt@C=40pt{
M'^{\,\sub}\ar@{->}[r]^-{f'}\ar@{->}[d]_-{g'} & N'^{\,\sub}\ar@{->}[d]^-{g}\\
M^\sub\ar@{->}[r]_-{f}  & N^\sub}\]
we have $g^{-1}f_{!!}\SF\simeq f'_{!!}g'^{-1}\SF,\hspace{3pt} g^{!}f_{\ast}\SF\simeq f'_{\ast}g'^{!}\SF$,

\item[\rm (7)]
$f^!(\SG\otimes\rho_{N!}\SL)\simeq f^!\SG\otimes \rho_{M!}f^{-1}\SL$,

\item[\rm (8)]
$\rihom^\sub(\mathcal{J}, \SF)\otimes\rho_{M!}\SK\simeq
\rihom^\sub(\mathcal{J}, \SF\otimes \rho_{M!}\SK)$.
\end{itemize}
\end{theorem}

Let us summarize the commutativity of the various functors in the table below. 
Here, $``\circ"$ means that the functors commute,
and $``\times"$ they do not.
\begin{table}[h]
\begin{equation*}
   \begin{tabular}{l||c|c|c|c|c|c|c}
    {} & $\otimes$ & $f^{-1}$ & $\bfR f_\ast$ & $f^!$ & $\bfR f_{!!}$\\ \hline \hline 
    $\overset{}{\underset{}{\bfR\rho_{\ast}}}$ & $\times$ & $\times$ & $\circ$ & $\circ$ & $\times$ \\ \hline
$\overset{}{\underset{}{\rho_{\ast}^{\RR-c}}}$ & $\circ$ & $\circ$ & $\circ$ & $\circ$ & $\times$ \\ \hline
    $\overset{}{\underset{}{\rho^{-1}}}$ & $\circ$ & $\circ$ & $\circ$ & $\times$ & $\circ$ \\ \hline
    $\overset{}{\underset{}{\rho_{!}}}$ & $\circ$ & $\circ$ & $\times$  
& $\times$ & $\times$ \\ \hline 
   \end{tabular}
   \end{equation*}
\end{table}

Recall that we assume that $\Bbbk$ is a field in this subsection.
See Subsection \ref{subsec2.2} for the details of $\Bbbk_{M^\sub}$-algebras and modules over them.
See also \cite[\S\S 3.1, 3.2]{Pre08}.
We just remark that we have the following facts by the definitions.
\begin{remark}\label{rem2.26}
\begin{itemize}
\item[(1)]
Let $\R$ be a subanalytic sheaf.
Then, the following two conditions are equivalent:
\begin{itemize}
\item[(i)]
a subanalytic sheaf $\R$ is a $\Bbbk_{M^\sub}$-algebra,
\item[(ii)]
there exist morphisms
$\mu_{\R}\colon\R\otimes\R\to\R,\
\varepsilon_\R\colon \Bbbk_{M^\sub}\to\R$
 of subanalytic sheaves
 such that the following diagrams are commutative:
 \[\hspace{-27pt}
\xymatrix@M=5pt@R=20pt@C=40pt{
\R\ar@{->}[r]^-\sim\ar@{->}[d]_-{\rm id_{\R}} & \Bbbk_{M^\sub}\otimes\R\ar@{->}[d]^-{\varepsilon_{\R}\otimes\R}\\
\R & \R\otimes\R\ar@{->}[l]^-{\mu_{\R}},
}\hspace{11pt}
\xymatrix@M=5pt@R=20pt@C=40pt{
\R\ar@{->}[r]^-\sim\ar@{->}[d]_-{\rm id_{\R}} & \R\otimes\Bbbk_{M^\sub}\ar@{->}[d]^-{\R\otimes\varepsilon_{\R}}\\
\R & \R\otimes\R\ar@{->}[l]^-{\mu_{\R}}
},\hspace{11pt}
\xymatrix@M=5pt@R=20pt@C=40pt{
\R\otimes\R\otimes\R\ar@{->}[r]^-{\mu_\R\otimes\R}\ar@{->}[d]_-{\R\otimes\mu_\R} &
\R\otimes\R\ar@{->}[d]^-{\mu_{\R}}\\
\R\otimes\R \ar@{->}[r]_-{\mu_{\R}} & \R.
}\]
\end{itemize}
\item[(2)]
Let $\R$ be a $\Bbbk_{M^\sub}$-algebra and $\SM$ a subanalytic sheaf.
Then the following two conditions are equivalent:
\begin{itemize}
\item[(i)]
a subanalytic sheaf $\SM$ is a $\R$-module,
\item[(ii)]
there exists a morphism
$\mu_{\SM}\colon\R\otimes\SM\to\SM$
 of subanalytic sheaves
 such that the following diagrams are commutative:
 \[\hspace{-27pt}
\xymatrix@M=5pt@R=20pt@C=40pt{
\SM\ar@{->}[r]^-\sim\ar@{->}[d]_-{\rm id_{\SM}} & \Bbbk_{M^\sub}\otimes\SM\ar@{->}[d]^-{\varepsilon_{\SA}\otimes\SM}\\
\SM & \R\otimes\SM\ar@{->}[l]^-{\mu_{\SM}},
}\hspace{11pt}
\xymatrix@M=5pt@R=20pt@C=40pt{
\R\otimes\R\otimes\SM\ar@{->}[r]^-{\mu_\R\otimes\SM}\ar@{->}[d]_-{\R\otimes\mu_\SM} &
\R\otimes\SM\ar@{->}[d]^-{\mu_{\SM}}\\
\R\otimes\SM \ar@{->}[r]_-{\mu_{\SM}} & \SM.
}\]
\end{itemize}

\item[(3)]
Let $\SA$ be a (classical) sheaf of $\Bbbk$-algebras on $M$.
Then, a subanalytic sheaf $\rho_{M!}\SA$ is $\Bbbk_{M^\sub}$-algebra
and the following functors are well-defined:
\begin{align*}
\rho_{M!}\colon\Mod(\SA)\to \Mod(\rho_{M!}\SA),\\
\rho_{M}^{-1}\colon\Mod(\rho_{M!}\SA)\to\Mod(\SA).
\end{align*}
\end{itemize}
\end{remark}

\subsection{Ind-Sheaves}
The theory of (classical) sheaves is not well suited to the study of various objects in Analysis
which are not defined by local properties, such as holomorphic functions with tempered growth. 
The notion of ind-sheaves was introduced by M.\:Kashiwara and P.\:Schapira in \cite{KS01}
to treat functions with such as growth conditions in the formalism of sheaves.
In this subsection,
we shall briefly recall the notion of ind-sheaves.
References are made to \cite{KS01, KS06}. 

Let $\Bbbk$ be a field and $M$ a good topological space,
that is, a topological space which is locally compact, 
Hausdorff, countable at infinity and has finite soft dimension. 
We denote by $\Mod^c(\Bbbk_M)$ the category of sheaves 
of $\Bbbk$-vector spaces on $M$ with compact support.
An ind-sheaf of $\Bbbk$-vector spaces on $M$ is an ind-object of $\Mod^c(\Bbbk_M)$,
that is, inductive limit 
$$\displaystyle``\varinjlim_{i\in I}"\SF_i := \varinjlim_{i\in I}\Hom_{\Mod^c(\Bbbk_M)}(\ \cdot\ ,\ \SF_i)$$
 of a small filtrant inductive system $\{\SF_i\}_{i\in I}$ in $\Mod^c(\Bbbk_M)$.
 We call an ind-sheaf of $\Bbbk$-vector spaces an ind-sheaf, for simplicity.
 Let us denote $\I\Bbbk_M$ the category of ind-sheaves of $\Bbbk$-vector spaces on $M$.
 Note that it is abelian.
 Note also that there exists a natural exact embedding $\iota_M : \Mod(\Bbbk_M)\to\I\Bbbk_M$ 
of categories. We sometimes omit it.
It has an exact left adjoint $\alpha_M$, 
that has in turn an exact fully faithful
left adjoint functor $\beta_M$: 
 \[\xymatrix@C=60pt{\Mod(\Bbbk_M)  \ar@<1.0ex>[r]^-{\iota_{M}} 
 \ar@<-1.0ex>[r]_- {\beta_{M}} & \I\Bbbk_M 
\ar@<0.0ex>[l]|-{\alpha_{M}}}.\]
The category $\I\Bbbk_M$ does not have enough injectives. 
Nevertheless, we can construct the derived category $\BDC(\I\Bbbk_M)$ 
for ind-sheaves and the Grothendieck six operations among them. 
We denote by $\otimes$ and $\rihom$ the operations 
of tensor product and internal hom respectively. 
If $f\colon M\to N$ be a continuous map, we denote 
by $f^{-1}, \rmR f_\ast, f^!$ and $\rmR f_{!!}$ the 
operations of inverse image,
direct image, proper inverse image and proper direct image,
respectively. 
We set also $\rhom := \alpha_M\circ\rihom$. 
We thus have (derived) functors\\
\begin{minipage}{61mm}
\begin{align*}
\iota_M &: \BDC(\Bbbk_M)\longrightarrow \BDC(\I\Bbbk_M),\\
\alpha_M &: \BDC(\I\Bbbk_M)\longrightarrow  \BDC(\Bbbk_M),\\
\beta_M &: \BDC(\Bbbk_M)\longrightarrow  \BDC(\I\Bbbk_M),\\
\\
{}_Z(\cdot)&\colon\BDC(\I\Bbbk_M)\longrightarrow\BDC(\I\Bbbk_M),\\
\bfR\Gamma_Z&\colon\BDC(\I\Bbbk_M)\longrightarrow\BDC(\I\Bbbk_M),
\end{align*}
\end{minipage}
\begin{minipage}{61mm}
\begin{align*}
(\cdot)\otimes(\cdot) &: \BDC(\I\Bbbk_M)\times\BDC(\I\Bbbk_M)\longrightarrow \BDC(\I\Bbbk_M), \\
\rihom(\cdot, \cdot) &: \BDC(\I\Bbbk_M)^{\op}\times\BDC(\I\Bbbk_M)\longrightarrow \bfD^+(\I\Bbbk_M), \\
\rhom(\cdot, \cdot) &: \BDC(\I\Bbbk_M)^{\op}\times\BDC(\I\Bbbk_M)\longrightarrow \bfD^+(\Bbbk_M), \\
\rHom(\cdot, \cdot) &: \BDC(\I\Bbbk_M)^{\op}\times\BDC(\I\Bbbk_M)\longrightarrow \bfD^+(\Bbbk), \\
\bfR f_\ast &: \BDC(\I\Bbbk_M)\longrightarrow \BDC(\I\Bbbk_N),\\
f^{-1} &: \BDC(\I\Bbbk_N)\longrightarrow \BDC(\I\Bbbk_M),\\
\bfR f_{!!} &: \BDC(\I\Bbbk_M)\longrightarrow \BDC(\I\Bbbk_N),\\
f^! &: \BDC(\I\Bbbk_N)\longrightarrow \BDC(\I\Bbbk_M),
\end{align*}
\end{minipage}
where $Z$ is a locally closed subset of $M$.
These functor have many properties as similar to classical sheaves,
see \cite{KS01, KS06} for details.
We just recall the commutativity of the various functors.
Let us summarize them in the table below.
Here, $``\circ"$ means that the functors commute,
and $``\times"$ they do not.
\begin{table}[h]
\begin{equation*}
   \begin{tabular}{l||c|c|c|c|c|c|c}
    {} & $\otimes$ & $f^{-1}$ & $\rmR f_\ast$ & $f^!$ & $\rmR f_{!!}$ \\ \hline \hline 
    $\overset{}{\underset{}{\iota}}$ & $\circ$ & $\circ$ & $\circ$ & $\circ$ & $\times$ \\ \hline
    $\overset{}{\underset{}{\alpha}}$ & $\circ$ & $\circ$ & $\circ$ & $\times$ & $\circ$ \\ \hline
    $\overset{}{\underset{}{\beta}}$ & $\circ$ & $\circ$ & $\times$ & $\times$ & $\times$\\ \hline 
   \end{tabular}
   \end{equation*}
\end{table}

\newpage
At the end of this subsection, 
we shall recall the notion of ind-sheaves with ring actions.
An ind-$\Bbbk_M$-algebra is the date of an ind-sheaf $R$ and morphisms of ind-sheaves
$$\mu_R\colon R\otimes R\to R,\ \varepsilon_R\colon \iota_M\Bbbk_M\to R$$
such that the following diagrams are commutative:
\[\hspace{-27pt}
\xymatrix@M=5pt@R=20pt@C=40pt{
R\ar@{->}[r]^-\sim\ar@{->}[d]_-{\rm id_{R}} & \iota_M\Bbbk_{M}\otimes R\ar@{->}[d]^-{\varepsilon_{R}\otimes R}\\
R & R\otimes R\ar@{->}[l]^-{\mu_{R}},
}\hspace{11pt}
\xymatrix@M=5pt@R=20pt@C=40pt{
R\ar@{->}[r]^-\sim\ar@{->}[d]_-{\rm id_{R}} & R\otimes\iota_M\Bbbk_{M}\ar@{->}[d]^-{R\otimes\varepsilon_{R}}\\
R &  R\otimes R\ar@{->}[l]^-{\mu_{R}},
}\hspace{11pt}
\xymatrix@M=5pt@R=20pt@C=40pt{
R\otimes R\otimes R\ar@{->}[r]^-{\mu_R\otimes R}\ar@{->}[d]_-{R\otimes\mu_R} &
R\otimes R\ar@{->}[d]^-{\mu_{R}}\\
R\otimes R \ar@{->}[r]_-{\mu_{R}} & R.
}\]
Let $R$ be an ind-$\Bbbk_M$-algebra.
An $R$-module $F$ is the date of an ind-sheaf $F$ and a morphism $\mu_F\colon R\otimes F\to F$
 such that the following diagrams are commutative:
 \[\hspace{-27pt}
\xymatrix@M=5pt@R=20pt@C=40pt{
F\ar@{->}[r]^-\sim\ar@{->}[d]_-{\rm id_{F}} & \iota_M\Bbbk_{M}\otimes F\ar@{->}[d]^-{\varepsilon_{R}\otimes F}\\
F & R\otimes F\ar@{->}[l]^-{\mu_{F}},
}\hspace{11pt}
\xymatrix@M=5pt@R=20pt@C=40pt{
R\otimes R\otimes F\ar@{->}[r]^-{\mu_R\otimes F}\ar@{->}[d]_-{\R\otimes\mu_F} &
R\otimes F\ar@{->}[d]^-{\mu_{F}}\\
R\otimes F \ar@{->}[r]_-{\mu_{F}} & F.
}\]
A morphism of $R$-modules from $F$ to $G$ is a morphism $\varphi\colon F\to G$ of ind-sheaves
such that the following diagram is commutative:
 \[\hspace{-27pt}
\xymatrix@M=5pt@R=20pt@C=40pt{
R\otimes F\ar@{->}[r]^-{R\otimes \varphi}\ar@{->}[d]_-{\mu_F}
& R\otimes G\ar@{->}[d]^-{\mu_G}\\
F\ar@{->}[r]_-{\varphi} & G.
}\]
Let us denote by $\I R$ the category of $R$-modules.
Note that this is abelian, see \cite[\S 5.4]{KS01} for the details.

For an $R$-module $F$,
we denote by $\nu_F\colon F\to\ihom(R, F)$ the corresponding morphism of $\mu_F$
by the isomorphism $\Hom_{\I\Bbbk_M}(R\otimes F, F)\simeq \Hom_{\I\Bbbk_M}(F, \ihom(R, F))$
and set
$e_F \colon F\simeq \Bbbk_M\otimes F\xrightarrow{\ \varepsilon_R\otimes F} R\otimes F,\
e_F^\ast\colon \ihom(R, F)\xrightarrow{\ \ihom(\varepsilon_R, F)\ }\ihom(\Bbbk_M, F)\simeq F.$
Note that in the classical case of module over ring,
the analogous morphisms of $\mu_F, \nu_F, e_F$ and $e_F^\ast$ are the morphisms
$r\otimes f\mapsto rf, f\mapsto (r\mapsto rf), f\mapsto 1\otimes f$
and $\varphi\mapsto \varphi(1)$, respectively.
Moreover, for an ind-$\Bbbk_M$-algebra $R$
we defne an ind-$\Bbbk_M$-algebra $R^\op$
by the date of an ind-$\Bbbk_M$-algebra $R$, a morphism
$\mu_{R^\op}\colon R\otimes R \to R$ of ind-sheaves corresponding to $a\otimes b\mapsto b\otimes a$
and a morphism $\varepsilon_{A^\op} := \varepsilon_{A}$.

\newpage
The tensor product functor and the internal hom functor
\begin{align*}
(\cdot)\underset{R}{\otimes}(\cdot) &\colon \I R^\op\times \I R\to \I\Bbbk_M,\\
\ihom_R(\cdot, \cdot) &\colon (\I R)^\op\times \I R\to \I\Bbbk_M
\end{align*}
are defined by
\begin{align*}
F\underset{R}{\otimes}G &:= \Coker\(\mu_F'\otimes G - F\otimes\mu_G\colon F\otimes R\otimes G\to F\otimes G \),\\
\ihom_R(F, G) &:= \Ker\(\ihom(\mu_F, G) - \ihom(F, \nu_G)'\colon \ihom(F, G)\to \ihom(R\otimes F, G)\) 
\end{align*}
where the morphism $\mu_F'$ is the composition $$F\otimes R \simeq R\otimes F\xrightarrow{\ \mu_F\ }F$$
and the morphism $ \ihom(F, \nu_G)'$ is the composition
$$\ihom(F, G)\xrightarrow{\ \ihom(F, \nu_G)\ }\ihom(F,  \ihom(R, G)) \simeq \ihom(R\otimes F, G).$$
Note that there exist isomorphisms $R\underset{R}{\otimes}F\simeq F, \ihom_R(R, F)\simeq F$ in $\I\Bbbk_M$,
see \cite[Lem.\:5.4.8]{KS01} for the details.

Let us denote by $\BDC(\I R)$ the derived category of $\I R$.
For any three ind-$\Bbbk_M$-algebras $R_1, R_2, R_3$
the following functors are well-defined:
\begin{align*}
(\cdot)\overset{\LL}{\underset{R_2}{\otimes}}(\cdot) &\colon
\BDC(\I(R_1\otimes R_2^\op))\times\BDC(\I(R_2\otimes R_3^\op))\to\bfD^+(\I(R_1\otimes R_3^\op)),\\
\rihom_{R_1}(\cdot, \cdot) &\colon
\BDC(\I(R_1\otimes R_2^\op))^\op\times\BDC(\I(R_1\otimes R_3^\op))\to\bfD^+(\I(R_2\otimes R_3^\op)),
\end{align*}
see \cite[Thms.\:5.4.19 (a), 5.4.21 (b)]{KS01} for the details.
Let $f\colon M\to N$ be a morphism of good topological spaces and $S$ an ind-$\Bbbk_N$-algebra.
Then, the following functors are well-defined:
\begin{align*}
\bfR f_\ast&\colon \BDC(\I(f^{-1}S))\to \BDC(\I S),\\
f^{-1}&\colon\BDC(\I S)\to\BDC(\I(f^{-1}S)),\\
\bfR f_{!!}&\colon \BDC(\I(f^{-1}S))\to \BDC(\I S),
\end{align*}
see \cite[Thms.\:5.5.1]{KS01} for the details.
These functors have many properties as similar to classical sheaves, see \cite[\S\S 5.4, 5.5]{KS01} for the details.
We just recall that
there exist isomorphisms 
\begin{align*}
\rihom_{R_2}\({}_2F_3, \rihom_{R_1}\({}_1F_2, {}_1F_4\)\)
&\simeq
\rihom_{R_1}\({}_1F_2\underset{R_2}{\otimes}{}_2F_3, {}_1F_4\),\\
\rihom_S\(G, \bfR f_\ast F\)
&\simeq
\bfR f_\ast\rihom_{f^{-1}S}\(f^{-1}G, F\)
\end{align*}
where $R_i\ (i = 1, 2, 3, 4)$ are ind-$\Bbbk_M$-algebras,
${}_iF_j \in \BDC(\I(R_i\otimes R_j^\op))\ (i, j = 1, 2, 3, 4)$,
$S$ is an ind-$\Bbbk_N$-algebra and $F\in\BDC(\I(f^{-1}S)), G\in\BDC(\I S)$.
Moreover, for any sheaf $\SA$ of $\Bbbk$-algebras whose flat dimension is finite,
the functors $\alpha_M, \beta_M$ induce functors (see \cite[\S 5.6]{KS01} for the details)
\begin{align*}
\alpha_M\colon\I(\beta_M\SA)\to\Mod(\SA),\\
\beta_M\colon\Mod(\SA)\to \I(\beta_M\SA),\\
\alpha_M\colon\BDC(\I(\beta_M\SA))\to \BDC(\SA),\\
\beta_M\colon\BDC(\SA)\to \BDC(\I(\beta_M\SA)).
\end{align*}

\subsection{Relation between Ind-sheaves and Subanalytic Sheaves}\label{subsec2.5}
In this subsection, we shall recall the relation between ind-sheaves and subanalytic sheaves.
Reference are made to \cite[\S\S 6.3, 7.1]{KS01} and \cite[\S A.2]{Pre13}.

Let $M$ be a real analytic manifold and $\Bbbk$ a field.
We denote
by $\Mod_{\RR-c}^c(\Bbbk_M)$ the abelian category of $\RR$-constructible sheaves on $M$ with compact support
and
by $\I_{\RR-c}\Bbbk_M$ the category of ind-objects of $\Mod_{\RR-c}^c(\Bbbk_M)$.
Moreover let us denote by $\BDC_{\I\RR-c}(\I\Bbbk_M)$
the full triangulated subcategory of $\BDC(\I\Bbbk_{M})$
consisting of objects whose cohomologies are contained in $\I_{\RR-c}\Bbbk_{M}$.
A functor $J_M\colon \I\Bbbk_M\to\Mod(\Bbbk_M^\sub)$ is defined by
$$J_M\left(``\varinjlim_{i\in I}"\SF_i\right) := \varinjlim_{i\in I}\rho_{M\ast}\SF_i.$$
Note that for any $F\in\I\Bbbk_M$ the subanalytic sheaf $J_MF$ is given by 
$\Gamma(U; J_MF) = \Hom_{\I\Bbbk_M}(\iota_M\Bbbk_U, F)$
for each open subanalytic subset $U$.
Note also that the functor $J_M$ is left exact and admits a left adjoint
$$I_M\colon \Mod(\Bbbk_M^\sub)\to\I\Bbbk_M$$
which is fully faithful, exact and commutes with filtrant inductive limits.
Then we have (derived) functors
\begin{align*}
I_M&\colon \BDC(\Bbbk_M^\sub)\to\BDC(\I\Bbbk_M),\\
\bfR J_M&\colon \BDC(\I\Bbbk_M)\to\BDC(\Bbbk_M^\sub)
\end{align*}
and a pair $(I_M, \bfR J_M)$ of functors is an adjoint pair
and there exists a canonical isomorphism $\id\simto\bfR J_M\circ I_M$.
Moreover, we have $$\bfR J_M\rihom(I_M(\cdot),\ \cdot) \simeq \rihom^\sub(\cdot, \bfR J_M(\cdot)).$$

\begin{theorem}[{\cite[Thm.\:6.3.5]{KS01}}, see also {\cite[A.2.1]{Pre13}}]\label{thm2.26}
There exists an equivalence of abelian categories:
\[\xymatrix@M=7pt@C=45pt{
\Mod(\Bbbk_M^\sub)\ar@<0.8ex>@{->}[r]^-{I_M}_-\sim
&
\I_{\RR-c}\Bbbk_M
\ar@<0.8ex>@{->}[l]^-{J_M}.
}\]
Furthermore, there exists an equivalence of triangulated categories:
\[\xymatrix@M=7pt@C=45pt{
\BDC(\Bbbk_M^\sub)\ar@<0.8ex>@{->}[r]^-{I_M}_-\sim
&
\BDC_{\I{\RR-c}}(\I\Bbbk_M)
\ar@<0.8ex>@{->}[l]^-{\bfR J_M}.
}\]
\end{theorem}
Let us summarize the commutativity between the various functors and functors $I, \bfR J$.
We will denote by $$\lambda_M\colon \BDC_{\I{\RR-c}}(\I\Bbbk_M)\simto \BDC(\Bbbk_M^\sub)$$
the inverse functor of $I_M\colon \BDC(\Bbbk_M^\sub)\simto \BDC_{\I{\RR-c}}(\I\Bbbk_M)$.
Let $f\colon M\to N$ be a morphism of real analytic manifolds.
Then, we have
\begin{align*}
\alpha_M\circ I_M &\simeq \rho_{M}^{-1},\\
I_M\circ\rho_{M!}&\simeq\beta_M,\\
I_M\circ f^{-1}&\simeq f^{-1}\circ I_N,\\
\bfR f_{!!}\circ I_M&\simeq I_N\circ \bfR f_{!!},\\
I_M\circ f^{!}&\simeq f^{!}\circ I_N,\\
I_M(\cdot\otimes \cdot) &\simeq I_M(\cdot)\otimes I_M(\cdot),
\end{align*}
and 
\begin{align*}
\bfR J_M\circ \iota_M&\simeq \rho_{M\ast},\\
\rho_{M}^{-1}\circ\bfR J_M&\simeq \alpha_M,\\
\bfR J_M\circ \beta_M&\simeq \rho_{M!},\\
\bfR f_{\ast}\circ \bfR J_M &\simeq \bfR J_N\circ \bfR f_{\ast} ,\\
\bfR J_M\circ f^{!}&\simeq f^{!}\circ \bfR J_N.
\end{align*}
Moreover, we have
\begin{align*}
I_M\circ\rho_{M\ast}^{\RR-c} &\simeq \iota_M|_{\Mod_{\RR-c}(\Bbbk_M)},\\
\lambda_M\circ f^{-1}&\simeq f^{-1}\circ \lambda_N,\\
\bfR f_{!!}\circ \lambda_M&\simeq \lambda_N\circ \bfR f_{!!},\\
\lambda_M(\cdot\otimes \cdot) &\simeq \lambda_M(\cdot)\otimes \lambda_M(\cdot).
\end{align*}

At the end of this subsection, 
let us consider Theorem \ref{thm2.26} with ring actions.
Let $\SA$ is a sheaf of algebras whose flat dimension is finite.
Since the functors $I_M$ and $\lambda_M$ are exact and commute with the tensor product functor $\otimes$,
the following functors are well-defined:
\begin{align*}
I_M&\colon\Mod(\rho_{M!}\SA)\to\I(\beta_M\SA),\\
\lambda_M&\colon\I(\beta_M\SA)\cap\I_{\RR-c}\Bbbk_M\to\Mod(\rho_{M!}\SA).
\end{align*}
See also Remark \ref{rem2.26}.
Hence, by Theorem \ref{thm2.26}, there exists an equivalence of abelian categories:
\[\xymatrix@M=7pt@C=45pt{
\Mod(\rho_{M!}\SA)\ar@<0.8ex>@{->}[r]^-{I_M}_-\sim
&
\I(\beta_M\SA)\cap\I_{\RR-c}\Bbbk_M
\ar@<0.8ex>@{->}[l]^-{\lambda_M}.
}\]
Furthermore, there exists an equivalence of triangulated categories:
\[\xymatrix@M=7pt@C=45pt{
\BDC(\rho_{M!}\SA)\ar@<0.8ex>@{->}[r]^-{I_M}_-\sim
&
\BDC(\I(\beta_M\SA))\cap\BDC_{\I{\RR-c}}(\I\Bbbk_M)
\ar@<0.8ex>@{->}[l]^-{\lambda_M}.
}\]
Moreover, for any $\M\in\BDC(\rho_{M!}\SA)$ and
any $G\in\BDC(\I(\beta_M\SA))\cap\BDC_{\I{\RR-c}}(\I\Bbbk_M)$,
one has
$$\bfR J_M\rihom_{\beta_M\SA}(I_M\M, G)
\simeq
\rihom_{\rho_{M!}\SA}^\sub(\M, \lambda_MG).$$
This claim follows from the fact that 
$\bfR J_M\rihom(I_M(\cdot),\ \cdot) \simeq \rihom^\sub(\cdot, \bfR J_M(\cdot))$
and the definition of the bifunctor $\ihom_{\beta_M\SA}$.
See also Remark \ref{rem2.26}.

\subsection{Ind-Sheaves on Bordered Spaces}
In this subsection, 
we shall recall the notions of bordered spaces and ind-sheaves on bordered spaces.
Reference are made to \cite[\S 3]{DK16}
for the details of bordered spaces and ind-sheaves on bordered spaces.

Throughout this subsection, we assume that $\Bbbk$ is a field.

\subsubsection{Bordered Spaces}
A bordered space is a pair $M_{\infty} = (M, \che{M})$ of
a good topological space $\che{M}$ and an open subset $M\subset\che{M}$.
A morphism $f\colon (M, \che{M})\to (N, \che{N})$ of bordered spaces
is a continuous map $f\colon M\to N$ such that the first projection
$\che{M}\times\che{N}\to\che{M}$ is proper on
the closure $\var{\Gamma}_f$ of the graph $\Gamma_f$ of $f$ 
in $\che{M}\times\che{N}$. 
Note that a morphism $\che{f}\colon \che{M}\to\che{N}$ such that $\che{f}(M)\subset N$
induces a morphism of bordered spaces from $(M, \che{M})$ to $(N, \che{N})$.
The category of good topological spaces is embedded into that
of bordered spaces by the identification $M = (M, M)$. 

For a locally closed subset $Z$ of $M$,
we set $Z_\infty := (Z, \var{Z})$, where $\var{Z}$ is the closure of $Z$ in $\che{M}$,
and denote by $i_{Z_\infty}\colon Z_\infty\to M_\infty$
a morphism of bordered spaces induced by the natural embedding $i_Z\colon Z\hookrightarrow M$. 
Moreover, let us denote by $j_{M_\infty}\colon M_\infty\to\che{M}$
a morphism of bordered spaces induced by the natural embedding $j_M\colon M\hookrightarrow \che{M}$.

\subsubsection{Ind-Sheaves on Bordered Spaces}
A quotient category
\begin{align*}
\BDC(\I\Bbbk_{M_\infty}) &:= 
\BDC(\I\Bbbk_{\che{M}})/\BDC(\I\Bbbk_{\che{M}\backslash M})
\end{align*}
is called the category of ind-sheaves on $M_{\infty} = (M, \che{M})$,
where $\BDC(\I\Bbbk_{\che{M}\backslash M})$
is identified with its essentially image in $\BDC(\I\Bbbk_{\che{M}})$
by the fully faithful functor $\bfR i_{\che{M}\backslash M !!} \simeq \bfR i_{\che{M}\backslash M\ast}$. 
Here $i_{\che{M}\backslash M}\colon \che{M}\backslash M\to M$ is the closed embedding.
An object of $\BDC(\I\Bbbk_{M_\infty})$ is called  ind-sheaves on $M_{\infty}$.
The quotient functor
\[\q_{M_\infty} : \BDC(\I\Bbbk_{\che{M}})\to\BDC(\I\Bbbk_{M_\infty})\]
has a left adjoint $\bfl_{M_\infty}$ and a right 
adjoint $\bfr_{M_\infty}$, both fully faithful, defined by 
\begin{align*}
\bfl_{M_\infty} &\colon \BDC(\I\Bbbk_{M_\infty})\to \BDC(\I\Bbbk_{\che{M}}),\ \q F \mapsto \iota_{\che{M}}\Bbbk_M\otimes F,\\
\bfr_{M_\infty} &\colon \BDC(\I\Bbbk_{M_\infty})\to \BDC(\I\Bbbk_{\che{M}}),\ \q F \mapsto \rihom(\iota_{\che{M}}\Bbbk_M, F).
\end{align*}
Moreover, they induce equivalences of categories:
\begin{align*}
\bfl_{M_\infty} &\colon \BDC(\I\Bbbk_{M_\infty})\simto
\{F\in\BDC(\I\Bbbk_{\che{M}})\ |\ \iota_{\che{M}}\Bbbk_{M}\otimes F\simto F\},\\
\bfr_{M_\infty} &\colon \BDC(\I\Bbbk_{M_\infty})\simto
\{F\in\BDC(\I\Bbbk_{\che{M}})\ |\ F\simto\rihom(\iota_{\che{M}}\Bbbk_{M}, F)\}.
\end{align*}

\newpage
For a morphism $f\colon M_\infty\to N_\infty$ 
of bordered spaces, 
we have the Grothendieck operations 
\begin{align*}
(\cdot)\otimes(\cdot) &\colon
\BDC(\I\Bbbk_{M_\infty})\times\BDC(\I\Bbbk_{M_\infty})\to \BDC(\I\Bbbk_{M_\infty}),\\
\rihom(\cdot, \cdot) &\colon
\BDC(\I\Bbbk_{M_\infty})^\op\times\BDC(\I\Bbbk_{M_\infty})\to \BDC(\I\Bbbk_{M_\infty}),\\
\bfR f_\ast&\colon\BDC(\I\Bbbk_{M_\infty})\to \BDC(\I\Bbbk_{N_\infty}),\\
f^{-1}&\colon\BDC(\I\Bbbk_{N_\infty})\to \BDC(\I\Bbbk_{M_\infty}),\\
\bfR f_{!!}&\colon\BDC(\I\Bbbk_{M_\infty})\to \BDC(\I\Bbbk_{N_\infty}),\\
f^!&\colon\BDC(\I\Bbbk_{N_\infty})\to \BDC(\I\Bbbk_{M_\infty}),
\end{align*} 
see \cite[Defs\:3.3.1, 3.3.4]{DK16} for the definition.
Moreover, these functors have many properties as similar to classical sheaves.
We shall skip the explanation of them.
See \cite[\S 3.3]{DK16} for the details.

Note that the functors
$j_{M_\infty}^{-1}\simeq j_{M_\infty}^! \colon \BDC(\I\Bbbk_{\che{M}})\to\BDC(\I\Bbbk_{M_\infty})$ 
are isomorphic to the quotient functor
and the functor $\bfR j_{M_\infty!!}\colon\BDC(\I\Bbbk_{M_\infty})\to\BDC(\I\Bbbk_{\che{M}})$
(resp.\ $\bfR j_{M_\infty\ast}\colon\BDC(\I\Bbbk_{M_\infty})\to\BDC(\I\Bbbk_{\che{M}})$)
is isomorphic to the functor $\bfl$ (resp.\ $\bfr$).

Note also that there exists an embedding functor 
$$\iota_{M_\infty}\colon\BDC(\Bbbk_{M}) \hookrightarrow \BDC(\I\Bbbk_{M_\infty}),
\hspace{7pt}
\SF\mapsto
j_{M_\infty}^{-1}\iota_{\che{M}}j_{M!}\SF\
(\,\simeq j_{M_\infty}^{-1}\iota_{\che{M}}\bfR j_{M\ast}\SF)$$
which has an exact left adjoint $$\alpha_{M_\infty}\colon \BDC(\I\Bbbk_{M_\infty})\to\BDC(\Bbbk_{M}),
\hspace{7pt}
F\mapsto
j_M^{-1}\alpha_{\che{M}}\bfR j_{M_\infty!!}F\
(\,\simeq j_M^{-1}\alpha_{\che{M}}\bfR j_{M_\infty\ast}F)
$$
that has in turn an exact fully faithful left adjoint functor
$$\beta_{M_\infty}\colon\BDC(\Bbbk_{M}) \hookrightarrow \BDC(\I\Bbbk_{M_\infty}),
\hspace{7pt}
\SF\mapsto
j_{M_\infty}^{-1}\beta_{\che{M}}j_{M!}\SF\
(\,\simeq j_{M_\infty}^{-1}\beta_{\che{M}}\bfR j_{M\ast}\SF).$$

It is clear that the quotient category 
$\BDC(\Bbbk_{M_\infty}) := 
\BDC(\Bbbk_{\che{M}})/\BDC(\Bbbk_{\che{M}\backslash M})$
is equivalent to the derived category $\BDC(\Bbbk_{M})$
of the abelian category $\Mod(\Bbbk_M)$.
We sometimes write $\BDC(\Bbbk_{M_\infty})$ for $\BDC(\Bbbk_{M})$,
when considered as a full subcategory of $\BDC(\I\Bbbk_{M_\infty})$.
%

Let $\SA$ be a sheaf of $\Bbbk$-algebras on $M$ whose flat dimension is finite.
Then we have a functor 
$$\rihom_{\pi^{-1}\beta_M\SA}(\pi^{-1}\beta_M(\cdot), \cdot)\colon
\BDC(\SA)\to \BDC(\I\Bbbk_{M\times\RR_\infty})$$
which is defined by
$$(\M, \q_{M\times\RR_\infty}F)
\mapsto
\q_{M\times\RR_\infty}\rihom_{\overline{\pi}^{-1}\beta_M\D_M}(\overline{\pi}^{-1}\beta_M\M, F),$$
where $\overline{\pi}^{-1}\colon M\times \overline{\RR} \to M$ is the projection.
This functor is needed to define the enhanced solution functors in Definition \ref{def3.33}.

\newpage
\subsection{Enhanced Ind-Sheaves}\label{subsec2.7}
It was a long-standing problem to generalize
the Riemann--Hilbert correspondence for regular holonomic $\D$-modules
to the (not necessarily regular) holonomic $\D$-module case. 
One of the difficulties was that we could not find
an appropriate substitute of the target category of the regular case.
In \cite{DK16}, the authors solved it by using the enhanced ind-sheaves.
In this subsection,
we shall recall a more general notion of enhanced ind-sheaves on topological spaces,
that is, the notion of enhanced ind-sheaves on bordered spaces.
Reference are made to \cite{KS16-2} and \cite{DK16-2}.
We also refer to \cite{DK16} and \cite{KS16}
for the notion of enhanced ind-sheaves on good topological spaces.

Let $M_\infty = (M, \che{M})$ be a bordered space.
We set $\RR_\infty := (\RR, \var{\RR})$ for 
$\var{\RR} := \RR\sqcup\{-\infty, +\infty\}$,
and let $t\in\RR$ be the affine coordinate. 
We consider the morphisms of bordered spaces
\[M_\infty\times\RR^2_\infty\xrightarrow{p_1,\ p_2,\ \mu}M_\infty
\times\RR_\infty\overset{\pi}{\longrightarrow}M_\infty\]
given by the maps $p_1(x, t_1, t_2) := (x, t_1)$, $p_2(x, t_1, t_2) := (x, t_2)$,
$\mu(x, t_1, t_2) := (x, t_1+t_2)$ and $\pi (x,t) := x$. 
Then the convolution functors for ind-sheaves on $M_\infty \times \RR_\infty$
\begin{align*}
(\cdot)\Potimes(\cdot)&\colon
\BDC(\I\Bbbk_{M_\infty\times \RR_\infty})\times \BDC(\I\Bbbk_{M_\infty\times \RR_\infty})
\to\BDC(\I\Bbbk_{M_\infty\times \RR_\infty}),\\
\Prihom(\cdot, \cdot)&\colon
\BDC(\I\Bbbk_{M_\infty\times \RR_\infty})^\op\times \BDC(\I\Bbbk_{M_\infty\times \RR_\infty})
\to\BDC(\I\Bbbk_{M_\infty\times \RR_\infty})
\end{align*}
 are defined by
\begin{align*}
F_1\Potimes F_2 & := \rmR\mu_{!!}(p_1^{-1}F_1\otimes p_2^{-1}F_2),\\
\Prihom(F_1, F_2) & := \rmR p_{1\ast}\rihom(p_2^{-1}F_1, \mu^!F_2),
\end{align*}
where $F_1, F_2\in\BDC(\I\Bbbk_{M_\infty\times \RR_\infty})$.
Now we define the triangulated category 
of enhanced ind-sheaves on a bordered space $M_\infty$ by 
$$\BEC(\I\Bbbk_{M_\infty}) :=
\BDC(\I\Bbbk_{M_\infty \times\RR_\infty})/\pi^{-1}\BDC(\I\Bbbk_{M_\infty}).$$
The quotient functor
\[\Q_{M_\infty} \colon \BDC(\I\Bbbk_{M_\infty\times\RR_\infty})\to\BEC(\I\Bbbk_{M_\infty})\]
has fully faithful left and right adjoints 
\begin{align*}
\bfL_{M_\infty}^\rmE\colon
\BEC(\I\Bbbk_{M_\infty}) \to\BDC(\I\Bbbk_{M_\infty\times\RR_\infty}),\\
\bfR_{M_\infty}^\rmE \colon
\BEC(\I\Bbbk_{M_\infty}) \to\BDC(\I\Bbbk_{M_\infty\times\RR_\infty})
\end{align*}
 defined by 
\begin{align*}
\bfL_{M_\infty}^\rmE(\Q_{M_\infty}(F))
&:= \iota_{M_\infty\times\RR_\infty}(\Bbbk_{\{t\geq0\}}\oplus\Bbbk_{\{t\leq 0\}})\Potimes F,\\
\bfR_{M_\infty}^\rmE(\Q_{M_\infty}(F))
&:=\Prihom(\iota_{M_\infty\times\RR_\infty}(\Bbbk_{\{t\geq0\}}\oplus\Bbbk_{\{t\leq 0\}}), F)
\end{align*}
for $F\in\BDC(\I\Bbbk_{M_\infty\times\RR_\infty})$,
where $\{t\geq0\}$ stands for $\{(x, t)\in M\times\RR\ |\ t\geq0\}$
and $\{t\leq0\}$ is defined similarly.
We sometimes denote $\Q_{M_\infty}$ (resp.\ $\bfL_{M_\infty}^\rmE, \bfR_{M_\infty}^\rmE$ )
by $\Q$ (resp.\ $\bfL^\rmE, \bfR^\rmE$) for short.
Moreover they induce equivalences of categories
\begin{align*}
\bfL_{M_\infty}^\rmE &\colon \BEC(\I\Bbbk_{M_\infty})\simto
\left\{F\in\BDC(\I\Bbbk_{M_\infty\times\RR_\infty})\ |\
\iota_{M_\infty\times\RR_\infty}(\Bbbk_{\{t\geq0\}}\oplus\Bbbk_{\{t\leq0\}})\Potimes F\simto F\right\},\\
\bfR_{M_\infty}^\rmE &\colon \BEC(\I\Bbbk_{M_\infty})\simto
\left\{F\in\BDC(\I\Bbbk_{M_\infty\times\RR_\infty})\ |\ F\simto
\Prihom(\iota_{M_\infty\times\RR_\infty}(\Bbbk_{\{t\geq0\}}\oplus\Bbbk_{\{t\leq0\}}), F)\right\}.
\end{align*}
Then we have the following standard t-structure on $\BEC(\I\CC_{M_\infty})$
which is induced by the standard t-structure on $\BDC(\I\CC_{M_\infty\times\RR_\infty})$
\begin{align*}
\bfE^{\leq 0}(\I\CC_{M_\infty}) & = \{K\in \BEC(\I\CC_{M_\infty})\ | \ 
\bfL_{M_\infty}^{\rmE}K\in \bfD^{\leq 0}(\I\CC_{M_\infty\times\RR_\infty})\},\\
\bfE^{\geq 0}(\I\CC_{M_\infty}) & = \{K\in \BEC(\I\CC_{M_\infty})\ | \ 
\bfL_{M_\infty}^{\rmE}K\in \bfD^{\geq 0}(\I\CC_{M_\infty\times\RR_\infty})\}.
\end{align*}
We denote by 
\[\SH^n \colon \BEC(\I\CC_{M_\infty})\to\bfE^0(\I\CC_{M_\infty})\]
the $n$-th cohomology functor, where we set 
$$\bfE^0(\I\CC_{M_\infty}) :=
\bfE^{\leq 0}(\I\CC_{M_\infty})\cap\bfE^{\geq 0}(\I\CC_{M_\infty}).$$

The convolution functors for enhanced ind-sheaves on bordered spaces are well defined.
We denote them by the same symbols:
\begin{align*}
(\cdot)\Potimes(\cdot) \colon
\BEC(\I\Bbbk_{M_\infty})\times\BEC(\I\Bbbk_{M_\infty})\to \BEC(\I\Bbbk_{M_\infty}),\\
\Prihom(\cdot, \cdot)\colon
\BEC(\I\Bbbk_{M_\infty})^\op\times\BEC(\I\Bbbk_{M_\infty})\to \BEC(\I\Bbbk_{M_\infty}).
\end{align*} 
For a morphism $f \colon M_\infty \to N_\infty $ of bordered spaces,
 we can define also the operations 
\begin{align*}
\bfE f_\ast&\colon\BEC(\I\Bbbk_{M_\infty})\to \BEC(\I\Bbbk_{N_\infty}),\\
\bfE f^{-1}&\colon\BEC(\I\Bbbk_{N_\infty})\to \BEC(\I\Bbbk_{M_\infty}),\\
\bfE f_{!!}&\colon\BEC(\I\Bbbk_{M_\infty})\to \BEC(\I\Bbbk_{N_\infty}),\\
\bfE f^!&\colon\BEC(\I\Bbbk_{N_\infty})\to \BEC(\I\Bbbk_{M_\infty})
\end{align*} 
for enhanced ind-sheaves on bordered spaces.
For example, 
we define $\bfE f_\ast \big( \Q_{M_\infty}F\big) :=
\Q_{N_\infty}\big(\bfR f_{\RR_\infty\ast}F\big)$
for $F\in\BDC(\I\Bbbk_{M_\infty\times\RR_\infty})$,
where $f_{\RR_\infty}\colon M_\infty \times \RR_{\infty} \to 
N_\infty \times \RR_{\infty}$ is the natural morphism of bordered spaces associated to $f$ . 
The other operations are defined similarly.  
Note that there exists a morphism $\bfE f_{!!}\to\bfE f_\ast$
of functors from $\BEC(\I\Bbbk_{M_\infty})$ to $\BEC(\I\Bbbk_{N_\infty})$
and it is an isomorphism if $f$ is proper.

Moreover we have external hom functors
\begin{align*}
\rihom^\rmE(\cdot, \cdot)&\colon
\BEC(\I\Bbbk_{M_\infty})^\op\times\BEC(\I\Bbbk_{M_\infty})\to \BDC(\I\Bbbk_{M_\infty}),\\
\rhom^\rmE(\cdot, \cdot)&\colon
\BEC(\I\Bbbk_{M_\infty})^\op\times\BEC(\I\Bbbk_{M_\infty})\to \BDC(\Bbbk_{M}),\\
\rHom^\rmE(\cdot, \cdot)&\colon
\BEC(\I\Bbbk_{M_\infty})^\op\times\BEC(\I\Bbbk_{M_\infty})\to \BDC(\Bbbk),
\end{align*}
which are defined by 
\begin{align*}
\rihom^\rmE(K_1, K_2)&:=\bfR\pi_\ast\rihom(\bfL_{M_\infty}^\rmE K_1, \bfL_{M_\infty}^\rmE K_2),\\
\rhom^\rmE(K_1, K_2) &:= \alpha_{M_\infty}\rihom^\rmE(K_1, K_2),\\
\rHom^\rmE(K_1, K_2) &:= \bfR\Gamma\big(M; \rhom^\rmE(K_1, K_2)\big),
\end{align*}
for $K_1, K_2\in\BEC(\I\Bbbk_{M_\infty})$.
For $F\in\BDC(\I\Bbbk_{M_\infty})$ and $K\in\BEC(\I\Bbbk_{M_\infty})$ the objects 
\begin{align*}
\pi^{-1}F\otimes K & :=\Q_{M_\infty}(\pi^{-1}F\otimes \bfL_{M_\infty}^\rmE K),\\
\rihom(\pi^{-1}F, K) & :=\Q_{M_\infty}\big(\rihom(\pi^{-1}F, \bfR_{M_\infty}^\rmE K)\big). 
\end{align*}
in $\BEC(\I\Bbbk_{M_\infty})$ are well defined. 
Hence, we have functors
\begin{align*}
\pi^{-1}(\cdot)\otimes (\cdot)&\colon
\BDC(\I\Bbbk_{M_\infty})\times\BEC(\I\Bbbk_{M_\infty})\to \BDC(\I\Bbbk_{M_\infty\times\RR_\infty}),\\
\rihom(\pi^{-1}(\cdot),\ \cdot)&\colon
\BDC(\I\Bbbk_{M_\infty})^\op\times\BEC(\I\Bbbk_{M_\infty})\to \BDC(\I\Bbbk_{M_\infty\times\RR_\infty}).
\end{align*}
We set $$\Bbbk_{M_\infty}^\rmE := \Q_{M_\infty}\q 
\(``\underset{a\to +\infty}{\varinjlim}"\
\iota_{\che{M}\times\var{\RR}}\Bbbk_{\{t\geq a\}}\)\in\BEC(\I\Bbbk_{M_\infty}).$$
Then we have a natural embedding functor
$$e_{M_\infty} \colon \BDC(\I\Bbbk_{M_\infty}) \to \BEC(\I\Bbbk_{M_\infty}),
\hspace{17pt}
F\mapsto\Bbbk_{M_\infty}^\rmE\otimes\pi^{-1}F.$$
Let us define $$\omega_{M_\infty}^\rmE := e_{M_\infty}(\iota_{M_\infty}\omega_M)\in\BEC(\I\Bbbk_{M_\infty})$$
where $\omega_M\in\BDC(\Bbbk_{M_\infty}) ( \simeq \BDC(\Bbbk_M))$ is the dualizing complex,
see \cite[Definition 3.1.16]{KS90} for the details.
Then we have the Verdier duality functor for enhanced ind-sheaves on bordered spaces
$$\rmD_{M_\infty}^\rmE \colon\BEC(\I\Bbbk_{M_\infty})^{\op}\to\BEC(\I\Bbbk_{M_\infty}),
\hspace{17pt}
K\mapsto\Prihom(K, \omega_{M_\infty}^\rmE).$$
 Note that there exists an isomorphism
 $\rmD_{M_\infty}^\rmE(e_{M_\infty}\iota_{M_\infty}\SF)
\simeq
e_{M_\infty}(\iota_{M_\infty}\rmD_M\SF)$
in $\BEC(\I\Bbbk_{M_\infty})$
for any $\SF\in\BDC(\Bbbk_M)$.

Let $i_0 \colon M_\infty\to M_\infty\times\RR_\infty$ be a morphism of bordered spaces
induced by $x\mapsto (x, 0)$.
We set
\begin{align*}
\I\sh_{M_\infty} &:= i_0^!\circ \bfR_{M_\infty}^{\rmE} \colon
\BEC(\I\Bbbk_{M_\infty}) \to \BDC(\I\Bbbk_{M_\infty}),\\
\sh_{M_\infty} &:= \alpha_{M_\infty}\circ \I\sh_{M_\infty} \colon
\BEC(\I\Bbbk_{M_\infty}) \to \BDC(\Bbbk_{M})
\end{align*}
and call them the ind-sheafification functor, the sheafification functor,
for enhanced ind-sheaves on bordered spaces, respectively.
Note that a pair $(e_{M_\infty}, \I\sh_{M_\infty})$ is an adjoint pair
and there exist isomorphisms
$F\simto \I\sh_{M_\infty}e_{M_\infty}F$
for $F\in\BDC(\I\Bbbk_M)$ and 
 $\SF\simto \sh_{M_\infty}e_{M_\infty}\iota_{M_\infty}\SF$
for $\SF\in\BDC(\Bbbk_M)$.
See \cite[\S 3]{DK21} for details.

For a continuous function $\varphi \colon U\to \RR$ defined on an open subset $U\subset M$,
we set the exponential enhanced ind-sheaf by 
\[\EE_{U|M_\infty}^\varphi := 
\Bbbk_{M_\infty}^\rmE\Potimes
\Q_{M_\infty}\iota_{M_\infty\times\RR_\infty}\Bbbk_{\{t+\varphi\geq0\}}
, \]
where $\{t+\varphi\geq0\}$ stands for 
$\{(x, t)\in M\times{\RR}\ |\ x\in U, t+\varphi(x)\geq0\}$. 

\newpage
\section{Main Results}\label{sec3}
The main theorem of this paper is Theorems \ref{main1}, \ref{main2}, \ref{main3} and \ref{main4}.

\subsection{Subanalytic Sheaves on Real Analytic Bordered spaces}\label{subsec3.1}
The notion of subanalytic sheaves on bordered spaces was introduced by M.\:Kashiwara.
Although it has already been explained in \cite[\S\S 3.4--3.7]{Kas16},
we will explain in detail, in this subsection again\footnote{In \cite{Kas16},
the author introduced the notion of subanalytic sheaves on subanalytic bordered spaces.
In this paper, we shall only consider them on real analytic bordered spaces.}.

A real analytic bordered space is a bordered space $M_\infty = (M, \che{M})$
such that $\che{M}$ is a real analytic manifold and $M$ is an open subanalytic subset.
A morphism $f\colon (M, \che{M})\to (N, \che{N})$ of real analytic bordered spaces is
a morphism of  bordered spaces such that
the graph $\Gamma_f$ of $f$ is a subanalytic subset of $\che{M}\times\che{N}$.
Note that a morphism $\che{f}\colon \che{M}\to\che{N}$ of real analytic manifolds
such that $\che{f}(M)\subset N$ induces a morphism of bordered spaces of real analytic bordered spaces
from $(M, \che{M})$ to $(N, \che{N})$.
The category of real analytic manifolds is embedded into that
of real analytic bordered spaces by the identification $M = (M, M)$. 

Let $M_\infty = (M, \che{M})$ be a real analytic bordered space.
We denote by $\Op_{M_\infty}^{\sub}$ (resp.\,$\Op_{M_\infty}^{\sub,\,c}$)
the category of open subsets of $M$ 
which are subanalytic (resp.\,subanalytic and relatively compact) in $\che{M}$.
It is clear that they admit finite products and fiber products.
Although the following proposition are well known by experts, we will prove in this paper.

\begin{proposition}
Let $M_\infty = (M, \che{M})$ be a real analytic bordered space.
\begin{itemize}
\item[\rm (1)]
The category $\Op_{M_\infty}^{\sub}$ can be endowed with the following Grothendieck topology:
a subset $S\subset \Ob\((\Op_{M_\infty}^{\sub})_U\)$ is a covering of $U\in\Ob\(\Op_{M_\infty}^{\sub}\)$ 
if for any compact subset $K$ of $\che{M}$ there exists a finite subset $S'\subset S$ of $S$
such that $\displaystyle K\cap U = K\cap\bigcup_{V\in S'}V$.

\item[\rm(2)]
The category $\Op_{M_\infty}^{\sub,\,c}$ can be endowed with the following Grothendieck topology:
a subset $S\subset \Ob\((\Op_{M_\infty}^{\sub,\,c})_U\)$ is a covering of $U\in\Ob\(\Op_{M_\infty}^{\sub,\,c}\)$ 
if for any compact subset $K$ of $\che{M}$ there exists a finite subset $S'\subset S$ of $S$
such that $\displaystyle K\cap U = K\cap\bigcup_{V\in S'}V$.
\end{itemize}
\end{proposition}

\begin{proof}
Since the proof of (2) is similar, we only prove (1).
It is clear that the condition (GT1) of Definition \ref{def-GT} is satisfied.

Let us prove that  the condition (GT2) of Definition \ref{def-GT} is satisfied.
Let $S_1\subset \Ob((\Op_{M_\infty}^\sub)_U)$ be a covering of $U\in\Ob(\Op_{M_\infty}^\sub)$
which is a refinement of $S_2\subset  \Ob((\Op_{M_\infty}^\sub)_U)$ and $K$ a compact subset of $\che{M}$.
Then there exists a finite subset $S_1'$ of $S_1$ such that $K\cap U = K\cap(\cup_{V\in S_1'}V)$.
Since $S_1$ is a refinement of $S_2$, for any $V_1\in S_1'$,
there exists $V_2\in S_2$ such that $V_1\subset V_2$.
We write $S'_2$ for a subset of $S_2$ consisting of such elements.
Since $\cup_{V_1\in S_1'}V_1\subset \cup_{V_2\in S_2'}V_2$,
we have
$$K\cap U = K\cap \bigcup_{V_1\in S_1'}V_1 \subset K\cap \bigcup_{V_2\in S_2'}V_2
\hspace{7pt}(\subset K\cap U)$$
and hence $K\cap U = K\cap (\cup_{V_2\in S_2'}V_2)$.
This implies that $S_2$ is a covering of $U$
and the condition (GT2) of Definition \ref{def-GT} is satisfied.

Let us prove that  the condition (GT3) of Definition \ref{def-GT} is satisfied.
Let $S \subset \Ob((\Op_{M_\infty}^\sub)_U)$ be a covering of $U\in\Ob(\Op_{M_\infty}^\sub)$,
$V\in\Ob((\Op_{M_\infty}^\sub)_U)$ and $K$ a compact subset of $\che{M}$.
Then there exists a finite subset $S'$ of $S$ such that $K\cap U = K\cap(\cup_{W\in S'}W)$.
Then the set $V\times_US' := \{V\cap W\ |\ W\in S'\}$ is a finite subset of $V\times_US$
and we have 
$$K\cap \bigcup_{\tl{W}\in V\times_US'}\tl{W}
=
K\cap\bigcup_{W\in S'}(V\cap W)
=
\(K\cap\bigcup_{W\in S'}W\)\cap V
=
(K\cap U)\cap V
=
K\cap V.$$
This implies that $V\times_US$ is a covering of $V$
and the condition (GT3) of Definition \ref{def-GT} is satisfied.

Let us prove that  the condition (GT4) of Definition \ref{def-GT} is satisfied.
Let $S_1$ be a covering of $U\in\Ob(\Op_{M_\infty}^\sub)$, $S_2$ a subset of $\Ob((\Op_{M_\infty}^\sub)_U)$
such that for any $V\in S_1$, $V\times_U S_2$ is a covering of $V$,
and $K$ a compact subset of $\che{M}$.
Then there exists a finite subset $S_1'$ of $S_1$ such that $K\cap U = K\cap (\cup_{V\in S_1'}V)$.
Moreover, for any $V\in S_1'\ (\subset S_1)$ there exists a finite subset $\tl{S_V}$ of $V\times_US_2$
such that $K\cap V = K\cap (\cup_{\tl{W}\in \tl{S_V}}\tl{W})$.
Since $\tl{S_V}$ is a finite subset of $V\times_US_2$,
there exist $m_V\in\mathbb{N}$ and $W_1^V,\ldots, W_{m_V}^V\in S_2$ such that 
$$\tl{S_V} = \{V\cap W_1^V,\ldots, V\cap W_{m_V}^V\}.$$
Let us set $S_V := \{W_1^V,\ldots, W_{m_V}^V\}$ 
and set $$S_2' := \bigcup_{V\in S_1'}S_V.$$
Then $S_2'$ is a finite subset of $S_2$ and we have
$$K\cap U
=
 \bigcup_{V\in S_1'}(K\cap V)
=
 \bigcup_{V\in S_1'}\(K\cap \(\bigcup_{\tl{W}\in \tl{S_V}}\tl{W}\)\)
 =
 K\cap \bigcup_{V\in S_1'}\bigcup_{\tl{W}\in \tl{S_V}}\tl{W}
  =
 K\cap \bigcup_{V\in S_1'}\bigcup_{W\in S_V}(V\cap W).
$$
Since $\bigcup_{V\in S_1'}\bigcup_{W\in S_V}(V\cap W)\subset
\bigcup_{V\in S_1'}\bigcup_{W\in S_V}W =  \bigcup_{W\in S_2'}W$,
we have $K\cap U\subset K\cap\bigcup_{W\in S_2'}W$
and hence $$K\cap U =  K\cap\bigcup_{W\in S_2'}W.$$
This implies that $S_2$ is a covering of $U$
and the condition (GT4) of Definition \ref{def-GT} is satisfied.
\end{proof}

Let us denote by $M_\infty^{\sub}$ (resp.\,$M_\infty^{\sub,c}$)
the site $\Op_{M_\infty}^{\sub}$ (resp.\,$\Op_{M_\infty}^{\sub,\,c}$) with the above Grothendieck topology. 
Note that the forgetful functor induces an equivalence of categories:
$$for\colon\Mod(\Bbbk_{M_\infty^{\sub}})\simto \Mod(\Bbbk_{M_\infty^{\sub,\,c}}).$$
Indeed, for $\SF\in\Mod(\Bbbk_{M_\infty^{\sub,\,c}})$ and for $V\in\Ob(\Op_{M_\infty}^{\sub})$ we set
$$\tl{\SF}(V) := \varprojlim_{U\in\Ob\(\Op_{M_\infty}^{\sub,\,c}\)}\SF(U\cap V).$$
Then we have $\tl{\SF}\in\Mod(\Bbbk_{M_\infty^{\sub}})$ and 
the functor $\SF\mapsto \tl{\SF}$ is the inverse functor of the forgetful functor.
\newpage

\begin{definition}
Let $M_\infty = (M, \che{M})$ be a real analytic bordered space.
An object of $\Mod(\Bbbk_{M_\infty^{\sub}})\simeq\Mod(\Bbbk_{M_\infty^{\sub,\,c}})$
is called a subanalytic sheaf on $M_\infty$.
We shall write $\Mod(\Bbbk_{M_\infty}^{\sub})$
instead of $\Mod(\Bbbk_{M_\infty^{\sub}})\simeq\Mod(\Bbbk_{M_\infty^{\sub,\,c}})$, for simplicity.
\end{definition}
The category of subanalytic sheaves on $M_\infty$ is abelian,
and admits projective limits and inductive limits,
by Theorem \ref{thm2.9} (1), (2) and (3).
Moreover, subanalytic sheaves can be characterized as follows:
\begin{lemma}
Let $M_\infty = (M, \che{M})$ be a real analytic bordered space.
We assume that a presheaf $\SF$ on $M_\infty^{\sub,\,c}$ satisfies the following two conditions:
\begin{itemize}
\item[\rm(1)]
$\SF(\emptyset)=0$,

\item[\rm(2)]
for any open subsets $U, V$ of $M$ which are subanalytic and relatively compact in $\che{M}$,
the following sequence is exact:
\[\xymatrix@M=5pt@R=0pt@C=20pt{
0\ar@{->}[r] & \SF(U\cup V)\ar@{->}[r]&\SF(U)\oplus\SF(V)\ar@{->}[r]&\SF(U\cap V).\\
{}& s \ar@{|->}[r] & (s|_U, s|_V) & {}\\
{}& {}  & (t, u) \ar@{|->}[r] & t|_{U\cap V}- u|_{U\cap V}
}\]
\end{itemize}
Then the presheaf $\SF$ is a sheaf on $M_\infty^{\sub,\,c}$.
\end{lemma}

\begin{proof}
First, let us prove that, 
for any finite family $\{U_i\}_{i=1}^n$ in $\Ob\(\Op_{M_\infty}^{\sub,\,c}\)$,
a sequence 
\[\xymatrix@M=5pt@R=0pt@C=30pt{
0\ar@{->}[r] & \SF\(\overset{n}{\underset{i=1}{\bigcup}}U_i\)\ar@{->}[r]^-\alpha
&\overset{n}{\underset{i=1}{\bigoplus}}\SF(U_i)\ar@{->}[r]^-\beta
&\underset{1\leq i < j \leq n}{\bigoplus}\SF(U_{i}\cap U_j)\\
{}& s \ar@{|->}[r] & (s|_{U_i})_{i=1}^n & {}\\
{}& {}  & (t_i)_{i=1}^n \ar@{|->}[r] & \(t_i|_{U_{i}\cap U_j}-t_j|_{U_{i}\cap U_j}\)_{1\leq i < j \leq n}
}\]
is exact.
We shall prove it by induction on $n$.
The cases of $n=1, 2$ are obvious.
Let us assume that $n > 2$ and the results of the case of $k\leq n-1$ have been proved.
We shall consider the following commutative diagram:
\[\xymatrix@M=5pt@R=30pt@C=30pt{
{} & {} & 0\ar@{->}[d] & 0\ar@{->}[d]\\
0\ar@{->}[r] & \SF\(\overset{n}{\underset{i=1}{\bigcup}}U_i\)\ar@{->}[r]^-{\alpha'}\ar@{->}[rd]^-{\alpha}
&\SF\(\overset{n-1}{\underset{i=1}{\bigcup}}U_i\)\oplus\SF(U_n)\ar@{->}[r]^-{\beta'}\ar@{->}[d]^-{\alpha''}
&\SF\(\overset{n-1}{\underset{i=1}{\bigcup}}U_{i}\cap U_n\)\ar@{->}[d]^-{\gamma}\\
{} & & \overset{n}{\underset{i=1}{\bigoplus}}\SF(U_i)
\ar@{->}[rd]^-{\beta}\ar@{->}[d]^-{\beta''}\ar@{->}[r]^-{\beta'''}
 & \overset{n-1}{\underset{i=1}{\bigoplus}}\SF(U_i\cap U_n)\ar@{->}[d]\\
{} & {} & \underset{1\leq i < j \leq n-1}{\bigoplus}\SF(U_i\cap U_j)\ar@{->}[r]
& \underset{1\leq i < j \leq n}{\bigoplus}\SF(U_i\cap U_j),
}\]
where
\begin{align*}
&\alpha'(s) := \(s|_{\bigcup_{i=1}^{n-1} U_i}, s|_{U_n}\), \hspace{17pt}
\beta'(t, u) := t|_{\bigcup_{i=1}^{n-1}U_i\cap U_n} - u|_{\bigcup_{i=1}^{n-1}U_i\cap U_n},\\
&\alpha''(t, u) := \(\(t|_{U_i}\)_{i=1}^{n-1}, u\), \hspace{17pt}
\beta''\((t_i)_{i=1}^{n}\) := \(t_i|_{U_{i}\cap U_j}-t_j|_{U_{i}\cap U_j}\)_{1\leq i < j \leq n-1},\\
&\gamma(v) := (v|_{U_i\cap U_n})_{i=1}^{n-1}, \hspace{40pt}
\beta'''\((t_i)_{i=1}^{n}\) := \(t_i|_{U_{i}\cap U_n}-t_n|_{U_{i}\cap U_n}\)_{i=1}^{n-1}.
\end{align*}
Remark that $\alpha = \alpha''\circ\alpha'$ and $\beta = \beta''\oplus \beta'''$.
By the assumption (2),
a sequence 
\[\xymatrix@M=5pt@R=0pt@C=20pt{
0\ar@{->}[r] & \SF\(\overset{n-1}{\underset{i=1}{\bigcup}}U_i\cup U_n\)\ar@{->}[r]^-{\alpha'}
&\SF\(\overset{n-1}{\underset{i=1}{\bigcup}}U_i\)\oplus\SF(U_n)\ar@{->}[r]^-{\beta'}
&\SF\(\overset{n-1}{\underset{i=1}{\bigcup}}U_i\cap U_n\)}\]
is exact.
Moreover, by the induction hypothesis,
$\gamma\colon\SF\(\overset{n-1}{\underset{i=1}{\bigcup}}U_i\cap U_n\)\to
\overset{n-1}{\underset{i=1}{\bigoplus}}\SF(U_i\cap U_n)$ is injective,
and a sequence
\[\xymatrix@M=5pt@R=0pt@C=20pt{
0\ar@{->}[r] & \SF\(\overset{n-1}{\underset{i=1}{\bigcup}}U_i\)\ar@{->}[r]
&\overset{n-1}{\underset{i=1}{\bigoplus}}\SF(U_i)\ar@{->}[r]
&\underset{1\leq i < j \leq n-1}{\bigoplus}\SF(U_{i}\cap U_j)\\
{}& s' \ar@{|->}[r] & \(s'|_{U_i}\)_{i=1}^{n-1} & {}\\
{}& {}  & \(t_i\)_{i=1}^{n-1} \ar@{|->}[r] & \(t_i|_{U_i\cap U_j}-t_j|_{U_i\cap U_j}\)_{1\leq i < j \leq n-1}
}\]
is exact, and hence the following sequence is exact:
\[\xymatrix@M=5pt@R=0pt@C=20pt{
0\ar@{->}[r] & \SF\(\overset{n-1}{\underset{i=1}{\bigcup}}U_i\)\oplus\SF(U_n)\ar@{->}[r]^-{\alpha''}
&\overset{n-1}{\underset{i=1}{\bigoplus}}\SF(U_i)\oplus\SF(U_n)\ar@{->}[r]^-{\beta''}
&\underset{1\leq i < j \leq n-1}{\bigoplus}\SF(U_{i}\cap U_j).}\]
Since $\alpha', \alpha''$ are injective, $\alpha$ is also injective.
It is clear that $\beta\circ\alpha = 0$ and hence we have $\Image\alpha\subset\Ker\beta$.
Let $(t_i)_{i=1}^n\in \Ker\beta$.
Then we have $ \beta\((t_i)_{i=1}^n\) = 0$ and hence $\beta''\((t_i)_{i=1}^n\) = \beta'''\((t_i)_{i=1}^n\)  = 0$.
Since $\Image\alpha''=\Ker\beta''$, there exists $(t, u)\in \SF\(\overset{n-1}{\underset{i=1}{\bigcup}}U_i\)\oplus\SF(U_n)$
such that $\alpha''(t, u)= \((t_i)_{i=1}^{n-1}, t_n\)$.
Moreover, since $\gamma(\beta'(t, u)) = \beta'''(\alpha''(t, u)) 
= \beta'''\((t_i)_{i=1}^{n-1}, t_n\) = \beta'''\((t_i)_{i=1}^{n}\) = 0$
and $\gamma$ is injective, we have $\beta'(t, u) = 0$.
So that, $(t, u)\in \Ker\beta' = \Image\alpha'$,
and hence there exists $s\in\SF\(\overset{n}{\underset{i=1}{\bigcup}}U_i\)$ such that $\alpha'(s) = (t, u)$.
Moreover we have $\alpha(s) = \alpha''(\alpha'(s)) = \alpha''(t, u) = ((t_i)_{i=1}^{n-1}, t_n) = (t_i)_{i=1}^{n}$,
and hence $\Image\alpha = \Ker\beta$.

Let us prove that the presheaf $\SF$ is a sheaf on $M_\infty^{\sub, c}$,
that is, for any covering $S$ of $U\in\Ob\(\Op_{M_\infty}^{\sub, c}\)$
a sequence 
\[\xymatrix@M=5pt@R=0pt@C=20pt{
0\ar@{->}[r] & \SF(U)\ar@{->}[r]^-\varphi
&\underset{V\in S}{\prod}\SF(V)\ar@{->}[r]^-\psi
&\underset{V', V''\in S}{\prod}\SF(V'\cap V'')\\
{}& s \ar@{|->}[r] & \(s|_{V}\)_{V\in S} & {}\\
{}& {}  & \(t_V\)_{V\in S} \ar@{|->}[r] & \(t_{V'}|_{V'\cap V''}-t_{V''}|_{V'\cap V''}\)_{V', V''\in S}
}\]
is exact.
Since $U$ is relatively compact in $\che{M}$ and 
the definition of Grothendieck topology on $M_\infty^{\sub,\,c}$,
there exists a finite subset $\{U_i\}_{i=1}^n$ of $S$
such that $U = \bigcup_{i=1}^nU_i$,
and hence a sequence
\[\xymatrix@M=5pt@R=0pt@C=30pt{
0\ar@{->}[r] & \SF(U)\ar@{->}[r]^-\alpha
&\overset{n}{\underset{i=1}{\bigoplus}}\SF(U_i)\ar@{->}[r]^-\beta
&\underset{1\leq i < j \leq n}{\bigoplus}\SF(U_{i}\cap U_j)\\
{}& s \ar@{|->}[r] & (s|_{U_i})_{i=1}^n & {}\\
{}& {}  & (t_i)_{i=1}^n \ar@{|->}[r] & \(t_i|_{U_{i}\cap U_j}-t_j|_{U_{i}\cap U_j}\)_{1\leq i < j \leq n}
}\]
is exact, as already proved.
Let $s\in \F(U)$ and assume that $\varphi(s) = 0$.
Then we have $\(s_V\)_{V\in S}=0$, and hence $\(s_{U_i}\)_{i=1}^n = 0$.
This implies that $\alpha(s)=0$.
Since $\alpha$ is injective, we have $s = 0$ and hence $\varphi$ is injective.
It is clear that $\psi\circ\varphi = 0$ and hence $\Image\varphi\subset \Ker\psi$.
Let $\(t_V\)_{V\in S}\in\Ker\psi$.
Then we have $0 = \psi\(\(t_V\)_{V\in S}\) = \(t_{V'}|_{V'\cap V''}-t_{V''}|_{V'\cap V''}\)_{V', V''\in S}$.
In particular, we have $\(t_i|_{U_{i}\cap U_j}-t_j|_{U_{i}\cap U_j}\)_{1\leq i < j \leq n} = 0$.
This implies that $\beta\(\(t_{U_i}\)_{i=1}^n\) = 0$.
Since $\(t_{U_i}\)_{i=1}^n\in\Ker\beta = \Image\alpha$,
there exists $s\in\F(U)$ such that for any $i=1, \ldots, n$ one has $s|_{U_i} = t_{U_i}$.
Let us prove that for any $V\in S\setminus\{U_i\ |\ i=1,\ldots, n\}$ one has $s|_V = t_V$.
Since $\(t_V\)_{V\in S}\in\Ker\psi$, in particular,
for any $i=1,\ldots, n$ we have 
$$(s|_V -  t_V)|_{V\cap U_i}
=
(s|_V)|_{V\cap U_i} - t_V|_{V\cap U_i}
= (s|_{U_i})|_{V\cap U_i} - t_V|_{V\cap U_i} = t_{U_i}|_{V\cap U_i} - t_V|_{V\cap U_i} = 0.
$$
Moreover, a map $\SF(V)\to\overset{n}{\underset{i=1}{\bigoplus}}\SF(V\cap U_i)$ is injective,
as already proved, and hence $s|_V = t_V$.
This implies that $\varphi(s) = \(t_V\)_{V\in S}$ and we have $\Image\varphi = \Ker\psi$.

Therefore, the proof is completed.
\end{proof}

From now on, let us describe various functors for subanalytic sheaves.
As already seen in Theorem \ref{thm2.8},
there exists the left exact functor
\[(\cdot)^{+}\colon {\rm Psh}(\Bbbk_{M_\infty^{\sub}})\to{\rm Psh}(\Bbbk_{M_\infty^{\sub}})\]
where for any $\SF\in {\rm Psh}(\Bbbk_{M_\infty}^\sub)$ and any $U\in\Ob\(\Op_{M_\infty}\)$
we set $$\SF^+(U) := \varinjlim_{S\in\Cov(U)}\SF(S)$$
and $\SF(S)$ is the kernel of a map
$$
\prod_{V\in S}\SF(V)\rightarrow\prod_{V', V''\in S}\SF(V'\cap V''),
\hspace{17pt}
\(t_V\)_{V\in S}\longmapsto (t_{V'}|_{V'\cap V''}-t_{V''}|_{V'\cap V''})_{V', V''\in S}.$$
Moreover, we have the exact functor
\[(\cdot)^{++}\colon {\rm Psh}(\Bbbk_{M_\infty^{\sub}})\to\Mod(\Bbbk_{M_\infty}^\sub)\]
which is called the sheafification functor
and is the left adjoint of natural embedding functor $\Mod(\Bbbk_{M_\infty}^\sub)\to {\rm Psh}(\Bbbk_{M_\infty^{\sub}})$.

Note that there exists the natural morphism $\rho_{M_\infty}\colon M\to M_\infty^\sub$ of sites.
We sometimes write $\rho$ instead of $\rho_{M_\infty}$, for simplicity.
Then we have the functors
\begin{align*}
\rho_{M_\infty\ast}&\colon\Mod(\Bbbk_{M})\to\Mod(\Bbbk_{M_\infty}^\sub),\\
\rho_{M_\infty}^{-1}&\colon\Mod(\Bbbk_{M_\infty}^\sub)\to\Mod(\Bbbk_{M})
\end{align*}
which are defined by 
$$(\rho_{M_\infty\ast}\SF)(U) := \SF(U)$$
for any $\SF\in\Mod(\Bbbk_M)$ and any $U\in \Ob\(\Op_{M_\infty}^\sub\)$
, and by
$$\rho_{M_\infty}^{-1}\SG := \(\rho_{\rm pre}^{-1}\SG\)^{++}$$
for any $\SG\in\Mod(\Bbbk_{M_\infty}^\sub)$.
Here, the presheaf $\rho_{\rm pre}^{-1}\SG$ is given by
$\displaystyle\rho_{\rm pre}^{-1}\SG(V) := \varinjlim_{V\subset U}\SG(U)$
for any open subset $V$ of $M$
and $U$ ranges through the family $\Ob\(\Op_{M_\infty}^\sub\)$.
Note that a pair $(\rho_{M_\infty}^{-1}, \rho_{M_\infty\ast})$ is an adjoint pair,
the functor $\rho_{M_\infty\ast}$ is left exact and the functor $\rho_{M_\infty}^{-1}$ is exact,
as already seen in Proposition \ref{prop2.12}.
Note also that there exists a canonical isomorphism $\rho_{M_\infty}^{-1}\circ\rho_{M_\infty\ast}\simto \id$ of functors
and an isomorphism 
$$\rho_{M_\infty\ast}\shom(\rho_{M_\infty}^{-1}\SF, \SG) \simeq \shom(\SF, \rho_{M_\infty\ast}\SG)$$
in $\Mod(\Bbbk_{M})$ for any $\SF\in\Mod(\Bbbk_{M_\infty}^\sub)$ and any $\SG\in\Mod(\Bbbk_M)$.
Hence the functor $\rho_{M_\infty\ast}$ is a fully faithful functor.

A sheaf $\SF\in\Mod(\Bbbk_M)$ is said to be $\RR$-constructible on $M_\infty$
if $j_{M!}\SF\in\Mod_{\RR-c}(\Bbbk_{\che{M}})$
where $j_M\colon M\hookrightarrow \che{M}$ is the natural embedding.
Let us denote by $\Mod_{\RR-c}(\Bbbk_{M_\infty})$
the category of $\RR$-constructible sheaves on $M_\infty$.
Note that the restriction of $\rho_{M_\infty\ast}$ to $\Mod_{\RR-c}(\Bbbk_{M_\infty})$ is exact.
We shall denote by 
$$\rho_{M_\infty\ast}^{\RR-c}\colon \Mod_{\RR-c}(\Bbbk_{M_\infty})\to \Mod(\Bbbk_{M_\infty}^\sub)$$
the exact functor.
By the same argument as \cite[Prop.\:1.1.14]{Pre08},
we have the left adjoint functor
$$\rho_{M_\infty!}\colon\Mod(\Bbbk_M)\to\Mod(\Bbbk_{M_\infty}^\sub)$$
of $\rho_{M_\infty}^{-1}$ which is given by the following:
for any $\SF\in\Mod(\Bbbk_M)$, $\rho_{M_\infty!}\SF$ is the sheaf associated to
the presheaf $U\mapsto \Gamma(\var{U}; \SF|_{\var{U}})$,
where $U\in\Ob\(\Op_{M_\infty}^\sub\)$ and $\var{U}$ is the closure of $U$ in $M$.
Note that the functor $\rho_{M!}$ is exact and fully faithful,
and there exist the canonical isomorphism $\id\simto\rho_{M_\infty}^{-1}\circ\rho_{M_\infty!}$ of functors 
and an isomorphism
$$\rho_{M_\infty}^{-1}\shom(\rho_{M_\infty!}\SF, \SG)\simeq\shom(\SF, \rho_{M_\infty}^{-1}\SG)$$
 in $\Mod(\Bbbk_M)$ for any $\SF\in\Mod(\Bbbk_M)$ and any $\SG\in\Mod(\Bbbk_{M_\infty}^\sub)$.

Let $f\colon M_\infty\to N_\infty$ be a morphism of real analytic bordered spaces.
Then we obtain a morphism of sites from $M_\infty^\sub$ to $N_\infty^\sub$ associated
with a map $\Op_{N_\infty}^\sub\to\Op_{M_\infty}^\sub, V\mapsto f^{-1}(V)$.
We shall denote by the same symbol $f\colon M_\infty^\sub\to N_\infty^\sub$.
As already seen in Definition \ref{def2.11} and Proposition \ref{prop2.12},
there exist the direct image functor and the inverse image functor
\begin{align*}
f_\ast&\colon\Mod(\Bbbk_{M_\infty}^\sub)\to \Mod(\Bbbk_{N_\infty}^\sub),\\
f^{-1}&\colon\Mod(\Bbbk_{N_\infty}^\sub)\to\Mod(\Bbbk_{M_\infty}^\sub)
\end{align*}
which are defined by 
$$(f_\ast\SF)(V):= \SF(f^{-1}(V))
\simeq\Hom_{\Mod(\Bbbk_{M_\infty}^\sub)}(\Bbbk_{f^{-1}(V)}, \SF)$$
for any $\SF\in\Mod(\Bbbk_{M_\infty}^\sub)$ and any $V\in\Ob\(\Op_{M_\infty}^\sub\)$,
and by
$$f^{-1}\SG := (f^{-1}_{\rm pre}\SG)^{++}$$
for any $\SG \in \Mod(\Bbbk_{N_\infty}^\sub)$.
Here, the presheaf $f^{-1}_{\rm pre}\SG$ is given by  
$$(f^{-1}_{\rm pre}\SG)(U) := \varinjlim_{U\subset f^{-1}(V)}\SG(V)$$
for any $U\in\Ob\(\Op_{M_\infty}^\sub\)$ and 
$V$ ranges through the family
$\{W\in \Ob\(\Op_{N_\infty}^\sub\)\ |\ U\subset f^{-1}(W)\}$.
We have already seen in Proposition \ref{prop2.12},
a pair $(f^{-1}, f_\ast)$ is an adjoint pair,
the direct image functor $f_\ast$ is left exact and the inverse image functor $f^{-1}$ is exact.

For any $U\in\Ob\(\Op_{M_\infty}^\sub\)$, we obtain a real analytic bordered space $(U, \che{M})$.
We shall denote by $U_\infty = (U, \che{M})$\footnote{It is clear that $(U, \var{U}) \simeq (U, \che{M})$ as bordered spaces.}
the real analytic bordered space
and denote by $i_{U_\infty}\colon U_\infty\to M_\infty$ a morphism of real analytic bordered spaces
induced by the natural embedding $U\hookrightarrow M$.
Hence, we obtain a morphism $i_{U_\infty}\colon U_\infty^\sub\to M_\infty^\sub$
of real analytic bordered spaces.
Moreover, we have a morphism from $M_\infty^\sub$ to $U_\infty^\sub$
induced by the map $\Op_{U_\infty}^\sub\to \Op_{M_\infty}^\sub, V\mapsto V$,
and the inverse image functor of it which is denoted by $i_{U_\infty!!}$\footnote{
We shall write $i_{U_\infty!!}$ instead of $i_{U_\infty!}$ in the notion of Subsection \ref{subsec2.2}.}
is left adjoint to the functor $i_{U_\infty}^{-1}$,
namely,
for any $\SF\in\Mod(\Bbbk_{U_\infty}^\sub)$ and any $\SG\in\Mod(\Bbbk_{M_\infty}^\sub)$,
we have
$$\Hom_{\Mod(\Bbbk_{M_\infty}^\sub)}(i_{U_\infty!!}\SF, \SG)
\simeq\Hom_{\Mod(\Bbbk_{U_\infty}^\sub)}(\SF, i_{U_\infty}^{-1}\SG).$$
We have already seen in Subsection \ref{subsec2.2},
there exist functors
\begin{align*}
(\cdot)_{U_\infty} := i_{U_\infty!!}i_{U_\infty}^{-1} &\colon\Mod(\Bbbk_{M_\infty}^\sub)\to\Mod(\Bbbk_{M_\infty}^\sub),\\
\Gamma_{U_\infty} := i_{U_\infty\ast}i_{U_\infty}^{-1}&\colon\Mod(\Bbbk_{M_\infty}^\sub)\to\Mod(\Bbbk_{M_\infty}^\sub),\\
\Gamma(U;\ \cdot\ )&\colon\Mod(\Bbbk_{M_\infty}^\sub)\to\Mod(\Bbbk).
\end{align*}

As already seen in Subsection \ref{subsec2.2},
we have the tensor product functor and the internal hom functor\footnote{
In this paper, we shall write $\ihom^\sub$ instead of the internal hom functor on $\Mod(\Bbbk_{M_\infty}^\sub)$,
that is, $\shom_{\Bbbk_{M_\infty^{\sub}}}(\cdot, \cdot)$ in the notion of Subsection \ref{subsec2.2}.}
\begin{align*}
(\cdot)\otimes(\cdot)&\colon
\Mod(\Bbbk_{M_\infty}^\sub)\times \Mod(\Bbbk_{M_\infty}^\sub)\to\Mod(\Bbbk_{M_\infty}^\sub),\\
\ihom^\sub(\cdot, \cdot)&\colon
\Mod(\Bbbk_{M_\infty}^\sub)^\op\times \Mod(\Bbbk_{M_\infty}^\sub)\to\Mod(\Bbbk_{M_\infty}^\sub)
\end{align*}
which are defined by
\begin{align*}
\SF\otimes\SG &:= \(\SF\overset{\rm pre}{\otimes}\SG\)^{++},\\
\(\ihom^\sub(\SF, \SG)\)(U) &:= \Hom_{\Mod(\Bbbk_{U_\infty}^\sub)}(i_{U_\infty}^{-1}\SF,\,i_{U_\infty}^{-1}\SG)
\end{align*}
for any $\SF, \SG\in\Mod(\Bbbk_{M})$ and any $U\in\Ob\(\Op_{M_\infty}^\sub\)$.
Here, the presheaf $\SF\overset{\rm pre}{\otimes}\SG$ is given by
$$\(\SF\overset{\rm pre}{\otimes}\SG\)(U) := \SF(U)\otimes\SG(U)$$
for any $U\in\Ob\(\Op_{M_\infty}^\sub\)$.
Moreover we have external hom functors 
\begin{align*}
\shom^\sub(\cdot, \cdot) := \rho_{M_\infty}^{-1}\ihom^\sub
&\colon
\Mod(\Bbbk_{M_\infty}^\sub)^\op\times \Mod(\Bbbk_{M_\infty}^\sub)\to\Mod(\Bbbk_M),\\
\Hom_{\Mod(\Bbbk_{M_\infty}^\sub)}(\cdot, \cdot)&\colon
\Mod(\Bbbk_{M_\infty}^\sub)^\op\times \Mod(\Bbbk_{M_\infty}^\sub)\to\Mod(\Bbbk).
\end{align*}

For a morphism $f\colon M\to N$ of real analytic manifolds,
the proper direct image functor 
$$f_{!!}\colon\Mod(\Bbbk_{M_\infty}^\sub)\to \Mod(\Bbbk_{N_\infty}^\sub)$$
is given by 
$$(f_{!!}\SF)(V) := \varinjlim_{U}\Hom_{\Mod(\Bbbk_{M_\infty}^\sub)}(\Bbbk_{f^{-1}(V)}, \SF_{U_\infty})$$
for any $V\in\Ob\(\Op_{N_\infty}^{\sub,c}\)$.
Here, $U$ ranges through the family $\Ob\(\Op_{M_\infty}^{\sub,c}\)$ such that $f^{-1}(V)\cap\var{U}\to V$ is proper,
$\var{U}$ is the closure of $U$ in $M$.
Note that the functor $f_{!!}$ is left exact and
if $f\colon M\to N$ is proper then we have $f_\ast\simeq f_{!!}$.

These functor have many properties as similar to classical sheaves.
We shall skip the explanation of it.

We shall write $\bfD^\ast(\Bbbk_{M_\infty}^\sub)$ ($\ast = {\rm ub}, +, -, {\rm b}$)
instead of $\bfD^\ast\(\Mod\(\Bbbk_{M_\infty}^{\sub}\)\)$
and denoted by $\BDC_{\RR-c}(\Bbbk_{M_\infty})$
the full triangulated subcategory of $\BDC(\Bbbk_M)$ consisting of objects
whose cohomologies are contained in $\Mod_{\RR-c}(\Bbbk_{M_\infty})$.
Recall that the abelian category $\Mod(\Bbbk_{M_\infty}^\sub)$ admits enough injectives,
see Theorem \ref{thm2.9} (6).
For a morphism $f\colon M_\infty\to N_\infty$ of real analytic bordered spaces
and $U\in\Ob\(\Op_{M_\infty}^\sub\)$,
there exist (derived) functors:
\begin{align*}
\bfR\rho_{M_\infty\ast}&\colon\BDC(\Bbbk_M)\hookrightarrow\BDC(\Bbbk_{M_\infty}^\sub),\\
\rho_{M_\infty\ast}^{\RR-c}&\colon\BDC_{\RR-c}(\Bbbk_{M_\infty})\hookrightarrow\BDC(\Bbbk_{M_\infty}^\sub),\\
\rho_{M_\infty}^{-1}&\colon\BDC(\Bbbk_{M_\infty}^\sub)\longrightarrow\BDC(\Bbbk_M),\\
\rho_{M_\infty!}&\colon\BDC(\Bbbk_M)\hookrightarrow\BDC(\Bbbk_{M_\infty}^\sub),\\
(\cdot)^{++}&\colon\BDC({\rm Psh}(\Bbbk_{M_\infty}^\sub))\longrightarrow\BDC(\Bbbk_{M_\infty}^\sub),\\
\bfR f_\ast&\colon\BDC(\Bbbk_{M_\infty}^\sub)\longrightarrow \BDC(\Bbbk_{N_\infty}^\sub),\\
f^{-1}&\colon\BDC(\Bbbk_{N_\infty}^\sub)\longrightarrow\BDC(\Bbbk_{M_\infty}^\sub),\\
\bfR f_{!!}&\colon\BDC(\Bbbk_{M_\infty}^\sub)\longrightarrow \BDC(\Bbbk_{N_\infty}^\sub),\\
(\cdot)\otimes(\cdot)&\colon
\BDC(\Bbbk_{M_\infty}^\sub)\times \BDC(\Bbbk_{M_\infty}^\sub)\longrightarrow\BDC(\Bbbk_{M_\infty}^\sub),\\
\rihom^\sub(\cdot, \cdot)&\colon
\BDC(\Bbbk_{M_\infty}^\sub)^\op\times \BDC(\Bbbk_{M_\infty}^\sub)
\longrightarrow\bfD^+(\Bbbk_{M_\infty}^\sub),\\
\rhom^\sub(\cdot, \cdot)&\colon
\BDC(\Bbbk_{M_\infty}^\sub)^\op\times \BDC(\Bbbk_{M_\infty}^\sub)\longrightarrow\bfD^+(\Bbbk_{M_\infty}),\\
\rHom(\cdot, \cdot)&\colon
\BDC(\Bbbk_{M_\infty}^\sub)^\op\times \BDC(\Bbbk_{M_\infty}^\sub)\longrightarrow\bfD^+(\Bbbk),\\
(\cdot)_U&\colon\BDC(\Bbbk_{M_\infty}^\sub)\longrightarrow\BDC(\Bbbk_{M_\infty}^\sub),\\
\bfR\Gamma_U&\colon\BDC(\Bbbk_{M_\infty}^\sub)\longrightarrow\BDC(\Bbbk_{M_\infty}^\sub),\\
\bfR\Gamma(U;\ \cdot\ )&\colon\BDC(\Bbbk_{M_\infty}^\sub)\longrightarrow\BDC(\Bbbk).
\end{align*}

Note that the functor $\bfR f_{!!}$ admits a right adjoint functor,
see e.g. \cite[\S 3.6]{Kas16}:
$$f^!\colon \BDC(\Bbbk_{N_\infty}^\sub)\longrightarrow\BDC(\Bbbk_{M_\infty}^\sub).$$
Note also that if $f\colon M\to N$ is topological submersive
then $f^!(\cdot)\simeq \rho_{M_\infty\ast}\omega_{M/N}\otimes f^{-1}(\cdot)$,
where $\omega_{M/N}\in\BDC(\Bbbk_M)$ is
the relative dualizing complex which is defined in \cite[Def.\:3.1.16 (i)]{KS90}. 

Moreover, these functors have many properties as similar to classical (subanalytic) sheaves.
\begin{proposition}\label{prop3.4}
Let $f\colon M_\infty\to N_\infty$ be a morphism of real analytic bordered spaces
and $\SF, \SF_1, \SF_2\in\BDC(\Bbbk_{M_\infty}^\sub),
\SG, \SG_1, \SG_2\in\BDC(\Bbbk_{N_\infty}^\sub),
\SK\in\BDC(\Bbbk_M), \SL\in\BDC(\Bbbk_N), \mathcal{J}\in\BDC_{\RR-c}(\Bbbk_{M_\infty})$.
\begin{itemize}
\item[\rm (1)]
$\bfR \rho_{M_\infty\ast}\shom(\rho_{M_\infty}^{-1}\SF, \SK)
\simeq
\rihom^\sub(\SF, \bfR \rho_{M_\infty\ast}\SK),\\[5pt]
\rhom^\sub(\rho_{M_\infty!}\SK, \SF)
\simeq
\rhom(\SK, \rho_{M_\infty}^{-1}\SF),$

\item[\rm (2)]
$\rihom^\sub(\bfR f_{!!}\SF, \SG)
\simeq
\bfR f_\ast\rihom^\sub(\SF, f^!\SG),\\[5pt]
\bfR f_\ast\rihom^\sub(f^{-1}\SG, \SF)
\simeq
\rihom^\sub(\SG, \bfR f_\ast\SF),\\[5pt]
\rihom^\sub(\SF_1\otimes\SF_2, \SF)
\simeq
\rihom^\sub\left(\SF_1, \rihom^\sub(\SF_2, \SF)\right),$

\item[\rm (3)]
$f^{-1}(\SF_1\otimes\SF_2)\simeq f^{-1}\SF_1\otimes f^{-1}\SF_2,\\[5pt]
\bfR f_{!!}\left(\SF\otimes f^{-1}\SG\right)\simeq \bfR f_{!!}\SF\otimes \SG,\\[5pt]
f^!\rihom^\sub(\SG_1, \SG_2)\simeq\rihom^\sub(f^{-1}\SG_1, f^!\SG_2),$

\item[\rm (4)]
for a cartesian diagram 
\[\xymatrix@M=5pt@R=20pt@C=40pt{
M'_\infty\ar@{->}[r]^-{f'}\ar@{->}[d]_-{g'} & N'_\infty\ar@{->}[d]^-{g}\\
M_\infty\ar@{->}[r]_-{f}  & N_\infty}\]
we have $g^{-1}\bfR f_{!!}\SF\simeq \bfR f'_{!!}g'^{-1}\SF,\hspace{3pt}
g^{!}\bfR f_{\ast}\SF\simeq \bfR f'_{\ast}g'^{!}\SF$,

\item[\rm (5)]
$f^!(\SG\otimes\rho_{N_\infty!}\SL)\simeq f^!\SG\otimes \rho_{M_\infty!}f^{-1}\SL,\\[5pt]
\rihom^\sub(\mathcal{J}, \SF)\otimes\rho_{M_\infty!}\SK\simeq
\rihom^\sub(\mathcal{J}, \SF\otimes \rho_{M_\infty!}\SK)$,

\item[\rm (6)]
there exist the commutativity of the various functors
\begin{table}[h]
\begin{equation*}
   \begin{tabular}{l||c|c|c|c|c|c|c}
    {} & $\otimes$ & $f^{-1}$ & $\bfR f_\ast$ & $f^!$ & $\bfR f_{!!}$\\ \hline \hline 
    $\overset{}{\underset{}{\bfR\rho_{\ast}}}$ & $\times$ & $\times$ & $\circ$ & $\circ$ & $\times$ \\ \hline
$\overset{}{\underset{}{\rho_{\ast}^{\RR-c}}}$ & $\circ$ & $\circ$ & $\circ$ & $\circ$ & $\times$ \\ \hline
    $\overset{}{\underset{}{\rho^{-1}}}$ & $\circ$ & $\circ$ & $\circ$ & $\times$ & $\circ$ \\ \hline
    $\overset{}{\underset{}{\rho_{!}}}$ & $\circ$ & $\circ$ & $\times$  
& $\times$ & $\times$ \\ \hline 
   \end{tabular}
   \end{equation*}
\end{table}

\noindent
where $``\circ"$ means that the functors commute, and $``\times"$ they do not.
\end{itemize}
\end{proposition}
Since the proof of this proposition is similar to the case of  classical sheaves, we shall skip the proof.

As we have already seen in Subsection \ref{subsec2.2},
we can consider subanalytic sheaves with ring actions.
We just remark that $\Bbbk_{M_\infty^\sub}$-algebras
and modules over them can be characterized as similar to Remark \ref{rem2.26} (1), (2),
and there exist functors
\begin{align*}
\bfR f_\ast&\colon\bfD(f^{-1}\R)\longrightarrow \bfD(\R),\\
f^{-1}&\colon\bfD(\R)\longrightarrow\bfD(f^{-1}\R),\\
(\cdot)\underset{\R}{\overset{\bfL}{\otimes}}(\cdot)&\colon
\bfD(\R^\op)\times \bfD(\R)\longrightarrow\bfD(\Bbbk_M^\sub),\\
\rihom_{\R}^\sub(\cdot, \cdot)&\colon
\bfD(\R)^\op\times \bfD(\R)\longrightarrow\bfD(\Bbbk_M),\\
\rhom_{\R}^\sub(\cdot, \cdot) := \rho_M^{-1}\rihom_{\R}^\sub&\colon
\bfD(\R)^\op\times \bfD(\R)\longrightarrow\bfD(\Bbbk_M),\\
\rHom_{\R}(\cdot, \cdot)&\colon
\bfD(\R)^\op\times \bfD(\R)\longrightarrow\bfD(\Bbbk)
\end{align*}
where $\R$ is a $\Bbbk_{M_\infty^\sub}$-algebra. 
In this paper, we shall write $\ihom_{\R}^\sub$ instead of the internal hom functor on $\Mod(\R)$,
that is, $\shom_{\R}(\cdot, \cdot)$ in the notion of Subsection \ref{subsec2.2}.

At the end of this subsection,
we shall describe a relation between ind-sheaves on $M_\infty$ and subanalytic sheaves on $M_\infty$.
Recall that $j_{M_\infty}\colon M_\infty\to \che{M}$ is a morphism of real analytic bordered spaces
associated to the natural embedding $j_M\colon M\hookrightarrow\che{M}$.
Then for any $\SF\in\Mod(\Bbbk_{M_\infty}^\sub)$ and any $\SG\in\Mod(\Bbbk_{\che{M}}^\sub)$ we have
\begin{align*}
j_{M_\infty}^{-1}\bfR j_{M_\infty\ast}\SF&\overset{\sim}{\longleftarrow} \SF,\\
j_{M_\infty}^{-1}\bfR j_{M_\infty!!}\SF& \simto\SF,\\
\bfR j_{M_\infty!!}j_{M_\infty}^{-1}\SG &\simeq \rho_{\che{M}\ast}\Bbbk_M\otimes \SG,\\\
\bfR j_{M_\infty\ast}j_{M_\infty}^{-1}\SG &\simeq \rihom^\sub(\rho_{\che{M}\ast}\Bbbk_M, \SG),
\end{align*}
and hence functors $j_{M_\infty}^{-1}, \bfR j_{M_\infty!!}, \bfR j_{M_\infty\ast}$ induce equivalences of categories:
\begin{align*}
\tl{j}\colon\BDC(\Bbbk_{M_\infty}^\sub) \simto
 &\,\{\SF\in \BDC(\Bbbk_{\che{M}}^\sub)\ |\  \rho_{\che{M}\ast}\Bbbk_M\otimes \SF\simto \SF\}\\
 \simeq\
 &\,\{\SF\in \BDC(\Bbbk_{\che{M}}^\sub)\ |\  \SF\simto\rihom^\sub\(\rho_{\che{M}\ast}\Bbbk_M, \SF\) \}.
  \end{align*}
Recall that the category of ind-sheaves on $M_\infty$ is defined by 
 \begin{align*}
\BDC(\I\Bbbk_{M_\infty}) &:= 
\BDC(\I\Bbbk_{\che{M}})/\BDC(\I\Bbbk_{\che{M}\backslash M})
\end{align*} 
and there exist (derived) functors
\begin{align*}
I_{\che{M}}&\colon \BDC(\Bbbk_{\che{M}}^\sub)\to\BDC(\I\Bbbk_{\che{M}}),\\
\bfR J_{\che{M}}&\colon \BDC(\I\Bbbk_{\che{M}})\to\BDC(\Bbbk_{\che{M}}^\sub).
\end{align*}    
Let us consider the following functors
\begin{align*}
I_{M_\infty}&\colon\BDC(\Bbbk_{M_\infty}^\sub)\to\BDC(\I\Bbbk_{M_\infty}),
\hspace{7pt}
\SF\mapsto \q I_{\che{M}}\bfR j_{M_\infty!!}\SF,\\
\bfR J_{M_\infty}&\colon\BDC(\I\Bbbk_{M_\infty})\to\BDC(\Bbbk_{M_\infty}^\sub),
\hspace{7pt}
F\mapsto j_{M_\infty}^{-1}\bfR J_{\che{M}}\bfR j_{M_\infty\ast}F,
\end{align*}  
where $\q\colon \BDC(\I\Bbbk_{\che{M}})\to\BDC(\I\Bbbk_{M_\infty})$ is the quotient functor.

We denote by $\BDC_{\I{\RR-c}}(\I\Bbbk_{M_\infty})$ the full triangulated subcategory of $\BDC(\I\Bbbk_{M_\infty})$
consisting of objects such that $\bfR j_{M_\infty!!}F\in\BDC_{\I\RR-c}(\I\Bbbk_{\che{M}})$.
Then the following lemma follows from \cite[Lem.\:7.1.3]{KS01}.
\begin{lemma}\label{lem3.5}
Let $f\colon M_\infty\to N_\infty$ be a morphism of real analytic bordered spaces
associated with a morphism $\che{f}\colon\che{M}\to\che{N}$ of real analytic manifolds.
The functors below are well defined:
\begin{itemize}
\item[\rm (1)]
$\iota_{M_\infty}\colon \BDC_{\RR-c}(\Bbbk_M)\to\BDC_{\I\RR-c}(\I\Bbbk_{M_\infty})$,

\item[\rm (2)]
$\beta_{M_\infty}\colon \BDC(\Bbbk_M)\to\BDC_{\I\RR-c}(\I\Bbbk_{M_\infty})$,

\item[\rm (3)]
$(\cdot)\otimes(\cdot)\colon
\BDC_{\I\RR-c}(\I\Bbbk_{M_\infty})\times \BDC_{\I\RR-c}(\I\Bbbk_{M_\infty})
\to\BDC_{\I\RR-c}(\I\Bbbk_{M_\infty})$,

\item[\rm (4)]
$\rihom(\iota_{M_\infty}(\cdot),\ \cdot)\colon
\BDC_{\RR-c}(\Bbbk_{M_\infty})^{\op}\times \BDC_{\I\RR-c}(\I\Bbbk_{M_\infty})
\to\BDC_{\I\RR-c}(\I\Bbbk_{M_\infty})$,

\item[\rm (5)]
$f^{-1}\colon\BDC_{\I\RR-c}(\I\Bbbk_{N_\infty})\to\BDC_{\I\RR-c}(\I\Bbbk_{M_\infty})$,

\item[\rm (6)]
$\bfR f_{!!}\colon\BDC_{\I\RR-c}(\I\Bbbk_{M_\infty})\to\BDC_{\I\RR-c}(\I\Bbbk_{N_\infty})$,

\item[\rm (7)]
$f^{!}\colon\BDC_{\I\RR-c}(\I\Bbbk_{N_\infty})\to\BDC_{\I\RR-c}(\I\Bbbk_{M_\infty})$.
\end{itemize}
\end{lemma}

\begin{proof}
(1)
Let $\SF\in\BDC_{\RR-c}(\Bbbk_{M})$.
Then we have $j_{M!}\SF\in\BDC_{\RR-c}(\Bbbk_{\che{M}})$.
Note that there exists an isomorphism 
$\bfR j_{M_\infty!!}\iota_{M_\infty}\SF \simeq \iota_{\che{M}}j_{M!}\SF$
in $\BDC(\I\Bbbk_{\che{M}})$.
Since the functor
$\iota_{\che{M}}\colon\BDC_{\RR-c}(\Bbbk_{\che{M}})\to\BDC_{\I\RR-c}(\I\Bbbk_{\che{M}})$
is well defined,
we have $\bfR j_{M_\infty!!}\iota_{M_\infty}\SF\in\BDC_{\I\RR-c}(\I\Bbbk_{\che{M}})$.
This implies that $\iota_{M_\infty}\SF\in\BDC_{\I\RR-c}(\I\Bbbk_{M_\infty})$.
\bigskip

\noindent
(2)
Let $\SF\in\BDC(\Bbbk_{M})$.
Then we have $j_{M!}\SF\in\BDC(\Bbbk_{\che{M}})$.
Note that there exists an isomorphism 
$\bfR j_{M_\infty!!}\beta_{M_\infty}\SF \simeq \beta_{\che{M}}j_{M!}\SF$
in $\BDC(\I\Bbbk_{\che{M}})$.
By \cite[Lem.\:7.1.3 (v)]{KS01},
we have $\beta_{\che{M}}j_{M!}\SF\in \BDC_{\I\RR-c}(\Bbbk_{\che{M}})$,
and hence $\bfR j_{M_\infty!!}\beta_{M_\infty}\SF\in \BDC_{\I\RR-c}(\Bbbk_{\che{M}})$.
This implies that $\beta_{M_\infty}\SF\in\BDC_{\I\RR-c}(\I\Bbbk_{M_\infty})$.
\bigskip

\noindent
(3)
Let $F_1, F_2\in \BDC_{\I\RR-c}(\I\Bbbk_{M_\infty})$.
Then we have $\bfR j_{M_\infty!!}F_1, \bfR j_{M_\infty!!}F_2\in\BDC_{\I\RR-c}(\I\Bbbk_{\che{M}})$.
Note that there exists an isomorphism
$\bfR j_{M_\infty!!}(F_1\otimes F_2)\simeq\bfR j_{M_\infty!!}F_1\otimes\bfR j_{M_\infty!!}F_2$.
By \cite[Lem.\:7.1.3 (iii)]{KS01},
we have $\bfR j_{M_\infty!!}F_1\otimes\bfR j_{M_\infty!!}F_2\in \BDC_{\I\RR-c}(\Bbbk_{\che{M}})$,
and hence $\bfR j_{M_\infty!!}(F_1\otimes F_2)\in \BDC_{\I\RR-c}(\Bbbk_{\che{M}})$.
This implies that $F_1\otimes F_2\in\BDC_{\I\RR-c}(\I\Bbbk_{M_\infty})$.
\bigskip

\noindent
(4)
Let $\SF\in\BDC_{\RR-c}(\Bbbk_{M_\infty})$ and $G\in \BDC_{\I\RR-c}(\I\Bbbk_{M_\infty})$.
Then we have $j_{M!}\SF\in\BDC_{\RR-c}(\Bbbk_{\che{M}})$ 
and $\bfR j_{M_\infty!!}G\in\BDC_{\I\RR-c}(\I\Bbbk_{\che{M}})$.
Note that there exist isomorphisms in $\BDC(\I\Bbbk_{\che{M}})$
\begin{align*}
\bfR j_{M_\infty!!}\rihom(\iota_{M_\infty}\SF, G)
&\simeq
\bfR j_{M_\infty!!}j_{M_\infty}^{-1}\bfR j_{M_\infty\ast}
\rihom(\iota_{M_\infty}\SF, j_{M_\infty}^{-1}\bfR j_{M_\infty!!}G)\\
&\simeq
\bfR j_{M_\infty!!}j_{M_\infty}^{-1}
\rihom(\bfR j_{M_\infty!!}\iota_{M_\infty}\SF, \bfR j_{M_\infty!!}G)\\
&\simeq
\iota_{\che{M}}\Bbbk_M\otimes
\rihom(\iota_{\che{M}}j_{M!}\SF, \bfR j_{M_\infty!!}G),
\end{align*}
where the first isomorphism follows from
$j_{M_\infty}^{-1}\bfR j_{M_\infty\ast}\simeq\id$ and $\id\simeq j_{M_\infty}^{-1}\bfR j_{M_\infty!!}$,
in the second isomorphism we used the adjointness of $(\bfR j_{M_\infty!!}, j_{M_\infty}^{-1})$
and the last isomorphism follows from
$\bfR j_{M_\infty!!}j_{M_\infty}^{-1}(\cdot)\simeq\iota_{\che{M}}\Bbbk_M\otimes(\cdot)$.
By \cite[Lem.\:7.1.3 (iii), (iv)]{KS01},
we have $\iota_{\che{M}}\Bbbk_M\otimes
\rihom(\iota_{\che{M}}j_{M!}\SF, \bfR j_{M_\infty!!}G)\in \BDC_{\I\RR-c}(\Bbbk_{\che{M}})$,
and hence $\bfR j_{M_\infty!!}\rihom(\iota_{M_\infty}\SF, G)\in \BDC_{\I\RR-c}(\Bbbk_{\che{M}})$.
This implies that $\rihom(\iota_{M_\infty}\SF, G)\in\BDC_{\I\RR-c}(\I\Bbbk_{M_\infty})$.
\bigskip

\noindent
(5)
Let $G\in \BDC_{\I\RR-c}(\I\Bbbk_{N_\infty})$.
Then we have $\bfR j_{N_\infty!!}G\in\BDC_{\I\RR-c}(\I\Bbbk_{\che{N}})$.
Note that there exists an isomorphism 
$\bfR j_{M_\infty!!}f^{-1}G \simeq \che{f}^{-1}\bfR j_{N_\infty!!}G$
in $\BDC(\I\Bbbk_{\che{N}})$.
Since the functor $\che{f}\colon \BDC(\I\Bbbk_{\che{N}})\to \BDC(\I\Bbbk_{\che{M}})$
is well defined,
we have $\bfR j_{M_\infty!!}f^{-1}G\in \BDC_{\I\RR-c}(\I\Bbbk_{\che{M}})$.
This implies that $f^{-1}G\in \BDC_{\I\RR-c}(\I\Bbbk_{M_\infty})$.
\bigskip

\noindent
(6)
Let $F\in \BDC_{\I\RR-c}(\I\Bbbk_{M_\infty})$.
Then we have $\bfR j_{M_\infty!!}F\in\BDC_{\I\RR-c}(\I\Bbbk_{\che{M}})$.
Note that there exists an isomorphism 
$\bfR j_{N_\infty!!}\bfR f_{!!}F \simeq \che{f}_{!!}\bfR j_{M_\infty!!}F$
in $\BDC(\I\Bbbk_{\che{M}})$.
By \cite[Lem.\:7.1.3 (i)]{KS01},
we have $\bfR j_{N_\infty!!}\bfR f_{!!}F\in \BDC_{\I\RR-c}(\I\Bbbk_{\che{N}})$.
This implies that $\bfR f_{!!}F\in \BDC_{\I\RR-c}(\I\Bbbk_{N_\infty})$.

\bigskip

\noindent
(7)
Let $G\in \BDC_{\I\RR-c}(\I\Bbbk_{N_\infty})$.
Then we have $\bfR j_{N_\infty!!}G\in\BDC_{\I\RR-c}(\I\Bbbk_{\che{N}})$.
Note that there exist isomorphism in $\BDC(\I\Bbbk_{\che{M}})$
\begin{align*}
\bfR j_{M_\infty!!}f^{!}G \simeq 
\bfR j_{M_\infty!!}j_{M_\infty}^{-1}\che{f}^{!}\bfR j_{N_\infty!!}G
\simeq
\iota_{\che{M}}\Bbbk_M\otimes \che{f}^{!}\bfR j_{N_\infty!!}G
\end{align*}
where in the last isomorphism we used 
$\bfR j_{M_\infty!!}j_{M_\infty}^{-1}\simeq\iota_{\che{M}}\Bbbk_M\otimes(\cdot)$.
By \cite[Lem.\:7.1.3 (ii), (iii)]{KS01} and the fact that $\iota_{\che{M}}\Bbbk_M\in\BDC_{\I\RR-c}(\I\Bbbk_{\che{M}})$,
we have $\iota_{\che{M}}\Bbbk_M\otimes \che{f}^{!}\bfR j_{N_\infty!!}G\in \BDC_{\I\RR-c}(\I\Bbbk_{\che{M}})$,
and hence $\bfR j_{M_\infty!!}f^{!}G\in \BDC_{\I\RR-c}(\I\Bbbk_{\che{M}})$.
This implies that $f^{!}G\in \BDC_{\I\RR-c}(\I\Bbbk_{M_\infty})$.
\end{proof}

Let us describe a relation between ind-sheaves on $M_\infty$ and subanalytic sheaves on $M_\infty$.
\begin{proposition}\label{prop3.6}
Let $M_\infty = (M, \che{M})$ be a real analytic bordered space.
Then we have
\begin{itemize}
\item[\rm (1)]
a pair $(I_{M_\infty}, \bfR J_{M_\infty})$ is an adjoint pair
and there exists a canonical isomorphism $\id\simto\bfR J_{M_\infty}\circ I_{M_\infty}$,

\item[\rm (2)]
there exists an equivalence of triangulated categories:
\[\xymatrix@M=7pt@C=45pt{
\BDC(\Bbbk_{M_\infty}^\sub)\ar@<0.8ex>@{->}[r]^-{I_{M_\infty}}_-\sim
&
\BDC_{\I{\RR-c}}(\I\Bbbk_{M_\infty})
\ar@<0.8ex>@{->}[l]^-{\bfR J_{M_\infty}}.
}\]
\end{itemize}
\end{proposition}

\begin{proof}
(1)
Let $\SF\in\BDC(\Bbbk_{M_\infty}^\sub)$ and $G\in\BDC(\I\Bbbk_{M_\infty})$.
Then we have
\begin{align*}
\Hom_{\BDC(\I\Bbbk_{M_\infty})}(I_{M_\infty}\SF, G)
&=
\Hom_{\BDC(\I\Bbbk_{M_\infty})}(\q I_{\che{M}}\bfR j_{M_\infty!!}\SF, G)\\
&\simeq
\Hom_{\BDC(\I\Bbbk_{\che{M}})}(I_{\che{M}}\bfR j_{M_\infty!!}\SF, \bfR j_{M_\infty\ast}G)\\
&\simeq
\Hom_{\BDC(\Bbbk_{\che{M}}^\sub)}(\bfR j_{M_\infty!!}\SF, \bfR J_{\che{M}}\bfR j_{M_\infty\ast}G)\\
&\simeq
\Hom_{\BDC(\Bbbk_{M_\infty}^\sub)}(\SF, \bfR j_{M_\infty}^{-1}\bfR J_{\che{M}}\bfR j_{M_\infty\ast}G)\\
&\simeq
\Hom_{\BDC(\Bbbk_{M_\infty}^\sub)}(\SF, \bfR J_{M_\infty}G),
\end{align*}
where in the first isomorphism we used the fact that $\q=j_{M_\infty}^{-1}$
and a pair $(j_{M_\infty}^{-1}, \bfR j_{M_\infty\ast})$ is an adjoint pair,
the second isomorphism follows from a pair $(I_{\che{M}}, \bfR J_{\che{M}})$ is an adjoint pair
and the third isomorphism follows from a pair $(\bfR j_{M_\infty!!}, j_{M_\infty}^{-1})$ is an adjoint pair.
This implies that a pair $(I_{M_\infty}, \bfR J_{M_\infty})$ is an adjoint pair.

Hence, there exists a natural morphism $\id\to \bfR J_{M_\infty}\circ I_{M_\infty}$ of functors.
Moreover, for any $\SF\in\BDC(\I\Bbbk_{M_\infty})$,
we have isomorphisms in $\BDC(\I\Bbbk_{M_\infty})$
\begin{align*}
(\bfR J_{M_\infty}\circ I_{M_\infty})(\SF)
&\simeq
j_{M_\infty}^{-1}\bfR J_{\che{M}}\bfR j_{M_\infty\ast}\q I_{\che{M}}\bfR j_{M_\infty!!}\SF\\
&\simeq
j_{M_\infty}^{-1}\bfR J_{\che{M}}\bfR j_{M_\infty\ast}j_{M_\infty}^{-1} I_{\che{M}}\bfR j_{M_\infty!!}\SF\\
&\simeq
j_{M_\infty}^{-1}\bfR J_{\che{M}}\rihom(\iota_{\che{M}}\Bbbk_{M}, I_{\che{M}}\bfR j_{M_\infty!!}\SF)\\
&\simeq
j_{M_\infty}^{-1}\bfR J_{\che{M}}\rihom(I_{\che{M}}\rho_{\che{M}}\Bbbk_{M}, I_{\che{M}}\bfR j_{M_\infty!!}\SF)\\
&\simeq
j_{M_\infty}^{-1}\rihom^\sub(\rho_{\che{M}}\Bbbk_{M}, \bfR J_{\che{M}}I_{\che{M}}\bfR j_{M_\infty!!}\SF)\\
&\simeq
j_{M_\infty}^{-1}\rihom^\sub(\rho_{\che{M}}\Bbbk_{M}, \bfR j_{M_\infty!!}\SF)\\
&\simeq
j_{M_\infty}^{-1}\bfR j_{M_\infty\ast}j_{M_\infty}^{-1}\bfR j_{M_\infty!!}\SF\\
&\simeq\SF
\end{align*}
where the second isomorphism follows from $\q = j_{M_\infty}^{-1}$,
in the third isomorphism we used the fact that
$\bfR j_{M_\infty\ast}j_{M_\infty}^{-1}(\cdot)\simeq\rihom(\iota_{\che{M}}\Bbbk_M, \cdot)$,
the fourth isomorphism follows from
$I_{\che{M}}\circ\rho_{\che{M}}^{\RR-c} = \iota_{\che{M}}|_{\Mod_{\RR-c}(\Bbbk_{\che{M}})}$,
the fifth isomorphism follows from the adjointness of $(I_{\che{M}}, \bfR J_{\che{M}})$
and in the sixth we used the fact $\id\simto \bfR J_{\che{M}}\circ I_{\che{M}}$.
\bigskip

\noindent
(2)
First, let us prove that
for any $\SF\in\BDC(\Bbbk_{M_\infty}^\sub)$ one has $I_{M_\infty}(\SF)\in\BDC_{\I\RR-c}(\I\Bbbk_{M_\infty})$.
There exist isomorphisms in $\BDC(\I\Bbbk_{\che{M}})$
\begin{align*}
\bfR j_{M_\infty!!}I_{M_\infty}(\SF)
&\simeq
\bfR j_{M_\infty!!}\q I_{\che{M}}\bfR j_{M_\infty!!}\SF\\
&\simeq
\iota_{\che{M}}\Bbbk_M\otimes I_{\che{M}}\bfR j_{M_\infty!!}\SF\\
&\simeq
I_{\che{M}}\rho_{\che{M}}\Bbbk_M\otimes I_{\che{M}}\bfR j_{M_\infty!!}\SF\\
&\simeq
I_{\che{M}}(\rho_{\che{M}}\Bbbk_M\otimes \bfR j_{M_\infty!!}\SF)
\end{align*}
where the third isomorphism follows from
$I_{\che{M}}\circ\rho_{\che{M}}^{\RR-c} = \iota_{\che{M}}|_{\Mod_{\RR-c}(\Bbbk_{\che{M}})}$
and in the fourth isomorphism we used the fact that
$I_{\che{M}}(\cdot\otimes\cdot)\simeq I_{\che{M}}(\cdot)\otimes I_{\che{M}}(\cdot)$.
Since $I_{\che{M}}(\rho_{\che{M}}\Bbbk_M\otimes \bfR j_{M_\infty!!}\SF)\in \BDC_{\I\RR-c}(\I\Bbbk_{\che{M}})$,
we have $\bfR j_{M_\infty!!}I_{M_\infty}(\SF)\in \BDC_{\I\RR-c}(\I\Bbbk_{\che{M}})$.
This implies $I_{M_\infty}(\SF)\in\BDC_{\I\RR-c}(\I\Bbbk_{M_\infty})$.

By (1), the functor $I_{M_\infty}$ is fully faithful.
Let us prove that the functor $I_{M_\infty}$ is essentially surjective.
Let $G\in\BDC_{\I{\RR-c}}(\I\Bbbk_{M_\infty})$.
Then we have $\bfR j_{M_\infty!!}G\in\BDC_{\I{\RR-c}}(\I\Bbbk_{\che{M}})$. 
By Theorem \ref{thm2.26}, there exists $\SF\in \BDC(\Bbbk_{\che{M}}^\sub)$
such that $\bfR j_{M_\infty!!}G\simeq I_{\che{M}}\SF$
and hence we have $G\simeq j_{M_\infty}^{-1}I_{\che{M}}\SF$.
Moreover, there exist isomorphisms in $\BDC(\I\Bbbk_{M_\infty})$
\begin{align*}
j_{M_\infty}^{-1}I_{\che{M}}\bfR j_{M_\infty!!}j_{M_\infty}^{-1}\SF
&\simeq
j_{M_\infty}^{-1}I_{\che{M}}(\rho_{\che{M}}\Bbbk_M\otimes\SF)\\
&\simeq
j_{M_\infty}^{-1}(I_{\che{M}}\rho_{\che{M}}\Bbbk_M\otimes I_{\che{M}}\SF)\\
&\simeq
j_{M_\infty}^{-1}(\iota_{\che{M}}\Bbbk_M\otimes I_{\che{M}}\SF)\\
&\simeq
j_{M_\infty}^{-1}\bfR j_{M_\infty!!}j_{M_\infty}^{-1}I_{\che{M}}\SF\\
&\simeq
j_{M_\infty}^{-1}I_{\che{M}}\SF
\end{align*}
where in the first isomorphism we used
$(\bfR j_{M_\infty!!}\circ j_{M_\infty}^{-1})(\cdot) \simeq \rho_{\che{M}\ast}\Bbbk_M\otimes(\cdot)$,
in the second isomorphism we used the fact that
$I_{\che{M}}(\cdot\otimes\cdot)\simeq I_{\che{M}}(\cdot)\otimes I_{\che{M}}(\cdot)$,
the third isomorphism follows from
$I_{\che{M}}\circ\rho_{\che{M}}^{\RR-c} = \iota_{\che{M}}|_{\Mod_{\RR-c}(\Bbbk_{\che{M}})}$
and in the fourth isomorphism we used the fact that
$(\bfR j_{M_\infty!!}\circ j_{M_\infty}^{-1})(\cdot)\simeq \iota_{\che{M}}\Bbbk_{M}\otimes(\cdot)$.
Hence we have 
$$G\simeq j_{M_\infty}^{-1}I_{\che{M}}\SF
\simeq
j_{M_\infty}^{-1}I_{\che{M}}\bfR j_{M_\infty!!}j_{M_\infty}^{-1}\SF
\simeq
I_{M_\infty}(j_{M_\infty}^{-1}\SF).$$
This implies that the functor $I_{M_\infty}$ is essentially surjective.

Therefore, the proof is completed.
\end{proof}

We will denote by $$\lambda_{M_\infty}\colon \BDC_{\I{\RR-c}}(\I\Bbbk_{M_\infty})\simto \BDC(\Bbbk_{M_\infty}^\sub)$$
the inverse functor of $I_{M_\infty}\colon \BDC(\Bbbk_{M_\infty}^\sub)
 \simto \BDC_{\I{\RR-c}}(\I\Bbbk_{M_\infty})$.
\begin{proposition}\label{prop3.7}
Let $f\colon M_\infty\to N_\infty$ be a morphism of real analytic bordered spaces
associated with a morphism $\che{f}\colon\che{M}\to\che{N}$ of real analytic manifolds.
\begin{itemize}
\item[\rm (1)]
For any $\SF\in\BDC(\Bbbk_{M_\infty}^\sub)$ and any $G\in\BDC(\I\Bbbk_{M_\infty})$,
we have
$$\bfR J_{M_\infty}\rihom(I_{M_\infty}\SF, G)
\simeq
\rihom^\sub(\SF, \bfR J_{M_\infty}G).$$

\item[\rm (2)]
For any $\SL\in\BDC(\Bbbk_M)$, any $\SF, \SF_1, \SF_2\in\BDC(\Bbbk_{M_\infty}^\sub)$ and any $\SG\in\BDC(\Bbbk_{N_\infty}^\sub)$,
we have
\begin{itemize}
\item[\rm (i)]
$\alpha_{M_\infty}I_{M_\infty}\SF\simeq \rho_{M_\infty}^{-1}\SF$,
\item[\rm (ii)]
$I_{M_\infty}\rho_{M_\infty!}\SL\simeq\beta_{M_\infty}\SL$,
\item[\rm (iii)]
$I_{M_\infty}f^{-1}\SG\simeq f^{-1}I_{N_\infty}\SG$,
\item[\rm (iv)]
$\bfR f_{!!}I_{M_\infty}\SF\simeq I_{N_\infty}\bfR f_{!!}\SF$,
\item[\rm (v)]
$I_{M_\infty}f^{!}\SG\simeq f^{!}I_{N_\infty}\SG$,
\item[\rm (vi)]
$I_{M_\infty}(\SF_1\otimes \SF_2) \simeq I_{M_\infty}(\SF_1)\otimes I_{M_\infty}(\SF_2)$.
\end{itemize}

\item[\rm (3)]
For any $\SL\in\BDC(\Bbbk_{M})$, any $F\in\BDC(\I\Bbbk_{M_\infty})$ and any $G\in\BDC(\I\Bbbk_{N_\infty})$,
we have
\begin{itemize}
\item[\rm (i)]
$\bfR J_{M_\infty}\iota_{M_\infty}\SL\simeq \rho_{M_\infty\ast}\SL$,
\item[\rm (ii)]
$\rho_{M_\infty}^{-1}\bfR J_{M_\infty}F\simeq \alpha_{M_\infty}F$,
\item[\rm (iii)]
$\bfR J_{M_\infty}\beta_{M_\infty}\SL\simeq \rho_{M_\infty!}\SL$,
\item[\rm (iv)]
$\bfR f_{\ast}\bfR J_{M_\infty}F\simeq \bfR J_{N_\infty}\bfR f_{\ast}F$,
\item[\rm (v)]
$\bfR J_{M_\infty}f^{!}G\simeq f^{!}\bfR J_{N_\infty}G$.
\end{itemize}

\item[\rm (4)]
For any $F, F_1, F_2\in\BDC_{\I\RR-c}(\I\Bbbk_{M_\infty})$, any $G\in\BDC_{\I\RR-c}(\I\Bbbk_{N_\infty})$
and any $\SL\in\BDC_{\RR-c}(\Bbbk_{M_\infty})$,
we have
\begin{itemize}
\item[\rm (i)]
$I_{M_\infty}\rho_{M_\infty\ast}^{\RR-c}\SL\simeq \iota_{M_\infty}\SL$,
\item[\rm (ii)]
$\lambda_{M_\infty}f^{-1}G\simeq f^{-1}\lambda_{N_\infty}G$,
\item[\rm (iii)]
$\bfR f_{!!}\lambda_{M_\infty}F\simeq \lambda_{N_\infty}\bfR f_{!!}F$,
\item[\rm (iv)]
$\lambda_{M_\infty}(F_1\otimes F_2) \simeq \lambda_{M_\infty}(F_1)\otimes \lambda_{M_\infty}(F_2)$.
\end{itemize}
\end{itemize}
\end{proposition}

\begin{proof}
(1)
Let $\SF\in\BDC(\Bbbk_{M_\infty}^\sub)$ and $G\in\BDC(\I\Bbbk_{M_\infty})$.
Then we have
\begin{align*}
\bfR J_{M_\infty}\rihom(I_{M_\infty}\SF, G) &= 
j_{M_\infty}^{-1}\bfR J_{\che{M}}\bfR j_{M_\infty\ast}
\rihom(j_{M_\infty}^{-1}I_{\che{M}}\bfR j_{M_\infty!!}\SF, G)\\
&\simeq
j_{M_\infty}^{-1}\rihom^\sub(\bfR j_{M_\infty!!}\SF, \bfR J_{\che{M}}\bfR j_{M_\infty\ast}G)\\
&\simeq
\rihom^\sub(j_{M_\infty}^{-1}\bfR j_{M_\infty!!}\SF, j_{M_\infty}^{-1}\bfR J_{\che{M}}\bfR j_{M_\infty\ast}G)\\
&\simeq
\rihom^\sub(\SF, \bfR J_{M_\infty}G),
\end{align*}
where in the second isomorphism we used the adjointness of
$(j_{M_\infty}^{-1}, \bfR j_{M_\infty\ast})$ and $(I_{\che{M}}, \bfR j_{\che{M}})$
and the last isomorphism follows from $j_{M_\infty}^{-1}\circ\bfR j_{M_\infty!!}\simeq\id$.
\bigskip

\noindent
(2)(i)
Let $\SF\in \BDC(\Bbbk_{M_\infty}^\sub)$.
Then we have 
\begin{align*}
\alpha_{M_\infty}I_{M_\infty}\SF
&\simeq
j_M^{-1}\alpha_{\che{M}}\bfR j_{M_\infty!!}j_{M_\infty}^{-1}I_{\che{M}}\bfR j_{M_\infty!!}\SF\\
&\simeq
j_M^{-1}\alpha_{\che{M}}(\iota_{\che{M}}\Bbbk_M\otimes I_{\che{M}}\bfR j_{M_\infty!!}\SF)\\
&\simeq
j_M^{-1}\alpha_{\che{M}}\iota_{\che{M}}\Bbbk_M\otimes
j_M^{-1}\alpha_{\che{M}}I_{\che{M}}\bfR j_{M_\infty!!}\SF\\
&\simeq
j_M^{-1}\rho_{\che{M}}^{-1}\bfR j_{M_\infty!!}\SF\\
&\simeq
\rho_{M_\infty}^{-1}j_{M_\infty}^{-1}\bfR j_{M_\infty!!}\SF\\
&\simeq \rho_{M_\infty}^{-1}\SF,
\end{align*}
where in the second isomorphism we used the fact that
$\bfR j_{M_\infty!!}j_{M_\infty}^{-1}\simeq\iota_{\che{M}}\Bbbk_M\otimes(\cdot)$
and the fourth isomorphism follows from 
$\alpha_{\che{M}}\circ I_{\che{M}}\simeq\rho_{\che{M}}^{-1}$.
\medskip

\noindent
(ii)
Let $\SL\in \BDC(\Bbbk_{M})$.
Since $I_{\che{M}}\circ\rho_{\che{M}!}\simeq\beta_{\che{M}}$,
we have 
\begin{align*}
I_{M_\infty}\rho_{M_\infty!}\SL
&\simeq
j_{M_\infty}^{-1}I_{\che{M}}\bfR j_{M_\infty!!}\rho_{M_\infty!}\SL\\
&\simeq
j_{M_\infty}^{-1}I_{\che{M}}\rho_{\che{M}!}j_{M!}\SL\\
&\simeq
j_{M_\infty}^{-1}\beta_{\che{M}}j_{M!}\SL\\
&\simeq\beta_{M_\infty}\SF.
\end{align*}
\medskip

\noindent
(iii)
Let $\SG\in\BDC(\Bbbk_{N_\infty}^\sub)$.
Since $I_{\che{M}}\circ\che{f}^{-1}\simeq\che{f}^{-1}\circ I_{\che{N}}$,
we have 
\begin{align*}
I_{M_\infty}f^{-1}\SG
&\simeq
j_{M_\infty}^{-1}I_{\che{M}}\bfR j_{M_\infty!!}f^{-1}\SG\\
&\simeq
j_{M_\infty}^{-1}I_{\che{M}}\che{f}^{-1}\bfR j_{N_\infty!!}\SG\\
&\simeq
j_{M_\infty}^{-1}\che{f}^{-1}I_{\che{N}}\bfR j_{N_\infty!!}\SG\\
&\simeq
{f}^{-1}j_{N_\infty}^{-1}I_{\che{N}}\bfR j_{N_\infty!!}\SG\\
&\simeq f^{-1}I_{N_\infty}\SG.
\end{align*}
\medskip

\noindent
(iv)
Let $\SF\in\BDC(\Bbbk_{M_\infty}^\sub)$.
Since $I_{\che{N}}\circ\che{f}_{!!}\simeq\che{f}_{!!}\circ I_{\che{M}}$,
we have 
\begin{align*}
\bfR f_{!!}I_{M_\infty}\SF
&\simeq
\bfR f_{!!}j_{M_\infty}^{-1}I_{\che{M}}\bfR j_{M_\infty!!}\SF\\
&\simeq
j_{N_\infty}^{-1}\bfR \che{f}_{!!}I_{\che{M}}\bfR j_{M_\infty!!}\SF\\
&\simeq
j_{N_\infty}^{-1}I_{\che{N}}\bfR \che{f}_{!!}\bfR j_{M_\infty!!}\SF\\
&\simeq
j_{N_\infty}^{-1}I_{\che{N}}\bfR j_{N_\infty!!}\bfR{f}_{!!}\SF\\
&\simeq
I_{N_\infty}\bfR{f}_{!!}\SF.
\end{align*}
\medskip

\noindent
(v)
Let $G\in\BDC(\Bbbk_{N_\infty}^\sub)$.
By Lemma \ref{lem3.5} (7) we have $f^!I_{N_\infty}G\in\BDC_{\I\RR-c}(\I\Bbbk_{M_\infty})$
and hence by Proposition \ref{prop3.6} (2) we have  
$I_{M_\infty}\bfR J_{M_\infty} f^!I_{N_\infty}G\simeq f^!I_{N_\infty}G$.
Then we have
\begin{align*}
f^!I_{N_\infty}G
\simeq
I_{M_\infty}\bfR J_{M_\infty} f^!I_{N_\infty}G
\simeq
I_{M_\infty}f^!\bfR J_{N_\infty}I_{N_\infty}G
\simeq
I_{M_\infty}f^!G,
\end{align*}
where the second isomorphism follows from (3)(v)
and in the last isomorphism we used $\bfR J_{M_\infty}\circ I_{M_\infty}\simeq\id$.
\medskip

\noindent
(vi)
Let $\SF_1, \SF_2\in\BDC(\Bbbk_{M_\infty}^\sub)$.
Then we have
$$\bfR j_{M_\infty!!}(\SF_1\otimes \SF_2)
\simeq
\bfR j_{M_\infty!!}(j_{M_\infty}^{-1}\bfR j_{M_\infty!!}\SF_1\otimes \SF_2)
\simeq
\bfR j_{M_\infty!!}\SF_1\otimes\bfR j_{M_\infty!!}\SF_2,$$
and hence we have
$$I_{\che{M}}\bfR j_{M_\infty!!}(\SF_1\otimes \SF_2)
\simeq
I_{\che{M}}\bfR j_{M_\infty!!}\SF_1\otimes I_{\che{M}}\bfR j_{M_\infty!!}\SF_2.$$
So that, we have 
\begin{align*}
I_{M_\infty}(\SF_1\otimes \SF_2)
&\simeq
j_{M_\infty}^{-1}I_{\che{M}}\bfR j_{M_\infty!!}(\SF_1\otimes \SF_2)\\
&\simeq
j_{M_\infty}^{-1}I_{\che{M}}\bfR j_{M_\infty!!}\SF_1
\otimes j_{M_\infty}^{-1}I_{\che{M}}\bfR j_{M_\infty!!}\SF_2\\
&\simeq 
I_{M_\infty}\SF_1\otimes I_{M_\infty}\SF_2.
\end{align*}
\bigskip

\noindent
(3)(i)
Let $F\in\BDC(\I\Bbbk_{M_\infty})$.
Since $\bfR J_{\che{M}}\circ\iota_{\che{M}}\simeq \rho_{\che{M}\ast}$,
we have
\begin{align*}
\bfR J_{M_\infty}\iota_{M_\infty}F
&\simeq
j_{M_\infty}^{-1}\bfR J_{\che{M}}\bfR j_{M_\infty\ast}\iota_{M_\infty}F\\
&\simeq
j_{M_\infty}^{-1}\bfR J_{\che{M}}\iota_{\che{M}}\bfR j_{M_\infty\ast}F\\
&\simeq
j_{M_\infty}^{-1}\rho_{\che{M}\ast}\bfR j_{M_\infty\ast}F\\
&\simeq
\rho_{M_\infty\ast}j_{M_\infty}^{-1}\bfR j_{M_\infty\ast}F\\
&\simeq \rho_{M_\infty\ast}F.
\end{align*}
\medskip

\noindent
(ii)
Let $F\in\BDC(\I\Bbbk_{M_\infty})$.
Since $\rho_{\che{M}}^{-1}\circ\bfR J_{\che{M}}\simeq \alpha_{\che{M}}$,
we have
\begin{align*}
\rho_{M_\infty}^{-1}\bfR J_{M_\infty}F
&\simeq
\rho_{M_\infty}^{-1}j_{M_\infty}^{-1}\bfR J_{\che{M}}\bfR j_{M_\infty\ast}F\\
&\simeq
j_{M_\infty}^{-1}\rho_{\che{M}}^{-1}\bfR J_{\che{M}}\bfR j_{M_\infty\ast}F\\
&\simeq
j_{M_\infty}^{-1}\alpha_{\che{M}}\bfR j_{M_\infty\ast}F\\
&\simeq
\alpha_{M_\infty}j_{M_\infty}^{-1}\bfR j_{M_\infty\ast}F\\
&\simeq\alpha_{M_\infty}F.
\end{align*}
\medskip

\noindent
(iii)
Let $\SL\in\BDC(\Bbbk_{M})$.
By using (2)(ii) and the fact that $\bfR J_{M_\infty}\circ I_{M_\infty}\simeq\id$,
we have
$$\bfR J_{M_\infty}\beta_{M_\infty}\SL\simeq
\bfR J_{M_\infty}I_{M_\infty}\rho_{M_\infty!}\SL\simeq
 \rho_{M_\infty!}F.$$
 \medskip

\noindent
(iv)
Let $F\in\BDC(\I\Bbbk_{M_\infty})$ and $\G\in\BDC(\Bbbk_{N_\infty}^\sub)$.
Then we have
\begin{align*}
\Hom_{\BDC(\I\Bbbk_{N_\infty})}(\G, \bfR f_\ast\bfR J_{M_\infty}F)
&\simeq
\Hom_{\BDC(\I\Bbbk_{M_\infty})}(I_{M_\infty}f^{-1}\G, F)\\
&\simeq
\Hom_{\BDC(\I\Bbbk_{M_\infty})}(f^{-1}I_{N_\infty}\G, F)\\
&\simeq
\Hom_{\BDC(\I\Bbbk_{N_\infty})}(\G, \bfR J_{N_\infty}\bfR f_\ast F),
\end{align*}
where the first and last isomorphisms follow from
adjointness of $(f^{-1}, \bfR f_{\ast})$ and $(I, \bfR J)$
and in the second isomorphism we used (2)(iii). 
Hence there exists an isomorphism
$\bfR f_\ast\bfR J_{M_\infty}F\simeq
\bfR J_{N_\infty}\bfR f_\ast F$
 in $\BDC(\I\Bbbk_{N_\infty})$.
 \medskip

\noindent
(v)
Let $G\in\BDC(\I\Bbbk_{N_\infty})$.
Since $\bfR J_{\che{M}}\circ \che{f}^!\simeq \che{f}^!\circ\bfR J_{\che{N}}$,
we have
\begin{align*}
\bfR J_{M_\infty}f^!G
&\simeq
j_{M_\infty}^{-1}\bfR J_{\che{M}}\bfR j_{M_\infty\ast}f^!G\\
&\simeq
j_{M_\infty}^{-1}\bfR J_{\che{M}}\che{f}^!\bfR j_{N_\infty\ast}G\\
&\simeq
j_{M_\infty}^{-1}\che{f}^!\bfR J_{\che{N}}\bfR j_{N_\infty\ast}G\\
&\simeq
f^!j_{N_\infty}^{-1}\bfR J_{\che{N}}\bfR j_{N_\infty\ast}G\\
&\simeq
f^!\bfR J_{N_\infty}G.
\end{align*}
\bigskip

\noindent
(4)(i)
Let $\SL\in\BDC_{\RR-c}(\Bbbk_{M_\infty})$.
Then we have $\iota_{M_\infty}\SL\in\BDC_{\I\RR-c}(\I\Bbbk_{M_\infty})$ by Lemma \ref{lem3.5} (1).
Hence, by using (3)(i) and Proposition \ref{prop3.6} (2), we have
$$\iota_{M_\infty}\SL \simeq I_{M_\infty}\bfR J_{M_\infty}\iota_{M_\infty}\SL
\simeq I_{M_\infty}\rho_{M_\infty\ast}\SL.$$
\medskip

\noindent
(ii)
Let $G\in\BDC_{\I\RR-c}(\I\Bbbk_{N_\infty})$.
Then we have $G\simeq I_{N_\infty}\lambda_{N_\infty}G$ by Proposition \ref{prop3.6} (2).
Moreover, by using (2)(iii)
we have 
$$\lambda_{M_\infty}f^{-1}G\simeq \lambda_{M_\infty}f^{-1}I_{N_\infty}\lambda_{N_\infty}G
\simeq
\lambda_{M_\infty}I_{M_\infty}f^{-1}\lambda_{N_\infty}G
\simeq
f^{-1}\lambda_{N_\infty}G.
$$
\medskip

\noindent
(iii)
Let $F\in\BDC_{\I\RR-c}(\I\Bbbk_{M_\infty})$.
Then we have $F\simeq I_{M_\infty}\lambda_{M_\infty}F$ by Proposition \ref{prop3.6} (2).
Hence, by using (2)(iv)
we have 
$$\lambda_{N_\infty}f_{!!}F\simeq \lambda_{N_\infty}f_{!!}I_{M_\infty}\lambda_{M_\infty}F
\simeq
\lambda_{N_\infty}I_{N_\infty}f_{!!}\lambda_{M_\infty}F
\simeq
f_{!!}\lambda_{M_\infty}F.
$$
\medskip

\noindent
(iv)
Let $F_1, F_2\in\BDC_{\I\RR-c}(\I\Bbbk_{M_\infty})$.
Then we have $F_i\simeq I_{M_\infty}\lambda_{M_\infty}F_i\ (i=1,2)$ by Proposition \ref{prop3.6} (2).
Hence, by using (2)(vi)
we have 
\begin{align*}
\lambda_{M_\infty}(F_1\otimes F_2)
&\simeq 
\lambda_{M_\infty}(I_{M_\infty}\lambda_{M_\infty}F_1\,\otimes\,I_{M_\infty}\lambda_{M_\infty}F_2)\\
&\simeq
\lambda_{M_\infty}I_{M_\infty}(\lambda_{M_\infty}F_1\,\otimes\,\lambda_{M_\infty}F_2)\\
&\simeq
\lambda_{M_\infty}F_1\otimes \lambda_{M_\infty}F_2.
\end{align*}
\end{proof}

\subsection{Convolutions for Subanalytic Sheaves on Real Analytic Bordered Spaces}
In this subsection, let us define convolution functors
for subanalytic sheaved on real analytic bordered spaces.
Although it has already been explained in \cite[\S 5.1]{Kas16},
we will explain in detail, in this subsection again\footnote{In \cite{Kas16},
the author introduced convolution functors for subanalytic sheaves on subanalytic bordered spaces.
In this paper, we shall only consider them on real analytic bordered spaces.}.

Let $M_\infty = (M, \che{M})$ be a real analytic bordered space.
We set $\RR_\infty := (\RR, \var{\RR})$ for 
$\var{\RR} := \RR\sqcup\{-\infty, +\infty\}$,
and let $t\in\RR$ be the affine coordinate. 
We consider the morphisms of real analytic bordered spaces
\[M_\infty\times\RR^2_\infty\xrightarrow{p_1,\ p_2,\ \mu}M_\infty
\times\RR_\infty\overset{\pi}{\longrightarrow}M_\infty\]
given by the maps $p_1(x, t_1, t_2) := (x, t_1)$, $p_2(x, t_1, t_2) := (x, t_2)$,
$\mu(x, t_1, t_2) := (x, t_1+t_2)$ and $\pi (x,t) := x$. 

Then, the convolution functors for subanalytic sheaves on $M_\infty \times \RR_\infty$
\begin{align*}
(\cdot)\Potimes(\cdot)&\colon
\BDC(\Bbbk_{M_\infty\times \RR_\infty}^\sub)\times \BDC(\Bbbk_{M_\infty\times \RR_\infty}^\sub)
\to\BDC(\Bbbk_{M_\infty\times \RR_\infty}^\sub),\\
\Prihomsub(\cdot, \cdot)&\colon
\BDC(\Bbbk_{M_\infty\times \RR_\infty}^\sub)^\op\times \BDC(\Bbbk_{M_\infty\times \RR_\infty}^\sub)
\to\BDC(\Bbbk_{M_\infty\times \RR_\infty}^\sub)
\end{align*}
 are defined by
\begin{align*}
\SF_1\Potimes \SF_2 & := \rmR\mu_{!!}(p_1^{-1}\SF_1\otimes p_2^{-1}\SF_2),\\
\Prihomsub(\SF_1, \SF_2) & := \rmR p_{1\ast}\rihom^\sub(p_2^{-1}\SF_1, \mu^!\SF_2),
\end{align*}
for $\SF_1, \SF_2\in\BDC(\Bbbk_{M_\infty\times \RR_\infty}^\sub)$.
Note that for any $\SF, \SF_1, \SF_2, \SF_3\in\BDC(\Bbbk_{M_\infty\times\RR_\infty}^\sub)$
there exist isomorphisms
\begin{align*}
\SF_1\Potimes \SF_2
&\simeq
\SF_2\Potimes \SF_1,\\
\SF_1\Potimes\(\SF_2\Potimes \SF_3\)
&\simeq
\(\SF_1\Potimes \SF_2\)\Potimes \SF_3,\\
\Bbbk_{\{t=0\}}\Potimes \SF
&\simeq
\SF\simeq \Prihomsub(\Bbbk_{\{t=0\}}, \SF),
\end{align*}
where $\{t = 0\}$ stands for $\{(x, t)\in M\times {\RR}\ |\ t = 0\}$.
Hence, the category $\BDC(\Bbbk_{M_\infty\times\RR_\infty}^\sub)$ has
a structure of commutative tensor category with $\Potimes$ as tensor product functor
and $\Bbbk_{\{t=0\}}$ as unit object.

The convolution functors have several properties
as similar to the tensor product functor and the internal hom functor.
For a morphism of real analytic bordered spaces $f\colon M_\infty\to N_\infty$,
let us denoted by $f_{\RR_\infty}\colon M_\infty\times\RR_\infty\to M_\infty\times\RR_\infty$ 
the morphism $f\times\id_{\RR_\infty}$
of real analytic bordered spaces.

\begin{proposition}\label{prop3.8}\label{prop3.8}
Let $f\colon M_\infty\to\N_\infty$ be a morphism of real analytic bordered spaces,
$\SF, \SF_1, \SF_2, \in\BDC(\Bbbk_{M_\infty\times\RR_\infty}^\sub)$
and $\SG, \SG_1, \SG_2, \in\BDC(\Bbbk_{N_\infty\times\RR_\infty}^\sub)$.
There exist isomorphisms
\begin{align*}
\Prihomsub\(\SF_1\Potimes \SF_2, \SF\)
&\simeq
\Prihomsub\(\SF_1, \Prihomsub(\SF_2, \SF)\),\\
\Hom_{\BDC(\Bbbk_{M_\infty\times \RR_\infty}^\sub)}\(\SF_1\Potimes \SF_2, \SF\)
&\simeq
\Hom_{\BDC(\Bbbk_{M_\infty\times \RR_\infty}^\sub)}\(\SF_1, \Prihomsub(\SF_2, \SF)\),\\
\bfR f_{\RR_\infty\ast}\Prihomsub\(f_{\RR_\infty}^{-1}\SG, \SF\)
&\simeq
\Prihomsub\(\SG, \bfR f_{\RR_\infty\ast}\SF\),\\
f_{\RR_\infty}^{-1}\(\SF_1\Potimes\SF_2\)
&\simeq
f_{\RR_\infty}^{-1}\SF_1\Potimes f_{\RR_\infty}^{-1}\SF_2,\\
\Prihomsub\(\bfR f_{\RR_\infty!!}\SF, \SG\)
&\simeq
\bfR f_{\RR_\infty\ast}\Prihomsub\(\SF, f_{\RR_\infty}^!\SG\),\\
\bfR f_{\RR_\infty!!}\left(\SF\Potimes f_{\RR_\infty}^{-1}\SG\right)
&\simeq
\bfR f_{\RR_\infty!!}\SF\Potimes \SG,\\
f_{\RR_\infty}^!\Prihomsub(\SG_1, \SG_2)
&\simeq
\Prihomsub\(f_{\RR_\infty}^{-1}\SG_1, f_{\RR_\infty}^!\SG_2\).
\end{align*}
\end{proposition}

\begin{proof}
First, let us prove the second isomorphism.
By using the adjointness,
we have
\begin{align*}
\Hom_{\BDC(\Bbbk_{M_\infty\times \RR_\infty}^\sub)}
\(\SF_1\Potimes \SF_2, \SF\)
&\simeq
\Hom_{\BDC(\Bbbk_{M_\infty\times \RR_\infty}^\sub)}
\(\bfR\mu_{!!}(p_1^{-1}\SF_1\otimes p_2^{-1}\SF_2), \SF\)\\
&\simeq
\Hom_{\BDC(\Bbbk_{M_\infty\times \RR_\infty}^\sub)}
\(\SF_1, \bfR p_{1\ast}\rihom^\sub(p_2^{-1}\SF_2, \mu^{!}\SF)\)\\
&\simeq
\Hom_{\BDC(\Bbbk_{M_\infty\times \RR_\infty}^\sub)}
\(\SF_1, \Prihomsub(\SF_2, \SF)\).
\end{align*}

Let us prove the first isomorphism.
By using the second isomorphism,
for any $\SF_0\in\BDC(\Bbbk_{M_\infty\times\RR_\infty}^\sub)$,
there exist isomorphisms
\begin{align*}
&\Hom_{\BDC(\Bbbk_{M_\infty\times \RR_\infty}^\sub)}
\(\SF_0,\,\Prihomsub\(\SF_1\Potimes \SF_2, \SF_3\)\)\\
\simeq\
&\Hom_{\BDC(\Bbbk_{M_\infty\times \RR_\infty}^\sub)}
\(\SF_0\Potimes\(\SF_1\Potimes \SF_2\),\,\SF_3\)\\
\simeq\
&\Hom_{\BDC(\Bbbk_{M_\infty\times \RR_\infty}^\sub)}
\(\(\SF_0\Potimes\SF_1\)\Potimes \SF_2,\,\SF_3\)\\
\simeq\
&\Hom_{\BDC(\Bbbk_{M_\infty\times \RR_\infty}^\sub)}
\(\SF_0\Potimes\SF_1,\,\Prihomsub(\SF_2, \SF_3)\)\\
\simeq\
&\Hom_{\BDC(\Bbbk_{M_\infty\times \RR_\infty}^\sub)}
\(\SF_0, \Prihomsub\(\SF_1, \Prihomsub(\SF_2, \SF_3)\)\).
\end{align*}
Hence, we have 
$\Prihomsub\(\SF_1\Potimes \SF_2, \SF_3\)\simeq
\Prihomsub\(\SF_1, \Prihomsub(\SF_2, \SF_3)\)$.
\medskip

Let us denote by $\tl{f}_{\RR_\infty}\colon M_\infty\times\RR_\infty^2\to N_\infty\times\RR_\infty^2$
the morphism $f\times\id_{\RR_\infty^2}$ of real analytic bordered spaces.
Then there exist cartesian diagrams:
\[\xymatrix@M=5pt@R=20pt@C=40pt{
M_\infty\times \RR_\infty^2\ar@{->}[r]^-{\bigstar}\ar@{->}[d]_-{\tl{f}_{\RR_\infty}}\ar@{}[rd]|-\Box&
M_\infty\times \RR_\infty\ar@{->}[r]^-{\pi}\ar@{->}[d]_-{f_{\RR_\infty}}\ar@{}[rd]|-\Box&
M_\infty\ar@{->}[d]_-{f}\\
N_\infty\times \RR_\infty^2\ar@{->}[r]_-{\bigstar}&
N_\infty\times \RR_\infty\ar@{->}[r]_-{\pi}&
N_\infty,}\]
where $\bigstar = p_1, p_2, \mu$, respectively.
Hence, we have
\begin{align*}
\bfR f_{\RR_\infty\ast}\Prihomsub(f_{\RR_\infty}^{-1}\SG, \SF)
&\simeq
\bfR f_{\RR_\infty\ast}\bfR p_{1\ast}\rihom^\sub(p_2^{-1}f_{\RR_\infty}^{-1}\SG, \mu^!\SF)\\
&\simeq
\bfR p_{1\ast}\bfR \tl{f}_{\RR_\infty\ast}\rihom^\sub(\tl{f}_{\RR_\infty}^{-1}p_2^{-1}\SG, \mu^!\SF)\\
&\simeq
\bfR p_{1\ast}\rihom^\sub(p_2^{-1}\SG, \bfR \tl{f}_{\RR_\infty\ast}\mu^!\SF)\\
&\simeq
\bfR p_{1\ast}\rihom^\sub(p_2^{-1}\SG,\mu^! \bfR {f}_{\RR_\infty\ast}\SF)\\
&\simeq
\Prihomsub(\SG, \bfR f_{\RR_\infty\ast}\SF),
\end{align*}
where in the fourth isomorphism we used Proposition \ref{prop3.4} (4).
By using Proposition \ref{prop3.4} (3), (4) we have 
\begin{align*}
f_{\RR_\infty}^{-1}\(\SF_1\Potimes\SF_2\)
&\simeq
f_{\RR_\infty}^{-1}\bfR\mu_{!!}\(p_1^{-1}\SF_1\Potimes p_2^{-1}\SF_2\)\\
&\simeq
\bfR\mu_{!!}\tl{f}_{\RR_\infty}^{-1}\(p_1^{-1}\SF_1\Potimes p_2^{-1}\SF_2\)\\
&\simeq
\bfR\mu_{!!}\(\tl{f}_{\RR_\infty}^{-1}p_1^{-1}\SF_1\Potimes \tl{f}_{\RR_\infty}^{-1}p_2^{-1}\SF_2\)\\
&\simeq
\bfR\mu_{!!}\(p_1^{-1}{f}_{\RR_\infty}^{-1}\SF_1\Potimes p_2^{-1}{f}_{\RR_\infty}^{-1}\SF_2\)\\
&\simeq
f_{\RR_\infty}^{-1}\SF_1\Potimes f_{\RR_\infty}^{-1}\SF_2,
\end{align*}
and 
\begin{align*}
\bfR f_{\RR_\infty!!}\left(\SF\Potimes f_{\RR_\infty}^{-1}\SG\right)
&\simeq
\bfR f_{\RR_\infty!!}\bfR\mu_{!!}\left(p_1^{-1}\SF\Potimes p_2^{-1}f_{\RR_\infty}^{-1}\SG\right)\\
&\simeq
\bfR\mu_{!!}\bfR\tl{f}_{\RR_\infty!!}\left(p_1^{-1}\SF\Potimes \tl{f}_{\RR_\infty}^{-1}p_2^{-1}\SG\right)\\
&\simeq
\bfR\mu_{!!}\left(\bfR\tl{f}_{\RR_\infty!!}p_1^{-1}\SF\Potimes p_2^{-1}\SG\right)\\
&\simeq
\bfR\mu_{!!}\left(p_1^{-1}\bfR{f}_{\RR_\infty!!}\SF\Potimes p_2^{-1}\SG\right)\\
&\simeq
\bfR f_{\RR_\infty!!}\SF\Potimes \SG.
\end{align*}
Moreover, by using Proposition \ref{prop3.4} (3) we have
\begin{align*}
f_{\RR_\infty}^!\Prihomsub(\SG_1, \SG_2)
&\simeq
f_{\RR_\infty}^!\bfR p_{1\ast}\rihom^\sub(p_2^{-1}\SG_1, \mu^{!}\SG_2)\\
&\simeq
\bfR p_{1\ast}\tl{f}_{\RR_\infty}^!\rihom^\sub(p_2^{-1}\SG_1, \mu^{!}\SG_2)\\
&\simeq
\bfR p_{1\ast}\rihom^\sub(\tl{f}_{\RR_\infty}^{-1}p_2^{-1}\SG_1, \tl{f}_{\RR_\infty}^!\mu^{!}\SG_2)\\
&\simeq
\bfR p_{1\ast}\rihom^\sub(p_2^{-1}{f}_{\RR_\infty}^{-1}\SG_1, \mu^{!}{f}_{\RR_\infty}^!\SG_2)\\
&\simeq
\Prihomsub\(f_{\RR_\infty}^{-1}\SG_1, f_{\RR_\infty}^!\SG_2\).
\end{align*}
By using Proposition \ref{prop3.4} (2), (4) there exits isomorphisms
\begin{align*}
\Prihomsub\(\bfR f_{\RR_\infty!!}\SF, \SG\)
&\simeq
\bfR p_{1\ast}\rihom^\sub\(p_2^{-1}\bfR f_{\RR_\infty!!}\SF, \mu^!\SG\)\\
&\simeq
\bfR p_{1\ast}\rihom^\sub\(\bfR\tl{f}_{\RR_\infty!!}p_2^{-1}\SF, \mu^!\SG\)\\
&\simeq
\bfR p_{1\ast}\bfR\tl{f}_{\RR_\infty\ast}\rihom^\sub\(p_2^{-1}\SF, \bfR\tl{f}_{\RR_\infty}^{!}\mu^!\SG\)\\
&\simeq
\bfR{f}_{\RR_\infty\ast}\bfR p_{1\ast}\rihom^\sub\(p_2^{-1}\SF, \mu^!\bfR{f}_{\RR_\infty}^{!}\SG\)\\
&\simeq
\bfR f_{\RR_\infty\ast}\Prihomsub\(\SF, f_{\RR_\infty}^!\SG\).
\end{align*}
\end{proof}

\begin{proposition}\label{prop3.9}
Let $\SF, \SG, \SH\in\BDC(\Bbbk_{M_\infty\times\RR_\infty}^\sub)$
and $\SF_0\in\BDC(\Bbbk_{M_\infty}^\sub)$.
Then there exist isomorphisms
\begin{align*}
\pi^{-1}\SF_0\otimes \(\SF\Potimes \SG\)
&\simeq
\(\pi^{-1}\SF_0\otimes\SF\)\Potimes \SG,\\
\rihom^\sub\(\pi^{-1}\SF_0, \Prihomsub(\SF, \SG)\)
&\simeq
\Prihomsub\(\pi^{-1}\SF_0\otimes\SF, \SG\)\\
&\simeq
\Prihomsub\(\SF,\rihom^\sub(\pi^{-1}\SF_0, \SG\),\\
\bfR\pi_\ast\rihom^\sub\(\SF\Potimes \SG, \SH\)
&\simeq
\bfR\pi_\ast\rihom^\sub\(\SF, \Prihomsub(\SG, \SH)\).
\end{align*}
\end{proposition}

\begin{proof}
First let us remark that $\pi\circ p_1 =  \pi\circ p_2 = \pi\circ\mu$.

By using Proposition \ref{prop3.4} (3),
we have
\begin{align*}
\pi^{-1}\SF_0\otimes \(\SF\Potimes \SG\)
&\simeq
\pi^{-1}\SF_0\otimes \bfR\mu_{!!}\(p_1^{-1}\SF\otimes p_2^{-1}\SG\)\\
&\simeq
\bfR\mu_{!!}\(\mu^{-1}\pi^{-1}\SF_0\otimes \(p_1^{-1}\SF\otimes p_2^{-1}\SG\)\)\\
&\simeq
\bfR\mu_{!!}\(p_1^{-1}\pi^{-1}\SF_0\otimes \(p_1^{-1}\SF\otimes p_2^{-1}\SG\)\)\\
&\simeq
\bfR\mu_{!!}\(p_1^{-1}(\pi^{-1}\SF_0\otimes\SF)\otimes p_2^{-1}\SG\)\\
&\simeq
\(\pi^{-1}\SF_0\otimes\SF\)\Potimes \SG.
\end{align*}

By using Proposition \ref{prop3.4} (2),
\begin{align*}
\rihom^\sub\(\pi^{-1}\SF_0, \Prihomsub(\SF_1, \SF_2)\)
&\simeq
\rihom^\sub\(\pi^{-1}\SF_0, \bfR p_{1\ast}\rihom^\sub(p_2^{-1}\SF_1, \mu^!\SF_2)\)\\
&\simeq
\bfR p_{1\ast}\rihom^\sub\(p_1^{-1}\pi^{-1}\SF_0, \rihom^\sub(p_2^{-1}\SF_1, \mu^!\SF_2)\)\\
&\simeq
\bfR p_{1\ast}\rihom^\sub\(p_1^{-1}\pi^{-1}\SF_0\otimes p_2^{-1}\SF_1, \mu^!\SF_2)\).
\end{align*}
Moreover, by using Proposition \ref{prop3.4} (3) we have
\begin{align*}
\bfR p_{1\ast}\rihom^\sub\(p_1^{-1}\pi^{-1}\SF_0\otimes p_2^{-1}\SF_1, \mu^!\SF_2)\)
&\simeq
\bfR p_{1\ast}\rihom^\sub\(p_2^{-1}\pi^{-1}\SF_0\otimes p_2^{-1}\SF_1, \mu^!\SF_2)\)\\
&\simeq
\bfR p_{1\ast}\rihom^\sub\(p_2^{-1}(\pi^{-1}\SF_0\otimes \SF_1), \mu^!\SF_2)\)\\
&\simeq
\Prihomsub\(\pi^{-1}\SF_0\otimes\SF_1, \SF_2\)
\end{align*}
and 
by using Proposition \ref{prop3.4} (2), (3) we have 
\begin{align*}
&\bfR p_{1\ast}\rihom^\sub\(p_1^{-1}\pi^{-1}\SF_0\otimes p_2^{-1}\SF_1, \mu^!\SF_2)\)\\
\simeq\
&\bfR p_{1\ast}\rihom^\sub\(\mu^{-1}\pi^{-1}\SF_0\otimes p_2^{-1}\SF_1, \mu^!\SF_2)\)\\
\simeq\
&\bfR p_{1\ast}\rihom^\sub\(p_2^{-1}\SF_1, \rihom^\sub(\mu^{-1}\pi^{-1}\SF_0, \mu^!\SF_2)\)\\
\simeq\
&\bfR p_{1\ast}\rihom^\sub\(p_2^{-1}\SF_1, \mu^!\rihom^\sub(\pi^{-1}\SF_0, \SF_2)\)\\
\simeq\
&\Prihomsub\(\SF_1, \rihom^\sub(\pi^{-1}\SF_0, \SF_2)\).
\end{align*}

Let us prove the last assertion.
For any $\SG_0\in\BDC(\Bbbk_{M_\infty}^\sub)$,
there exist isomorphisms
\begin{align*}
&\Hom_{\BDC(\Bbbk_{M_\infty}^\sub)}
\(\SG_0,\,\bfR\pi_\ast\rihom^\sub\(\SF\Potimes \SG, \SH\)\)\\
&\simeq
\Hom_{\BDC(\Bbbk_{M_\infty\times \RR_\infty}^\sub)}
\(\pi^{-1}\SG_0,\,\rihom^\sub\(\SF\Potimes \SG, \SH\)\)\\
&\simeq
\Hom_{\BDC(\Bbbk_{M_\infty\times \RR_\infty}^\sub)}
\(\pi^{-1}\SG_0\otimes\(\SF\Potimes \SG\),\,\SH\)\\
&\simeq
\Hom_{\BDC(\Bbbk_{M_\infty\times \RR_\infty}^\sub)}
\(\(\pi^{-1}\SG_0\otimes\SF\)\Potimes \SG,\,\SH\)\\
&\simeq
\Hom_{\BDC(\Bbbk_{M_\infty\times \RR_\infty}^\sub)}
\(\pi^{-1}\SG_0\otimes\SF,\,\Prihomsub(\SG, \SH)\)\\
&\simeq
\Hom_{\BDC(\Bbbk_{M_\infty}^\sub)}
\(\SG_0, \bfR\pi_\ast\rihom^\sub\(\SF, \Prihomsub(\SG, \SH)\)\).\,
\end{align*}
where in the third (resp.\,fourth) isomorphism we used the first assertion
(resp.\,Proposition \ref{prop3.8}).
Hence, we have an isomorphism
$$\bfR\pi_\ast\rihom^\sub\(\SF\Potimes \SG, \SH\)
\simeq
\bfR\pi_\ast\rihom^\sub\(\SF, \Prihomsub(\SG, \SH)\).$$
\end{proof}

\begin{lemma}
For any $\SF\in\BDC(\Bbbk_{M_\infty\times\RR_\infty}^\sub)$ and any $\SG\in\BDC(\Bbbk_{N_\infty}^\sub)$,
We have
\begin{align*}
\pi^{-1}\SG\Potimes\SF
&\simeq
\pi^{-1}(\SG\otimes \bfR\pi_{!!}\SF),\\
\Prihomsub(\pi^{-1}\SG, \SF)
&\simeq
\pi^{!}\Prihomsub(\SG, \bfR\pi_{\ast}\SF),\\
\Prihomsub(\SF, \pi^{!}\SG)
&\simeq
\pi^{!}\Prihomsub(\bfR\pi_{!!}\SF, \SG).
\end{align*}
\end{lemma}

\begin{proof}
Note that there exist Cartesian diagrams (i=1, 2):
\[\xymatrix@M=5pt@R=20pt@C=40pt{
{M_\infty\times \RR_\infty^2}\ar@{}[rd]|-\Box\ar@{->}[r]^-{\mu}\ar@{->}[d]_-{p_i}&
{M_\infty\times \RR_\infty}\ar@{->}[d]_-{\pi}&
{M_\infty\times \RR_\infty^2}\ar@{->}[l]_-{p_1}\ar@{->}[d]_-{p_2}\ar@{}[ld]|-\Box\\
{M_\infty\times \RR_\infty}\ar@{->}[r]_-{\pi}&
{M_\infty}&
{M_\infty\times \RR_\infty}\ar@{->}[l]^-{\pi}.}\]
Then by using Proposition \ref{prop3.4} (3), (4),
we have
\begin{align*}
\pi^{-1}\SG\Potimes\SF
&\simeq
\bfR\mu_{!!}\(p_1^{-1}\pi^{-1}\SG\otimes p_2^{-1}\SF\)\\
&\simeq
\bfR\mu_{!!}\(\mu^{-1}\pi^{-1}\SG\otimes p_2^{-1}\SF\)\\
&\simeq
\pi^{-1}\SG\otimes \bfR\mu_{!!}p_2^{-1}\SF\\
&\simeq
\pi^{-1}\SG\otimes \pi^{-1}\bfR\pi_{!!}\SF\\
&\simeq
\pi^{-1}(\SF\otimes \bfR\pi_{!!}\SF).
\end{align*}
Moreover, by using  Proposition \ref{prop3.4} (2), (3), (4),
we have
\begin{align*}
\Prihomsub(\pi^{-1}\SG, \SF)
&\simeq
\bfR p_{1\ast}\rihom^\sub(p_2^{-1}\pi^{-1}\SG, \mu^!\SF)\\
&\simeq
\bfR p_{1\ast}\rihom^\sub(\mu^{-1}\pi^{-1}\SG, \mu^!\SF)\\
&\simeq
\bfR p_{1\ast}\mu^!\rihom^\sub(\pi^{-1}\SG, \SF)\\
&\simeq
\pi^!\bfR\pi_{\ast}\rihom^\sub(\pi^{-1}\SG, \SF)\\
&\simeq
\pi^{!}\Prihomsub(\SG, \bfR\pi_{\ast}\SF)
\end{align*}
and 
\begin{align*}
\Prihomsub(\SF, \pi^{!}\SG)
&\simeq
\bfR p_{1\ast}\rihom^\sub(p_2^{-1}\SF, \mu^!\pi^{!}\SG)\\
&\simeq
\bfR p_{1\ast}\rihom^\sub(p_2^{-1}\SF, p_2^!\pi^{!}\SG)\\
&\simeq
\bfR p_{1\ast}p_2^!\rihom^\sub(\SF, \pi^{!}\SG)\\
&\simeq
\pi^!\bfR \pi_{\ast}\rihom^\sub(\SF, \pi^{!}\SG)\\
&\simeq
\pi^{!}\Prihomsub(\bfR\pi_{!!}\SF, \SG).
\end{align*}
\end{proof}

Since $\bfR\pi_{!!}\Bbbk_{\{t\leq0\}} \simeq \bfR\pi_{!!}\Bbbk_{\{t\geq0\}}\simeq0$
and $\pi^{-1}\Bbbk_M\simeq\Bbbk_{M\times\RR}$, we have
\begin{corollary}\label{cor3.11}
For any $\SF\in\BDC(\Bbbk_{M_\infty\times\RR_\infty}^\sub)$
and any $\SG\in\BDC(\Bbbk_{M_\infty}^\sub)$,
we have
\begin{align*}
\rho_{M_\infty\times\RR_\infty\ast}\(\Bbbk_{\{t\geq0\}}\oplus\Bbbk_{\{t\leq0\}}\)\Potimes \pi^{-1}\SG
&\simeq 0\\
\Prihomsub\(\rho_{M_\infty\times\RR_\infty\ast}\(\Bbbk_{\{t\geq0\}}\oplus\Bbbk_{\{t\leq0\}}\), \pi^{-1}\SG\)
&\simeq 0,\\
\rho_{M_\infty\times\RR_\infty\ast}\Bbbk_{M\times\RR}\Potimes \SF
&\simeq
\pi^{-1}\bfR\pi_{!!}\SF\ (\,\simeq \pi^!\bfR\pi_{!!}\SF[-1]),\\
\Prihomsub\(\rho_{M_\infty\times\RR_\infty\ast}\Bbbk_{M\times\RR}, \SF\)
&\simeq
\pi^{!}\bfR\pi_{\ast}\SF\
(\,\simeq \pi^{-1}\bfR\pi_{\ast}\SF[1]).
\end{align*}
\end{corollary}

At the end of this subsection,
we shall prove the following proposition.
\begin{proposition}\label{prop3.12}
Let $M_\infty = (M, \che{M})$ be a real analytic bordered space.
\begin{itemize}
\item[\rm (1)]
For any $\SF_1, \SF_2\in\BDC(\Bbbk_{M_\infty\times\RR_\infty}^\sub)$,
there exits an isomorphism in $\BDC(\I\Bbbk_{M_\infty\times\RR_\infty}):$
$$I_{M_\infty\times\RR_\infty}\SF_1\Potimes I_{M_\infty\times\RR_\infty}\SF_2
\simeq
I_{M_\infty\times\RR_\infty}\(\SF_1\Potimes\SF_2\).$$

\item[\rm (2)]
For any $\SF\in\BDC(\Bbbk_{M_\infty\times\RR_\infty}^\sub)$
and any $G\in\BDC(\I\Bbbk_{M_\infty\times\RR_\infty})$,
there exist an isomorphism in $\BDC(\Bbbk_{M_\infty\times\RR_\infty}^\sub):$
$$\bfR J_{M_\infty\times\RR_\infty}\Prihom\(I_{M_\infty\times\RR_\infty}\SF, G\)
\simeq
\Prihomsub\(\SF, \bfR J_{M_\infty\times\RR_\infty}G\).$$

\item[\rm (3)]
The functor 
$$(\cdot)\Potimes(\cdot)\colon
\BEC_{\I\RR-c}(\I\Bbbk_{M_\infty\times\RR_\infty})\times\BEC_{\I\RR-c}(\I\Bbbk_{M_\infty\times\RR_\infty})
\to\BDC_{\I\RR-c}(\I\Bbbk_{M_\infty\times\RR_\infty})$$
is well defined.

\item[\rm (4)]
For any $F_1, F_2\in\BDC_{\I\RR-c}(\I\Bbbk_{M_\infty\times\RR_\infty})$,
there exits an isomorphism in $\BDC(\Bbbk_{M_\infty\times\RR_\infty}^\sub)$
$$\lambda_{M_\infty\times\RR_\infty}F_1\Potimes \lambda_{M_\infty\times\RR_\infty}F_2
\simeq
\lambda_{M_\infty\times\RR_\infty}\(F_1\Potimes F_2\).$$
\end{itemize}
\end{proposition}

\begin{proof}
(1)
Let us denote by $S$ the closure of 
$\{(t_1, t_2, t_3)\in\RR^3\ |\ t_1+t_2+t_3=0\}$ in $\var{\RR}^3$,
and consider the morphisms $\tl{p}_1, \tl{p}_2, \tl{\mu}\colon S\to\var{\RR}$
given by $\tl{p}_1(t_1, t_2, t_3) = t_1,\ \tl{p}_2(t_1, t_2, t_3) = t_2,\ \tl{\mu}(t_1, t_2, t_3) = t_1+t_2 = -t_3$.
We shall denote by the same symbols the corresponding morphisms $\che{M}\times S\to\che{M}\times\var{\RR}$.
Then there exists a commutative diagram
\[\xymatrix@M=7pt@C=35pt{
M_\infty\times\RR_\infty^2\ar@{->}[r]^-k\ar@{->}[d]_-u
&
\che{M}\times S\ar@{->}[d]^-{\tl{u}}\\
M_\infty\times\RR_\infty\ar@{->}[r]_-{j_{M_\infty\times\RR_\infty}}
&\che{M}\times\var{\RR},}\]
where $u = p_1, p_2, \mu$, and
$k$ is the morphism associated
by the embedding $\RR^2\hookrightarrow S, (t_1, t_2)\mapsto (t_1, t_2, -t_1-t_2)$.
Note that for any $F_1, F_2\in\BDC(\I\Bbbk_{M_\infty\times\RR_\infty})$
there exists an isomorphism 
$$
F_1\Potimes F_2\simeq
j_{M_\infty\times\RR_\infty}^{-1}\bfR\tl{\mu}_{!!}(
\tl{p}_1^{-1}\bfR j_{M_\infty\times\RR_\infty!!}F_1\otimes
\tl{p}_2^{-1}\bfR j_{M_\infty\times\RR_\infty!!}F_2).
$$
This assertion can be proved as similarly to \cite[Lem.\:4.3.9]{DK16}.
Then, we have isomorphism
$$
I_{M_\infty\times\RR_\infty}\SF_1\Potimes I_{M_\infty\times\RR_\infty}\SF_2
\simeq
j_{M_\infty\times\RR_\infty}^{-1}\bfR\tl{\mu}_{!!}\(
\tl{p}_1^{-1}\bfR j_{M_\infty\times\RR_\infty!!}I_{M_\infty\times\RR_\infty}\SF_1
\otimes
\tl{p}_2^{-1}\bfR j_{M_\infty\times\RR_\infty!!}I_{M_\infty\times\RR_\infty}\SF_2\).
$$
Moreover, we have isomorphisms
\begin{align*}
\tl{p}_1^{-1}\bfR j_{M_\infty\times\RR_\infty!!}I_{M_\infty\times\RR_\infty}\SF_1
&\simeq
\tl{p}_1^{-1}\bfR j_{M_\infty\times\RR_\infty!!}
j_{M_\infty\times\RR_\infty}^{-1}I_{\che{M}\times\var{\RR}}\bfR j_{M_\infty\times\RR_\infty!!}\SF_1\\
&\simeq
\tl{p}_1^{-1}(\iota_{\che{M}\times\var{\RR}}\Bbbk_{M\times\RR}
\otimes I_{\che{M}\times\var{\RR}}\bfR j_{M_\infty\times\RR_\infty!!}\SF_1)\\
&\simeq
\tl{p}_1^{-1}(I_{\che{M}\times\var{\RR}}\rho_{\che{M}\times\var{\RR}\ast}\Bbbk_{M\times\RR}
\otimes I_{\che{M}\times\var{\RR}}\bfR j_{M_\infty\times\RR_\infty!!}\SF_1)\\
&\simeq
\tl{p}_1^{-1}I_{\che{M}\times\var{\RR}}\(\rho_{\che{M}\times\var{\RR}\ast}\Bbbk_{M\times\RR}
\otimes \bfR j_{M_\infty\times\RR_\infty!!}\SF_1\)\\
&\simeq
I_{\che{M}\times S}\tl{p}_1^{-1}\(\rho_{\che{M}\times\var{\RR}\ast}\Bbbk_{M\times\RR}
\otimes \bfR j_{M_\infty\times\RR_\infty!!}\SF_1\)\\
&\simeq
I_{\che{M}\times S}\tl{p}_1^{-1}\(
\bfR j_{M_\infty\times\RR_\infty!!}j_{M_\infty\times\RR_\infty}^{-1}\bfR j_{M_\infty\times\RR_\infty!!}\SF_1\)\\
&\simeq
I_{\che{M}\times S}\(\tl{p}_1^{-1}\bfR j_{M_\infty\times\RR_\infty!!}\SF_1\),
\end{align*}
where in the second (resp.\:sixth) isomorphism we used the fact that 
$\bfR j_{M_\infty\times\RR_\infty!!}j_{M_\infty\times\RR_\infty}^{-1}
\simeq\iota_{\che{M}\times\var{\RR}}\Bbbk_{M\times\RR}\otimes(\cdot)$
(resp.\:$\bfR j_{M_\infty\times\RR_\infty!!}j_{M_\infty\times\RR_\infty}^{-1}
\simeq\rho_{\che{M}\times\var{\RR}\ast}\Bbbk_{M\times\RR}\otimes(\cdot)$).
See also the end of Subsection \ref{subsec2.5}.
In a similar way,
we have an isomorphism
$$\tl{p}_2^{-1}\bfR j_{M_\infty\times\RR_\infty!!}I_{M_\infty\times\RR_\infty}\SF_2
\simeq
I_{\che{M}\times S}\(\tl{p}_2^{-1}\bfR j_{M_\infty\times\RR_\infty!!}\SF_2\).$$
Hence, there exist isomorphisms 
\begin{align*}
I_{M_\infty\times\RR_\infty}\SF_1\Potimes I_{M_\infty\times\RR_\infty}\SF_2
&\simeq
j_{M_\infty\times\RR_\infty}^{-1}\bfR\tl{\mu}_{!!}I_{\che{M}\times S}
\(\tl{p}_1^{-1}\bfR j_{M_\infty\times\RR_\infty!!}\SF_1
\otimes
\tl{p}_2^{-1}\bfR j_{M_\infty\times\RR_\infty!!}\SF_2\)\\
&\simeq
j_{M_\infty\times\RR_\infty}^{-1}I_{\che{M}\times \var{\RR}}\bfR\tl{\mu}_{!!}
\(\tl{p}_1^{-1}\bfR j_{M_\infty\times\RR_\infty!!}\SF_1
\otimes
\tl{p}_2^{-1}\bfR j_{M_\infty\times\RR_\infty!!}\SF_2\)\\
&\simeq
I_{M_\infty\times\RR_\infty}j_{M_\infty\times\RR_\infty}^{-1}\bfR\tl{\mu}_{!!}
\(\tl{p}_1^{-1}\bfR j_{M_\infty\times\RR_\infty!!}\SF_1
\otimes
\tl{p}_2^{-1}\bfR j_{M_\infty\times\RR_\infty!!}\SF_2\).
\end{align*}
Moreover, we have
\begin{align*}
&j_{M_\infty\times\RR_\infty}^{-1}\bfR\tl{\mu}_{!!}
\(\tl{p}_1^{-1}\bfR j_{M_\infty\times\RR_\infty!!}\SF_1
\otimes
\tl{p}_2^{-1}\bfR j_{M_\infty\times\RR_\infty!!}\SF_2\)\\
\simeq\
&\bfR{\mu}_{!!}k^{-1}
\(\bfR k_{!!}{p}_1^{-1}\SF_1\otimes \bfR k_{!!}{p}_2^{-1}\SF_2\)\\
\simeq\
&\bfR{\mu}_{!!}
\({p}_1^{-1}\SF_1\otimes {p}_2^{-1}\SF_2\)
\end{align*}
and hence we have $I_{M_\infty\times\RR_\infty}\SF_1\Potimes I_{M_\infty\times\RR_\infty}\SF_2
\simeq
I_{M_\infty\times\RR_\infty}\(\SF_1\Potimes\SF_2\)$.
\medskip

\noindent
(2)
For any $\SF_0\in\BDC(\Bbbk_{M_\infty\times\RR_\infty}^\sub)$,
there exist isomorphisms
\begin{align*}
&\Hom_{\BDC(\Bbbk_{M_\infty\times\RR_\infty}^\sub)}\(
\SF_0, \bfR J_{M_\infty\times\RR_\infty}\Prihom\(I_{M_\infty\times\RR_\infty}\SF, G\)\)\\
\simeq\
&\Hom_{\BDC(\Bbbk_{M_\infty\times\RR_\infty}^\sub)}\(
I_{M_\infty\times\RR_\infty}\SF_0\Potimes I_{M_\infty\times\RR_\infty}\SF, G\)\\
\simeq\
&\Hom_{\BDC(\Bbbk_{M_\infty\times\RR_\infty}^\sub)}\(
I_{M_\infty\times\RR_\infty}\(\SF_0\Potimes \SF\), G\)\\
\simeq\
&\Hom_{\BDC(\Bbbk_{M_\infty\times\RR_\infty}^\sub)}\(\SF_0,
\Prihomsub\(\SF, \bfR J_{M_\infty\times\RR_\infty}G\)\),
\end{align*} 
where in the first and last isomorphisms we used Proposition \ref{prop3.6} (1) and Proposition \ref{prop3.8},
and the second isomorphism follows from the assertion (1).
Hence we have an isomorphism 
$\bfR J_{M_\infty\times\RR_\infty}\Prihom\(I_{M_\infty\times\RR_\infty}\SF, G\)
\simeq
\Prihomsub\(\SF, \bfR J_{M_\infty\times\RR_\infty}G\).$

\medskip

\noindent
(3)
Let $F_1, F_2\in\BDC_{\I\RR-c}(\I\Bbbk_{M_\infty\times\RR_\infty})$.
By Proposition \ref{prop3.6} (2), there exist $\SF_1, \SF_2\in\BDC(\Bbbk_{M_\infty}^\sub)$ such that
$F_1\simeq I_{M_\infty\times\RR_\infty}\SF_1,
F_2\simeq I_{M_\infty\times\RR_\infty}\SF_2$.
Moreover, by using the first assertion,
we have $$F_1\Potimes F_2 \simeq
I_{M_\infty\times\RR_\infty}\SF_1\Potimes I_{M_\infty\times\RR_\infty}\SF_1
\simeq
I_{M_\infty\times\RR_\infty}\(\SF_1\otimes\SF_2\).$$
This implies that $F_1\Potimes F_2\in \BDC_{\I\RR-c}(\I\Bbbk_{M_\infty\times\RR_\infty})$.
The proof is completed.
\medskip

\noindent
(4)
This assertion follows from Propositions \ref{prop3.6} (2) and the assertion (1).
\end{proof}

\subsection{Enhanced Subanalytic Sheaves}\label{subsec3.3}
In this subsection, let us define enhanced subanalytic sheaves in a similar way to the definition of enhanced ind-sheaves.

Let $M_\infty = (M, \che{M})$ be a real analytic bordered space and 
set $\RR_\infty := (\RR, \var{\RR})$ for $\var{\RR} := \RR\sqcup\{-\infty, +\infty\}$.
Let us set
$$\BEC(\Bbbk_{M_\infty}^\sub) :=
\BDC(\Bbbk_{M_\infty \times\RR_\infty}^\sub)/\pi^{-1}\BDC(\Bbbk_{M_\infty}^\sub)$$
and we shall call an object of $\BEC(\Bbbk_{M_\infty}^\sub)$ an enhanced subanalytic sheaf\footnote{It seems that a object of $\BDC(\Bbbk_{M_\infty\times\RR_\infty}^\sub)$ is called 
an enhanced subanalytic sheaf, in \cite{Kas16}.} on $M_\infty$.
The category $\pi^{-1}\BDC(\Bbbk_{M_\infty}^\sub)$ can be characterized as follows.
\begin{lemma}
For $\SF\in\BDC(\Bbbk_{M_\infty\times\RR_\infty}^\sub)$, the following five conditions are equivalent
\begin{itemize}
\item[\rm (i)]
$\SF\in\pi^{-1}\BDC(\Bbbk_{M_\infty}^\sub)$,

\item[\rm (ii)]
$\SF\simto\Bbbk_{M\times\RR}[1]\Potimes\SF$,

\item[\rm (iii)]
$\Prihomsub(\Bbbk_{M\times\RR}[1], \SF)\simto \SF$,

\item[\rm (iv)]
$\(\Bbbk_{\{t\geq0\}}\oplus\Bbbk_{\{t\leq0\}}\)\Potimes \SF \simeq 0$,

\item[\rm (v)]
$\Prihomsub\(\Bbbk_{\{t\geq0\}}\oplus\Bbbk_{\{t\leq0\}}, \SF\)\simeq 0$.

\end{itemize}
\end{lemma}

\begin{proof}
By using a distinguished triangle 
$$\Bbbk_{\{t\geq 0\}}\oplus\Bbbk_{\{t\leq 0\}}\longrightarrow
\Bbbk_{\{t=0\}}\longrightarrow\Bbbk_{M\times\RR}[1]\xrightarrow{+1}$$
and the fact that $\Bbbk_{\{t=0\}}\Potimes \SF\simeq \SF$
(resp.\,$\SF\simeq\Prihomsub(\Bbbk_{\{t=0\}}, \SF)$),
we have the condition (ii) (resp.\,(iii)) is equivalent to the condition (iv) (resp.\,(v)).

Let us prove three conditions (i), (ii), (iii) are equivalent.
Let us assume the condition (ii) (resp.\,(iii)) is satisfied.
Then we have $\SF\simto\Bbbk_{M\times\RR}\Potimes\SF[1]$
(resp.\,$\Prihomsub(\Bbbk_{M\times\RR}[1], \SF)\simto \SF$).
By Corollary \ref{cor3.11},
$\SF\simto\pi^{-1}\bfR\pi_{!!}[1]\SF$
(resp.\,$\pi^{-1}\bfR\pi_\ast\SF\overset{\sim}{\longleftarrow}\SF$).
Hence, the condition (i) is satisfied. 

Let us assume the condition (i) is satisfied.
By using Proposition \ref{prop3.7} (2)(iii),
we have $I_{M_\infty\times\RR_\infty}\SF\in\pi^{-1}\BDC(\I\Bbbk_{M_\infty})$,
and hence by \cite[Lem.\:4.4.3]{DK16} we have
\begin{align*}
\pi^{-1}\bfR\pi_\ast I_{M_\infty\times\RR_\infty}\SF
&\simto
I_{M_\infty\times\RR_\infty}\SF,\\
I_{M_\infty\times\RR_\infty}\SF
&\simto
\pi^{!}\bfR\pi_{!!} I_{M_\infty\times\RR_\infty}\SF.
\end{align*}
By using Proposition \ref{prop3.7} (3)(iv), (v) (resp.\,(2)(iv), (v)) and Proposition \ref{prop3.6} (2),
we have 
$$
\pi^{-1}\bfR\pi_\ast\SF
\simto
\SF\hspace{7pt}
(\text{resp.}\, \SF\simto\pi^{!}\bfR\pi_{!!} \SF).
$$
By Corollary \ref{cor3.11}, 
this implies that the condition (iii) (resp.\,(ii)) is satisfied.

Therefore, the proof is completed.
\end{proof}

Let us prove the quotient functor
\[\Q_{M_\infty}^\sub \colon \BDC(\Bbbk_{M_\infty\times\RR_\infty}^\sub)\to\BEC(\Bbbk_{M_\infty}^\sub)\]
has fully faithful left and right adjoints.
By Corollary \ref{cor3.11},
two functors
\begin{align*}
\bfL_{M_\infty}^{\rmE, \sub}&\colon
\BEC(\Bbbk_{M_\infty}^\sub) \to\BDC(\Bbbk_{M_\infty\times\RR_\infty}^\sub),
\hspace{7pt}
\\
\bfR_{M_\infty}^{\rmE, \sub} &\colon
\BEC(\Bbbk_{M_\infty}^\sub) \to\BDC(\Bbbk_{M_\infty\times\RR_\infty}^\sub)
\end{align*}
 which are defined by 
\begin{align*}
\bfL_{M_\infty}^{\rmE, \sub}(\Q_{M_\infty}^\sub(\cdot))
&:= \rho_{M_\infty\times\RR_\infty\ast}(\Bbbk_{\{t\geq0\}}\oplus\Bbbk_{\{t\leq 0\}})
\Potimes (\cdot) ,\\
\bfR_{M_\infty}^{\rmE, \sub}(\Q_{M_\infty}^\sub(\cdot)) 
&:=\Prihomsub(\rho_{M_\infty\times\RR_\infty\ast}(\Bbbk_{\{t\geq0\}}\oplus\Bbbk_{\{t\leq 0\}}), \cdot)
\end{align*}
are well defined.

\begin{lemma}\label{lem3.13}
The functors $\bfL_{M_\infty}^{\rmE, \sub}, \bfR_{M_\infty}^{\rmE, \sub}$ induce equivalences of categories
\begin{align*}
\bfL_{M_\infty}^{\rmE, \sub} &\colon \BEC(\Bbbk_{M_\infty}^\sub)\simto
\{\SF\in\BDC(\Bbbk_{M_\infty\times\RR_\infty}^\sub)\ |\
\rho_{M_\infty\times\RR_\infty\ast}(\Bbbk_{\{t\geq0\}}\oplus\Bbbk_{\{t\leq0\}})\Potimes \SF\simto \SF\},\\
\bfR_{M_\infty}^{\rmE, \sub} &\colon \BEC(\Bbbk_{M_\infty}^\sub)\simto
\{\SF\in\BDC(\Bbbk_{M_\infty\times\RR_\infty}^\sub)\ |\ \SF\simto
\Prihomsub(\rho_{M_\infty\times\RR_\infty\ast}(\Bbbk_{\{t\geq0\}}\oplus\Bbbk_{\{t\leq0\}}), \SF)\},
\end{align*}
respectively.

Moreover, the quotient functor admits
a left (resp.\,right) adjoint $\bfL_{M_\infty}^{\rmE, \sub}$ (resp.\,$\bfR_{M_\infty}^{\rmE, \sub}$).

\end{lemma}

\begin{proof}
By Corollary \ref{cor3.11} and the fact that 
for any $K\in\BEC(\Bbbk_{M_\infty}^\sub)$
there exist isomorphisms
\begin{align*}
&\rho_{M_\infty\times\RR_\infty\ast}(\Bbbk_{\{t\geq0\}}\oplus\Bbbk_{\{t\leq0\}})\Potimes\bfL^{\rmE, \sub}_{M_\infty}K
\simto\bfL^{\rmE, \sub}_{M_\infty}K,\\
&\bfR^{\rmE, \sub}_{M_\infty}K
\simto
\Prihomsub\(\rho_{M_\infty\times\RR_\infty\ast}(\Bbbk_{\{t\geq0\}}\oplus\Bbbk_{\{t\leq0\}}), 
\bfR^{\rmE, \sub}_{M_\infty}K\),
\end{align*}
functors 
\begin{align*}
\bfL_{M_\infty}^{\rmE, \sub} &\colon \BEC(\Bbbk_{M_\infty}^\sub)\to
\{\SF\in\BDC(\Bbbk_{M_\infty\times\RR_\infty}^\sub)\ |\
\rho_{M_\infty\times\RR_\infty\ast}(\Bbbk_{\{t\geq0\}}\oplus\Bbbk_{\{t\leq0\}})\Potimes \SF\simto \SF\},\\
\bfR_{M_\infty}^{\rmE, \sub} &\colon \BEC(\Bbbk_{M_\infty}^\sub)\to
\{\SF\in\BDC(\Bbbk_{M_\infty\times\RR_\infty}^\sub)\ |\ \SF\simto
\Prihomsub(\rho_{M_\infty\times\RR_\infty\ast}(\Bbbk_{\{t\geq0\}}\oplus\Bbbk_{\{t\leq0\}}), \SF)\}
\end{align*}
are well defined.
Let $K\in\BEC(\Bbbk_{M_\infty}^\sub)$.
Then there exist $\SF\in\BDC(\Bbbk_{M_\infty\times\RR_\infty}^\sub)$
such that $K\simeq \Q_{M_\infty}^\sub\SF$.
Let us prove that 
\begin{align*}
&\Q_{M_\infty}^\sub\(\rho_{M_\infty\times\RR_\infty\ast}(\Bbbk_{\{t\geq0\}}\oplus\Bbbk_{\{t\leq0\}})
\Potimes \SF\)
\simeq \Q_{M_\infty}^\sub(\SF),\\
&\SF\simto
\Q_{M_\infty}^\sub\(\Prihomsub(\rho_{M_\infty\times\RR_\infty\ast}(\Bbbk_{\{t\geq0\}}\oplus\Bbbk_{\{t\leq0\}}), \SF)\).
\end{align*}
Since there exists a distinguished triangle 
$$\Bbbk_{\{t\geq 0\}}\oplus\Bbbk_{\{t\leq 0\}}\longrightarrow
\Bbbk_{\{t=0\}}\longrightarrow\Bbbk_{M\times\RR}[1]\xrightarrow{+1},$$
it is enough to show that 
$$\Q_{M_\infty}^\sub\(\rho_{M_\infty\times\RR_\infty\ast}\Bbbk_{M\times\RR}[1]\Potimes \SF\)
\simeq
\Q_{M_\infty}^\sub\(\Prihomsub(\rho_{M_\infty\times\RR_\infty\ast}\Bbbk_{M\times\RR}[1], \SF)\)
\simeq0.$$
These assertions follows from Corollary \ref{cor3.11}.
So that, we have
\begin{align*}
\Q_{M_\infty}^\sub\bfL_{M_\infty}^{\rmE, \sub}K
\simeq
\Q_{M_\infty}^\sub\(\rho_{M_\infty\times\RR_\infty\ast}(\Bbbk_{\{t\geq0\}}\oplus\Bbbk_{\{t\leq0\}})
\Potimes \SF\)
\simeq
\Q_{M_\infty}^\sub(\SF)
\simeq K,\\
\Q_{M_\infty}^\sub\bfR_{M_\infty}^{\rmE, \sub}K
\simeq
\Q_{M_\infty}^\sub\(\Prihomsub(\rho_{M_\infty\times\RR_\infty\ast}(\Bbbk_{\{t\geq0\}}\oplus\Bbbk_{\{t\leq0\}}), \SF)\)
\simeq
\Q_{M_\infty}^\sub(\SF)
\simeq K.
\end{align*}
Hence we have $\Q_{M_\infty}^\sub\circ\bfL_{M_\infty}^{\rmE, \sub}\simeq\id,\ 
\Q_{M_\infty}^\sub\circ\bfR_{M_\infty}^{\rmE, \sub}\simeq\id$.
Moreover, it is clear that for any
$\SG_1\in\{\SF\in\BDC(\Bbbk_{M_\infty\times\RR_\infty}^\sub)\ |\
\rho_{M_\infty\times\RR_\infty\ast}(\Bbbk_{\{t\geq0\}}\oplus\Bbbk_{\{t\leq0\}})\Potimes \SF\simto \SF\}$
and any 
$\SG_2\in\{\SF\in\BDC(\Bbbk_{M_\infty\times\RR_\infty}^\sub)\ |\
\SF\simto\Prihomsub(\rho_{M_\infty\times\RR_\infty\ast}(\Bbbk_{\{t\geq0\}}\oplus\Bbbk_{\{t\leq0\}}), \SF)\}$,
we have 
$\bfL_{M_\infty}^{\rmE, \sub}\Q_{M_\infty}^\sub(\SG_1)\simeq\SG_1$
and 
$\bfR_{M_\infty}^{\rmE, \sub}\circ\Q_{M_\infty}^\sub(\SG_2)\simeq\SG_2$.
Therefore, there exist equivalences of categories
\begin{align*}
\bfL_{M_\infty}^{\rmE, \sub} &\colon \BEC(\Bbbk_{M_\infty}^\sub)\simto
\{\SF\in\BDC(\Bbbk_{M_\infty\times\RR_\infty}^\sub)\ |\
\rho_{M_\infty\times\RR_\infty\ast}(\Bbbk_{\{t\geq0\}}\oplus\Bbbk_{\{t\leq0\}})\Potimes \SF\simto \SF\},\\
\bfR_{M_\infty}^{\rmE, \sub} &\colon \BEC(\Bbbk_{M_\infty}^\sub)\simto
\{\SF\in\BDC(\Bbbk_{M_\infty\times\RR_\infty}^\sub)\ |\ \SF\simto
\Prihomsub(\rho_{M_\infty\times\RR_\infty\ast}(\Bbbk_{\{t\geq0\}}\oplus\Bbbk_{\{t\leq0\}}), \SF)\}.
\end{align*}

Let us prove that the quotient functor admits
a left (resp.\,right) adjoint $\bfL_{M_\infty}^{\rmE, \sub}$ (resp.\,$\bfR_{M_\infty}^{\rmE, \sub}$).
Let $\SF\in\BDC(\Bbbk_{M_\infty\times\RR_\infty}^\sub)$ and $K\in\BEC(\Bbbk_{M_\infty}^\sub)$.
Since functors $\bfL_{M_\infty}^{\rmE, \sub}, \bfR_{M_\infty}^{\rmE, \sub}$ induce fully faithful functors
\begin{align*}
\bfL_{M_\infty}^{\rmE, \sub} &\colon \BEC(\Bbbk_{M_\infty}^\sub)
\hookrightarrow\BDC(\Bbbk_{M_\infty\times\RR_\infty}^\sub),\\
\bfR_{M_\infty}^{\rmE, \sub} &\colon \BEC(\Bbbk_{M_\infty}^\sub)
\hookrightarrow\BDC(\Bbbk_{M_\infty\times\RR_\infty}^\sub),
\end{align*}
there exist isomorphisms
\begin{align*}
\Hom_{\BEC(\Bbbk_{M_\infty}^\sub)}
\(\Q_{M_\infty}^\sub\SF, K\)
&\simeq
\Hom_{\BDC(\Bbbk_{M_\infty\times\RR_\infty}^\sub)}
\(\bfL_{M_\infty}^{\rmE, \sub}\Q_{M_\infty}^\sub\SF, \bfL_{M_\infty}^{\rmE, \sub}K\),\\
\Hom_{\BEC(\Bbbk_{M_\infty}^\sub)}
\(K, \Q_{M_\infty}^\sub\SF\)
&\simeq
\Hom_{\BDC(\Bbbk_{M_\infty\times\RR_\infty}^\sub)}
\(\bfR_{M_\infty}^{\rmE, \sub}K, \bfR_{M_\infty}^{\rmE, \sub}\Q_{M_\infty}^\sub\SF\).
\end{align*}
Let us prove that there exist isomorphisms in $\BDC(\Bbbk_{M_\infty\times\RR_\infty}^\sub):$
\begin{align*}
&\Prihomsub\(\rho_{M_\infty\times\RR_\infty\ast}(\Bbbk_{\{t\geq0\}}\oplus\Bbbk_{\{t\leq0\}}),
 \bfL_{M_\infty}^{\rmE, \sub}K\)
\simeq
\bfR_{M_\infty}^{\rmE, \sub}K,\\
&\rho_{M_\infty\times\RR_\infty\ast}\(\Bbbk_{\{t\geq0\}}\oplus\Bbbk_{\{t\leq0\}}\)\Potimes\bfR_{M_\infty}^{\rmE, \sub}K
\simeq
\bfL_{M_\infty}^{\rmE, \sub}K.
\end{align*}
Since $K\in \BEC(\Bbbk_{M_\infty}^\sub)$,
there exists $\SF_0\in \BDC(\Bbbk_{M_\infty\times\RR_\infty}^\sub)$ such that $K\simeq \Q_{M_\infty}^\sub\SF_0$.
Moreover there exists $F_0\in\BDC_{\I\RR-c}(\I\Bbbk_{M_\infty\times\RR_\infty})$ such that $\SF_0\simeq\lambda_{M_\infty\times\RR_\infty}F_0$
by Proposition \ref{prop3.6} (2).
Then we have isomorphisms
\begin{align*}
 &\Prihomsub\(\rho_{M_\infty\times\RR_\infty\ast}(\Bbbk_{\{t\geq0\}}\oplus\Bbbk_{\{t\leq0\}}),
 \rho_{M_\infty\times\RR_\infty\ast}(\Bbbk_{\{t\geq0\}}\oplus\Bbbk_{\{t\leq0\}})\Potimes \SF_0\)\\
 \simeq\
 &\Prihomsub\(\rho_{M_\infty\times\RR_\infty\ast}(\Bbbk_{\{t\geq0\}}\oplus\Bbbk_{\{t\leq0\}}),
 \lambda_{M_\infty\times\RR_\infty}\iota_{M_\infty\times\RR_\infty}(\Bbbk_{\{t\geq0\}}\oplus\Bbbk_{\{t\leq0\}})\Potimes \lambda_{M_\infty\times\RR_\infty}F_0\)\\
 \simeq\
 &\Prihomsub\(\rho_{M_\infty\times\RR_\infty\ast}(\Bbbk_{\{t\geq0\}}\oplus\Bbbk_{\{t\leq0\}}),
 \lambda_{M_\infty\times\RR_\infty}\(\iota_{M_\infty\times\RR_\infty}(\Bbbk_{\{t\geq0\}}\oplus\Bbbk_{\{t\leq0\}})\Potimes F_0\)\)\\
 \simeq\
 &\bfR J_{M_\infty\times\RR_\infty
}\Prihom\(\iota_{M_\infty\times\RR_\infty}(\Bbbk_{\{t\geq0\}}\oplus\Bbbk_{\{t\leq0\}}),
 \iota_{M_\infty\times\RR_\infty}(\Bbbk_{\{t\geq0\}}\oplus\Bbbk_{\{t\leq0\}})\Potimes F_0\)\\
 \simeq\
 &\bfR J_{M_\infty\times\RR_\infty
}\Prihom\(\iota_{M_\infty\times\RR_\infty}(\Bbbk_{\{t\geq0\}}\oplus\Bbbk_{\{t\leq0\}}), F_0\)\\
 \simeq\
 &\Prihomsub\(\rho_{M_\infty\times\RR_\infty\ast}(
 \Bbbk_{\{t\geq0\}}\oplus\Bbbk_{\{t\leq0\}}), \lambda_{M_\infty\times\RR_\infty}F_0\)
\end{align*}
where in the first isomorphism we used Proposition \ref{prop3.7} (3)(i),
in the second isomorphism we used Proposition \ref{prop3.12} (4),
in the third and fifth isomorphisms we used Propositions \ref{prop3.7} (4)(i) and \ref{prop3.12} (2).
In the fourth isomorphism we used the fact that
for any $G\in\BDC(\I\Bbbk_{M_\infty\times\RR_\infty})$
there exists an isomorphism in $\BDC(\I\Bbbk_{M_\infty\times\RR_\infty})$:
\begin{align*}
&\Prihom\(\iota_{M_\infty\times\RR_\infty}(\Bbbk_{\{t\geq0\}}\oplus\Bbbk_{\{t\leq0\}}),
 \iota_{M_\infty\times\RR_\infty}(\Bbbk_{\{t\geq0\}}\oplus\Bbbk_{\{t\leq0\}})\Potimes G\)\\
 \simeq\
 &\Prihom\(\iota_{M_\infty\times\RR_\infty}(\Bbbk_{\{t\geq0\}}\oplus\Bbbk_{\{t\leq0\}}), G\).
 \end{align*}
This assertion can be proved as similar way to \cite[Cor.\:4.3.11]{DK16}.
Hence there exists an isomorphism in $\BDC(\Bbbk_{M_\infty\times\RR_\infty}^\sub):$
$$\Prihomsub\(\rho_{M_\infty\times\RR_\infty\ast}(\Bbbk_{\{t\geq0\}}\oplus\Bbbk_{\{t\leq0\}}),
 \bfL_{M_\infty}^{\rmE, \sub}K\)
\simeq \bfR_{M_\infty}^{\rmE, \sub}K.$$
Moreover, we have isomorphisms
\begin{align*}
 & \rho_{M_\infty\times\RR_\infty\ast}(\Bbbk_{\{t\geq0\}}\oplus\Bbbk_{\{t\leq0\}})\Potimes
 \Prihomsub\(\rho_{M_\infty\times\RR_\infty\ast}(\Bbbk_{\{t\geq0\}}\oplus\Bbbk_{\{t\leq0\}}), \SF_0\)\\
 \simeq\
 & \lambda_{M_\infty\times\RR_\infty}\iota_{M_\infty\times\RR_\infty}(\Bbbk_{\{t\geq0\}}\oplus\Bbbk_{\{t\leq0\}})\Potimes
 \Prihomsub\(\rho_{M_\infty\times\RR_\infty\ast}(\Bbbk_{\{t\geq0\}}\oplus\Bbbk_{\{t\leq0\}}), \lambda_{M_\infty\times\RR_\infty}F_0\)\\
 \simeq\
 & \lambda_{M_\infty\times\RR_\infty}\iota_{M_\infty\times\RR_\infty}(\Bbbk_{\{t\geq0\}}\oplus\Bbbk_{\{t\leq0\}})\Potimes
 \lambda_{M_\infty\times\RR_\infty}\Prihom\(\iota_{M_\infty\times\RR_\infty}(\Bbbk_{\{t\geq0\}}\oplus\Bbbk_{\{t\leq0\}}), F_0\)\\
 \simeq\
 & \lambda_{M_\infty\times\RR_\infty}\(\iota_{M_\infty\times\RR_\infty}(\Bbbk_{\{t\geq0\}}\oplus\Bbbk_{\{t\leq0\}})\Potimes
\Prihom\(\iota_{M_\infty\times\RR_\infty}(\Bbbk_{\{t\geq0\}}\oplus\Bbbk_{\{t\leq0\}}), F_0\)\)\\
 \simeq\
 & \lambda_{M_\infty\times\RR_\infty}\(\iota_{M_\infty\times\RR_\infty}(\Bbbk_{\{t\geq0\}}\oplus\Bbbk_{\{t\leq0\}})\Potimes F_0\)\\
 \simeq\
 & \lambda_{M_\infty\times\RR_\infty}\iota_{M_\infty\times\RR_\infty}(\Bbbk_{\{t\geq0\}}\oplus\Bbbk_{\{t\leq0\}})\Potimes  \lambda_{M_\infty\times\RR_\infty}F_0
\end{align*}
where in the first isomorphism we used Proposition \ref{prop3.7} (3)(i),
in the second isomorphisms we used Proposition \ref{prop3.12} (2),
in the third and fifth isomorphisms we used Proposition \ref{prop3.12} (4).
In the fourth isomorphism we used the fact that
for any $G\in\BDC(\I\Bbbk_{M_\infty\times\RR_\infty})$
there exists an isomorphism in $\BDC(\I\Bbbk_{M_\infty\times\RR_\infty})$:
\begin{align*}
\iota_{M_\infty\times\RR_\infty}(\Bbbk_{\{t\geq0\}}\oplus\Bbbk_{\{t\leq0\}})\Potimes
\Prihom\(\iota_{M_\infty\times\RR_\infty}(\Bbbk_{\{t\geq0\}}\oplus\Bbbk_{\{t\leq0\}}), G\)
 \simeq 
 \iota_{M_\infty\times\RR_\infty}(\Bbbk_{\{t\geq0\}}\oplus\Bbbk_{\{t\leq0\}})\Potimes G.
 \end{align*}
This assertion can be proved as similar way to \cite[Cor.\:4.3.11]{DK16}.
Hence there exists isomorphism in $\BDC(\Bbbk_{M_\infty\times\RR_\infty}^\sub):$
$$\rho_{M_\infty\times\RR_\infty\ast}\(\Bbbk_{\{t\geq0\}}\oplus\Bbbk_{\{t\leq0\}}\)\Potimes\bfR_{M_\infty}^{\rmE, \sub}K
\simeq
\bfL_{M_\infty}^{\rmE, \sub}K.$$
Therefore, we have 
\begin{align*}
\Hom_{\BDC(\Bbbk_{M_\infty\times\RR_\infty}^\sub)}
(\bfL_{M_\infty}^{\rmE, \sub}\Q_{M_\infty}^\sub\SF, \bfL_{M_\infty}^{\rmE, \sub}K)
&\simeq
\Hom_{\BDC(\Bbbk_{M_\infty\times\RR_\infty}^\sub)}
(\SF, \bfR_{M_\infty}^{\rmE, \sub}K),\\
\Hom_{\BDC(\Bbbk_{M_\infty\times\RR_\infty}^\sub)}
(\bfR_{M_\infty}^{\rmE, \sub}K, \bfR_{M_\infty}^{\rmE, \sub}\Q_{M_\infty}^\sub\SF)
&\simeq
\Hom_{\BDC(\Bbbk_{M_\infty\times\RR_\infty}^\sub)}
(\bfL_{M_\infty}^{\rmE, \sub}K, \SF),
\end{align*}
and hence there exist isomorphisms
\begin{align*}
\Hom_{\BEC(\Bbbk_{M_\infty}^\sub)}
(\Q_{M_\infty}^\sub\SF, K)
&\simeq
\Hom_{\BDC(\Bbbk_{M_\infty\times\RR_\infty}^\sub)}
(\SF, \bfR_{M_\infty}^{\rmE, \sub}K),\\
\Hom_{\BEC(\Bbbk_{M_\infty}^\sub)}
(K, \Q_{M_\infty}^\sub\SF)
&\simeq
\Hom_{\BDC(\Bbbk_{M_\infty\times\RR_\infty}^\sub)}
(\bfL_{M_\infty}^{\rmE, \sub}K, \SF).
\end{align*}
Therefore, the quotient functor admits a left (resp.\,right) adjoint $\bfL_{M_\infty}^{\rmE, \sub}$ (resp.\,$\bfR_{M_\infty}^{\rmE, \sub}$).
\end{proof}

We sometimes denote $\Q_{M_\infty}^\sub$
(resp.\ $\bfL_{M_\infty}^{\rmE, \sub}, \bfR_{M_\infty}^{\rmE, \sub}$ )
by $\Q^\sub$ (resp.\ $\bfL^{\rmE, \sub}, \bfR^{\rmE, \sub}$) for short.
Let us set
\begin{align*}
\bfE^{\leq 0}(\Bbbk_{M_\infty}^\sub) & = \{K\in \BEC(\Bbbk_{M_\infty}^\sub)\ | \ 
\bfL_{M_\infty}^{\rmE, \sub}K\in \bfD^{\leq 0}(\Bbbk_{M_\infty\times\RR_\infty}^\sub)\},\\
\bfE^{\geq 0}(\Bbbk_{M_\infty}^\sub) & = \{K\in \BEC(\Bbbk_{M_\infty}^\sub)\ | \ 
\bfL_{M_\infty}^{\rmE, \sub}K\in \bfD^{\geq 0}(\Bbbk_{M_\infty\times\RR_\infty}^\sub)\},
\end{align*}
where $\(\bfD^{\leq 0}(\Bbbk_{M_\infty\times\RR_\infty}^\sub),
\bfD^{\geq 0}(\Bbbk_{M_\infty\times\RR_\infty}^\sub)\)$
is the standard t-structure on $\BDC(\Bbbk_{M_\infty\times\RR_\infty}^\sub)$.

\begin{proposition}
A pair $\(\bfE^{\leq 0}(\Bbbk_{M_\infty}^\sub),
\bfE^{\geq 0}(\Bbbk_{M_\infty}^\sub)\)$
is a t-structure on $\BEC(\Bbbk_{M_\infty}^\sub)$.
\end{proposition}

\begin{proof}
It is enough to show that for any $K\in \BEC(\Bbbk_{M_\infty}^\sub)$
there exists a distinguished triangle 
$$K_1 \to K \to K_2 \xrightarrow{+1}$$
with $K_1\in \bfE^{\leq 0}(\Bbbk_{M_\infty}^\sub), K_2\in \bfE^{\geq 1}(\Bbbk_{M_\infty}^\sub)$.

Let $k\in \BEC(\Bbbk_{M_\infty})$.
Then there exists a distinguished triangle 
$$\SF_1 \to \bfL_{M_\infty}^{\rmE, \sub}K \to \SF_2 \xrightarrow{+1}$$
in $\BDC(\Bbbk_{M_\infty\times\RR_{\infty}}^\sub)$
with $\SF_1\in \bfD^{\leq 0}(\Bbbk_{M_\infty\times\RR_\infty}^\sub), 
\SF_2\in \bfD^{\geq 1}(\Bbbk_{M_\infty\times\RR_\infty}^\sub)$.
By Corollary \ref{cor3.11} and Lemma \ref{lem3.13},
we have an isomorphism 
$\pi^{-1}\bfR\pi_{!!}\bfL_{M_\infty}^{\rmE, \sub}K\simeq 0$
in $\BDC(\Bbbk_{M_\infty\times\RR_\infty}^\sub)$
and hence
we have an isomorphism 
$$\pi^{-1}\bfR\pi_{!!}\SF_1[1] \simeq \pi^{-1}\bfR\pi_{!!}\SF_2$$
in $\BDC(\Bbbk_{M_\infty\times\RR_\infty}^\sub)$.
Since functors $\bfR\pi_{!!}$ and $\pi^{-1}$ are left t-exact
with respect to standard t-structures,
we have $\pi^{-1}\bfR\pi_{!!}\SF_2 \in \bfD^{\geq 1}(\Bbbk_{M_\infty\times\RR_\infty}^\sub)$.
Since functors $\bfR\pi_{!!}[1]$ and $\pi^{-1}$ are right t-exact
with respect to standard t-structures,
we have $\pi^{-1}\bfR\pi_{!!}\SF_1[1] \in \bfD^{\leq 0}(\Bbbk_{M_\infty\times\RR_\infty}^\sub)$.
Therefore, we have 
$$\pi^{-1}\bfR\pi_{!!}\SF_1[1] \simeq \pi^{-1}\bfR\pi_{!!}\SF_2 \simeq 0$$
and hence
there exist isomorphisms
$$\bfL_{M_\infty}^{\rmE, \sub}\Q^\sub(\SF_1) \simeq \SF_1,\ 
\bfL_{M_\infty}^{\rmE, \sub}\Q^\sub(\SF_2) \simeq \SF_2$$
in $\BDC(\Bbbk_{M_\infty\times\RR_\infty}^\sub)$
and a distinguished triangle
$$\Q^\sub \SF_1 \to K \to \Q^\sub \SF_2 \xrightarrow{+1}$$
in $\BEC(\Bbbk_{M_\infty}^\sub)$.

The proof is completed.
\end{proof}

We denote by 
\[\SH^n \colon \BEC(\Bbbk_{M_\infty}^\sub)\to\bfE^0(\Bbbk_{M_\infty}^\sub)\]
the $n$-th cohomology functor, where we set 
$\bfE^0(\Bbbk_{M_\infty}^\sub) :=
\bfE^{\leq 0}(\Bbbk_{M_\infty}^\sub)\cap\bfE^{\geq 0}(\Bbbk_{M_\infty}^\sub)$.

By Proposition \ref{prop3.9},
the convolution functors can be lifted to the triangulated category $\BEC(\Bbbk_{M_\infty}^\sub)$.
We denote them by the same symbols $\Potimes$, $\Prihomsub$.
Namely, we obtain functors
\begin{align*}
(\cdot)\Potimes(\cdot)&\colon
\BEC(\Bbbk_{M_\infty}^\sub)\times \BEC(\Bbbk_{M_\infty}^\sub)\to\BEC(\Bbbk_{M_\infty}^\sub),\\
\Prihomsub(\cdot, \cdot)&\colon
\BEC(\Bbbk_{M_\infty}^\sub)^\op\times \BEC(\Bbbk_{M_\infty}^\sub)\to\BEC(\Bbbk_{M_\infty}^\sub)
\end{align*}
 which are defined by
\begin{align*}
\Q_{M_\infty}^\sub(\SF)\Potimes \Q_{M_\infty}^\sub(\SG)
& := 
\Q_{M_\infty}^\sub\(\SF\Potimes\SG\),\\
\Prihomsub(\Q_{M_\infty}^\sub(\SF), \Q_{M_\infty}^\sub(\SG))
& := \Q_{M_\infty}^\sub\(\Prihomsub(\SF, \SG)\),
\end{align*}
for $\SF, \SG\in\BDC(\Bbbk_{M_\infty\times \RR_\infty}^\sub)$.
Moreover, by Proposition \ref{prop3.4} (4),
for a morphism $f\colon M_\infty\to N_\infty$ of real analytic bordered spaces,
the following functors are well defined
\begin{align*}
\bfE f_\ast&\colon
\BEC(\Bbbk_{M_\infty}^\sub)\to\BEC(\Bbbk_{N_\infty}^\sub),
\hspace{7pt}
\Q_{M_\infty}^\sub\SF
\longmapsto
\Q_{N_\infty}^\sub\(\bfR f_{\RR_\infty\ast}\SF\)
\\
\bfE f^{-1}&\colon
\BEC(\Bbbk_{N_\infty}^\sub)\to\BEC(\Bbbk_{M_\infty}^\sub),
\hspace{7pt}
\Q_{N_\infty}^\sub\SG
\longmapsto
\Q_{M_\infty}^\sub\(f_{\RR_\infty}^{-1}\SG\),\\
\bfE f_{!!}&\colon
\BEC(\Bbbk_{M_\infty}^\sub)\to\BEC(\Bbbk_{N_\infty}^\sub),
\hspace{7pt}
\Q_{M_\infty}^\sub\SF
\longmapsto
\Q_{N_\infty}^\sub\(\bfR f_{\RR_\infty!!}\SF\),\\
\bfE f^{!}&\colon
\BEC(\Bbbk_{N_\infty}^\sub)\to\BEC(\Bbbk_{M_\infty}^\sub),
\hspace{7pt}
\Q_{N_\infty}^\sub\SG
\longmapsto
\Q_{M_\infty}^\sub\(f_{\RR_\infty}^{!}\SG\).
\end{align*}


Let us define external hom functors
\begin{align*}
\rihom^{\rmE,\sub}(\cdot, \cdot)&\colon
\BEC(\Bbbk_{M_\infty}^\sub)^\op\times\BEC(\Bbbk_{M_\infty}^\sub)\to \BDC(\Bbbk_{M_\infty}^\sub),\\
\rhom^{\rmE,\sub}(\cdot, \cdot)&\colon
\BEC(\Bbbk_{M_\infty}^\sub)^\op\times\BEC(\Bbbk_{M_\infty}^\sub)\to \BDC(\Bbbk_{M}),\\
\rHom^{\rmE, \sub}(\cdot, \cdot)&\colon
\BEC(\Bbbk_{M_\infty}^\sub)^\op\times\BEC(\Bbbk_{M_\infty}^\sub)\to \BDC(\Bbbk),
\end{align*}
by 
\begin{align*}
\rihom^{\rmE, \sub}\(\Q_{M_\infty}^\sub \SF_1, \Q_{M_\infty}^\sub \SF_2\)
&:=\bfR\pi_\ast\rihom^\sub(\SF_1, \SF_2),\\
\rhom^{\rmE, \sub}\(\Q_{M_\infty}^\sub \SF_1, \Q_{M_\infty}^\sub \SF_2\)
&:= \rho_{M_\infty\ast}\rihom^\rmE\(\Q_{M_\infty}^\sub\SF_1, \Q_{M_\infty}^\sub\SF_2\),\\
\rHom^{\rmE, \sub}\(\Q_{M_\infty}^\sub\SF_1, \Q_{M_\infty}^\sub\SF_2\)
&:= \bfR\Gamma\(M; \rhom^\rmE(\Q_{M_\infty}^\sub\SF_1, \Q_{M_\infty}^\sub\SF_2)\),
\end{align*}
for $\SF_1, \SF_2\in\BEC(\Bbbk_{M_\infty\times\RR_\infty}^\sub)$.
Note that for any $K_1, K_2\in\BEC(\Bbbk_{M_\infty}^\sub)$,
we have
$$\Hom_{\BEC(\Bbbk_{M_\infty}^\sub)}(K_1, K_2)\simeq\SH^0\rHom^{\rmE, \sub}\(K_1, K_2\).$$
Moreover, 
for $\SF_0\in\BDC(\Bbbk_{M_\infty}^\sub)$ and $\SF\in\BDC(\Bbbk_{M_\infty\times\RR_\infty}^\sub)$,
the objects 
\begin{align*}
\pi^{-1}\SF_0\otimes \Q_{M_\infty}^\sub\SF
& :=\Q_{M_\infty}^\sub\(\pi^{-1}\SF_0\otimes \SF\),\\
\rihom^\sub\(\pi^{-1}\SF_0, \Q_{M_\infty}^\sub\SF\)
& :=\Q_{M_\infty}^\sub\(\rihom^\sub(\pi^{-1}\SF_0, \Q_{M_\infty}^\sub\SF)\) 
\end{align*}
are well defined and hence the following functors are well defined
\begin{align*}
\pi^{-1}(\cdot)\otimes (\cdot)
&\colon
\BDC(\Bbbk_{M_\infty}^\sub)\times\BEC(\Bbbk_{M_\infty}^\sub)\to \BEC(\Bbbk_{M_\infty}^\sub),\\
\rihom^\sub\(\pi^{-1}(\cdot), \cdot\)&\colon
\BDC(\Bbbk_{M_\infty}^\sub)^\op\times\BEC(\Bbbk_{M_\infty}^\sub)\to \BEC(\Bbbk_{M_\infty}^\sub).
\end{align*}


At the end of this subsection,
let us prove that these functorss have several properties as similar to the classical sheaves.
\begin{proposition}\label{prop3.14}
Let $f\colon M_\infty\to N_\infty$ be a morphism of real analytic bordered spaces.

\begin{itemize}
\item[\rm (1)]
\begin{itemize}
\item[\rm (i)]
For any $K_1, K_2, K_3\in\BEC(\Bbbk_{M_\infty}^\sub)$,
one has
\begin{align*}
\Prihomsub\(K_1\Potimes K_2, K_3\)
&\simeq
\Prihomsub\(K_1, \Prihomsub(K_2, K_3)\),\\
\rihom^{\rmE, \sub}\(K_1\Potimes K_2, K_3\)
&\simeq
\rihom^{\rmE, \sub}\(K_1, \Prihomsub(K_2, K_3)\),\\
\rhom^{\rmE, \sub}\(K_1\Potimes K_2, K_3\)
&\simeq
\rhom^{\rmE, \sub}\(K_1, \Prihomsub(K_2, K_3)\),\\
\rHom^{\rmE, \sub}\(K_1\Potimes K_2, K_3\)
&\simeq
\rHom^{\rmE, \sub}\(K_1\Prihomsub(K_2, K_3)\),\\
\Hom_{\BEC(\Bbbk_{M_\infty}^\sub)}\(K_1 \Potimes K_2, K_3\)
&\simeq
\Hom_{\BEC(\Bbbk_{M_\infty}^\sub)}\(K_1, \Prihomsub(K_2, K_3)\).
\end{align*}

\item[\rm (ii)]
For any $K\in\BEC(\Bbbk_{M_\infty}^\sub)$ and any $L\in\BEC(\Bbbk_{N_\infty}^\sub)$,
one has
\begin{align*}
\bfE f_\ast\Prihomsub\(\bfE f^{-1}L, K\)
&\simeq
\Prihomsub\(L, \bfE f_{\ast}K\),\\
\bfR f_\ast\rihom^{\rmE, \sub}\(\bfE f^{-1}L, K\)
&\simeq
\rihom^{\rmE, \sub}\(L, \bfE f_{\ast}K\),\\
\bfR f_\ast\rhom^{\rmE, \sub}\(\bfE f^{-1}L, K\)
&\simeq
\rhom^{\rmE, \sub}\(L, \bfE f_{\ast}K\),\\
\rHom^{\rmE, \sub}\(\bfE f^{-1}L, K\)
&\simeq
\rHom^{\rmE, \sub}\(L, \bfE f_{\ast}K\),\\
\Hom_{\BEC(\Bbbk_{M_\infty}^\sub)}\(\bfE f^{-1}L, K\)
&\simeq
\Hom_{\BEC(\Bbbk_{N_\infty}^\sub)}\(L, \bfE f_{\ast}K\).
\end{align*}

\item[\rm (iii)]
For any $K\in\BEC(\Bbbk_{M_\infty}^\sub)$ and any $L\in\BEC(\Bbbk_{N_\infty}^\sub)$,
one has
\begin{align*}
\Prihomsub\(\bfE f_{!!}K, L\)
&\simeq
\bfE f_\ast\Prihomsub\(K, \bfE f^{!}L\),\\
\rihom^{\rmE, \sub}\(\bfE f_{!!}K, L\)
&\simeq
\bfR f_\ast\rihom^{\rmE, \sub}\(K, \bfE f^{!}L\),\\
\rhom^{\rmE, \sub}\(\bfE f_{!!}K, L\)
&\simeq
\bfR f_\ast\rhom^{\rmE, \sub}\(K, \bfE f^{!}L\),\\
\bfR\Hom^{\rmE, \sub}\(\bfE f_{!!}K, L\)
&\simeq
\bfR\Hom^{\rmE, \sub}\(K, \bfE f^{!}L\),\\
\Hom_{\BEC(\Bbbk_{N_\infty}^\sub)}\(\bfE f_{!!}K, L\)
&\simeq
\Hom_{\BEC(\Bbbk_{M_\infty}^\sub)}\(K, \bfE f^{!}L, \).
\end{align*}
\end{itemize}

\item[\rm (2) ]
For any $K, K_1, K_2\in\BEC(\Bbbk_{M_\infty}^\sub)$ and any $L, L_1, L_2\in\BEC(\Bbbk_{N_\infty}^\sub)$,
one has,
\begin{align*}
\bfE f^{-1}\(K_1\Potimes K_2\)
&\simeq
\bfE f^{-1}K_1\Potimes \bfE f^{-1}K_2,\\
\bfE f_{!!}\(K\Potimes\bfE f^{-1}L\)
&\simeq
\bfE f_{!!}K\Potimes L,\\
\bfE f^!\Prihomsub\(L_1, L_2\)
&\simeq
\Prihomsub\(\bfE f^{-1}L_1, \bfE f^!L_2\)
\end{align*}

\item[\rm (3)]
For a cartesian diagram 
\[\xymatrix@M=5pt@R=20pt@C=40pt{
M'_\infty\ar@{->}[r]^-{f'}\ar@{->}[d]_-{g'} & N'_\infty\ar@{->}[d]^-{g}\\
M_\infty\ar@{->}[r]_-{f}  & N_\infty}\]
and any $K\in \BEC(\Bbbk_{M_\infty}^\sub)$,
one has
\begin{align*}
\bfE g^{-1}\bfE f_{!!}K&\simeq \bfE f'_{!!}\bfE g'^{-1}K,\\
\bfE g^{!}\bfE f_{\ast}K&\simeq \bfE f'_{\ast}\bfE g'^{!}K.
\end{align*}

\item[\rm (4)]
\begin{itemize}
\item[\rm (i)]
For any $\SF\in\BDC(\Bbbk_{M_\infty}^\sub)$ and any $K_1, K_2\in \BEC(\Bbbk_{M_\infty}^\sub)$,
one has
\begin{align*}
\Prihomsub\(\pi^{-1}\SF\otimes K_1, K_2\)
&\simeq
\Prihomsub\(K_1, \rihom^\sub(\pi^{-1}\SF, K_2\)\\
&\simeq
\rihom^\sub\(\pi^{-1}\SF, \Prihomsub(K_1, K_2)\),\\
\rihom^{\rmE, \sub}\(\pi^{-1}\SF\otimes K_1, K_2\)
&\simeq
\rihom^{\rmE,\sub}\(K_1, \rihom^\sub(\pi^{-1}\SF, K_2\)\\
&\simeq
\rihom^{\sub}\(\SF, \rihom^{\rmE, \sub}(K_1, K_2)\),\\
\rhom^{\rmE, \sub}\(\pi^{-1}\SF\otimes K_1, K_2\)
&\simeq
\rhom^{\rmE,\sub}\(K_1, \rihom^\sub(\pi^{-1}\SF, K_2)\)\\
&\simeq
\rhom^{\sub}\(\SF, \rihom^{\rmE, \sub}(K_1, K_2)\),\\
\rHom^{\rmE, \sub}\(\pi^{-1}\SF\otimes K_1, K_2\)
&\simeq
\rHom^{\rmE,\sub}\(K_1, \rihom^\sub(\pi^{-1}\SF, K_2)\)\\
&\simeq
\rHom\(\SF, \rihom^{\rmE, \sub}(K_1, K_2)\),\\
\Hom_{\BEC(\Bbbk_{M_\infty}^\sub)}\(\pi^{-1}\SF\otimes K_1, K_2\)
&\simeq
\Hom_{\BEC(\Bbbk_{M_\infty}^\sub)}\(K_1, \rihom^\sub(\pi^{-1}\SF, K_2)\)\\
&\simeq
\Hom_{\BDC(\Bbbk_{M_\infty}^\sub)}\(\SF, \rihom^{\rmE, \sub}(K_1, K_2)\),\\
\end{align*}

\item[\rm (ii)]
For any $\SF\in\BDC(\Bbbk_{M_\infty}^\sub)$,
any $\SG\in\BDC(\Bbbk_{N_\infty}^\sub)$,
any $K\in \BEC(\Bbbk_{M_\infty}^\sub)$
and any $L\in \BEC(\Bbbk_{N_\infty}^\sub)$,
one has
\begin{align*}
\bfE f^{-1}\(\pi^{-1}\SF\otimes L\)
&\simeq
\pi^{-1}f^{-1}\SF\otimes \bfE f^{-1}L,\\
\bfE f_{!!}\(\pi^{-1}\SF\otimes \bfE f^{-1}L\)
&\simeq
\pi^{-1}\bfR f_{!!}\SF\otimes L,\\
\bfE f_{!!}\(\pi^{-1}f^{-1}\SF\otimes K\)
&\simeq
\pi^{-1}\SF\otimes \bfE f_{!!}K,\\
\bfE f^!\Prihomsub\(\pi^{-1}\SG, L\)
&\simeq
\Prihomsub\(\pi^{-1}f^{-1}\SG, \bfE f^!L\).
\end{align*}

\item[\rm (iii)]
For any $\SF\in\BDC(\Bbbk_{M_\infty}^\sub)$
and any $K, L\in\BEC(\Bbbk_{M_\infty}^\sub)$,
one has
\begin{align*}
\pi^{-1}\SF\otimes \(K\Potimes L\)
&\simeq
\(\pi^{-1}\SF\otimes K\)\Potimes L,\\
\rihom^\sub\(\pi^{-1}\SF, \Prihomsub(K, L)\)
&\simeq
\Prihomsub\(\pi^{-1}\SF\otimes K, L\)\\
&\simeq
\Prihomsub\(K,\rihom^\sub(\pi^{-1}\SF, L\),\\
\rihom^{\bfE, \sub}\(K\Potimes L, \SF\)
&\simeq
\rihom^{\bfE, \sub}\(K, \Prihomsub(L, \SF)\).
\end{align*}
\end{itemize}
\end{itemize}
\end{proposition}

\begin{proof}
(1)(i)
The first (resp.\,second) assertion follows from Proposition \ref{prop3.8} (resp.\,\ref{prop3.4} (2)).
The third (resp.\,fourth, fifth) assertion follows from the second (resp.\, third, fourth) one.
\medskip

\noindent
(ii)
The first (resp.\,second) assertion follows from Proposition \ref{prop3.8} (resp.\,\ref{prop3.4} (2)).
The third (resp.\,fourth, fifth) assertion follows from the second (resp.\, third, fourth) one.
\medskip

\noindent
(iii)
The first (resp.\,second) assertion follows from Proposition \ref{prop3.8} (resp.\,\ref{prop3.4} (2)).
The third (resp.\,fourth, fifth) assertion follows from the second (resp.\, third, fourth) one.
\medskip

\noindent
(2)
The three assertions follow from Proposition \ref{prop3.8}.
\medskip

\noindent
(3)
The two assertions follow from Proposition \ref{prop3.4} (4).
\medskip

\noindent
(4)(i)
By the definition 
we have
\begin{align*}
\Prihomsub\(\pi^{-1}\SF\otimes K_1, K_2\)
&\simeq
\Q_{M_\infty}^\sub\Prihomsub
\(\pi^{-1}\SF\otimes \bfL_{M_\infty}^{\rmE, \sub}K_1, \bfR_{M_\infty}^{\rmE, \sub}K_2\)
\end{align*}
and hence by Proposition \ref{prop3.9}
there exist isomorphisms
\begin{align*}
&\Q_{M_\infty}^\sub\Prihomsub
\(\pi^{-1}\SF\otimes \bfL_{M_\infty}^{\rmE, \sub}K_1, \bfR_{M_\infty}^{\rmE, \sub}K_2\)\\
\simeq
&\,\Q_{M_\infty}^\sub\Prihomsub
\(\bfL_{M_\infty}^{\rmE, \sub}K_1, \rihom^\sub(\pi^{-1}\SF, \bfR_{M_\infty}^{\rmE, \sub}K_2\)\\
\simeq
&\,\Prihomsub\(K_1, \rihom^\sub(\pi^{-1}\SF, K_2\)
\end{align*}
and
\begin{align*}
&\Q_{M_\infty}^\sub\Prihomsub
\(\pi^{-1}\SF\otimes \bfL_{M_\infty}^{\rmE, \sub}K_1, \bfR_{M_\infty}^{\rmE, \sub}K_2\)\\
\simeq
&\,\Q_{M_\infty}^\sub\rihom^\sub
\(\pi^{-1}\SF, \Prihomsub(\bfL_{M_\infty}^{\rmE, \sub}K_1, \bfR_{M_\infty}^{\rmE, \sub}K_2)\)\\
\simeq
&\,\rihom^\sub\(\pi^{-1}\SF, \Prihomsub(K_1, K_2\).
\end{align*}
\medskip

\noindent
By the definition we have
$$\rihom^{\rmE, \sub}\(\pi^{-1}\SF\otimes K_1, K_2\)
\simeq
\bfR\pi_{\ast}\rihom^\sub
\(\pi^{-1}\SF\otimes\bfL_{M_\infty}^{\rmE, \sub}K_1, \bfL_{M_\infty}^{\rmE, \sub}K_2\)$$
and hence by Proposition \ref{prop3.4} (2)
there exist isomorphisms
\begin{align*}
&\bfR\pi_{\ast}\rihom^\sub
\(\pi^{-1}\SF\otimes\bfL_{M_\infty}^{\rmE, \sub}K_1, \bfL_{M_\infty}^{\rmE, \sub}K_2\)\\
\simeq
&\bfR\pi_{\ast}\rihom^\sub
\(\bfL_{M_\infty}^{\rmE, \sub}K_1, \rihom^\sub\(\pi^{-1}\SF, \bfL_{M_\infty}^{\rmE, \sub}K_2\)\)\\
\simeq
&\rihom^{\rmE,\sub}\(K_1, \rihom^\sub(\pi^{-1}\SF, K_2\)
\end{align*}
and
\begin{align*}
&\bfR\pi_{\ast}\rihom^\sub
\(\pi^{-1}\SF\otimes\bfL_{M_\infty}^{\rmE, \sub}K_1, \bfL_{M_\infty}^{\rmE, \sub}K_2\)\\
\simeq
&\bfR\pi_{\ast}\rihom^\sub
\(\pi^{-1}\SF, \rihom^\sub\(\bfL_{M_\infty}^{\rmE, \sub}K_1, \bfL_{M_\infty}^{\rmE, \sub}K_2\)\)\\
\simeq
&\rihom^\sub\(\SF,
\bfR\pi_{\ast}\rihom^\sub\(\bfL_{M_\infty}^{\rmE, \sub}K_1, \bfL_{M_\infty}^{\rmE, \sub}K_2\)\)\\
\simeq
&\rihom^{\sub}\(\SF, \rihom^{\rmE, \sub}(K_1, K_2\).
\end{align*}
\medskip

\noindent
The third (resp.\,fourth, fifth) assertion follows from the second (resp.\,third, fourth) one.
\medskip

\noindent
The four assertions of (ii) follow from Proposition \ref{prop3.4} (3), (4).
\medskip

\noindent
The three assertions of (iii) follow from Proposition \ref{prop3.9}.
\end{proof}


\subsection{Relation between Enhanced Ind-Sheaves and Enhanced Subanalytic Sheaves }\label{subsec3.4}
In this subsection, we shall explain a relation between enhanced subanalytic sheaves and enhanced ind-sheaves.
Theorems \ref{main1} and \ref{main2} are main results of this paper.

Let $M_\infty = (M, \che{M})$ be a real analytic bordered space.
Let us consider a quotient category 
$$
\BEC_{\I\RR-c}(\I\Bbbk_{M_\infty}) :=
\BDC_{\I\RR-c}(\I\Bbbk_{M_\infty \times\RR_\infty})/\pi^{-1}\BDC_{\I\RR-c}(\I\Bbbk_{M_\infty}). 
$$
Note that this is a full triangulated subcategory of $\BEC(\I\Bbbk_{M_\infty})$
by using
$\pi^{-1}\BDC_{\I\RR-c}(\I\Bbbk_{M_\infty})
= \pi^{-1}\BDC(\I\Bbbk_{M_\infty})\cap\BDC_{\I\RR-c}(\I\Bbbk_{M_\infty\times\RR_\infty})$
and \cite[Prop.\:1.6.10]{KS90}.
Note also that $\Bbbk_{M_\infty}^\rmE\in\BEC_{\I\RR-c}(\I\Bbbk_{M_\infty})$.
Moreover, Propositions \ref{prop3.15} below follows from Lemma \ref{lem3.5} and Proposition \ref{prop3.12} (3).
\begin{proposition}\label{prop3.15}
Let $f\colon M_\infty\to N_\infty$ be a morphism of real analytic bordered spaces
associated with a morphism $\che{f}\colon\che{M}\to\che{N}$ of real analytic manifolds.
The functors below are well defined:
\begin{itemize}
\item[\rm (1)]
$e_{M_\infty}\colon \BDC_{\I\RR-c}(\I\Bbbk_{M_\infty})\to\BEC_{\I\RR-c}(\I\Bbbk_{M_\infty})$,

\item[\rm (2)]
$(\cdot)\Potimes(\cdot)\colon
\BEC_{\I\RR-c}(\I\Bbbk_{M_\infty})\times \BEC_{\I\RR-c}(\I\Bbbk_{M_\infty})
\to\BEC_{\I\RR-c}(\I\Bbbk_{M_\infty})$,

\item[\rm (3)]
$\bfE f^{-1}\colon\BEC_{\I\RR-c}(\I\Bbbk_{N_\infty})\to\BEC_{\I\RR-c}(\I\Bbbk_{M_\infty})$,

\item[\rm (4)]
$\bfE f_{!!}\colon\BEC_{\I\RR-c}(\I\Bbbk_{M_\infty})\to\BEC_{\I\RR-c}(\I\Bbbk_{N_\infty})$,

\item[\rm (5)]
$\bfE f^{!}\colon\BEC_{\I\RR-c}(\I\Bbbk_{N_\infty})\to\BEC_{\I\RR-c}(\I\Bbbk_{M_\infty})$.
\end{itemize}
\end{proposition}

By Proposition \ref{prop3.7} (2)(iii), (3)(v),
the following functors are well defined
\begin{align*}
I_{M_\infty}^\rmE&\colon\BEC(\Bbbk_{M_\infty}^\sub)\to\BEC(\I\Bbbk_{M_\infty}),
\hspace{7pt}
\Q_{M_\infty}^\sub\SF\mapsto \Q_{M_\infty}I_{M_\infty\times\RR_\infty}\SF,\\
J_{M_\infty}^\rmE&\colon\BEC(\I\Bbbk_{M_\infty})\to\BEC(\Bbbk_{M_\infty}^\sub),
\hspace{7pt}
\Q_{M_\infty}F\mapsto \Q_{M_\infty}^\sub\bfR J_{M_\infty\times\RR_\infty}F.
\end{align*}  

\begin{theorem}\label{main1}
Let $M_\infty = (M, \che{M})$ be a real analytic bordered space.
Then we have
\begin{itemize}
\item[\rm (1)]
a pair $(I_{M_\infty}^\rmE, J_{M_\infty}^\rmE)$ is an adjoint pair
and there exists a canonical isomorphism $\id\simto J_{M_\infty}^\rmE\circ I_{M_\infty}^\rmE$,

\item[\rm (2)]
there exists an equivalence of triangulated categories:
\[\xymatrix@M=7pt@C=45pt{
\BEC(\Bbbk_{M_\infty}^\sub)\ar@<0.8ex>@{->}[r]^-{I_{M_\infty}^\rmE}_-\sim
&
\BEC_{\I{\RR-c}}(\I\Bbbk_{M_\infty})
\ar@<0.8ex>@{->}[l]^-{J_{M_\infty}^\rmE}.
}\]
\end{itemize}
\end{theorem}

\begin{proof}
(1)
For any $K\in \BEC(\Bbbk_{M_\infty}^\sub)$ and any $L\in\BEC(\I\Bbbk_{M_\infty})$,
there exist isomorphisms
\begin{align*}
\Hom_{\BEC(\I\Bbbk_{M_\infty})}\(I_{M_\infty}^\rmE K, L\)
&\simeq
\Hom_{\BEC(\I\Bbbk_{M_\infty})}\(
\Q_{M_\infty}I_{M_\infty\times\RR_\infty}\bfL_{M_\infty}^{\rmE, \sub}K, L\)\\
&\simeq
\Hom_{\BEC(\Bbbk_{M_\infty}^\sub)}\(K, 
\Q_{M_\infty}^\sub\bfR J_{M_\infty\times\RR_\infty}\bfR_{M_\infty}^{\rmE}L\)\\
&\simeq
\Hom_{\BEC(\Bbbk_{M_\infty}^\sub)}\(K, J_{M_\infty}^\sub L\),
\end{align*}
where in the second isomorphism we used Proposition \ref{prop3.6} (1) and Lemma \ref{lem3.13}.
This implies that a pair $(I_{M_\infty}^\rmE, J_{M_\infty}^\rmE)$ is an adjoint pair.
Moreover, by Proposition \ref{prop3.6} (1) it is clear that $\id\simto J_{M_\infty}^\rmE\circ I_{M_\infty}^\rmE$.
\medskip

\noindent
(2)
Since the functor $I_{M_\infty\times\RR_\infty}\colon \BDC(\Bbbk_{M_\infty\times\RR_\infty}^\sub)\to 
\BDC_{\I\RR-c}(\I\Bbbk_{M_\infty\times\RR_\infty})$ is well defined,
for any $K\in\BEC(\Bbbk_{M_\infty}^\sub)$
we have $$I_{M_\infty}^\rmE K
=\Q_{M_\infty}I_{M_\infty\times\RR_\infty}\bfL_{M_\infty}^{\rmE, \sub}K
\in\BEC_{\I\RR-c}(\I\Bbbk_{M_\infty}).$$
Let $L\in \BEC_{\I\RR-c}(\I\Bbbk_{M_\infty})$.
Then there exists $G\in\BDC_{\I\RR-c}(\I\Bbbk_{M_\infty\times\RR_\infty})$
such that $L \simeq \Q_{M_\infty}G$ and hence
$$I_{M_\infty}^\rmE\bfR J_{M_\infty}^\rmE L \simeq
\Q_{M_\infty}I_{M_\infty\times\RR_\infty}\bfR J_{M_\infty\times\RR_\infty}G
\simeq \Q_{M_\infty}G \simeq L,$$
where in the second isomorphism we used Proposition \ref{prop3.6} (2).
Therefore, the proof is completed.
\end{proof}

We will denote by $\lambda_{M_\infty}^\rmE
\colon \BEC_{\I{\RR-c}}(\I\Bbbk_{M_\infty})\simto \BEC(\Bbbk_{M_\infty}^\sub)$
 the inverse functor of
$I_{M_\infty}^\rmE\colon \BEC(\Bbbk_{M_\infty}^\sub)\simto \BEC_{\I{\RR-c}}(\I\Bbbk_{M_\infty})$.
There exists a commutativity between the various functors and functors $I^\rmE, J^\rmE, \lambda^\rmE$ as below.

\begin{proposition}\label{prop3.18}
Let $f\colon M_\infty\to N_\infty$ be a morphism of real analytic bordered spaces
associated with a morphism $\che{f}\colon\che{M}\to\che{N}$ of real analytic manifolds.
\begin{itemize}
\item[\rm (1)]
For any $K, K_1, K_2\in\BEC(\Bbbk_{M_\infty}^\sub)$ and any $L\in\BEC(\I\Bbbk_{M_\infty})$,
we have
\begin{itemize}
\item[\rm (i)]
$I_{M_\infty}^\rmE\(K_1\Potimes K_2\)
\simeq
I_{M_\infty}^\rmE K_1\Potimes I_{M_\infty}^\rmE K_2,$

\item[\rm (ii)]
$J_{M_\infty}^\rmE\Prihom(I_{M_\infty}^\rmE K, L)
\simeq
\Prihomsub(K, J_{M_\infty}^\rmE L)$

\item[\rm (iii)]
$\bfR J_{M_\infty}\rihom^{\rmE}(I_{M_\infty}^\rmE K, L)
\simeq
\rihom^{\rmE, \sub}(K, J_{M_\infty}^\rmE L).$

\end{itemize}

\item[\rm (2)]
For any $K\in\BEC(\Bbbk_{M_\infty}^\sub)$ and any $L\in\BEC(\Bbbk_{N_\infty}^\sub)$,
we have
\begin{itemize}
\item[\rm (i)]
$I_{M_\infty}^\rmE\bfE f^{-1}L\simeq \bfE f^{-1}I_{N_\infty}^\rmE L$,

\item[\rm (ii)]
$\bfE f_{!!}I_{M_\infty}^\rmE K\simeq I_{N_\infty}^\rmE\bfE f_{!!}K$,

\item[\rm (iii)]
$I_{M_\infty}^\rmE\bfE f^{!}L\simeq \bfE f^{!}I_{N_\infty}^\rmE L$.
\end{itemize}

\item[\rm (3)]
For any $K\in\BEC(\I\Bbbk_{M_\infty})$ and any $L\in\BEC(\I\Bbbk_{N_\infty})$,
we have
\begin{itemize}
\item[\rm (i)]
$\bfE f_{\ast}J_{M_\infty}^\rmE K\simeq J_{N_\infty}^\rmE\bfE f_{\ast}K$,

\item[\rm (ii)]
$J_{M_\infty}^\rmE\bfE f^{!}L\simeq \bfE f^{!}J_{N_\infty}^\rmE L$.
\end{itemize}

\item[\rm (4)]
For any $K\in\BEC_{\I\RR-c}(\I\Bbbk_{M_\infty})$ and any $L\in\BEC_{\I\RR-c}(\I\Bbbk_{N_\infty})$,
we have
\begin{itemize}
\item[\rm (i)]
$\lambda_{M_\infty}^\rmE\bfE f^{-1}L\simeq \bfE f^{-1}\lambda_{N_\infty}^\rmE L$,

\item[\rm (ii)]
$\bfE f_{!!}\lambda_{M_\infty}^\rmE K\simeq \lambda_{N_\infty}^\rmE\bfE f_{!!}K$,

\item[\rm (iii)]
$\lambda_{M_\infty}^\rmE\(K_1\Potimes K_2\)
\simeq \lambda_{M_\infty}^\rmE K_1\Potimes \lambda_{M_\infty}^\rmE K_2$.
\end{itemize}

\end{itemize}
\end{proposition}

\begin{proof}
Let us denote by ${f}_{\RR_\infty}\colon M_\infty\times\RR_\infty\to N_\infty\times\RR_\infty$
the morphism $f\times\id_{\RR_\infty}$ of real analytic bordered spaces.
\medskip

\noindent
(1)
Let $K, K_1, K_2\in\BEC(\Bbbk_{M_\infty}^\sub)$, $L\in\BEC(\I\Bbbk_{M_\infty})$,
then there exist $\SF, \SF_1, \SF_2\in\BDC(\Bbbk_{M_\infty\times\RR_\infty}^\sub)$,
$G\in\BDC(\I\Bbbk_{M_\infty\times\RR_\infty})$
such that $K\simeq\Q_{M_\infty}^\sub\SF, K_1\simeq\Q_{M_\infty}^\sub\SF_1,
K_2\simeq\Q_{M_\infty}^\sub\SF_2, L\simeq\Q_{M_\infty}G$.
\smallskip

\noindent
(i)
By Proposition \ref{prop3.12} (1), we have
\begin{align*}
I_{M_\infty}^\rmE\(K_1\Potimes K_2\)
&\simeq
\Q_{M_\infty}I_{M_\infty\times\RR_\infty}\(\SF_1\Potimes \SF_2\)\\
&\simeq
\Q_{M_\infty}\(I_{M_\infty\times\RR_\infty}\SF_1\Potimes I_{M_\infty\times\RR_\infty}\SF_2\)\\
&\simeq
I_{M_\infty}^\rmE K_1\Potimes I_{M_\infty}^\rmE K_2.
\end{align*}

\noindent
(ii)
By Proposition \ref{prop3.12} (2), we have
\begin{align*}
J_{M_\infty}^\rmE\Prihom(I_{M_\infty}^\rmE K, L)
&\simeq
\Q_{M_\infty}^\sub \bfR J_{M_\infty\times\RR_\infty}\Prihom(I_{M_\infty\times\RR_\infty}\SF, G)\\
&\simeq
\Q_{M_\infty}^\sub\Prihomsub(\SF, \bfR J_{M_\infty\times\RR_\infty}G)\\
&\simeq
\Prihomsub(K, J_{M_\infty}^\rmE L).
\end{align*}

\noindent
(iii)
By Proposition \ref{prop3.7} (1), (3)(iv), we have
\begin{align*}
\bfR J_{M_\infty}\rihom^{\rmE}(I_{M_\infty}^\rmE K, L)
&\simeq
\bfR J_{M_\infty}\pi_\ast\rihom(I_{M_\infty\times\RR_\infty}\SF, \SG)\\
&\simeq
\pi_\ast\bfR J_{M_\infty\times\RR_\infty}\rihom(I_{M_\infty\times\RR_\infty}\SF, \SG)\\
&\simeq
\pi_\ast\rihom^\sub(\SF, \bfR J_{M_\infty\times\RR_\infty}\SG)\\
&\simeq
\rihom^{\rmE, \sub}(K, J_{M_\infty}^\rmE L).
\end{align*}
\medskip

\noindent
(2)
Let $K\in\BEC(\Bbbk_{M_\infty}^\sub)$, $L\in\BEC(\Bbbk_{N_\infty}^\sub)$.
Then there exist $\SF\in\BDC(\Bbbk_{M_\infty\times\RR_\infty}^\sub)$,
$\SG\in\BDC(\Bbbk_{N_\infty\times\RR_\infty}^\sub)$
such that $K\simeq\Q_{M_\infty}^\sub\SF,\ L\simeq\Q_{N_\infty}^\sub\SG$.
\smallskip

\noindent
(i)
By Proposition \ref{prop3.7} (2)(iii), we have
\begin{align*}
I_{M_\infty}^\rmE\bfE f^{-1}L
\simeq 
\Q_{M_\infty}I_{M_\infty\times\RR_\infty}f_{\RR_\infty}^{-1}\SG
\simeq 
\Q_{M_\infty}f_{\RR_\infty}^{-1}I_{N_\infty\times\RR_\infty}\SG\simeq 
\bfE f^{-1}I_{N_\infty}^\rmE L.
\end{align*}

\noindent
(ii)
By Proposition \ref{prop3.7} (2)(iv), we have
\begin{align*}
\bfE f_{!!}I_{M_\infty}^\rmE K
\simeq 
\Q_{N_\infty}f_{\RR_\infty!!}I_{M_\infty\times\RR_\infty}\SF
\simeq 
\Q_{N_\infty}I_{N_\infty\times\RR_\infty}f_{\RR_\infty!!}\SF
\simeq 
\bfE f_{!!}I_{N_\infty}^\rmE K.
\end{align*}

\noindent
(iii)
By Proposition \ref{prop3.7} (2)(v), we have
\begin{align*}
I_{M_\infty}^\rmE\bfE f^{!}L
\simeq 
\Q_{M_\infty}I_{M_\infty\times\RR_\infty}f_{\RR_\infty}^{!}\SG
\simeq 
\Q_{M_\infty}f_{\RR_\infty}^{!}I_{N_\infty\times\RR_\infty}\SG
\simeq 
\bfE f^{!}I_{N_\infty}^\rmE L.
\end{align*}
\medskip

\noindent
(3)
Let $K\in\BEC(\I\Bbbk_{M_\infty})$, $L\in\BEC(\I\Bbbk_{M_\infty})$.
Then there exist $F\in\BDC(\I\Bbbk_{M_\infty\times\RR_\infty})$,
$G\in\BDC(\I\Bbbk_{M_\infty\times\RR_\infty})$
such that $K\simeq\Q_{M_\infty} F,\ L\simeq\Q_{M_\infty} G$.
\smallskip

\noindent
(i)
By Proposition \ref{prop3.7} (3)(iv), we have
\begin{align*}
\bfE f_{\ast}J_{M_\infty}^\rmE K
\simeq 
\Q_{N_\infty}^\sub \bfR f_{\RR_\infty\ast}\bfR J_{M_\infty\times\RR_\infty}F
\simeq 
\Q_{N_\infty}^\sub \bfR J_{N_\infty\times\RR_\infty}\bfR f_{\RR_\infty\ast}F
\simeq 
J_{N_\infty}^\rmE\bfE f_{\ast} K.
\end{align*}

\noindent
(ii)
By Proposition \ref{prop3.7} (3)(v), we have
\begin{align*}
J_{M_\infty}^\rmE\bfE f^{!}L
\simeq 
\Q_{M_\infty}^\sub \bfR J_{M_\infty\times\RR_\infty}f_{\RR_\infty}^{!}G
\simeq 
\Q_{M_\infty}^\sub f_{\RR_\infty}^{!}\bfR J_{N_\infty\times\RR_\infty}G
\simeq 
\bfE f^{!}J_{N_\infty}^\rmE L.
\end{align*}
\medskip

\noindent
(4)
Let $K\in\BEC_{\I\RR-c}(\I\Bbbk_{M_\infty})$, $L\in\BEC_{\I\RR-c}(\I\Bbbk_{N_\infty})$.
\smallskip

\noindent
(i)
By using (2)(i) and Theorem \ref{main1} (2), we have
\begin{align*}
\lambda_{M_\infty}^\rmE\bfE f^{-1}L
\simeq
\lambda_{M_\infty}^\rmE\bfE f^{-1}I_{N_\infty}^\rmE \lambda_{N_\infty}^\rmE L
\simeq
\lambda_{M_\infty}^\rmE I_{M_\infty}^\rmE \bfE f^{-1}\lambda_{N_\infty}^\rmE L
\simeq
 \bfE f^{-1}\lambda_{N_\infty}^\rmE L.
\end{align*}

\noindent
(ii)
By using (2)(ii) and Theorem \ref{main1} (2), we have
\begin{align*}
\bfE f_{!!}\lambda_{M_\infty}^\rmE K
\simeq
\lambda_{N_\infty}^\rmE I_{N_\infty}^\rmE \bfE f_{!!}\lambda_{M_\infty}^\rmE K
\simeq
\lambda_{N_\infty}^\rmE\bfE f_{!!}I_{M_\infty}^\rmE\lambda_{M_\infty}^\rmE K
\simeq 
\lambda_{N_\infty}^\rmE\bfE f_{!!}K.
\end{align*}

\noindent
(iii)
By using (1)(i) and Theorem \ref{main1} (2), we have
\begin{align*}
\lambda_{M_\infty}^\rmE\(K_1\Potimes K_2\)
&\simeq
\lambda_{M_\infty}^\rmE\(
I_{M_\infty}^\rmE\lambda_{M_\infty}^\rmE K_1
\Potimes I_{M_\infty}^\rmE\lambda_{M_\infty}^\rmE K_2\)\\
&\simeq
\lambda_{M_\infty}^\rmE I_{M_\infty}^\rmE\(
\lambda_{M_\infty}^\rmE K_1\Potimes \lambda_{M_\infty}^\rmE K_2\)\\
&\simeq
\lambda_{M_\infty}^\rmE K_1\Potimes \lambda_{M_\infty}^\rmE K_2.
\end{align*}
\end{proof}

Let us prove that 
the functors $I_{M_\infty}^\rmE, J_{M_\infty}^\rmE$ are preserve the $\RR$-constructability.
Let us recall that an enhanced ind-sheaf $K$ on $M_\infty$ is $\RR$-constructible
if for any open subset $U$ of $M$ which is subanalytic and relatively compact in $\che{M}$
there exists $\SF\in\BDC_{\RR-c}(\Bbbk_{U_\infty\times\RR_\infty})$
such that $\bfE i_{U_\infty}^{-1}K\simeq
\Bbbk_{U_\infty}^{\rmE}\Potimes \Q_{U_\infty}\iota_{U_\infty\times\RR_\infty}\SF$.
We denote by $\BEC_{{\RR-c}}(\I\Bbbk_{M_\infty})$
the category of $\RR$-constructible enhanced ind-sheaves.
See \cite[\S3.3]{DK16-2} for the details.

We shall set $$\Bbbk_{M_\infty}^{\rmE, \sub} := \Q_{M_\infty}^\sub 
\(\underset{a\to +\infty}{\varinjlim}\ \rho_{M_\infty\times\RR_\infty\ast}\Bbbk_{\{t\geq a\}}
\)\in\BEC(\Bbbk_{M_\infty}^\sub).$$
\begin{lemma}\label{lem3.19}
There exist an isomorphism
$I_{M_\infty}^\rmE\Bbbk_{M_\infty}^{\rmE, \sub}\simeq \Bbbk_{M_\infty}^\rmE$
in $\BEC(\I\Bbbk_{M_\infty})$
and 
an isomorphism
$J_{M_\infty}^\rmE\Bbbk_{M_\infty}^{\rmE}\simeq \Bbbk_{M_\infty}^{\rmE, \sub}$
in $\BEC(\Bbbk_{M_\infty}^\sub)$.
\end{lemma}

\begin{proof}
By the definition of $\Bbbk_{M_\infty}^\rmE$,
we have $\Bbbk_{M_\infty}^\rmE
\simeq
\Q_{M_\infty}j_{M_\infty\times\RR_\infty}^{-1}\(``\underset{a\to +\infty}{\varinjlim}"\ \iota_{\che{M}\times\var{\RR}}\Bbbk_{\{t\geq a\}}\)$. 
Since $I_{\che{M}\times\var{\RR}}\circ\rho_{\che{M}\times\var{\RR}\ast}\simeq \iota_{\che{M}\times\var{\RR}}$
and the fact that a functor $I$ commutes with the filtrant inductive limits,
there exist isomorphisms in $\BDC(\I\Bbbk_{\che{M}\times\var{\RR}})$:
$$``\underset{a\to +\infty}{\varinjlim}"\ \iota_{\che{M}\times\var{\RR}}\Bbbk_{\{t\geq a\}}
\simeq
I_{\che{M}\times\var{\RR}}\(
\underset{a\to +\infty}{\varinjlim}\ \rho_{\che{M}\times\var{\RR}\ast}\Bbbk_{\{t\geq a\}}\).$$
Hence, we have
\begin{align*}
\Bbbk_{M_\infty}^\rmE
&\simeq
\Q_{M_\infty}j_{M_\infty\times\RR_\infty}^{-1}I_{\che{M}\times\var{\RR}}\(
\underset{a\to +\infty}{\varinjlim}\ \rho_{\che{M}\times\var{\RR}\ast}\Bbbk_{\{t\geq a\}}\)\\
&\simeq
I_{M_\infty}^\rmE\Q_{M_\infty}^\sub \(\underset{a\to +\infty}{\varinjlim}\ 
j_{M_\infty\times\RR_\infty}^{-1}\rho_{\che{M}\times\var{\RR}\ast}\Bbbk_{\{t\geq a\}}\)\\
&\simeq
I_{M_\infty}^\rmE\Q_{M_\infty}^\sub
\(\underset{a\to +\infty}{\varinjlim}\ \rho_{M_\infty\times\RR_\infty\ast}\Bbbk_{\{t\geq a\}}\)\\
&\simeq
I_{M_\infty}^\rmE\Bbbk_{M_\infty}^{\rmE, \sub}.
\end{align*}
The second assertion follows from the first assertion and Theorem \ref{main1} (1).
\end{proof}

\begin{proposition}\label{prop3.20}
The triangulated category $\BEC_{{\RR-c}}(\I\Bbbk_{M_\infty})$
is a full triangulated subcategory of $\BEC_{{\I\RR-c}}(\I\Bbbk_{M_\infty})$.
\end{proposition}

\begin{proof}
Let $K\in\BEC_{{\RR-c}}(\I\Bbbk_{M_\infty})$.
Since a pair $(I_{M_\infty}^\rmE, J_{M_\infty}^\rmE)$ is an adjoint pair,
we have a morphism $I_{M_\infty}^\rmE J_{M_\infty}^\rmE K\to K$ in $\BEC(\I\Bbbk_{M_\infty})$.
Since $K$ is $\RR$-constructible,
for any open subset $U$ of $M$ which is subanalytic and relatively compact in $\che{M}$
there exists $\SF^U\in\BDC_{\RR-c}(\Bbbk_{U_\infty\times\RR_\infty})$
such that $\bfE i_{U_\infty}^{-1}K\simeq
\Bbbk_{U_\infty}^{\rmE}\Potimes \Q_{U_\infty}\iota_{U_\infty\times\RR_\infty}\SF^U$.
By Proposition \ref{prop3.7} (4)(i), Theorem \ref{main1} (2), Proposition \ref{prop3.18} (1)(i) and Lemma \ref{lem3.19},
we have $$\Bbbk_{U_\infty}^{\rmE}\Potimes \Q_{U_\infty}\iota_{U_\infty\times\RR_\infty}\SF^U
\simeq I_{U_\infty}^\rmE\(\Bbbk_{U_\infty}^{\rmE, \sub}\Potimes
\Q_{U_\infty}^\sub\rho_{U_\infty\times\RR_\infty\ast}\SF^U\)\in\BEC_{\I\RR-c}(\I\Bbbk_{U_\infty}).$$
This implies that for any open subset $U$ of $M$ which is subanalytic and relatively compact in $\che{M}$,
$\bfE i_{U_\infty}^{-1}K\in\BEC_{\I\RR-c}(\I\Bbbk_{U_\infty})$.
Hence,  by Proposition \ref{prop3.18} (2)(i), (3)(ii)
there exist isomorphisms
\begin{align*}
\(I_{M_\infty}^\rmE J_{M_\infty}^\rmE K\)|_{U_\infty}
\simeq
I_{U_\infty}^\rmE J_{U_\infty}^\rmE (K|_{U_\infty})
\simeq
K|_{U_\infty}
\end{align*}
for any open subset $U$ of $M$ which is subanalytic and relatively compact in $\che{M}$.
This implies that $I_{M_\infty}^\rmE J_{M_\infty}^\rmE K\simto K$
and hence $K\in\BEC_{\I\RR-c}(\I\Bbbk_{M_\infty})$.
\end{proof}

\begin{definition}\label{def3.21}
Let $M_\infty = (M, \che{M})$ be a real analytic bordered space.
We say that an enhanced subanalytic sheaf $K$ is $\RR$-constructible 
if for any open subset $U$ of $M$ which is subanalytic and relatively compact in $\che{M}$
there exists $\SF\in\BDC_{\RR-c}(\Bbbk_{U_\infty\times\RR_\infty})$
such that $$\bfE i_{U_\infty}^{-1}K\simeq
\Bbbk_{U_\infty}^{\rmE, \sub}\Potimes \Q_{U_\infty}^\sub\rho_{U_\infty\times\RR_\infty\ast}\SF.$$

Let us denote by $\BEC_{{\RR-c}}(\Bbbk_{M_\infty}^\sub)$
the category of $\RR$-constructible enhanced subanalytic sheaves.
\end{definition}

\begin{theorem}\label{main2}
Let $M_\infty$ be a real analytic bordered space.
Then the functors $I_{M_\infty}^\rmE, J_{M_\infty}^\rmE$ induce an equivalence of categories
\[\xymatrix@M=7pt@C=45pt{
\BEC_{\RR-c}(\Bbbk_{M_\infty}^\sub)\ar@<0.8ex>@{->}[r]^-{I_{M_\infty}^\rmE}_-\sim
&
\BEC_{\RR-c}(\I\Bbbk_{M_\infty})
\ar@<0.8ex>@{->}[l]^-{J_{M_\infty}^\rmE}.
}\]
\end{theorem}

\begin{proof}
It is enough to show that the following functors are well defined:
\begin{align*}
I_{M_\infty}^\rmE&\colon\BEC_{\RR-c}(\Bbbk_{M_\infty}^\sub)\to\BEC_{\RR-c}(\I\Bbbk_{M_\infty}),\\
J_{M_\infty}^\rmE&\colon\BEC_{\RR-c}(\I\Bbbk_{M_\infty})\to\BEC_{\RR-c}(\Bbbk_{M_\infty}^\sub).
\end{align*}  

Let $K\in \BEC_{\RR-c}(\Bbbk_{M_\infty}^\sub)$
and 
 $U$ an open subset of $M$ which is subanalytic and relatively compact in $\che{M}$.
Then there exists $\SF^U\in\BDC_{\RR-c}(\Bbbk_{U_\infty\times\RR_\infty})$
such that $\bfE i_{U_\infty}^{-1}K\simeq
\Bbbk_{U_\infty}^{\rmE, \sub}\Potimes \Q_{M_\infty}^\sub\rho_{M_\infty\times\RR_\infty\ast}\SF^U.$
By Propositions \ref{prop3.7} (4)(i), \ref{prop3.18} (1)(i), (2)(i) and Lemma \ref{lem3.19},
there exist isomorphisms 
\begin{align*}
\bfE i_{U_\infty}^{-1}I_{M_\infty}^\rmE K
&\simeq
I_{U_\infty}^\rmE \bfE i_{U_\infty}^{-1}K\\
&\simeq
I_{U_\infty}^\rmE\(\Bbbk_{U_\infty}^{\rmE, \sub}\Potimes \Q_{M_\infty}^\sub\rho_{M_\infty\times\RR_\infty\ast}\SF^U\)\\
&\simeq
I_{U_\infty}^\rmE\Bbbk_{U_\infty}^{\rmE, \sub}\Potimes I_{U_\infty}^\rmE\Q_{M_\infty}^\sub\rho_{M_\infty\times\RR_\infty\ast}\SF^U
\\
&\simeq
\Bbbk_{U_\infty}^{\rmE}\Potimes \Q_{M_\infty}\iota_{M_\infty\times\RR_\infty}\SF^U.
\end{align*}
Therefore, we have $I_{M_\infty}^\rmE K\in \BEC_{\RR-c}(\I\Bbbk_{M_\infty})$.

 Let $K\in \BEC_{\RR-c}(\I\Bbbk_{M_\infty})$ and 
 $U$ an open subset of $M$ which is subanalytic and relatively compact in $\che{M}$.
Then there exists $\SF^U\in\BDC_{\RR-c}(\Bbbk_{U_\infty\times\RR_\infty})$
such that $\bfE i_{U_\infty}^{-1}K\simeq
\Bbbk_{U_\infty}^{\rmE}\Potimes \Q_{M_\infty}\iota_{M_\infty\times\RR_\infty}\SF^U.$
By Propositions \ref{prop3.7} (3)(i), \ref{prop3.18} (4)(i), (iii), Lemma \ref{lem3.19} and Proposition \ref{prop3.20},
there exist isomorphisms
\begin{align*}
\bfE i_{U_\infty}^{-1}J_{M_\infty}^\rmE K
&\simeq
J_{U_\infty}^\rmE\bfE i_{U_\infty}^{-1} K\\
&\simeq
J_{U_\infty}^\rmE \(
\Bbbk_{U_\infty}^{\rmE}\Potimes \Q_{U_\infty}\iota_{U_\infty\times\RR_\infty}\SF^U\)\\
&\simeq
J_{U_\infty}^\rmE\Bbbk_{U_\infty}^{\rmE}\Potimes 
J_{U_\infty}^\rmE\Q_{U_\infty}\iota_{U_\infty\times\RR_\infty}\SF^U\\
&\simeq
\Bbbk_{U_\infty}^{\rmE, \sub}\Potimes 
\Q_{U_\infty}^\sub\rho_{U_\infty\times\RR_\infty\ast}\SF^U.
\end{align*}
Therefore, we have $\lambda_{M_\infty}^\rmE K\in\BEC_{\RR-c}(\Bbbk_{M_\infty}^\sub)$.

The proof is completed.
\end{proof}

Let us summarize results of Proposition \ref{prop3.20} and Theorems \ref{main1}, \ref{main2} in the following commutative diagram:
\[\xymatrix@M=10pt@R=35pt@C=85pt{
\BEC(\Bbbk_{M_\infty}^\sub)\ar@<0.7ex>@{^{(}->}[r]^-{I_{M_\infty}^\rmE}
\ar@<-0.7ex>@{->}[rd]^-{I_{M_\infty}^\rmE}_-\sim
 & \BEC(\I\Bbbk_{M_\infty})\ar@<0.7ex>@{->}[l]^-{J_{M_\infty}^\rmE}\\
\BEC_{\RR-c}(\Bbbk_{M_\infty}^\sub)\ar@{}[u]|-{\bigcup}
\ar@<0.7ex>@{->}[rd]^-{I_{M_\infty}^\rmE}_-\sim
& \BEC_{\I\RR-c}(\I\Bbbk_{M_\infty}).\ar@{}[u]|-{\bigcup}
\ar@<2.5ex>@{->}[lu]^-{\lambda_{M_\infty}^\rmE}\\
{} & \BEC_{\RR-c}(\I\Bbbk_{M_\infty})\ar@{}[u]|-{\bigcup}
\ar@<1.0ex>@{->}[lu]^-{\lambda_{M_\infty}^\rmE}.
}\]

At the end of this subsection,
let us consider an embedding functor from
$\BDC(\Bbbk_{M_\infty}^\sub)$ to $\BEC(\Bbbk_{M_\infty}^\sub)$
 and a Verdier duality functor for enhanced subanalytic sheaves.
 
Let us define a functor
$$e_{M_\infty}^\sub \colon \BDC(\Bbbk_{M_\infty}^\sub) \to \BEC(\Bbbk_{M_\infty}^\sub),
\hspace{7pt}
\SF\mapsto \pi^{-1}\SF\otimes\Bbbk_{M_\infty}^{\rmE, \sub}.$$
By Proposition \ref{prop3.23} below, we have commutative diagrams:
\[\xymatrix@M=7pt@C=45pt{
\BDC(\Bbbk_{M_\infty}^\sub)\ar@{->}[r]^-{e_{M_\infty}^\sub}\ar@{->}[d]_-{I_{M_\infty}}
&
\BEC(\Bbbk_{M_\infty}^\sub)\ar@{->}[d]^-{I_{M_\infty}^\rmE}\\
\BDC(\I\Bbbk_{M_\infty})\ar@{->}[r]_-{e_{M_\infty}}
&
\BEC(\I\Bbbk_{M_\infty}),}
\hspace{17pt}
\xymatrix@M=7pt@C=45pt{
\BDC_{\I\RR-c}(\I\Bbbk_{M_\infty})\ar@{->}[r]^-{e_{M_\infty}}\ar@{->}[d]_-{\lambda_{M_\infty}}
&
\BEC_{\I\RR-c}(\I\Bbbk_{M_\infty})\ar@{->}[d]^-{\lambda_{M_\infty}^\rmE}\\
\BDC(\Bbbk_{M_\infty}^\sub)\ar@{->}[r]_-{e_{M_\infty}^\sub}
&
\BEC(\Bbbk_{M_\infty}^\sub).}
\]

\begin{proposition}\label{prop3.23}
For any $\SF\in\BDC(\Bbbk_{M_\infty}^\sub)$ and any $F\in\BDC_{\I\RR-c}(\I\Bbbk_{M_\infty})$,
we have
\begin{align*}
I_{M_\infty}^\rmE e_{M_\infty}^\sub\SF
&\simeq
e_{M_\infty}I_{M_\infty}\SF,\\
e_{M_\infty}^\sub\lambda_{M_\infty}F
&\simeq
\lambda_{M_\infty}^\rmE e_{M_\infty}F.
\end{align*}
Moreover, the functor
$e_{M_\infty}^\sub\colon \BDC(\Bbbk_{M_\infty}^\sub) \to \BEC(\Bbbk_{M_\infty}^\sub)$
is fully faithful.
\end{proposition}

\begin{proof}
Let $\SF\in\BDC(\Bbbk_{M_\infty}^\sub)$.
By Proposition \ref{prop3.7} (2)(iii), (vi) and Lemma \ref{lem3.19},
there exist isomorphisms in $\BEC(\I\Bbbk_{M_\infty})$:
$$I_{M_\infty}^\rmE e_{M_\infty}^\sub\SF
\simeq
I_{M_\infty}^\rmE\Bbbk_{M_\infty}^{\rmE, \sub}\otimes \pi^{-1}I_{M_\infty}\SF
\simeq
\Bbbk_{M_\infty}^{\rmE}\otimes \pi^{-1}I_{M_\infty}\SF
\simeq
e_{M_\infty}I_{M_\infty}\SF.$$
\smallskip

Let $F\in\BDC_{\I\RR-c}(\I\Bbbk_{M_\infty})$.
By Proposition \ref{prop3.7} (4)(ii), (iv) and Lemma \ref{lem3.19},
there exist isomorphisms in $\BEC(\Bbbk_{M_\infty}^\sub)$:
$$e_{M_\infty}^\sub \lambda_{M_\infty}F
\simeq
\Bbbk_{M_\infty}^{\rmE, \sub}\otimes \pi^{-1}\lambda_{M_\infty}F
\simeq
\lambda_{M_\infty}^\rmE\Bbbk_{M_\infty}^{\rmE}\otimes \pi^{-1}\lambda_{M_\infty}F
\simeq
\lambda_{M_\infty}^\rmE e_{M_\infty}F.$$

\medskip

\noindent
Let $\SF_1, \SF_2\in\BDC(\Bbbk_{M_\infty}^\sub)$.
By Proposition \ref{prop3.6}, the functor
$I_{M_\infty}\colon\BDC(\Bbbk_{M_\infty}^\sub)\to\BDC(\I\Bbbk_{M_\infty})$ is fully faithful
and hence there exists an isomorphism 
$$\Hom_{\BDC(\Bbbk_{M_\infty}^\sub)}(\SF_1, \SF_2)
\simeq
\Hom_{\BDC(\I\Bbbk_{M_\infty})}(I_{M_\infty}\SF_1, I_{M_\infty}\SF_2).$$
Since the functor
$e_{M_\infty}\colon \BDC(\I\Bbbk_{M_\infty}) \to \BEC(\I\Bbbk_{M_\infty})$ is fully faithful,
we have an isomorphism
$$\Hom_{\BDC(\I\Bbbk_{M_\infty})}(I_{M_\infty}\SF_1, I_{M_\infty}\SF_2)
\simeq
\Hom_{\BEC(\I\Bbbk_{M_\infty})}(e_{M_\infty}I_{M_\infty}\SF_1, e_{M_\infty}I_{M_\infty}\SF_2).$$
Moreover, by the first assertion, we have
$$\Hom_{\BEC(\I\Bbbk_{M_\infty})}(e_{M_\infty}I_{M_\infty}\SF_1, e_{M_\infty}I_{M_\infty}\SF_2)
\simeq
\Hom_{\BEC(\I\Bbbk_{M_\infty})}(
I_{M_\infty}^\rmE e_{M_\infty}^\sub\SF_1, I_{M_\infty}^\rmE e_{M_\infty}^\sub\SF_2).$$
By Theorem \ref{main1}, the functor
$I_{M_\infty}^\rmE\colon\BEC(\Bbbk_{M_\infty}^\sub)\to\BEC(\I\Bbbk_{M_\infty})$ is fully faithful,
and hence
$$\Hom_{\BEC(\I\Bbbk_{M_\infty})}(
I_{M_\infty}^\rmE e_{M_\infty}^\sub\SF_1, I_{M_\infty}^\rmE e_{M_\infty}^\sub\SF_2)
\simeq
\Hom_{\BEC(\Bbbk_{M_\infty}^\sub)}(
e_{M_\infty}^\sub\SF_1, e_{M_\infty}^\sub\SF_2).$$
Therefore, we have 
$$\Hom_{\BDC(\Bbbk_{M_\infty}^\sub)}(\SF_1, \SF_2)\simeq
\Hom_{\BEC(\Bbbk_{M_\infty}^\sub)}(
e_{M_\infty}^\sub\SF_1, e_{M_\infty}^\sub\SF_2).$$
This implies that the functor
$e_{M_\infty}^\sub\colon \BDC(\Bbbk_{M_\infty}^\sub) \to \BEC(\Bbbk_{M_\infty}^\sub)$
is fully faithful.
\end{proof}

The functor $e^\sub$ commutes with several functors as below.
\begin{proposition}\label{prop3.24}
Let $f\colon M_\infty\to N_\infty$ be a morphism of real analytic bordered spaces.
For any $\SF, \SF_1, \SF_2\in\BDC(\Bbbk_{M_\infty}^\sub)$ and any $\SG\in\BDC(\Bbbk_{N_\infty}^\sub)$,
we have
\begin{align*}
e_{M_\infty}^\sub(\SF_1\otimes\SF_2)
&\simeq
e_{M_\infty}^\sub\SF_1\Potimes e_{M_\infty}^\sub\SF_2,\\
\bfE f_{!!}e_{M_\infty}^\sub\SF
&\simeq
e_{N_\infty}^\sub\bfR f_{!!}\SF,\\
\bfE f^{-1}e_{N_\infty}^\sub\SG
&\simeq
e_{M_\infty}^\sub f^{-1}\SG,\\
\bfE f^{!}e_{N_\infty}^\sub\SG
&\simeq
e_{M_\infty}^\sub f^{!}\SG.
\end{align*}
\end{proposition}

\begin{proof}
By Proposition \ref{prop3.14} (4)(iii) and the fact that
$\Bbbk_{M_\infty}^{\rmE, \sub}\Potimes \Bbbk_{M_\infty}^{\rmE, \sub}
\simeq\Bbbk_{M_\infty}^{\rmE, \sub}$,
we have
\begin{align*}
e_{M_\infty}^\sub(\SF_1\otimes\SF_2)
&\simeq
\Bbbk_{M_\infty}^{\rmE, \sub}\otimes\pi^{-1}(\SF_1\otimes\SF_2)\\
&\simeq
\(\Bbbk_{M_\infty}^{\rmE, \sub}\Potimes\Bbbk_{M_\infty}^{\rmE, \sub}\)
\otimes(\pi^{-1}\SF_1\otimes\pi^{-1}\SF_2)\\
&\simeq
\(\Bbbk_{M_\infty}^{\rmE, \sub}\otimes\pi^{-1}\SF_1\)
\Potimes\(\Bbbk_{M_\infty}^{\rmE, \sub}\otimes\pi^{-1}\SF_2\)\\
&\simeq
e_{M_\infty}^\sub\SF_1\Potimes e_{M_\infty}^\sub\SF_2.
\end{align*}

\noindent
The second and third assertions follow from Proposition \ref{prop3.14} (4)(ii)
and the fact that $\bfE f^{-1}\Bbbk_{N_\infty}^{\rmE, \sub}\simeq \Bbbk_{M_\infty}^{\rmE, \sub}$.

\noindent
Let  us prove the last assertion.
By Propositions \ref{prop3.6}, \ref{prop3.18} (3)(ii) and  \ref{prop3.23},
we have isomorphisms in $\BEC(\Bbbk_{M_\infty}^\sub)$:
$$\bfE f^!e_{N_\infty}^\sub\SG
\simeq
\bfE f^!e_{N_\infty}^\sub\lambda_{N_\infty}I_{N_\infty}\SG
\simeq
\lambda_{M_\infty}^\rmE \bfE f^!e_{N_\infty}I_{N_\infty}\SG.$$
By \cite[Prop.\:2.18]{KS16-2}, we have $\bfE f^!\circ e_{N_\infty} \simeq e_{M_\infty}\circ f^!$
and hence there exist isomorphisms
\begin{align*}
\bfE f^!e_{N_\infty}^\sub\SG
\simeq
\lambda_{M_\infty}^\rmE \bfE f^!e_{N_\infty}I_{N_\infty}\SG
\simeq
\lambda_{M_\infty}^\rmE e_{M_\infty} f^!I_{N_\infty}\SG
\simeq
e_{M_\infty}^\sub\lambda_{M_\infty}^\rmE I_{M_\infty}f^!\SG
\simeq
e_{M_\infty}^\sub f^!\SG
\end{align*}
where in the third isomorphism we used Propositions \ref{prop3.7} (2)(v), \ref{prop3.23}
and in the last isomorphism we used Proposition \ref{prop3.6}.
\end{proof}

Let $i_0 \colon M_\infty\to M_\infty\times\RR_\infty$ be the morphism of real analytic bordered spaces
induced by the map $M\to M\times\RR, x\mapsto (x, 0)$.
We set
\[\sh_{M_\infty}^\sub := i_0^!\circ \bfR_{M_\infty}^{\rmE, \sub}
\colon\BEC(\Bbbk_{M_\infty}^\sub) \to \BDC(\Bbbk_{M_\infty}^\sub) \]
and call it the subanalytic sheafification functor for enhanced subanalytic sheaves on real analytic bordered spaces.
It has the following properties.
\begin{proposition}
\end{proposition}
\begin{itemize}
\item[\rm (1)]
A pair $(e_{M_\infty}^\sub, \sh_{M_\infty}^\sub)$ is an adjoint pair.

\item[\rm (2)]
For any $\SF\in\BDC(\Bbbk_{M_\infty}^\sub)$,
one has $\SF\simto \sh_{M_\infty}^\sub e_{M_\infty}^\sub \SF$.

\item[\rm (3)]
For any $K\in\BEC(\I\Bbbk_{M_\infty})$,
one has $\bfR J_{M_\infty}\I\sh_{M_\infty} K\simeq \sh_{M_\infty}^\sub J_{M_\infty}^\rmE K$.

\item[\rm (4)]
Let $f\colon M_\infty\to N_\infty$ be a morphism of real analytic bordered spaces.
For any $K\in\BEC(\Bbbk_{M_\infty}^\sub)$ and any $L\in\BEC(\Bbbk_{N_\infty}^\sub)$,
one has
\begin{align*}
\bfR f_\ast\sh_{M_\infty}^\sub K
&\simeq
\sh_{N_\infty}^\sub \bfE f_\ast K,\\
f^!\sh_{M_\infty}^\sub L
&\simeq
\sh_{M_\infty}^\sub\bfE f^!L.
\end{align*}
\end{itemize}

\begin{proof}
First, let us prove the assertion (3).
Let $K\in\BEC(\I\Bbbk_{M_\infty})$.
Then there exists $F\in\BDC(\I\Bbbk_{M_\infty\times\RR_\infty})$
such that $K\simeq\Q_{M_\infty}(F)$.
Then we have
\begin{align*}
\bfR J_{M_\infty}\I\sh_{M_\infty} K
&\simeq
\bfR J_{M_\infty}i_0^!\bfR_{M_\infty}^{\rmE}K\\
&\simeq
i_0^!\bfR J_{M_\infty\times\RR_\infty}\bfR_{M_\infty}^{\rmE}K\\
&\simeq
i_0^!\bfR J_{M_\infty\times\RR_\infty}
\Prihom(\iota_{M_\infty\times\RR_\infty}(\Bbbk_{\{t\geq0\}}\oplus\Bbbk_{\{t\leq 0\}}), F)\\
&\simeq
i_0^!\bfR J_{M_\infty\times\RR_\infty}\Prihom(I_{M_\infty\times\RR_\infty}\rho_{M_\infty\times\RR_\infty\ast}
(\Bbbk_{\{t\geq0\}}\oplus\Bbbk_{\{t\leq 0\}}), F)\\
&\simeq
i_0^!\Prihomsub(\rho_{M_\infty\times\RR_\infty\ast}
(\Bbbk_{\{t\geq0\}}\oplus\Bbbk_{\{t\leq 0\}}), \bfR J_{M_\infty\times\RR_\infty}F)\\
&\simeq
i_0^! \bfR_{M_\infty}^{\rmE, \sub} J_{M_\infty}^\rmE K\\
&\simeq \sh_{M_\infty}^\sub J_{M_\infty}^\rmE K,
\end{align*}
where in the second (resp.\,fourth, fifth) isomorphism
we used Proposition \ref{prop3.7} (3)(v) (resp.\,(4)(i), Proposition \ref{prop3.12} (2)).
\medskip

\noindent
(1)
Let $\SF\in\BDC(\Bbbk_{M_\infty}^\sub)$ and $K\in\BEC(\Bbbk_{M_\infty}^\sub)$.
Then we have 
\begin{align*}
\Hom_{\BEC(\Bbbk_{M_\infty}^\sub)}\(e_{M_\infty}^\sub\SF, K\)
&\simeq
\Hom_{\BEC(\I\Bbbk_{M_\infty})}\(I_{M_\infty}^\rmE e_{M_\infty}^\sub\SF, I_{M_\infty}^\rmE K\)\\
&\simeq
\Hom_{\BEC(\I\Bbbk_{M_\infty})}\(e_{M_\infty}I_{M_\infty}\SF, I_{M_\infty}^\rmE K\)\\
&\simeq
\Hom_{\BDC(\Bbbk_{M_\infty}^\sub)}\(\SF,  \bfR J_{M_\infty}\I\sh_{M_\infty}I_{M_\infty}^\rmE K\)\\
&\simeq
\Hom_{\BDC(\Bbbk_{M_\infty}^\sub)}\(\SF,  \sh_{M_\infty}^\sub J_{M_\infty}^\rmE I_{M_\infty}^\rmE K\)\\
&\simeq
\Hom_{\BDC(\Bbbk_{M_\infty}^\sub)}\(\SF,  \sh_{M_\infty}^\sub K\)
\end{align*}
where in the first and last isomorphisms (resp.\:second isomorphism) we used Theorem \ref{main1} (resp.\:Proposition \ref{prop3.23}),
in the third isomorphism we used the fact that $(e_{M_\infty}, \I\sh_{M_\infty})$ is an adjoint pair
and Proposition \ref{prop3.6} (1),
and the fourth isomorphism follows from the assertion (3). 
This implies that a pair $(e_{M_\infty}^\sub, \sh_{M_\infty}^\sub)$ is an adjoint pair.
\medskip

\noindent
(2)
Let  $\SF\in\BDC(\Bbbk_{M_\infty}^\sub)$.
By the assertion (1), there exists a canonical morphism $\SF\to \sh_{M_\infty}^\sub e_{M_\infty}^\sub \SF$.
Moreover we have 
\begin{align*}
\sh_{M_\infty}^\sub e_{M_\infty}^\sub \SF
\simeq
\sh_{M_\infty}^\sub J_{M_\infty}^\rmE I_{M_\infty}^\rmE e_{M_\infty}^\sub \SF
\simeq
\bfR J_{M_\infty}\I\sh_{M_\infty}e_{M_\infty}I_{M_\infty}\SF
\simeq\SF
\end{align*}
where in the first isomorphism we used Theorem \ref{main1} (1),
in the second isomorphism we used the assertion (3) and Proposition \ref{prop3.23},
and in the last isomorphism we used the fact that $\I\sh_{M_\infty}\circ e_{M_\infty}\simeq\id$
and Proposition \ref{prop3.6} (1). 
\medskip

\noindent
(4)
Let $K\in\BEC(\Bbbk_{M_\infty}^\sub)$.
For any $\SF\in\BDC(\Bbbk_{N_\infty}^\sub)$,
we have
\begin{align*}
\Hom_{\BDC(\Bbbk_{N_\infty}^\sub)}\(\SF,
\bfR f_\ast\sh_{M_\infty}^\sub K\)
&\simeq
\Hom_{\BDC(\Bbbk_{M_\infty}^\sub)}\(
e_{M_\infty}^\sub f^{-1}\SF, K\)\\
&\simeq
\Hom_{\BDC(\Bbbk_{M_\infty}^\sub)}\(
\bfE f^{-1}e_{N_\infty}^\sub\SF, K\)\\
&\simeq
\Hom_{\BDC(\Bbbk_{M_\infty}^\sub)}\(\SF,
\sh_{N_\infty}^\sub \bfE f_\ast K\),
\end{align*}
where in the first and last isomorphisms
we used the fact that a pair $(f^{-1}, \bfR f_\ast)$ is an adjoint pair,
the assertion (1) and Proposition \ref{prop3.14} (1)(ii),
and in the second isomorphism we used Proposition \ref{prop3.24}.
This implies that there exists an isomorphism
$\bfR f_\ast\sh_{M_\infty}^\sub K\simeq\sh_{N_\infty}^\sub \bfE f_\ast K$.

Let $L\in\BEC(\Bbbk_{N_\infty}^\sub)$.
Then there exists $\SG\in\BDC(\Bbbk_{N_\infty\times\RR_\infty}^\sub)$
such that $L\simeq\Q_{N_\infty}\SG$.
We shall denote by ${f}_{\RR_\infty}\colon M_\infty\times\RR_\infty\to N_\infty\times\RR_\infty$
the morphism $f\times\id_{\RR_\infty}$ of real analytic bordered spaces.
Then we have isomorphisms in $\BDC(\Bbbk_{M_\infty}^\sub)$:
$$f^!\sh_{N_\infty}^\sub L
\simeq
f^!i_0^!\bfR_{N_\infty}^{\rmE,\sub}L
\simeq
i_0^!f_{\RR_\infty}^!\bfR_{N_\infty}^{\rmE,\sub}L
\simeq
i_0^!f_{\RR_\infty}^!\Prihomsub(
\rho_{M_\infty\times\RR_\infty\ast}(\Bbbk_{\{t\geq0\}}\oplus\Bbbk_{\{t\leq 0\}}), \SG).$$
Moreover, by Proposition \ref{prop3.8}, we have
\begin{align*}
f^!\sh_{N_\infty}^\sub L
&\simeq
i_0^!f_{\RR_\infty}^!\Prihomsub(
\rho_{M_\infty\times\RR_\infty\ast}(\Bbbk_{\{t\geq0\}}\oplus\Bbbk_{\{t\leq 0\}}), \SG)\\
&\simeq
i_0^!\Prihomsub(
f_{\RR_\infty}^{-1}\rho_{N_\infty\times\RR_\infty\ast}(\Bbbk_{\{t\geq0\}}\oplus\Bbbk_{\{t\leq 0\}}),
f_{\RR_\infty}^!\SG)\\
&\simeq
i_0^!\Prihomsub(
\rho_{M_\infty\times\RR_\infty\ast}(\Bbbk_{\{t\geq0\}}\oplus\Bbbk_{\{t\leq 0\}}), f_{\RR_\infty}^!\SG)\\
&\simeq
i_0^!\bfR_{M_\infty}^{\rmE, \sub}\Q_{M_\infty}^\sub f_{\RR_\infty}^!\SG\\
&\simeq
\sh_{M_\infty}^\sub\bfE f^!L.
\end{align*}
\end{proof}

Let us set $$\omega_{M_\infty}^{\rmE, \sub}
:= e_{M_\infty}^\sub(\rho_{M_\infty\ast}\omega_M)\in\BEC(\Bbbk_{M_\infty}^\sub)$$
where $\omega_M\in\BDC(\Bbbk_{M_\infty})\ (\,\simeq \BDC(\CC_M))$ is the dualizing complex,
see \cite[Def.\:3.1.16 (i)]{KS90} for the details.
Note that since $\omega_M\simeq j_M^{-1}\omega_{\che{M}}$,
we have $\omega_M\in\BDC_{\RR-c}(\Bbbk_{M_\infty})$.
We shall define a functor
$$\bfD_{M_\infty}^{\rmE, \sub} \colon\BEC(\Bbbk_{M_\infty}^\sub)^{\op}\to\BEC(\Bbbk_{M_\infty}^\sub),
\hspace{7pt}
K\mapsto
\Prihomsub(K, \omega_{M_\infty}^{\rmE, \sub}).
$$

\begin{lemma}\label{lem3.26}
There exist an isomorphism 
$I_{M_\infty}^\rmE\omega_{M_\infty}^{\rmE, \sub}\simeq \omega_{M_\infty}^\rmE$
in $\BEC(\I\Bbbk_{M_\infty})$
and 
an isomorphism
$J_{M_\infty}^\rmE\omega_{M_\infty}^{\rmE}\simeq \omega_{M_\infty}^{\rmE, \sub}$
in $\BEC(\Bbbk_{M_\infty}^\sub)$.
\end{lemma}

\begin{proof}
Since $\omega_M$ is $\RR$-constructible,
there exists an isomorphism
$\iota_{M_\infty}\omega_M\simeq I_{M_\infty}\rho_{M_\infty\ast}\omega_M$
in $\BDC(\I\Bbbk_{M_\infty})$.
Hence, we have
\begin{align*}
\omega_{M_\infty}^\rmE
:=
e_{M_\infty}(\iota_{M_\infty}\omega_M)
\simeq
e_{M_\infty}(I_{M_\infty}\rho_{M_\infty\ast}\omega_M)
\simeq
 I_{M_\infty}^\rmE e_{M_\infty}^\sub(\rho_{M_\infty\ast}\omega_M)
\simeq
 I_{M_\infty}^\rmE\omega_{M_\infty}^{\rmE, \sub},
\end{align*}
where in the third isomorphism we used Proposition \ref{prop3.23}.

The second assertion follows from the first assertion and Theorem \ref{main1} (1).
\end{proof}

\begin{proposition}\label{prop3.27}
Let $f\colon M_\infty\to N_\infty$ be a morphism of real analytic bordered spaces.
\begin{itemize}
\item[\rm (1)]
For any $K\in\BEC(\Bbbk_{M_\infty}^\sub)$ and any $L\in\BEC(\Bbbk_{N_\infty}^\sub)$,
one has
\begin{align*}
\bfE f^!\bfD_{N_\infty}^{\rmE, \sub}L
&\simeq
\bfD_{M_\infty}^{\rmE, \sub}\bfE f^{-1}L,\\
\bfE f_{\ast}\bfD_{M_\infty}^{\rmE, \sub}K
&\simeq
\bfD_{N_\infty}^{\rmE, \sub}\bfE f_{!!}K,\\
J_{M_\infty}^\rmE\bfD_{M_\infty}^\rmE I_{M_\infty}^\rmE K
&\simeq
\bfD_{M_\infty}^{\rmE, \sub}K.
\end{align*}

\item[\rm (2)]
For any $K\in\BEC_{\RR-c}(\Bbbk_{M_\infty}^\sub)$,
we have $\bfD_{M_\infty}^{\rmE, \sub}K\in\BEC_{\RR-c}(\Bbbk_{M_\infty}^\sub)$ and 
$\bfD_{M_\infty}^{\rmE, \sub}\bfD_{M_\infty}^{\rmE, \sub}K
\simeq K$.
In particular, there exists an equivalence of categories:
\[\bfD_{M_\infty}^{\rmE, \sub}\colon
\BEC_{\RR-c}(\Bbbk_{M_\infty}^\sub)^\op\simto\BEC_{\RR-c}(\Bbbk_{M_\infty}^\sub).\]
\end{itemize}
\end{proposition}

\begin{proof}
(1)
First, let us prove that $\bfE f^!\omega_{N_\infty}^{\rmE, \sub}\simeq
\omega_{M_\infty}^{\rmE, \sub}$.
By Proposition \ref{prop3.24} and the fact that $f^!\omega_N\simeq\omega_M$,
we have 
$$
\bfE f^!\omega_{N_\infty}^{\rmE, \sub}
\simeq
\bfE f^!e_{N_\infty}^{\rmE, \sub}(\rho_{N_\infty\ast}\omega_N)
\simeq
e_{M_\infty}^{\sub}(\rho_{M_\infty\ast}f^!\omega_N)
\simeq
e_{M_\infty}^{\sub}(\rho_{M_\infty\ast}\omega_M)
\simeq
\omega_{M_\infty}^{\rmE, \sub}.$$

Let $K\in\BEC(\Bbbk_{M_\infty}^\sub)$ and $L\in\BEC(\Bbbk_{N_\infty}^\sub)$.
By Proposition \ref{prop3.14} (2), we have
\begin{align*}
\bfE f^!\bfD_{N_\infty}^{\rmE, \sub}L
&\simeq
\bfE f^!\Prihomsub(L, \omega_{N_\infty}^{\rmE, \sub})\\
&\simeq
\Prihomsub(\bfE f^{-1}L, \bfE f^!\omega_{N_\infty}^{\rmE, \sub})\\
&\simeq
\Prihomsub(\bfE f^{-1}L, \omega_{M_\infty}^{\rmE, \sub})\\
&\simeq
\bfD_{M_\infty}^{\rmE, \sub}\bfE f^{-1}L.
\end{align*}
By Proposition \ref{prop3.14} (1)(iii),
we have
\begin{align*}
\bfE f_\ast\bfD_{M_\infty}^{\rmE, \sub}K
&\simeq
\bfE f_\ast\Prihomsub(K, \omega_{M_\infty}^{\rmE, \sub})\\
&\simeq
\bfE f_\ast\Prihomsub(K, \bfE f^!\omega_{N_\infty}^{\rmE, \sub})\\
&\simeq
\Prihomsub(\bfE f_{!!}K, \omega_{N_\infty}^{\rmE, \sub})\\
&\simeq
\bfD_{N_\infty}^{\rmE, \sub}\bfE f_{!!}K.
\end{align*}
By Proposition \ref{prop3.18} (1)(ii) and Lemma \ref{lem3.26},
we have
\begin{align*}
J_{M_\infty}^\rmE\bfD_{M_\infty}^\rmE I_{M_\infty}^\rmE K
&\simeq
J_{M_\infty}^\rmE\Prihom(I_{M_\infty}^\rmE K, \omega_{M_\infty}^{\rmE})\\
&\simeq
\Prihomsub(K, J_{M_\infty}^\rmE \omega_{M_\infty}^{\rmE})\\
&\simeq
\Prihomsub(K,\omega_{M_\infty}^{\rmE, \sub})\\
&\simeq
\bfD_{M_\infty}^{\rmE, \sub}K.
\end{align*}

\noindent
(2)
Let $K\in\BEC_{\RR-c}(\Bbbk_{M_\infty}^\sub)$.
By Theorem \ref{main2} we have $I_{M_\infty}^\rmE K\in\BEC_{\RR-c}(\I\Bbbk_{M_\infty})$.
Since $\bfD_{M_\infty}^\rmE\colon \BEC_{\RR-c}(\I\Bbbk_{M_\infty})^\op\to\BEC_{\RR-c}(\I\Bbbk_{M_\infty})$
(see, \cite[Prop.\:3.3.3 (ii)]{DK16-2}),
we have $\bfD_{M_\infty}^\rmE I_{M_\infty}^\rmE K\in\BEC_{\RR-c}(\I\Bbbk_{M_\infty})$.
Hence, by Theorem \ref{main2}, we have 
$J_{M_\infty}^\rmE\bfD_{M_\infty}^\rmE I_{M_\infty}^\rmE K\in\BEC_{\RR-c}(\Bbbk_{M_\infty}^\sub)$.
By the assertion (1), we have 
$J_{M_\infty}^\rmE\bfD_{M_\infty}^\rmE I_{M_\infty}^\rmE K
\simeq
\bfD_{M_\infty}^{\rmE, \sub}K$,
and hence $\bfD_{M_\infty}^{\rmE, \sub}K\in\BEC_{\RR-c}(\Bbbk_{M_\infty}^\sub)$.

Moreover, since $I_{M_\infty}^\rmE J_{M_\infty}^\rmE\bfD_{M_\infty}^\rmE I_{M_\infty}^\rmE K
\simeq
\bfD_{M_\infty}^\rmE I_{M_\infty}^\rmE K$
we have
\begin{align*}
\bfD_{M_\infty}^{\rmE, \sub}\bfD_{M_\infty}^{\rmE, \sub}K
&\simeq
J_{M_\infty}^\rmE\bfD_{M_\infty}^\rmE I_{M_\infty}^\rmE
J_{M_\infty}^\rmE\bfD_{M_\infty}^\rmE I_{M_\infty}^\rmE K\\
&\simeq
J_{M_\infty}^\rmE\bfD_{M_\infty}^\rmE \bfD_{M_\infty}^\rmE I_{M_\infty}^\rmE K\\
&\simeq
J_{M_\infty}^\rmE I_{M_\infty}^\rmE K\\
&\simeq
K
\end{align*}
where the third (resp.\,last) isomorphism follows from \cite[Prop.\:3.3.3 (ii)]{DK16-2})
(resp.\:Theorem \ref{main1} (1)).
\end{proof}

Several operations preserve the $\RR$-constructability as below.
Let us recall that 
a morphism $f\colon (M, \che{M})\to (N, \che{N})$ of real analytic bordered spaces
is called semi-proper if the second projection 
$\che{M}\times\che{N}\to\che{N}$ is proper on
the closure $\var{\Gamma}_f$ of the graph $\Gamma_f$ of $f$ 
in $\che{M}\times\che{N}$. 
\begin{proposition}
Let $f\colon M_\infty\to N_\infty$ be a morphism of real analytic bordered spaces
associated with a morphism $\che{f}\colon\che{M}\to\che{N}$ of real analytic manifolds.
The functors below are well defined:
\begin{itemize}
\item[\rm (1)]
$e_{M_\infty}^\sub\rho_{M_\infty\ast}\colon
\BDC_{\RR-c}(\Bbbk_{M_\infty})\to\BEC_{\RR-c}(\I\Bbbk_{M_\infty})$,

\item[\rm (2)]
$\bfE f_{\ast}\colon\BEC_{\RR-c}(\Bbbk_{M_\infty}^\sub)\to\BEC_{\RR-c}(\Bbbk_{N_\infty}^\sub)$,
if $f$ is semi-proper,

\item[\rm (3)]
$\bfE f^{-1}\colon\BEC_{\RR-c}(\Bbbk_{N_\infty}^\sub)\to\BEC_{\RR-c}(\Bbbk_{M_\infty}^\sub)$,

\item[\rm (4)]
$\bfE f_{!!}\colon\BEC_{\RR-c}(\Bbbk_{M_\infty}^\sub)\to\BEC_{\RR-c}(\Bbbk_{N_\infty}^\sub)$,
if $f$ is semi-proper,

\item[\rm (5)]
$\bfE f^{!}\colon\BEC_{\RR-c}(\Bbbk_{N_\infty}^\sub)\to\BEC_{\RR-c}(\Bbbk_{M_\infty}^\sub)$.
\end{itemize}

\end{proposition}

\begin{proof}
(1)
Let $\SF\in\BDC_{\RR-c}(\Bbbk_{M_\infty})$ and 
$U$ be an open subset of $M$ which is subanalytic and relatively compact in $\che{M}$.
We set $\SF^U := \Bbbk_{\{t = 0\}}\otimes \pi^{-1}\SF|_U$.
By Propositions \ref{prop3.4} (6) and \ref{prop3.24},
we have 
$$\bfE i_{U_\infty}^{-1}\(e_{M_\infty}^\sub\rho_{M_\infty\ast}\SF\)
\simeq
e_{U_\infty}^\sub\rho_{U_\infty\ast}(\SF|_U)
\simeq\Bbbk_{U_\infty}^{\rmE, \sub}\otimes
\pi^{-1}\rho_{U_\infty\ast}(\SF|_U).$$
Since 
$\Bbbk_{U_\infty}^{\rmE, \sub}\simeq
\Bbbk_{U_\infty}^{\rmE, \sub}\Potimes\rho_{U_\infty\times\RR_\infty\ast}\Bbbk_{\{t=0\}}$,
there exist isomorphisms
\begin{align*}
\Bbbk_{U_\infty}^{\rmE, \sub}\otimes
\pi^{-1}\rho_{U_\infty\ast}(\SF|_U)
&\simeq
\(\Bbbk_{U_\infty}^{\rmE, \sub}\Potimes\rho_{U_\infty\times\RR_\infty\ast}\Bbbk_{\{t=0\}}\)
\otimes \rho_{U_\infty\times\RR_\infty\ast}\pi^{-1}\SF|_U\\
&\simeq
\Bbbk_{U_\infty}^{\rmE, \sub}\Potimes
\rho_{U_\infty\times\RR_\infty\ast}\(\Bbbk_{\{t = 0\}}\otimes \pi^{-1}\SF|_U\)\\
&\simeq
\Bbbk_{U_\infty}^{\rmE, \sub}\Potimes
\rho_{U_\infty\times\RR_\infty}(\SF^U),
\end{align*}
where in the first isomorphism we used Proposition \ref{prop3.14} (4)(iii).
Since $\SF$ is $\R$-constructible, 
we have $\SF^U\in\BDC_{\RR-c}(\Bbbk_{U_\infty})$.
This implies that $e_{M_\infty}^{\sub}\rho_{M_\infty\ast}\SF$ is $\RR$-constructible.
\medskip

\noindent
(2)
Let $K\in\BEC_{\RR-c}(\Bbbk_{M_\infty}^\sub)$.
By Theorem \ref{main1} and Proposition \ref{prop3.18} (3)(i), we have isomorphisms 
\begin{align*}
\bfE f_{\ast}K
\simeq
\bfE f_{\ast}\lambda_{M_\infty}^\rmE I_{M_\infty}^\rmE K
\simeq
\lambda_{N_\infty}^\rmE\bfE f_{\ast}I_{M_\infty}^\rmE K.
\end{align*}
Since $K$ is $\RR$-constructible,
we have $I_{M_\infty}^\rmE K\in\BEC_{\RR-c}(\I\Bbbk_{M_\infty})$
by Theorem \ref{main2},
and hence, by \cite[Prop.\:3.3.3 (iv)]{DK16-2} we have 
$\bfE f_{\ast}I_{M_\infty}^\rmE K\in\BEC_{\RR-c}(\I\Bbbk_{N_\infty})$.
This implies that $\bfE f_{\ast}K\simeq
\lambda_{N_\infty}^\rmE\bfE f_{\ast}I_{M_\infty}^\rmE K$ is $\RR$-constructible by Theorem \ref{main2}.
\medskip

\noindent
(3)
Let $L\in\BEC_{\RR-c}(\Bbbk_{N_\infty}^\sub)$.
By Theorem \ref{main1} and Proposition \ref{prop3.18} (4)(i), we have isomorphisms 
\begin{align*}
\bfE f^{-1}K
\simeq
\bfE f^{-1}\lambda_{N_\infty}^\rmE I_{N_\infty}^\rmE L
\simeq
\lambda_{M_\infty}^\rmE\bfE f^{-1}I_{N_\infty}^\rmE L.
\end{align*}
Since $L$ is $\RR$-constructible,
we have $I_{N_\infty}^\rmE L\in\BEC_{\RR-c}(\I\Bbbk_{N_\infty})$
by Theorem \ref{main2},
and hence, by \cite[Prop.\:3.3.3 (iii)]{DK16-2} we have 
$\bfE f^{-1}I_{N_\infty}^\rmE L\in\BEC_{\RR-c}(\I\Bbbk_{M_\infty})$.
This implies that $\bfE f^{-1}L\simeq
\lambda_{M_\infty}^\rmE\bfE f^{-1}I_{N_\infty}^\rmE L$ is $\RR$-constructible by Theorem \ref{main2}.
\medskip

\noindent
(4)
Let $K\in\BEC_{\RR-c}(\Bbbk_{M_\infty}^\sub)$.
Then we have $$\bfE f_{!!}K\simeq \bfE f_{!!} \bfD_{M_\infty}^{\rmE, \sub}\bfD_{M_\infty}^{\rmE, \sub}K
\simeq
 \bfD_{N_\infty}^{\rmE, \sub}\bfE f_{\ast}\bfD_{M_\infty}^{\rmE, \sub}K$$
by Proposition \ref{prop3.27}.
This implies that $\bfE f_{!!}K$ is $\RR$-constructible by the assertion (2) and Proposition \ref{prop3.27} (2).
\medskip

\noindent
(5)
Let $L\in\BEC_{\RR-c}(\Bbbk_{N_\infty}^\sub)$.
Then we have $$\bfE f^{!}K\simeq \bfE f^{!} \bfD_{N_\infty}^{\rmE, \sub}\bfD_{N_\infty}^{\rmE, \sub}L
\simeq
 \bfD_{M_\infty}^{\rmE, \sub}\bfE f^{-1}\bfD_{N_\infty}^{\rmE, \sub}L$$
by Proposition \ref{prop3.27}.
This implies that $\bfE f^{!}K$ is $\RR$-constructible by the assertion (3) and Proposition \ref{prop3.27} (2).
\end{proof}

Moreover, convolution functors preserve the $\RR$-constructability as below.
\begin{proposition}
\begin{itemize}
\item[\rm (1)]
The functors 
\begin{align*}
(\cdot)\Potimes(\cdot)&\colon
\BEC_{\RR-c}(\Bbbk_{M_\infty}^\sub)\times \BEC_{\RR-c}(\Bbbk_{M_\infty}^\sub)
\to\BEC_{\RR-c}(\Bbbk_{M_\infty}^\sub),\\
\Prihomsub(\cdot, \cdot)&\colon
\BEC_{\RR-c}(\Bbbk_{M_\infty}^\sub)^\op\times \BEC_{\RR-c}(\Bbbk_{M_\infty}^\sub)
\to\BEC_{\RR-c}(\Bbbk_{M_\infty}^\sub)
\end{align*}
are well defined.

\item[\rm (2)]
For any $K, L\in\BEC_{\RR-c}(\Bbbk_{M_\infty}^\sub)$,
one has
\begin{itemize}
\item[\rm (i)]
$\bfD_{M_\infty}^{\rmE, \sub}\(K\Potimes L\)
\simeq
\Prihomsub(K, \bfD_{M_\infty}^{\rmE,\sub} L),$

\item[\rm (ii)]
$\bfD_{M_\infty}^{\rmE, \sub}\Prihomsub(K, L)
\simeq
K\Potimes \bfD_{M_\infty}^{\rmE, \sub} L,$

\item[\rm (iii)]
$\Prihomsub(K, L)
\simeq
\Prihomsub\(\bfD_{M_\infty}^{\rmE,\sub}L, \bfD_{M_\infty}^{\rmE,\sub} K\),$

\item[\rm (iv)]
$\rihom^{\rmE, \sub}(K, L)
\simeq
\rihom^{\rmE, \sub}\(\bfD_{M_\infty}^{\rmE,\sub}L, \bfD_{M_\infty}^{\rmE,\sub}K\),$

\item[\rm (v)]
$\rhom^{\rmE, \sub}(K, L)
\simeq
\rhom^{\rmE, \sub}\(\bfD_{M_\infty}^{\rmE,\sub}L, \bfD_{M_\infty}^{\rmE,\sub}K\).$
\end{itemize}
\end{itemize}
\end{proposition}

\begin{proof}
(1)
Let $K, L\in\BEC_{\RR-c}(\Bbbk_{M_\infty}^\sub)$ and
$U$ be an open subset of $M$ which is subanalytic and relatively compact in $\che{M}$.
Then there exist $\SF, \SG\in\BDC_{\RR-c}(\Bbbk_{U_\infty\times\RR_\infty})$ such that 
\begin{align*}
\bfE i_{U_\infty}^{-1}K\simeq
\Bbbk_{U_\infty}^{\rmE, \sub}\Potimes \Q_{U_\infty}^\sub\rho_{U_\infty\times\RR_\infty\ast}\SF,\hspace{9pt}
\bfE i_{U_\infty}^{-1}L\simeq
\Bbbk_{U_\infty}^{\rmE, \sub}\Potimes \Q_{U_\infty}^\sub\rho_{U_\infty\times\RR_\infty\ast}\SG.
\end{align*}
Hence ,we have
\begin{align*}
\bfE i_{U_\infty}^{-1}\(K\Potimes L\)
&\simeq
\bfE i_{U_\infty}^{-1}K\Potimes\bfE i_{U_\infty}^{-1}L\\
&\simeq
\(\Bbbk_{U_\infty}^{\rmE, \sub}\Potimes \Q_{U_\infty}^\sub\rho_{U_\infty\times\RR_\infty\ast}\SF\)
\Potimes
\(\Bbbk_{U_\infty}^{\rmE, \sub}\Potimes \Q_{U_\infty}^\sub\rho_{U_\infty\times\RR_\infty\ast}\SG\)\\
&\simeq
\Bbbk_{U_\infty}^{\rmE, \sub}\Potimes \Q_{U_\infty}^\sub\(\rho_{U_\infty\times\RR_\infty\ast}\SF
\Potimes\rho_{U_\infty\times\RR_\infty\ast}\SG\)
\end{align*}
where in the first isomorphism we used Proposition \ref{prop3.14} (2)
and in the last isomorphism we used
$\Bbbk_{U_\infty}^{\rmE, \sub}\Potimes \Bbbk_{U_\infty}^{\rmE, \sub}
\simeq\Bbbk_{U_\infty}^{\rmE, \sub}$.
Since $\SF, \SG\in\BDC_{\RR-c}(\Bbbk_{U_\infty\times\RR_\infty})$,
there exists an isomorphism
$\rho_{U_\infty\times\RR_\infty\ast}\SF
\Potimes\rho_{U_\infty\times\RR_\infty\ast}\SG
\simeq
\rho_{U_\infty\times\RR_\infty\ast}\(\mu_{!!}(p_1^{-1}\SF\otimes p_2^{-1}\SG)\)$
and $\mu_{!!}(p_1^{-1}\SF\otimes p_2^{-1}\SG)\in\BDC_{\RR-c}(\Bbbk_{U_\infty\times\RR_\infty})$.
Therefore, $K\Potimes L$ is $\RR$-constructible.

Moreover, by using the assertion (2)(i), (iii),
there exist isomorphisms 
\begin{align*}
\Prihomsub(K, L)
&\simeq
\Prihomsub\(\bfD_{M_\infty}^{\rmE, \sub}L, \bfD_{M_\infty}^{\rmE, \sub}K\)\\
&\simeq
\bfD_{M_\infty}^{\rmE, \sub}\(\bfD_{M_\infty}^{\rmE, \sub}L\Potimes K\)\\
&\simeq
\bfD_{M_\infty}^{\rmE, \sub}\(K\Potimes \bfD_{M_\infty}^{\rmE, \sub}L\),
\end{align*}
and hence by Proposition \ref{prop3.27} (2) and the first assertion of (1),
$\Prihomsub(K, L)$ is $\RR$-constructible.
\medskip

\noindent
(2)(i)
By Proposition \ref{prop3.14} (1)(i),
we have
\begin{align*}
\bfD_{M_\infty}^{\rmE, \sub}\(K\Potimes L\)
&\simeq
\Prihomsub\(K\Potimes L, \omega_{M_\infty}^{\rmE, \sub}\)\\
&\simeq
\Prihomsub\(K, \Prihomsub(L, \omega_{M_\infty}^{\rmE, \sub})\)\\
&\simeq
\Prihomsub\(K, \bfD_{M_\infty}^{\rmE, \sub}L\).
\end{align*}

\noindent
(ii)
Since $L$ is $\RR$-constructible,
we have an isomorphism $L\simeq \bfD_{M_\infty}^{\rmE, \sub}\bfD_{M_\infty}^{\rmE, \sub}L$
by Proposition \ref{prop3.27} (2).
Hence, by using the assertion (2)(i),
we have
\begin{align*}
\bfD_{M_\infty}^{\rmE, \sub}\Prihomsub(K, L)
&\simeq
\bfD_{M_\infty}^{\rmE, \sub}\Prihomsub(K,
\bfD_{M_\infty}^{\rmE, \sub}\bfD_{M_\infty}^{\rmE, \sub}L)\\
&\simeq
\bfD_{M_\infty}^{\rmE, \sub}\bfD_{M_\infty}^{\rmE, \sub}\(
K\Potimes \bfD_{M_\infty}^{\rmE, \sub}L\)\\
&\simeq
K\Potimes \bfD_{M_\infty}^{\rmE, \sub} L,
\end{align*}
where
in the last isomorphism we used Proposition \ref{prop3.27} (2) and the assertion (1).
\smallskip

\noindent
(iii)
By Proposition \ref{prop3.14} (1)(i),
we have
\begin{align*}
\bfD_{M_\infty}^{\rmE, \sub}\(K\Potimes\bfD_{M_\infty}^{\rmE, \sub}L\)
&\simeq
\Prihomsub\(K\Potimes\bfD_{M_\infty}^{\rmE, \sub}L, \omega_{M_\infty}^{\rmE, \sub}\)\\
&\simeq
\Prihomsub\(\bfD_{M_\infty}^{\rmE, \sub}L, \Prihomsub(K, \omega_{M_\infty}^{\rmE, \sub})\)\\
&\simeq
\Prihomsub\(\bfD_{M_\infty}^{\rmE, \sub}L, \bfD_{M_\infty}^{\rmE, \sub}K\).
\end{align*}
Since $\Prihomsub(K, L)\simeq
\bfD_{M_\infty}^{\rmE, \sub}\(K\Potimes\bfD_{M_\infty}^{\rmE, \sub}L\)$,
we have
\begin{align*}
\Prihomsub(K, L)\simeq
\Prihomsub\(\bfD_{M_\infty}^{\rmE,\sub}L, \bfD_{M_\infty}^{\rmE,\sub}K\).
\end{align*}

\noindent
(iv)
First, let us prove $\Bbbk_{M_\infty}^{\rmE, \sub}\Potimes K\simeq K$.
Since $\Bbbk_{M_\infty}^{\rmE, \sub}\simeq
\lambda_{M_\infty}^\rmE \Bbbk_{M_\infty}^{\rmE}$
and $K$ is $\RR$-constructible,
we have 
$$\Bbbk_{M_\infty}^{\rmE, \sub}\Potimes K
\simeq
\lambda_{M_\infty}^\rmE \Bbbk_{M_\infty}^{\rmE}\Potimes
\lambda_{M_\infty}^\rmE I_{M_\infty}^\rmE K
\simeq
\lambda_{M_\infty}^\rmE \(\Bbbk_{M_\infty}^{\rmE}\Potimes I_{M_\infty}^\rmE K\)
\simeq \lambda_{M_\infty}^\rmE I_{M_\infty}^\rmE K
\simeq K,$$
where in the first isomorphism we used Theorem \ref{main1} (1) and Lemma \ref{lem3.19},
in the second isomorphism Proposition \ref{prop3.18} (4)(iii),
in the last isomorphisms we used Theorem \ref{main1} 
and in the third isomorphism we used the fact that $\Bbbk_{M_\infty}^{\rmE}\Potimes K'\simeq K'$
for any $K'\in\BEC_{\RR-c}(\I\Bbbk_{M_\infty})$.
Hence, there exist isomorphisms 
\begin{align*}
\rihom^{\rmE, \sub}(K, L)
&\simeq
\rihom^{\rmE, \sub}(\Bbbk_{M_\infty}^{\rmE, \sub}\Potimes K, L)\\
&\simeq
\rihom^{\rmE, \sub}\(\Bbbk_{M_\infty}^{\rmE, \sub}, \Prihomsub(K, L)\)\\
&\simeq
\rihom^{\rmE, \sub}\(\Bbbk_{M_\infty}^{\rmE, \sub},
\Prihomsub(\bfD_{M_\infty}^{\rmE,\sub}L,\, \bfD_{M_\infty}^{\rmE,\sub}K)\)\\
&\simeq
\rihom^{\rmE, \sub}\(\Bbbk_{M_\infty}^{\rmE, \sub}\Potimes\bfD_{M_\infty}^{\rmE,\sub}L,\,
\bfD_{M_\infty}^{\rmE,\sub}K\)\\
&\simeq
\rihom^{\rmE, \sub}\(\bfD_{M_\infty}^{\rmE,\sub}L, \bfD_{M_\infty}^{\rmE,\sub}K\), 
\end{align*}
where in the second and fourth isomorphisms we used Proposition \ref{prop3.14} (1)(i),
in the third isomorphism we used the assertion (iii)
and in the last isomorphism we used the fact 
$\Bbbk_{M_\infty}^{\rmE, \sub}\Potimes\bfD_{M_\infty}^{\rmE,\sub}L
\simeq \bfD_{M_\infty}^{\rmE,\sub}L$.
\medskip

The assertion (v) follows from the assertion (iv).
\end{proof}

\subsection{Irregular Riemann--Hilbert Correspondence and Enhanced Subanalytic Sheaves}
In this subsection,
we will explain a relation between \cite[Thm.\:9.5.3]{DK16} and \cite[Thm.\:6.3]{Kas16}.
Theorems \ref{main3} and \ref{main4} are main results of this paper.

\subsubsection{Main Results of \cite{DK16} and \cite{Kas16}}\label{subsec3.5.1}
Let $X$ be a complex manifold and denote by $X_\RR$ the underlying real analytic manifold of $X$. 
We denote by $\SO_{X}$ and $\D_{X}$ the sheaves of holomorphic functions 
and holomorphic differential operators on $X$, respectively. 
Let $\BDC(\D_{X})$ be the bounded derived category of left $\D_{X}$-modules. 
Moreover we denote by $\BDCcoh(\D_{X})$, $\BDChol(\D_{X})$ and $\BDCrh(\D_{X})$
the full triangulated subcategories of $\BDC(\D_{X})$
consisting of objects with coherent, holonomic and regular holonomic cohomologies, respectively.
For a morphism $f \colon X\to Y$ of complex manifolds, 
denote by $\Dotimes$, $\bfD f_\ast$, $\bfD f^\ast$, 
$\DD_{X} \colon \BDCcoh(\SD_{X})^{\op} \simto \BDCcoh(\SD_{X})$  
the standard operations for $\SD$-modules. 
The classical solution functor is defined by  
$$
\Sol_X \colon \BDC (\SD_X)^{\op}\to\BDC(\CC_X),
\hspace{7pt}\M \longmapsto \rhom_{\D_X}(\M, \SO_X).
$$

Let $M$ be a real analytic manifold of dimension $n$.
We denote by $\SC_M^\infty$ the sheaf of complex functions of class $\SC^\infty$ on $M$
and by $\Db_M$ the sheaf of Schwartz's distributions on $M$.
\begin{definition}[{\cite[Def.\:7.2.3]{KS01}}]
Let $U$ be an open subset of $M$.
\begin{itemize}\item[(1)]
One say that 
$f\in\SC_M^{\infty}(U)$ has polynomial growth at $p\in M$
if for a local coordinate system $(x_1, \ldots, x_n)$ around $p$,
there exist a sufficiently small compact neighborhood $K$ of $p$
and a positive integer $N$ such that 
$$\sup_{x\in K\cap U}\dist(x, K\setminus U)^N|f(x)| < +\infty.$$

A function $f\in\SC_M^{\infty}(U)$ is said to be tempered at $p\in M$
if all its derivatives have polynomial growth at $p$,
and is said to be tempered if tempered at any point of $M$. 

Let us denote by $\SC_{M}^{\infty, \rmt}(U)$
the subset of $\SC_{M}^{\infty}(U)$ consisting of functions which are tempered.

\item[(2)]
Let us defined $\Db_M^\rmt(U)$ by the image of the restriction map
$$\Gamma(M; \Db_M)\to \Gamma(U; \Db_M).$$

\end{itemize}
\end{definition}

Note that subanalytic presheaves $U\mapsto\SC_{M}^{\infty, \rmt}(U)$ and 
$U\mapsto \Db_M^\rmt(U)$ are subanalytic sheaves,
see \cite[\S 7.2]{KS01}, also \cite[\S 3.3]{Pre08}.

We shall write $\BDC(\CC_X^\sub)$, $\BEC(\CC_X^\sub)$, $\BEC_{\RR-c}(\CC_X^\sub)$
instead of $\BDC(\CC_{X_\RR}^\sub)$, $\BEC(\CC_{X_\RR}^\sub)$, $\BEC_{\RR-c}(\CC_{X_\RR}^\sub)$, respectively.
\begin{definition}[{\cite[\S 7.3]{KS01}, \cite[\S 5.2]{DK16} and also \cite[\S 3.3]{Pre08}}]
Let us denote by $X^c$ the complex conjugate manifold of $X$.
An object $\SO_X^\rmt\in\BDC(\CC_X^\sub)$ is defined by  
$$\SO_X^\rmt :=
\rihom^\sub_{\rho_X!\D_{X^c}}(\rho_{X^{c}!}\SO_{X^c}, \SC_{X_\RR}^{\infty, \rmt})
\simeq
\rihom^\sub_{\rho_X!\D_{X^c}}(\rho_{X^{c}!}\SO_{X^c}, \Db_{X_\RR}^{\rmt})$$
and is called the subanalytic sheaf of tempered holomorphic functions on $X$.

Moreover, the tempered solution functor is defined by 
$$
\Sol_X^\rmt \colon\BDC(\D_X)^{\op}\to\BDC(\I\CC_X), 
\hspace{7pt}
\SM \longmapsto \rihom_{\beta_X\D_X}(\beta_X\M, I_X\SO_X^\rmt). 
$$
\end{definition}

Note that an ind-sheaf $I_{X}\SO_X^\rmt$ is nothing but the ind-sheaf of tempered holomorphic functions on $X$
which is denoted by $\SO_X^\rmt$ in \cite[\S 7.3]{KS01}.
Note also that there exist isomorphisms
$\rho_X^{-1}\SO_X^\rmt\simeq \alpha_XI_X\SO_X^\rmt \simeq \SO_X$
and hence we have
$\alpha_X\Sol_X^\rmt(\M)\simeq \Sol_X(\M)$ for any $\M\in\BDCcoh(\D_X)$.

\begin{definition}[{\cite[Def.\:8.1.1]{DK16} and \cite[Def.\:7.2.1]{KS16}}]\label{def3.32}
Let $\tl{k}_M\colon M\times \RR_\infty\to M\times \PP^1\RR$ be the natural morphism of real analytic bordered spaces
and $t$ be a coordinate of $\RR$,
where $\PP^1\RR$ is the real projective line.
An object $\Db_{M\times\RR_\infty}^\rmt\in\BDC(\I\CC_{M\times\RR_\infty})$ is defined by
$$\Db_{M\times\RR_\infty}^\rmt :=
 \tl{k}_M^{!}I_{M\times\PP^1\RR}\Db_{M\times\PP^1\RR}^\rmt
 \simeq
I_{M\times\RR_\infty} \tl{k}_M^{!}\Db_{M\times\PP^1\RR}^\rmt$$
and $\Db_{M}^{\T}\in\BDC(\I\CC_{M\times\RR_\infty})$ is defined by 
the complex, concentrated in $-1$ and $0$:
$$\Db_{M\times\RR_\infty}^\rmt\xrightarrow{\ \partial_t-1\ }\Db_{M\times\RR_\infty}^\rmt.$$

Moreover, we set $\Db_{M}^\rmE := \Q_M(\Db_M^\T)\in\BEC(\I\CC_{M})$
and call it the enhanced ind-sheaf of tempered distributions.
\end{definition}
Note that we have $\SH^i(\Db_{M}^\T) = 0$ for any $i\neq -1$
and hence there exists an isomorphism
$\Db_{M}^\T\simeq \Ker(\partial_t-1)[1]$ in $\BDC(\I\CC_{M\times\RR_\infty})$.

\begin{definition}[{\cite[Def.\:8.2.1]{DK16} and \cite[Def.\:7.2.3]{KS16}}]\label{def3.33}
Let $\tl{i}\colon X\times\RR_\infty\to X\times\PP^1\CC$ be the natural morphism of bordered spaces
and $\tau\in\CC\subset\PP^1\CC$ the affine coordinate 
such that $t = \tau|_\RR$,
where $\PP^1\CC$ is the complex projective line.
An object $\SO_X^\rmE\in\BEC(\I\CC_X)$ is defined by
\begin{align*}
\SO_X^\rmE &:= \Q_X\rihom_{\pi_{X^c}^{-1}\beta_{X^c}\D_{X^c}}(
\pi_{X^c}^{-1}\beta_{X^c}\SO_{X^c}, \Db_{X_\RR}^{\T})\\
&\,\simeq
\Q_X\tl{i}^!\rihom_{p^{-1}\beta_{\PP^1\CC}\D_{\PP^1\CC}}
(p^{-1}\beta_\PP\SE_{\CC|\PP^1\CC}^{\tau},
 I_{X\times\PP^1\CC}\SO_{X\times\PP^1\CC}^\rmt)[2],
 \end{align*}
 where $\SE_{\CC|\PP}^{\tau}$ is the meromorphic connection associated to $d+d\tau$,
$p\colon X\times\PP^1\CC\to \PP^1\CC$ is the projection
and $\pi_{X^c}\colon X^c \times\RR_\infty\to X^c$ is the morphisms of bordered spaces associated
with the projection $X^c\times\RR\to X^c$. 

It is called the enhanced ind-sheaf of tempered holomorphic functions on $X$.

Moreover, the enhanced solution functor is defined by 
$$
\Sol_X^\rmE \colon \BDC(\D_X)^\op\to\BEC(\I\CC_X), 
\hspace{7pt} 
\M \longmapsto \rihom_{\pi^{-1}\beta_X\D_X}(\pi^{-1}\beta_X\M, \SO_X^\rmE),
$$
where $\pi\colon X\times\RR_\infty\to X$ is the morphism of bordered spaces
associated with the first projection $X\times\RR\to X$.
\end{definition}

Note that $\SO_X^\rmE$ is isomorphic to the enhanced ind-sheaf induced by 
the Dolbeault complex with coefficients in $\Db_{X_\RR}^{\T}[-1]$:
$$\Db_{X_\RR}^{\T}[-1]\xrightarrow{\ \var{\partial}\ }\Omega_{X^c}^1\otimes_{\SO_{X^c}}\Db_{X_\RR}^{\T}[-1]
\xrightarrow{\ \var{\partial}\ }\cdots\xrightarrow{\ \var{\partial}\ }
\Omega_{X^c}^{d_X}\otimes_{\SO_{X^c}}\Db_{X_\RR}^{\T}[-1],$$
where $\Omega_{X^c}^p$ is the sheaf of $p$-differential forms with coefficients in $\SO_{X^c}$
and $d_X$ is the complex dimension of $X$. 

Note that $\I\sh_X\SO_X^\rmE \simeq I_X\SO_X^\rmt$
and hence there exists an isomorphism
$\I\sh_X\Sol_X^\rmE(\M)\simeq\Sol_X^\rmt(\M)$
for any $\M\in\BDCcoh(\D_X)$.

Let us recall the main results of \cite{DK16}.
\begin{theorem}[{\cite[Thm.\:9.5.3\footnote{
In \cite{DK16}, although Theorem 9.5.3 was stated by using the enhanced de Rham functor,
we can obtain a similar statement by using the enhanced solution functor.} and 9.6.1]{DK16}}]\label{thmDK}
The enhanced solution functor induces an embedding functor
\[ \Sol_{X}^\rmE \colon \BDChol(\D_{X})^{\op}
\hookrightarrow\BEC_{\RR-c}(\I\CC_{X}).\]
Moreover for any $\M\in\BDChol(\D_{X})$ there exists an isomorphism in $\BDC(\D_{X}):$
\[\M\simto \RH_{X}^{\rmE}\Sol_{X}^{\rmE}(\M),\]  
where $\RH_{X}^{\rmE}(K) := \rhom^{\rmE}(K, \SO_{X}^{\rmE})$.
\end{theorem}

Let us recall the main results of \cite{Kas16}.
\begin{definition}[{\cite[\S 5.2]{Kas16}}]\label{def3.35}
Let $\tl{k}_M\colon M\times \RR_\infty\to M\times \PP^1\RR$ be the natural morphism of real analytic bordered spaces.
An object $\Db_{M}^{\T, \sub}\in\BDC(\CC_{M\times\RR_\infty}^\sub)$ is defined by
the complex, concentrated in $-1$ and $0$:
$$\tl{k}^!\Db_{M\times\PP^1\RR}^\rmt\xrightarrow{\ \partial_t-1\ }\tl{k}^!\Db_{M\times\PP^1\RR}^\rmt.$$
\end{definition}
Note that we have $\SH^i(\Db_{M}^{\T, \sub}) = 0$ for any $i\neq -1$
and hence there exists an isomorphism
$\Db_{M}^{\T,\sub}\simeq \Ker(\partial_t-1)[1]$ in $\BDC(\CC_{M\times\RR_\infty}^\sub)$.
Remark that the notion $\Db^{\T}$ in \cite[\S 5.2]{Kas16} is equal to $\Db^{\T, \sub}[-1]$ in our notion.

\begin{definition}[{\cite[\S\S 5.3, 5.4]{Kas16}}]\label{def3.36}
An object $\SO_X^{\T, \sub}\in\BDC(\CC_{X\times\RR_\infty}^\sub)$ is defined by 
the Dolbeault complex with coefficients in $\Db_{X_\RR}^{\T,\sub}[-1]$:
$$\Db_{X_\RR}^{\T,\sub}[-1]\xrightarrow{\ \var{\partial}\ }\Omega_{X^c}^1\otimes_{\SO_{X^c}}\Db_{X_\RR}^{\T,\sub}[-1]
\xrightarrow{\ \var{\partial}\ }\cdots\xrightarrow{\ \var{\partial}\ }
\Omega_{X^c}^{d_X}\otimes_{\SO_{X^c}}\Db_{X_\RR}^{\T,\sub}[-1].$$

Moreover, we set
$$
\Sol_X^{\T,\sub} \colon \BDC(\D_X)^\op\to\BDC(\CC_{X\times\RR_\infty}^\sub), 
\hspace{7pt} 
\M \longmapsto \rihom^\sub_{\pi^{-1}\rho_{X!}\D_X}(\pi^{-1}\rho_{X!}\M, \SO_X^{\T,\sub}).
$$
\end{definition}

Note also that there exists an isomorphism in $\BDC(\Bbbk_{X\times\RR_\infty}^\sub)$:
$$\SO_X^{\T, \sub}\simeq
\rihom_{\pi_{X^c}^{-1}\rho_{X^c!}\D_{X^c}}^\sub\(
\pi_{X^c}^{-1}\rho_{X^c!}\SO_{X^c}, \Db_{X_\RR}^{\T, \sub}[-1]\).$$

\begin{theorem}[{\cite[Thms.\:6.2 and 6.3\footnote{
In \cite{Kas16}, although Theorem 6.3 was stated by using the enhanced de Rham functor,
we can obtain a similar statement by using the enhanced solution functor.}]{Kas16}}]
There exists an embedding functor
\[ \Sol_{X}^{\T,\sub} \colon \BDChol(\D_{X})^{\op}
\hookrightarrow\BDC(\CC_{X\times\RR_\infty}^\sub).\]
Moreover for any $\M\in\BDChol(\D_{X})$ there exists an isomorphism in $\BDC(\D_X):$
\[\M\simto \rhom^{\rmE, \sub}\(\Sol_{X}^{\T, \sub}(\M), \SO_X^{\T, \sub}\).\] 
\end{theorem}

\subsubsection{Relation between \cite[Thm.\:9.5.3]{DK16} and \cite[Thm.\:6.3]{Kas16}}
Let us explain a relation between \cite[Thm.\:9.5.3]{DK16} and \cite[Thm.\:6.3]{Kas16}.
\begin{definition}\label{def3.37}
Let us define
$$\SO_X^{\rmE, \sub} := 
\Q_X^\sub\(\SO_X^{\T, \sub}[1]\)\in\BEC(\CC_X^\sub)$$
and set
$$
\Sol_X^{\rmE,\sub} \colon \BDC(\D_X)^\op\to\BEC(\CC_{X}^\sub), 
\hspace{7pt} 
\M \longmapsto \rihom^\sub_{\pi^{-1}\rho_{X!}\D_X}(\pi^{-1}\rho_{X!}\M, \SO_X^{\rmE,\sub}).
$$
\end{definition}
By the definition, it is clear that 
$$
\SO_X^{\rmE, \sub}\simeq
\Q_X^\sub\rihom_{\pi_{X^c}^{-1}\rho_{X^c!}\D_{X^c}}^\sub\(
\pi_{X^c}^{-1}\rho_{X^c!}\SO_{X^c}, \Db_{X_\RR}^{\T, \sub}\),
$$
and for any $\M\in\BDC(\D_X)$ one has
$$\Sol_X^{\rmE, \sub}(\M)\simeq
\Q_X^\sub\(\Sol_X^{\T, \sub}(\M)\)[1].$$

\begin{lemma}\label{lem3.38}
\begin{itemize}
\item[\rm (1)]
There exists an isomorphism
$\Db_{M}^{\T}\simeq I_{M\times\RR_\infty}\Db_{M}^{\T, \sub}$
in $\BDC(\I\CC_{M\times\RR_\infty})$.

\item[\rm (2)]
There exists an isomorphism
$\SO_X^{\rmE,\sub}\simeq J_X^\rmE\SO_X^{\rmE}$
in $\BEC(\CC_X^\sub)$.

\item[\rm (3)]
For any $\M\in\BDC(\D_X)$,
there exists an isomorphism in $\BEC(\CC_X^\sub):$
$$\Sol_X^{\rmE, \sub}(\M)\simeq J_X^\rmE\Sol_X^{\rmE}(\M).$$

\item[\rm (4)]
For any $\M\in\BDC(\D_X)$,
there exists an isomorphism in $\BDC(\CC_{X\times\RR_\infty}^\sub):$
$$\Sol_X^{\T, \sub}(\M)[1]\simeq \bfR_{X}^{\rmE, \sub}\Sol_X^{\rmE, \sub}(\M).$$
\end{itemize}
\end{lemma}

\begin{proof}
(1)
Since the functor $I_{M\times\RR_\infty}$ is exact,
we have isomorphisms
\begin{align*}
\Db_{M}^{\T}
&\simeq
\Ker\(\partial_t-1\colon\Db_{M\times\RR_\infty}^\rmt\to\Db_{M\times\RR_\infty}^\rmt\)\\
&\simeq
\Ker\(\partial_t-1\colon I_{M\times\RR_\infty} \tl{k}_M^{!}\Db_{M\times\PP^1\RR}^\rmt\to
I_{M\times\RR_\infty} \tl{k}_M^{!}\Db_{M\times\PP^1\RR}^\rmt\)\\
&\simeq
 I_{M\times\RR_\infty}\Ker\(
 \partial_t-1\colon \tl{k}_M^{!}\Db_{M\times\PP^1\RR}^\rmt\to\tl{k}_M^{!}\Db_{M\times\PP^1\RR}^\rmt\)\\
&\simeq I_{M\times\RR_\infty}\Db_{M}^{\T, \sub},
\end{align*}
see Definition \ref{def3.32} for the details of $\Db_{M}^{\T}$, Definition \ref{def3.35} for the details of $\Db_{M}^{\T, \sub}$.
\medskip

\noindent
(2)
Since $\beta_{X^c}\simeq I_{X^c}\circ\rho_{X^c!}$ and 
$\pi_{X^c}^{-1}\circ I_{X^c}\simeq I_{X^c\times\RR_\infty}\circ\pi_{X^c}^{-1}$
(see \S 2.5 for the first isomorphism, Proposition \ref{prop3.7} (2)(iii) for the second isomorphism),
we have isomorphisms
\begin{align*}
J_X^\rmE\SO_X^\rmE
&\simeq
\Q_X\bfR J_{X\times\RR_\infty}\rihom_{\pi_{X^c}^{-1}\beta_{X^c!}\D_{X^c}}\(
\pi_{X^c}^{-1}\beta_{X^c}\SO_{X^c}, \Db_{X_\RR}^{\T}\)\\
&\simeq
\Q_X\bfR J_{X\times\RR_\infty}\rihom_{\pi_{X^c}^{-1}\beta_{X^c!}\D_{X^c}}\(
\pi_{X^c}^{-1}I_{X^c}\rho_{X^c!}\SO_{X^c}, \Db_{X_\RR}^{\T}\)\\
&\simeq
\Q_X\bfR J_{X\times\RR_\infty}\rihom_{\pi_{X^c}^{-1}\beta_{X^c!}\D_{X^c}}\(
I_{X^c\times\RR_\infty}\pi_{X^c}^{-1}\rho_{X^c!}\SO_{X^c}, \Db_{X_\RR}^{\T}\).
\end{align*}
Moreover, there exist isomorphisms
\begin{align*}
&\Q_X\bfR J_{X\times\RR_\infty}\rihom_{\pi_{X^c}^{-1}\beta_{X^c!}\D_{X^c}}\(
I_{X^c\times\RR_\infty}\pi_{X^c}^{-1}\rho_{X^c!}\SO_{X^c}, \Db_{X_\RR}^{\T}\)\\
\simeq\
&\Q_X\rihom^\sub_{\pi_{X^c}^{-1}\rho_{X^c!}\D_{X^c}}\(
\pi_{X^c}^{-1}\rho_{X^c!}\SO_{X^c}, \bfR J_{X_\RR\times\RR_\infty}\Db_{X_\RR}^{\T}\)\\
\simeq\
&\Q_X\rihom^\sub_{\pi_{X^c}^{-1}\rho_{X^c!}\D_{X^c}}\(
\pi_{X^c}^{-1}\rho_{X^c!}\SO_{X^c}, \Db_{X_\RR}^{\T,\sub}\)\\
\simeq\
&\SO_X^{\rmE,\sub}
\end{align*}
where
in the second isomorphism we used the assertion (1) and Proposition \ref{prop3.6} (1).
\medskip

\noindent
(3)
Let $\M\in\BDC(\D_X)$.
By the fact that $\beta_{X}\simeq I_{X}\circ\rho_{X!}$ and the assertion (2),
there exist isomorphisms
\begin{align*}
J_X^\rmE\Sol_X^\rmE(\M)
&\simeq
J_X^\rmE\rihom_{\pi^{-1}\beta_X\D_X}(\pi^{-1}\beta_X\M, \SO_X^\rmE)\\
&\simeq
J_X^\rmE\rihom_{\pi^{-1}\beta_X\D_X}(\pi^{-1} I_{X}\rho_{X!}\M, \SO_X^\rmE)\\
&\simeq
J_X^\rmE\rihom_{\pi^{-1}\beta_X\D_X}(I_{X\times\RR_\infty}\pi^{-1}\rho_{X!}\M, \SO_X^\rmE)\\
&\simeq
\rihom_{\pi^{-1}\rho_{X!}\D_X}^\sub(\pi^{-1}\rho_{X!}\M, J_X^\rmE\SO_X^\rmE)\\
&\simeq
\rihom_{\pi^{-1}\rho_{X!}\D_X}^\sub(\pi^{-1}\rho_{X!}\M, \SO_X^{\rmE, \sub})\\
&\simeq
\Sol_X^{\rmE, \sub}(\M).
\end{align*}

\noindent
(4)
First let us prove that 
$$\Prihomsub\(\rho_{X\times\RR_\infty\ast}\(
\Bbbk_{\{t\geq0\}}\oplus\Bbbk_{\{t\leq 0\}}\), \Db_{X_\RR}^{\T, \sub}\)
\simeq\Db_{X_\RR}^{\T, \sub}.$$
By the assertion (1) and Proposition \ref{prop3.18} (1)(ii),
\begin{align*}
&\Prihomsub\(\rho_{X\times\RR_\infty\ast}\(
\Bbbk_{\{t\geq0\}}\oplus\Bbbk_{\{t\leq 0\}}\), \Db_{X_\RR}^{\T, \sub}\)\\
\simeq\
&\Prihomsub\(\rho_{X\times\RR_\infty\ast}\(
\Bbbk_{\{t\geq0\}}\oplus\Bbbk_{\{t\leq 0\}}\), \bfR J_{X\times\RR_\infty}\Db_{X_\RR}^{\T}\)\\
\simeq\
&\bfR J_{X\times\RR_\infty}\Prihom\(I_{X\times\RR_\infty}\rho_{X\times\RR_\infty\ast}\(
\Bbbk_{\{t\geq0\}}\oplus\Bbbk_{\{t\leq 0\}}\), \Db_{X_\RR}^{\T}\)\\
\simeq\
&\bfR J_{X\times\RR_\infty}\Prihom\(\iota_{X\times\RR_\infty}\(
\Bbbk_{\{t\geq0\}}\oplus\Bbbk_{\{t\leq 0\}}\), \Db_{X_\RR}^{\T}\).
\end{align*}
Moreover, by using the fact that
$\Prihom\(\iota_{X\times\RR_\infty}\(
\Bbbk_{\{t\geq0\}}\oplus\Bbbk_{\{t\leq 0\}}\), \Db_{X_\RR}^{\T}\)
\simeq\Db_{X_\RR}^{\T}$,
see e.g. \cite[Prop.\:8.1.3]{DK16},
we have
$$\bfR J_{X\times\RR_\infty}\Prihom\(\iota_{X\times\RR_\infty}\(
\Bbbk_{\{t\geq0\}}\oplus\Bbbk_{\{t\leq 0\}}\), \Db_{X_\RR}^{\T}\)
\simeq \bfR J_{X\times\RR_\infty}\Db_{X_\RR}^{\T}
\simeq\Db_{X_\RR}^{\T,\sub}.$$
Hence, we proved $\Prihomsub\(\rho_{X\times\RR_\infty\ast}\(
\Bbbk_{\{t\geq0\}}\oplus\Bbbk_{\{t\leq 0\}}\), \Db_{X_\RR}^{\T, \sub}\)
\simeq\Db_{X_\RR}^{\T, \sub}.$

Next, we shall prove that 
$$\Prihomsub(\rho_{X\times\RR_\infty\ast}(
\Bbbk_{\{t\geq0\}}\oplus\Bbbk_{\{t\leq 0\}},\SO_X^{\T,\sub}))
\simeq\SO_X^{\T,\sub}.$$
By the fact that $\SO_X^{\T, \sub}\simeq
\rihom_{\pi_{X^c}^{-1}\rho_{X^c!}\D_{X^c}}^\sub\(
\pi_{X^c}^{-1}\rho_{X^c!}\SO_{X^c}, \Db_{X_\RR}^{\T, \sub}[-1]\),$
we have 
\begin{align*}
&\Prihomsub\(\rho_{X\times\RR_\infty\ast}\(
\Bbbk_{\{t\geq0\}}\oplus\Bbbk_{\{t\leq 0\}},\SO_X^{\T,\sub}\)\)\\
\simeq\
&\Prihomsub\(\rho_{X\times\RR_\infty\ast}\(
\Bbbk_{\{t\geq0\}}\oplus\Bbbk_{\{t\leq 0\}}\),
\rihom_{\pi_{X^c}^{-1}\rho_{X^c!}\D_{X^c}}^\sub\(
\pi_{X^c}^{-1}\rho_{X^c!}\SO_{X^c}, \Db_{X_\RR}^{\T, \sub}\)\)[-1]\\
\simeq\
&\rihom_{\pi_{X^c}^{-1}\rho_{X^c!}\D_{X^c}}^\sub\(\pi_{X^c}^{-1}\rho_{X^c!}\SO_{X^c},
\Prihomsub\(\rho_{X\times\RR_\infty\ast}\(
\Bbbk_{\{t\geq0\}}\oplus\Bbbk_{\{t\leq 0\}}\), \Db_{X_\RR}^{\T, \sub}\)\)[-1]\\
\simeq\
&\rihom_{\pi_{X^c}^{-1}\rho_{X^c!}\D_{X^c}}^\sub
\(\pi_{X^c}^{-1}\rho_{X^c!}\SO_{X^c}, \Db_{X_\RR}^{\T, \sub}\)[-1]\\
\simeq\
&\SO_X^{\T, \sub}.
\end{align*}
By the definition, it is clear that 
\begin{align*}
&\bfR_{X}^{\rmE, \sub}\Sol_X^{\rmE, \sub}(\M)\\
\simeq\
&\Prihomsub\(\rho_{X\times\RR_\infty\ast}\(
\Bbbk_{\{t\geq0\}}\oplus\Bbbk_{\{t\leq 0\}}\),
\rihom^\sub_{\pi^{-1}\rho_{X!}\D_X}(\pi^{-1}\rho_{X!}\M, \SO_X^{\T,\sub}\)[1]\\
\simeq\
&\rihom^\sub_{\pi^{-1}\rho_{X!}\D_X}(\pi^{-1}\rho_{X!}\M,
\Prihomsub\(\rho_{X\times\RR_\infty\ast}\(
\Bbbk_{\{t\geq0\}}\oplus\Bbbk_{\{t\leq 0\}}\),\SO_X^{\T,\sub}\)[1]\\
\simeq\
&\rihom^\sub_{\pi^{-1}\rho_{X!}\D_X}\(\pi^{-1}\rho_{X!}\M, \SO_X^{\T,\sub}\)[1]\\
\simeq\
&\Sol_X^{\T, \sub}(\M)[1].
\end{align*}

\end{proof}

\begin{theorem}\label{main3}
The functor $\Sol_X^{\rmE,\sub}$ induces an embedding functor
\[ \Sol_{X}^{\rmE, \sub} \colon \BDChol(\D_{X})^{\op}
\hookrightarrow\BEC_{\RR-c}(\CC_{X}^\sub).\]
Moreover for any $\M\in\BDChol(\D_{X})$ there exists an isomorphism in $\BDC(\D_X):$
\[\M\simto \RH_{X}^{\rmE,\sub}\Sol_{X}^{\rmE,\sub}(\M),\]  
where $\RH_{X}^{\rmE, \sub}(K) := \rhom^{\rmE,\sub}(K, \SO_{X}^{\rmE,\sub})$.
\end{theorem}

\begin{proof}
First, let us prove that $\Sol_X^{\rmE,\sub}(\M)\in\BEC_{\RR-c}(\CC_X^\sub)$ for any $\M\in\BDChol(\D_X)$.
Let $\M\in\BDChol(\D_X)$.
By Theorem \ref{thmDK}, we have 
$\Sol_X^{\rmE}(\M)\in\BEC_{\RR-c}(\I\CC_X)$
and hence by Theorem \ref{main2} we have $J_X^\rmE\Sol_X^{\rmE}(\M)\in\BEC_{\RR-c}(\CC_X^\sub)$.
This implies that $\Sol_X^{\rmE,\sub}(\M)\in\BEC_{\RR-c}(\CC_X^\sub)$ by Lemma \ref{lem3.38} (3).
Hence, a functor
$\Sol_{X}^{\rmE, \sub} \colon \BDChol(\D_{X})^{\op}\to\BEC_{\RR-c}(\CC_{X}^\sub)$
is well defined. 
\medskip

For any $\M, \N\in\BDChol(\D_X)$,
there exist isomorphisms
\begin{align*}
\Hom_{\BDChol(\D_X)}(\M, \N)
&\simeq
\Hom_{\BEC_{\RR-c}(\I\CC_X)}\(\Sol_X^{\rmE}(\N), \Sol_X^{\rmE}(\M)\)\\
&\simeq
\Hom_{\BEC_{\RR-c}(\CC_X^\sub)}\(J_X^\rmE\Sol_X^{\rmE}(\N), J_X^\rmE\Sol_X^{\rmE}(\M)\)\\
&\simeq
\Hom_{\BEC_{\RR-c}(\CC_X^\sub)}\(\Sol_X^{\rmE,\sub}(\N), \Sol_X^{\rmE,\sub}(\M)\),
\end{align*}
where in the first (resp.\,second, third) isomorphism 
we used Theorem \ref{thmDK} (resp.\,Theorem \ref{main2}, Lemma \ref{lem3.38} (3)).
This implies that the functor 
$\Sol_{X}^{\rmE, \sub} \colon \BDChol(\D_{X})^{\op}\to\BEC_{\RR-c}(\CC_{X}^\sub)$
is fully faithful.
\medskip

Let $\M\in\BDChol(\D_X)$.
By using the adjointness, there exist isomorphisms
\begin{align*}
&\Hom_{\BDC(\D_X)}\(\M, \RH_{X}^{\rmE,\sub}\Sol_{X}^{\rmE,\sub}(\M)\)\\
\simeq\
&\Hom_{\BDC(\D_X)}\(
\M, \rhom^{\rmE,\sub}(\Sol_{X}^{\rmE,\sub}(\M), \SO_{X}^{\rmE,\sub})\)\\
\simeq\
&\Hom_{\BDC(\D_X)}\(
\M, \rho_X^{-1}\bfR \pi_\ast\rihom^{\sub}(\Sol_{X}^{\rmE,\sub}(\M), \SO_{X}^{\rmE,\sub})\)\\
\simeq\
&\Hom_{\BDC(\D_X)}\(
\pi^{-1}\rho_{X!}\M, \rihom^{\sub}(\Sol_{X}^{\rmE,\sub}(\M), \SO_{X}^{\rmE,\sub})\)\\
\simeq\
&\Hom_{\BEC_{\RR-c}(\CC_X^\sub)}\(\Sol_{X}^{\rmE,\sub}(\M),
\rihom_{\pi^{-1}\rho_{X!}\D_X}^{\sub}(\pi^{-1}\rho_{X!}\M, \SO_{X}^{\rmE,\sub})\)\\
\simeq\
&\Hom_{\BEC_{\RR-c}(\CC_X^\sub)}\(\Sol_X^{\rmE,\sub}(\M), \Sol_X^{\rmE,\sub}(\M)\)\\
\simeq\
&\Hom_{\BDC(\D_X)}\(\M, \M\).
\end{align*} 
Hence, there exists a canonical morphism 
$$\M\to \RH_{X}^{\rmE,\sub}\Sol_{X}^{\rmE,\sub}(\M)$$
which is given by the identity map $\id_{\M}$ of $\M$.
Let us prove that it is isomorphism.

By Lemma \ref{lem3.38} (2),
we have isomorphisms 
$$\RH_{X}^{\rmE,\sub}\Sol_{X}^{\rmE,\sub}(\M)
\simeq
\rhom^{\rmE,\sub}(\Sol_{X}^{\rmE,\sub}(\M), \SO_{X}^{\rmE,\sub})\\
\simeq
\rhom^{\rmE,\sub}(\Sol_{X}^{\rmE,\sub}(\M), J_X^\rmE\SO_{X}^{\rmE})$$
and by Proposition \ref{prop3.18} (1)(iii) we have 
\begin{align*}
\rhom^{\rmE,\sub}(\Sol_{X}^{\rmE,\sub}(\M), J_X^\rmE\SO_{X}^{\rmE})
&\simeq
\rho^{-1}\rihom^{\rmE,\sub}(\Sol_{X}^{\rmE,\sub}(\M), J_X^\rmE\SO_{X}^{\rmE})\\
&\simeq
\rho^{-1}\bfR J_X\rihom^{\rmE}(I_X^\rmE\Sol_{X}^{\rmE,\sub}(\M), \SO_{X}^{\rmE}).
\end{align*}
By Proposition \ref{prop3.7} (3)(ii) and Lemma \ref{lem3.38} (3),
there exists an isomorphism in $\BDC(\CC_{X})$
\begin{align*}
\rho^{-1}\bfR J_X\rihom^{\rmE}(I_X^\rmE\Sol_{X}^{\rmE,\sub}(\M), \SO_{X}^{\rmE})
&\simeq
\alpha_X\rihom^{\rmE}(I_X^\rmE\Sol_{X}^{\rmE,\sub}(\M), \SO_{X}^{\rmE})\\
&\simeq
\rhom^{\rmE}(I_X^\rmE\Sol_{X}^{\rmE,\sub}(\M), \SO_{X}^{\rmE})\\
&\simeq
\rhom^{\rmE}(I_X^\rmE J_X^\rmE\Sol_{X}^{\rmE}(\M), \SO_{X}^{\rmE}).
\end{align*}
Since $\M\in\BDChol(\D_X)$,
$\Sol_{X}^{\rmE}(\M)$ is $\RR$-constructible by the first assertion
and hence there exists an isomorphism 
$I_X^\rmE J_X^\rmE\Sol_{X}^{\rmE}(\M) \simeq\Sol_{X}^{\rmE}(\M) $
by Theorem \ref{main2}.
By Theorem \ref{thmDK},
we have $$\rhom^{\rmE}(\Sol_{X}^{\rmE}(\M), \SO_{X}^{\rmE})
\simeq\RH_{X}^{\rmE}\Sol_{X}^{\rmE}(\M)\simeq\M.$$
Therefore, there exists an isomorphism  
$\M\simto \RH_{X}^{\rmE,\sub}\Sol_{X}^{\rmE,\sub}(\M).$

\end{proof}

\begin{theorem}\label{main4}
\begin{itemize}
\item[\rm(1)]
For any $\M\in\BDChol(\D_X)$, there exists an isomorphism in $\BEC(\I\CC_{X}):$
$$\Sol_X^{\rmE}(\M)\simeq I_X^\rmE \Sol_X^{\rmE, \sub}(\M).$$

\item[\rm(2)]
For any $\M\in\BDC(\D_X)$, there exists an isomorphism in $\BDC(\CC_{X\times\RR_\infty}^\sub):$
$$\Sol_X^{\T, \sub}(\M)[1]\simeq\bfR_X^{\rmE, \sub} J_X^{\rmE}\Sol_X^{\rmE}(\M).$$

Moreover, there exists a commutative diagram:
\[\xymatrix@M=10pt@R=20pt@C=55pt{
\BDChol(\D_X)^\op\ar@{^{(}->}[r]^-{\Sol_X^{\T,\sub}(\cdot)[1]}\ar@{^{(}->}[rd]_-{\Sol_X^\rmE}
& \BDC(\CC_{X\times\RR_\infty}^\sub)\\
{}&\BEC_{\RR-c}(\I\CC_X)\ar@{^{(}->}[u]_-{\bfR_X^{\rmE,\sub}\circ J_X^\rmE}.
}\]
\end{itemize}
\end{theorem}
\begin{proof}
(1)
Let $\M\in\BDChol(\D_X)$.
Since $\Sol_X^\rmE(\M)$ is $\RR$-constructible,
there exists an isomorphism 
$I_X^\rmE \Sol_X^{\rmE, \sub}(\M)\simeq I_X^\rmE J_X^\rmE\Sol_X^{\rmE}(\M)\simeq\Sol_X^{\rmE}(\M)$
by Theorem \ref{main2} and Lemma \ref{lem3.38} (3).
\medskip

\noindent
(2)
This follows from Lemma \ref{lem3.38} (3), (4).
\end{proof}

\newpage
Let us summarize results of Theorems \ref{main1}, \ref{main2}, \ref{main3} and \ref{main4}
 in the following commutative diagram:
\[\xymatrix@M=7pt@R=35pt@C=60pt{
{}&{}&\BDC(\CC_{X\times\RR_\infty}^\sub) & {}\\
\BDChol(\D_X)^\op\ar@{^{(}->}[r]_-{\Sol_X^{\rmE, \sub}}
\ar@{^{(}->}[rru]^-{\Sol_X^{\T, \sub}(\cdot)[1]}\ar@{^{(}->}[rd]_-{\Sol_X^{\rmE}}
 & \BEC_{\RR-c}(\CC_X^\sub)\ar@{}[r]|-{\text{\large $\subset$}}
 \ar@<-1.0ex>@{->}[d]_-{I_X^\rmE}\ar@{}[d]|-\wr
  & \BEC(\CC_X^\sub)\ar@{^{(}->}[u]_-{\bfR_X^{\rmE, \sub}}
  \ar@<-1.0ex>@{^{(}->}[rd]_-{I_X^\rmE}
   \ar@<-1.0ex>@{->}[d]_-{I_X^\rmE}\ar@{}[d]|-\wr\\
{}&\BEC_{\RR-c}(\I\CC_X)\ar@{}[r]|-{\text{\large $\subset$}}
\ar@<-1.0ex>@{->}[u]_-{J_X^\rmE}
&\BEC_{\I\RR-c}(\I\CC_X)\ar@{}[r]|-{\text{\large $\subset$}}
\ar@<-1.0ex>@{->}[u]_-{J_X^\rmE}
& \BEC(\I\CC_X).\ar@<-1.0ex>@{->}[lu]_-{J_X^\rmE}
}\]

\begin{remark}
One can consider $\CC$-constructability for enhanced subanalytic sheaves as similar to \cite[Def.\:3.19]{Ito20}.
We shall skip the explanation of it.
\end{remark}

\end{document}